\documentclass[11pt]{article}

\usepackage{amsmath,amsthm,amssymb,amsfonts, graphicx}
\usepackage{graphics}
\usepackage{authblk}
\usepackage{upgreek}
\usepackage{amsrefs,hyperref}
\usepackage{cleveref}
\usepackage{dsfont}
\theoremstyle{plain}
\newtheorem{theorem}{Theorem}[section]
\newtheorem{corollary}[theorem]{Corollary}

\newtheorem{lemma}[theorem]{Lemma}
\newtheorem{proposition}[theorem]{Proposition}

\newtheorem{question}[theorem]{Open Question}
\theoremstyle{definition}
\newtheorem{definition}[theorem]{Definition}
\theoremstyle{remark}
\newtheorem{remark}[theorem]{Remark}
\numberwithin{equation}{section}
\newcommand{\diff}{\mathop{}\!\mathrm{d}}

\DeclareMathOperator{\tr}{tr}

\DeclareMathOperator{\spn}{span}
\DeclareMathOperator{\supp}{supp}

\title{Finite-time blowup for the Fourier-restricted Euler and hypodissipative Navier--Stokes model equations}
\author[1]{Evan Miller}
\affil[1]{University of Maine, Department of Mathematics and Statistics

evan.miller1@maine.edu}

\begin{document}

\maketitle

\begin{abstract}
In this paper, we introduce the Fourier-restricted Euler and hypodissipative Navier--Stokes equations. These equations are analogous to the Euler and hypodissipative Navier--Stokes equations respectively, but with the Helmholtz projection replaced by a projection onto a more restrictive constraint space; the $(u\cdot\nabla)u$ nonlinearity is otherwise unchanged. The constraint space restricts the divergence-free velocity to specific Fourier modes, which have a dyadic shell structure, and are constructed iteratively using permutations.

In the inviscid case---and in the hypo-viscous case when $\alpha<\frac{\log(3)}{6\log(2)} \approx .264$---we prove finite-time blowup for a set of solutions with a discrete group of symmetries. Our blowup Ansatz is odd, permutation symmetric, and mirror symmetric about the plane $x_1+x_2+x_3=0$. The Fourier-restricted Euler and hypodissipative Navier--Stokes equations respect both the energy equality and the identity for enstrophy growth from the full Euler and hypodissipative Navier--Stokes equations respectively, which is a substantial advance over the previous literature on Euler and Navier--Stokes model equations.
\end{abstract}

\tableofcontents

\section*{Acknowledgements}

This material is based on work supported by the Pacific Institute for the Mathematical Sciences, while the author was a PIMS postdoctoral fellow at the University of British Columbia. 
This postdoctoral fellowship was supported by the Natural Sciences and Engineering Research Council of Canada (NSERC) under grant no. 568577-2022.
This material is also based on work supported by the National Science Foundation (USA) under grant no. DMS-1928930 while the author was in residence at the Simons Laufer Mathematical Sciences Institute (formerly MSRI) during the Summer of 2023.
This material is also based upon work supported by the Swedish Research Council under grant no. 2021-06594 while the author was in residence at Institut Mittag-Leffler in Djursholm, Sweden during the Autumn 2023 semester.

\section{Introduction} \label{IntroSection}

The incompressible Euler equation is among the oldest PDEs, and yet much about its solutions, including whether smooth solutions in three dimensions can form singularities in finite-time, remains unknown. The incompressible Euler equation with no external force is given by
\begin{align}
    \partial_t u +(u\cdot\nabla) u+\nabla p
    &=0 \\
    \nabla\cdot u&= 0,
\end{align}
where $u\in\mathbb{R}^3$ is the velocity and $p$ is the pressure.
The pressure does not evolve independently; rather it is determined entirely by the velocity by making use of the divergence free constraint. This allows the Euler equation to be expressed in terms of the Helmholtz projection onto the space divergence free vector fields:
\begin{equation}
    \partial_t u+
    \mathbb{P}_{df}((u\cdot\nabla)u)
    =0.
\end{equation}

The Euler equation governs the motion of an inviscid fluid, in which there is no internal friction.
When the viscous effects of the internal friction of the fluid are considered, 
it is necessary to add a dissipative term, giving the Navier--Stokes equation:
\begin{equation}
    \partial_t u-\nu\Delta u 
    +\mathbb{P}_{df}((u\cdot\nabla)u)=0,
\end{equation}
where $\nu>0$ is the kinematic viscosity.

The hypodissipative Navier--Stokes equation interpolates between the Euler equation and the Navier--Stokes equation.
This equation has a fractional Laplacian dissipative term and is given by
\begin{equation}
    \partial_t u+\nu (-\Delta)^\alpha u
    +\mathbb{P}_{df}((u\cdot\nabla)u)
    =0,
\end{equation}
where $0<\alpha<1$. When $\alpha>1$, we refer to this as the hyperdissipative Navier--Stokes equation. This family of equations has an energy equality for strong solutions with
\begin{equation}
    \frac{1}{2}\|u(\cdot,t)\|_{L^2}^2
    +\nu\int_0^t 
    \|u(\cdot,\tau)\|_{\dot{H}^\alpha}^2 \diff\tau
    =\frac{1}{2}\left\|u^0\right\|_{L^2}^2,
\end{equation}
and for the Euler equation we have the energy equality
\begin{equation}
    \frac{1}{2}\|u(\cdot,t)\|_{L^2}^2
    =\frac{1}{2}\left\|u^0\right\|_{L^2}^2.
\end{equation}
The energy equality for the hyperdissipative Navier--Stokes equation is scale critical when $\alpha=\frac{5}{4}$, and there is global regularity for all solutions of the hyperdissipative Navier--Stokes equation when $\alpha\geq \frac{5}{4}$, and in fact logarithmically into the supercritical range \cites{TaoLog,BMRlog}. Furthermore, Colombo and Haffter recently proved that for any fixed initial data, $u^0\in H^\delta, \delta>0$, the initial data $u^0$ gives rise to a global smooth solution for all $\alpha>\frac{5}{4}-\epsilon$, where $\epsilon>0$ depends only on $\delta$ and $\left\|u^0\right\|_{H^\delta}$, the size of the initial data \cite{ColomboHaffter}.

In this paper, we will introduce a model equation for the Euler and hypo-dissipative Navier--Stokes equations that exhibits finite-time blowup. This equation will have the same nonlinearity involving the self-advection of the velocity, $(u\cdot\nabla)u$, but the Helmholtz projection will be replaced by a projection onto a more restrictive constraint space $\dot{H}^s_\mathcal{M}\subset \dot{H}^s_{df}$.

The Fourier-restricted Euler equation will be given by
\begin{equation}
    \partial_t u+
    \mathbb{P}_{\mathcal{M}}
    ((u\cdot\nabla)u)
    =0,
\end{equation}
and the Fourier-restricted, hypodissipative Navier--Stokes equation will be given by
\begin{equation}
    \partial_t u+\nu (-\Delta)^\alpha u
    +\mathbb{P}_{\mathcal{M}}
    ((u\cdot\nabla)u)
    =0.
\end{equation}
Note that the divergence free constraint still holds, because our new constraint space $\dot{H}^s_{\mathcal{M}}$ is a subspace of the space of divergence free vector fields.
In particular, this means these Fourier-restricted Euler and hypodissipative Navier--Stokes equations can be expressed in divergence form as
\begin{align}
    \partial_t u+ \mathbb{P}_\mathcal{M}
    \nabla\cdot(u\otimes u) &=0 \\
    \partial_t u +\nu(-\Delta)^{\alpha}
    + \mathbb{P}_\mathcal{M}
    \nabla\cdot(u\otimes u) &=0,
\end{align}
respectively.

Moreover, the $(u\cdot\nabla)u$ nonlinearity has not been altered; only the Helmholtz projection has been changed. This is a substantial improvement on existing model equations for Navier--Stokes, which required the nonlinearity to be changed in a more fundamental way \cites{TaoModifiedNS,TaoModifiedEuler,KatzPavlovic,FriedlanderPavlovic,MontgomerySmith,GallagherPaicuToy,MillerStrainModel,ConstantinLaxMajda,DeGregorio,Constantin}. In particular, both the structure of energy conservation/dissipation and enstrophy growth are preserved by this model equation in a way that was not the case for previous models.

In order to give the model equation explicitly, we need to define the constraint space $\dot{H}^s_{\mathcal{M}}$, which we will do in the next section. Our Ansatz for blowup will be odd, permutation symmetric, and mirror symmetric about the plane $x_1+x_2+x_3=0$. As a result of these symmetries, it will have a geometric structure that involves planar stretching and axial compression at the origin, with $\spn\{(1,1,1)\}$ acting as the axis of compression. 
Using this Ansatz, we will prove finite-time blowup for the Fourier-restricted Euler equation and the Fourier-restricted, hypodissipative Navier--Stokes equation when $\alpha<\frac{\log(3)}{6\log(2)} \approx
.264$.

\subsection{Definition of the constraint space} \label{bunchofdefs}

First, we must define a number of important vectors. Let
\begin{equation}
    \sigma =
    \left(\begin{array}{c}
         1  \\
         1 \\
         1
    \end{array}\right).
\end{equation}
For all $m \in\mathbb{Z}^+$, we will define the frequencies $k^m, h^m$, and $j^m$ by
\begin{align}
    k^m
    &=
    2^{2m}\sigma +3^m\left(\begin{array}{c}
         1 \\
         0 \\
         -1
    \end{array}\right) \\
    h^m
    &=
    2^{2m+1}\sigma +3^m\left(\begin{array}{c}
         1 \\
         1 \\
         -2
    \end{array}\right) \\
    j^m
    &=
    2^{2m+1}\sigma +3^m\left(\begin{array}{c}
         2 \\
         -1 \\
         -1
    \end{array}\right).
    \end{align}

A permutation $P\in\mathcal{P}_3$ is, of course, a bijection 
$P:\{1,2,3\} \to \{1,2,3\}$.
We will define the permutation of a vector $v\in\mathbb{R}^3$ to be the permutation of its entries
\begin{equation}
    P(v)_i=v_{P(i)},
\end{equation}
and we will take $\mathcal{P}[v]$ to be the set of all permutations of this vector. 
We will take the set of frequencies for our model equation to be the permutations of $\pm k^m,\pm h^m, \pm j^m$, using these vectors to define a series of frequency shells.
For all $n\in\mathbb{Z}^+$ even, $n=2m$, let
\begin{equation}
\mathcal{M}^+_n=
\mathcal{P}[k^m],
\end{equation}
and for all $n\in\mathbb{Z}^+$ odd, $n=2m+1$, let
\begin{equation}
    \mathcal{M}^+_n=
    \mathcal{P}[h^m]\cup \mathcal{P}[j^m].
\end{equation}
We use these shells to define all of our positive frequencies,
\begin{align}
    \mathcal{M}^+ 
    &=\bigcup_{m=0}^\infty
    \mathcal{P}\left[k^m\right]
    \cup 
    \mathcal{P}\left[h^m\right]
    \cup 
    \mathcal{P}\left[j^m\right] \\
    &=\bigcup_{n=0}^\infty
    \mathcal{M}^+_n.
    \end{align}
    We define our negative frequencies by reflection
    \begin{align}
    \mathcal{M}^- &=-\mathcal{M}^+ \\
    \mathcal{M}_n^-&=-\mathcal{M}_n^+,
    \end{align}
    and putting it all together, we let
    \begin{equation}
    \mathcal{M}=
    \mathcal{M}^+ \cup \mathcal{M}^-.
    \end{equation}

To illustrate, the first several shells are given by:
\begin{align}
    \mathcal{M}^+_0
    &=
    \left\{
    \left(\begin{array}{c}
         2 \\ 1 \\ 0
    \end{array}\right),
    \left(\begin{array}{c}
         2 \\ 0 \\ 1
    \end{array}\right),
    \left(\begin{array}{c}
         1 \\ 2 \\ 0
    \end{array}\right),
    \left(\begin{array}{c}
         1 \\ 0 \\ 2
    \end{array}\right),
    \left(\begin{array}{c}
         0 \\ 2 \\ 1
    \end{array}\right),
    \left(\begin{array}{c}
         0 \\ 1 \\ 2
    \end{array}\right)
    \right\} \\
    \mathcal{M}^+_1
    &=
    \left\{
    \left(\begin{array}{c}
         3 \\ 3 \\ 0
    \end{array}\right),
    \left(\begin{array}{c}
         3 \\ 0 \\ 3
    \end{array}\right),
    \left(\begin{array}{c}
         0 \\ 3 \\ 3
    \end{array}\right),
    \left(\begin{array}{c}
         4 \\ 1 \\ 1
    \end{array}\right),
    \left(\begin{array}{c}
         1 \\ 4 \\ 1
    \end{array}\right),
    \left(\begin{array}{c}
         1 \\ 1 \\ 4
    \end{array}\right)
    \right\} \\
    \mathcal{M}^+_2
    &=
    \left\{
    \left(\begin{array}{c}
         7 \\ 4 \\ 1
    \end{array}\right),
    \text{  and permutations}
    \right\} \\
    \mathcal{M}^+_3
    &=
    \left\{
    \left(\begin{array}{c}
         11 \\ 11 \\ 2
    \end{array}\right),
    \left(\begin{array}{c}
         14 \\ 5 \\ 5
    \end{array}\right),
    \text{  and permutations}
    \right\}.
    \end{align}
Note that each shell has 6 frequencies, but that the odd shell require 2 canonical frequencies rather than 1, because the repeated index means each canonical frequency only has 3 permutations.

For all $k \in \mathcal{M}$, we construct our constraint space by restricting the Fourier amplitudes of the velocity to the span of $v^k$ at each frequency $k\in\mathcal{M}$, where
\begin{align}
    v^k
    &=
    \frac{P_k^\perp(\sigma)}
    {\left|P_k^\perp(\sigma)\right|} \\
    P_k^\perp (\sigma)
    &=
    \sigma-\frac{k\cdot \sigma}{|k|^2}k.
\end{align}

\begin{definition}
    For all $s\geq 0$ and for all
    $u\in \dot{H}^s\left(\mathbb{T}^3;\mathbb{R}^3\right)$,
    we will say that 
    $u\in \dot{H}^s_{\mathcal{M}}
    \left(\mathbb{T}^3\right)$ if:
    \begin{equation}
        \supp{\hat{u}}\subset \mathcal{M},
    \end{equation}
    and for all $k\in \mathcal{M},$
    \begin{equation}
        \hat{u}(k)\in 
        \spn\left\{v^k\right\}.
    \end{equation}
\end{definition}

\begin{remark}
   Note that these frequencies can be constructed dyadically using permutations because for all $m\in\mathbb{Z}^+$,
\begin{align}
    h^m &=
    k^m+P_{12}\left(k^m\right) \\
    j^m &=
    k^m+P_{23}\left(k^m\right) \\
    k^{m+1}&=h^m+j^m.
\end{align} 
\end{remark}

\begin{remark}
    Note that by definition $k\cdot v^k=0$ for all $k\in \mathcal{M}$. This guarantees that $\nabla\cdot u=0$ for all $u\in H^s_\mathcal{M}$, because the divergence free condition can be expressed in Fourier space as $k\cdot \hat{u}(k)=0$.
\end{remark}

\begin{definition}
    For any vector field $u\in \dot{H}^s\left(\mathbb{T}^3;
    \mathbb{R}^3\right)$,
    and any permutation $P\in \mathcal{P}_3$,
    let
    \begin{equation}
        u^P(x)=Pu(P^{-1}x).
    \end{equation}
    We will likewise take the permutation of the Fourier transform to be given by
    \begin{equation}
    \hat{u}^P(k)=P\hat{u}(P^{-1}k).
    \end{equation}
     We will say that a vector field $u$ is permutation symmetric if 
     for all $P\in \mathcal{P}_3$,
     \begin{equation}
         u=u^P.
     \end{equation}
\end{definition}

\begin{remark}
    Note that $u\in \dot{H}^s$ is permutation symmetric if and only if for all $x\in\mathbb{T}^3$,
    and for all $P\in\mathcal{P}_3$,
    \begin{equation}
    u(Px)=Pu(x).
    \end{equation}
\end{remark}

\begin{remark}
In addition to permutation symmetry, we will consider $\sigma$-mirror symmetry, vector fields which have a mirror symmetry about the plane $\sigma\cdot x=0$.
Formally, a vector field is $\sigma$-mirror symmetric if for all $x\in\mathbb{T}^3$,
\begin{equation}
    u(x)=M_\sigma u(M_\sigma x),
\end{equation}
where $M_\sigma=I_3
-\frac{2}{3}\sigma\otimes\sigma$.
We call symmetry with respect to this map $\sigma$-mirror symmetry, because if 
$x=\lambda \sigma+y$, where $\sigma\cdot y=0$,
then
\begin{equation}
M_\sigma x=-\lambda \sigma+ y.
\end{equation}
\end{remark}

\subsection{Statement of the main results}

We will establish a general local and global wellposedness theory, and we will prove finite-time blowup for a certain class initial data with the symmetries described above.
For the Fourier-restricted Euler equation, we prove local wellposedness. For the Fourier-restricted hypodissipative Navier--Stokes equation with $0<\alpha<\frac{\log(3)}{4\log(2)}$, we prove local wellposedness for general initial data and global wellposedness for small initial data.
For the Fourier-restricted hypodissipative Navier--Stokes equation with 
$\alpha \geq \frac{\log(3)}{4\log(2)}$, we prove global well posedness for general initial data. 
For the Fourier-restricted Euler equation, and the Fourier-restricted hypodissipative Navier--Stokes equation with 
$0<\alpha<\frac{\log(3)}{6\log(2)}$,
we prove finite-time blowup for a carefully constructed class of odd, permutation symmetric, $\sigma$-mirror symmetric initial data.
This leaves the range 
$\frac{\log(3)}{6\log(2)}\leq \alpha < \frac{\log(3)}{4\log(2)}$ where the global regularity problem remains open for large initial data.

\begin{theorem} \label{RestrictedEulerExistenceThmIntro}
For all $u^0\in \dot{H}^s_{\mathcal{M}}, 
    s\geq \frac{\log(3)}{2\log(2)}$, 
    there exists a unique solution of the Fourier-restricted Euler equation $u\in C^1\left([0,T_{max});
    \dot{H}^s_{\mathcal{M}}\right)$.
    We have a lower bound on the time of existence
    \begin{equation}
    T_{max}\geq \frac{1}{\mathcal{C}_*
    \left\|u^0\right\|_{\dot{H}
    ^\frac{\log(3)}{2\log(2)}}},
    \end{equation}
    where $\mathcal{C}_*>0$ is an absolute constant independent of $u^0$ and $s$,
    and consequently if $T_{max}<+\infty$, 
    then for all $0\leq t<T_{max}$,
    \begin{equation}
    \|u(\cdot,t)\|_{\dot{H}^\frac{\log(3)}{2\log(2)}}
    \geq 
    \frac{1}{\mathcal{C}_*(T_{max}-t)}.
    \end{equation}
Furthermore, for all $0\leq t<T_{max}$, this solution satisfies
    the energy equality
    \begin{equation}
        \|u(\cdot,t)\|_{L^2}
        =
        \left\|u^0\right\|_{L^2}.
    \end{equation}
\end{theorem}

\begin{theorem} \label{RestrictedHypoExistenceThmIntro}
    Fix the fractional dissipation 
    $0<\alpha< \frac{\log(3)}{4\log(2)},$ 
    and the viscosity $\nu>0$.
    Suppose $u^0\in \dot{H}^s_\mathcal{M}$, where
    $\frac{\log(3)}{2\log(2)}-2\alpha
    <s\leq \frac{\log(3)}{2\log(2)}$.
    Then there exists a unique smooth solution of the Fourier-restricted, hypodissipative Navier--Stokes equation, $u\in C\left([0,T_{max});
    \dot{H}^s_{\mathcal{M}}\right)$, where
    \begin{equation}
    T_{max}\geq 
    \frac{\nu^{\frac{1}{\rho}-1}}
    {\left(C_{s,\alpha}\left\|u^0\right\|_{\dot{H}^s}
    \right)^\frac{1}{\rho}}
    \end{equation}
    with $C_{s,\alpha}>0$ is an absolute constant independent of $u^0$ and $\nu$, and
    \begin{equation}
    \rho=1-\frac{\frac{\log(3)}{2\log(2)}-s}{2\alpha}.
    \end{equation}
    Furthermore, if $T_{max}<+\infty$, 
    then for all $0\leq t<T_{max}$,
    \begin{equation}
    \left\|u(\cdot,t)\right\|_{\dot{H}^s}
    \geq 
    \frac{\nu^{1-\rho}}{C_{s,\alpha}
    \left(T_{max}-t\right)^\rho}.
    \end{equation}
    For $s=\frac{\log(3)}{2\log(2)}$,
    we have the lower bound on the existence time
    \begin{equation}
    T_{max}\geq \frac{1}{\mathcal{C}_*
    \left\|u^0\right\|_{\dot{H}
    ^\frac{\log(3)}{2\log(2)}}},
    \end{equation}
    and consequently if $T_{max}<+\infty$, 
    then for all $0\leq t<T_{max}$,
    \begin{equation}
    \|u(\cdot,t)\|_{\dot{H}^\frac{\log(3)}{2\log(2)}}
    \geq 
    \frac{1}{\mathcal{C}_*(T_{max}-t)}.
    \end{equation}
    Furthermore, for all $0\leq t<T_{max}$, this solution satisfies 
    the energy equality
        \begin{equation}
        \frac{1}{2}\|u(\cdot,t)\|_{L^2}^2
        +
        \nu \int_0^t
        \|u(\cdot,\tau)\|_{\dot{H}^\alpha}^2
        \diff\tau 
        =
        \frac{1}{2}\left\|u^0\right\|_{L^2}^2.
        \end{equation}
     Finally, for all positive times we have higher regularity, with
    $u\in C^\infty\left((0,T_{max});C^\infty
    \left(\mathbb{T}^3\right)\right)$,
    or equivalently
    $u\in C^\infty\left((0,T_{max})\times 
    \mathbb{T}^3)\right)$.
\end{theorem}

\begin{theorem} \label{GwpThmIntro}
    Fix the degree of dissipation $\alpha\geq \frac{\log(3)}{4\log(2)}$, the regularity $s>\frac{\log(3)}{2\log(2)}-2\alpha, s\geq 0$, and the viscosity $\nu>0$.
    Then for all $u^0\in \dot{H}^s_\mathcal{M}$, there exists a unique, global smooth solution of the Fourier-restricted hypodissipative Navier--Stokes equation $u\in C\left([0,+\infty);\dot{H}^s_\mathcal{M}\right)$,
    and for all $0\leq t<+\infty$,
    \begin{equation}
    \|u(\cdot,t)\|_{\dot{H}^s}^2
    \leq
    \left\|u^0\right\|_{\dot{H}^s}^2
    \exp\left(\frac{\mathcal{C}_s^2}{4(2\pi)^{
    4\alpha-\frac{\log(3)}{2\log(2)}}}
    \frac{\left\|u^0\right\|_{L^2}^2}{\nu^2}
    \right),
    \end{equation}
    where $\mathcal{C}_s>0$ is an absolute constant depending only on $s$.
    Furthermore, for all $0\leq t<+\infty$, this solution satisfies 
    the energy equality
        \begin{equation}
        \frac{1}{2}\|u(\cdot,t)\|_{L^2}^2
        +
        \nu \int_0^t
        \|u(\cdot,\tau)\|_{\dot{H}^\alpha}^2
        \diff\tau 
        =
        \frac{1}{2}\left\|u^0\right\|_{L^2}^2.
        \end{equation}
     Finally, for all positive times we have higher regularity, with
    $u\in C^\infty\left((0,+\infty);C^\infty
    \left(\mathbb{T}^3\right)\right)$,
    or equivalently
    $u\in C^\infty\left((0,+\infty)\times 
    \mathbb{T}^3)\right)$.
\end{theorem}

\begin{theorem} \label{SmallDataGwpIntro}
    Suppose $u^0\in \dot{H}^s_{\mathcal{M}}$, where $s>\frac{\log(3)}{2\log(2)}-2\alpha$ and $0<\alpha<\frac{\log(3)}{4\log(2)}$. Further suppose that 
    \begin{equation}
    \left\|u^0\right\|_{\dot{H}^{\frac{\log(3)}{2\log(2)}-2\alpha}}
    <
    M_{\alpha,s} \nu,
    \end{equation}
    where $M_{\alpha,s}>0$ is an absolute constant depending on only $\alpha$ and $s$.
    Then there exists a global smooth solution of the Fourier-restricted hypodissipative Navier--Stokes equation
    $u\in C\left([0,+\infty);\dot{H}^s_\mathcal{M}\right)$,
    and for all $0< t<+\infty$,
    \begin{align}
    \|u(\cdot,t)\|_{
    \dot{H}^{\frac{\log(3)}{2\log(2)}-2\alpha}}
    &\leq
    \left\|u^0\right\|_{
    \dot{H}^{\frac{\log(3)}{2\log(2)}-2\alpha}} \\
    \|u(\cdot,t)\|_{\dot{H}^s}
    &\leq
    \left\|u^0\right\|_{\dot{H}^s}.
    \end{align}
\end{theorem}

Both the Fourier-restricted Euler and hypodissipative Navier--Stokes equations have smooth solutions that blowup in finite-time. In the inviscid case, such blowup is generic for odd, permutation symmetric, $\sigma$-mirror symmetric solutions, while in the viscous case, there is only finite-time blowup when the dissipation is sufficiently weak and the viscosity is sufficiently small relative to a Lyapunov functional.

\begin{theorem} 
\label{RestrictedEulerBlowupIntro}
Suppose $u^0\in \dot{H}^\frac{\log(3)}{2\log(2)}
_{\mathcal{M}}$ is odd, permutation symmetric, and $\sigma$-mirror symmetric. Then $u^0$ can be expressed of the form
    \begin{equation}
    u^0(x)
    =
    -2\sum_{n=0}^{+\infty}
    \psi_n(0)\sum_{k\in\mathcal{M}^+_n}
    v^k\sin(2\pi k\cdot x).
    \end{equation}
    These symmetries are preserved by the dynamics of the Fourier-restricted Euler equation, and so for all $0\leq t<T_{max}$,
    \begin{equation}
    u(x,t)
    =
    -2\sum_{n=0}^{+\infty}
    \psi_n(t)\sum_{k\in\mathcal{M}^+_n}
    v^k\sin(2\pi k\cdot x),
    \end{equation}
    where $u\in C^1\left([0,T_{max});
    \dot{H}^\frac{\log(3)}{2\log(2)}
    _{\mathcal{M}}\right)$
    is the unique solution of the Fourier-restricted Euler equation.
    
    Furthermore, for all $0\leq t<T_{max}$ and for all $n\in\mathbb{Z}^+,$ the Fourier coefficients satisfy the system of ODEs
    \begin{equation} \label{EulerODEintro}
        \partial_t \psi_n
        =
        \sqrt{2}\pi\beta_{n-1}
        \left(\sqrt{3}\right)^{n}
        \psi_{n-1}^2
        -\sqrt{2}\pi\beta_n
        \left(\sqrt{3}\right)^{n+1}
        \psi_n \psi_{n+1},
    \end{equation}
    where for all $n\in\mathbb{Z}^+$,
    \begin{equation}
        \beta_n=\frac{1}{\left(
        1+\frac{1}{2}
        \left(\frac{3}{4}\right)^n
        \right)^\frac{1}{2}},
    \end{equation}
    and by convention
    \begin{equation}
        \psi_{-1},\beta_{-1}:=0.
    \end{equation}

    If $u^0$ is not identically zero, then this solution blows up in finite-time with
    \begin{equation}
    T_{max}\leq 
    \left(\left(\frac{3^\frac{1}{4}}
    {12\left(3^\frac{1}{4}-1\right)}
    \right)^\frac{1}{2}
    \left\|u^0\right\|_{L^2}
    -H(0) \right)
    \frac{6\sqrt{3}}{\pi\left(
    3^\frac{3}{8}-3^\frac{5}{16}\right)
    \left\|u^0\right\|_{L^2}^2},
    \end{equation}
    where
    \begin{equation}
    H(0)
    =
    \sum_{n=0}^\infty
    \left(\sqrt{3}\right)^{-\frac{n}{4}}
    \psi_n(0).    
    \end{equation}
\end{theorem}

\begin{theorem} 
\label{RestrictedHypoBlowupIntro}
    Suppose $u^0\in \dot{H}^\frac{\log(3)}
    {2\log(2)}_{\mathcal{M}}$, is odd, permutation symmetric, and $\sigma$-mirror symmetric; and can therefore be expressed in the form
    \begin{equation}
    u^0(x)
    =
    -2\sum_{n=0}^{+\infty}
    \psi_n(0)\sum_{k\in\mathcal{M}^+_n}
    v^k\sin(2\pi k\cdot x).
    \end{equation}
    Then, as in the inviscid case, these symmetries are preserved by the dynamics of the Fourier-restricted hypodissipative Navier--Stokes equation, and for all $0\leq t<T_{max}$,
    \begin{equation}
    u(x,t)
    =
    -2\sum_{n=0}^{+\infty}
    \psi_n(t)\sum_{k\in\mathcal{M}^+_n}
    v^k\sin(2\pi k\cdot x),
    \end{equation}
    where $u\in C\left([0,T_{max});
    \dot{H}^\frac{\log(3)}{2\log(2)}_{\mathcal{M}}\right)$
    is the unique solution of the Fourier-restricted, hypodissipative Navier--Stokes equation.
    Furthermore, for all $0<t<T_{max}$ and for all $n\in\mathbb{Z}^+,$ the Fourier coefficients satisfy the system of ODEs
    \begin{equation} \label{NavierStokesODEintro}
        \partial_t \psi_n
        =
        -(12\pi^2)^\alpha \mu_n^\alpha \nu
        \left(\sqrt{3}\right)^{2\Tilde{\alpha}n}
        \psi_n
        +\sqrt{2}\pi\beta_{n-1}
        \left(\sqrt{3}\right)^{n}
        \psi_{n-1}^2
        -\sqrt{2}\pi\beta_n
        \left(\sqrt{3}\right)^{n+1}
        \psi_n \psi_{n+1},
    \end{equation}
    where
    \begin{equation}
        \Tilde{\alpha}=\frac{2\log(2)}{\log(3)}\alpha
    \end{equation}
    and for all $n\in\mathbb{Z}^+$,
    \begin{align}
        \beta_n &=
        \frac{1}{\left(1+\frac{1}{2}
        \left(\frac{3}{4}\right)^n
        \right)^\frac{1}{2}} \\
        \mu_n &=
        \frac{1}{\left(1+\frac{2}{3}
        \left(\frac{3}{4}\right)^n
        \right)^\frac{1}{2}}
    \end{align}
    and by convention
    \begin{equation}
        \psi_{-1},\beta_{-1}:=0.
    \end{equation}

    Further suppose that $\alpha<\frac{\log(3)}{6\log(2)}$, and for some $\Tilde{\alpha}<s<\frac{1}{3}$
    \begin{equation}
    H_s(0)=\sum_{n=0}^\infty 
    \left(\sqrt{3}\right)^{-sn}\psi_n(0)
    > 
    C_{\alpha,s} \nu,
    \end{equation}
    where
    \begin{equation}
    C_{\alpha,s }=
    \frac{\sqrt{3}\left(12\pi^2\right)^\alpha}
    {\pi\left(1-\left(\sqrt{3}\right)^{-(1+s-4\Tilde{\alpha})}\right)^\frac{1}{2}
    \left(\left(\sqrt{3}\right)^{1-s}-
    \left(\sqrt{3}\right)^\frac{1+s}{2}\right)
    \left(1-
    \left(\sqrt{3}\right)^{-(1+s)}
    \right)^\frac{1}{2}}.
    \end{equation}
    Then the solution blows up in finite-time with
    \begin{equation}
        T_{max}\leq
        \frac{1}
        {\kappa_s H_s(0)},
    \end{equation}
    where
    \begin{equation}
    \kappa_s
    =
    \frac{\pi}{\sqrt{3}} \left(
    \left(\sqrt{3}\right)^{1-s}
    -\left(\sqrt{3}\right)^\frac{1+s}{2}\right)
    \left(1-\left(\sqrt{3}\right)^{-(1+s)}\right).
    \end{equation}
    Note that this finite-time blowup is implied by a singular lower bound on the Lyapunov functional
    \begin{equation}
    H_s(t)>
    \frac{H_s(0)}
    {1-\kappa_s H_s(0)t}.
    \end{equation}
\end{theorem}

\begin{remark}
    Note that \Cref{RestrictedHypoBlowupIntro} guarantees finite-time blowup within the set of odd, permutation-symmetric, $\sigma$-mirror symmetric initial data with positive coefficients $\psi_n$ as long as the viscosity is sufficiently small.
    In the inviscid case, blowup is generic in the symmetry class for all nontrivial initial data. 
    In the hypo-viscous case with the degree of dissipation in the range $0<\alpha<\frac{\log(3)}{6\log(2)}$, blowup is generic at high Reynolds number subject to these symmetry and sign constraints. This leaves the range $\frac{\log(3)}{6\log(2)}\leq \alpha<\frac{\log(3)}{4\log(2)}$ where the problem of finite-time blowup remains open for large initial data. We summarize the regularity theory in a table below.
    \end{remark}

\begin{center}
 \begin{tabular}{| |c| c| |} 
 \hline
  Parameter $\alpha$ 
 & Behaviour of solutions  \\ [0.5ex] 
 \hline\hline
  $\alpha=0, \nu=0$ & Finite-time blowup generic subject to symmetry and sign conditions  \\ [.5ex]
 \hline
   $0<\alpha<\frac{\log(3)}{6\log(2)}$ &
  Finite-time blowup for large initial data
  (subject to above conditions)\\ [.5ex]
 \hline
  $\alpha\geq \frac{\log(3)}{4\log(2)}$ & Global regularity for generic initial data \\ [.5ex]
 \hline
  $0<\alpha<\frac{\log(3)}{4\log(2)}$ & Global regularity for small initial data\\ [.5ex] 
  \hline
 $\frac{\log(3)}{6\log(2)} \leq \alpha
 <\frac{\log(3)}{4\log(2)}$ &
  Finite-time blowup open for large initial data \\ [.5ex]
 \hline
\end{tabular}
\end{center}

\begin{remark}
While the dynamics of the Fourier-restricted model equations for which we have proven finite-time blowup are dramatically simpler than the full Euler and hypodissipative Navier--Stokes equations, it does nonetheless show that the $(u\cdot\nabla)u$ nonlinearity can result in a cascade of energy to arbitrarily high modes fast enough to lead to finite-time blowup, even for divergence free vector fields. Because the symmetry class we considered for the restricted model equation---odd, permutation symmetric, and $\sigma$-mirror symmetric---is also preserved by the full Euler and Navier--Stokes equation, this subspace of solutions is an intriguing candidate for the finite-time blowup problem.
\end{remark}

\begin{question} \label{FullEulerOpenQuestion}
    Do there exist odd, permutation symmetric solutions of the incompressible Euler equation $u\in C\left([0,T_{max});C^\infty\right)$, that blow up in finite-time $T_{max}<+\infty$?
\end{question}

\begin{remark}
    While the choice of Fourier modes for the Fourier-restricted model equation is very specifically motivated to yield a nonlinear term with a dyadic structure, it is also based on geometric structures which are relevant to the full Euler equation as well. In addition to considering data satisfying the conditions of \Cref{RestrictedHypoBlowupIntro} as initial data for the full Euler equation, this analysis suggests a fairly broad class of initial data as possible candidates for finite-time blowup for the full Euler equation:
    \begin{equation}
    u^0(x)=-\sum_{\sigma\cdot k>0} 
    \gamma_k P_k^\perp(\sigma) \sin(2\pi k\cdot x),
    \end{equation}
    where $\gamma$ is nonnegative and permutation symmetric in the sense that for all $P\in\mathcal{P}_3$ and $k\in\mathbb{Z}^3$, 
    \begin{equation}
        \gamma_k=\gamma_{P(k)}.
    \end{equation}
    See \Cref{GeometricSection} and particularly \Cref{FullEulerRemark} for further discussion.
\end{remark}

\subsection{Organization}

In \Cref{FurtherResultsSection}, we will state several regularity criteria for the restricted model equations and discuss the relationship of the results in this paper to the previous literature.
In \Cref{Prelims}, we will define a number of important spaces and establish their key properties.
In \Cref{WellPosedSection}, we will prove local wellposedness for the Fourier-restricted Euler and hypodissipative Navier--Stokes equations, as well as global regularity in the viscous case for generic initial data when $\alpha\geq \frac{\log(3)}{4\log(2)}$ and for small initial data when $\alpha< \frac{\log(3)}{4\log(2)}$.
In \Cref{PermutationSection}, we will discuss permutation symmetry for both the Euler and hypodissipative Navier--Stokes equations and their respective restricted model equations.

In \Cref{DynamicsSection}, we will study the dynamics of the Fourier-restricted Euler and hypodissipative Navier--Stokes equations for odd, permutation symmetric solutions. We will use this symmetry class to reduce the Fourier-restricted Euler and hypodissipative Navier--Stokes equations to an infinite system of ODEs very similar to the Friedlander-Katz-Pavlovi\'c dyadic Euler and Navier--Stokes equations introduced in \cites{FriedlanderPavlovic,KatzPavlovic}. 
In \Cref{BlowupSection}, we will use this infinite system of ODEs to prove \Cref{RestrictedEulerBlowupIntro,RestrictedHypoBlowupIntro}, showing finite-time blowup both for the Fourier-restricted Euler equation and for the Fourier-restricted, hypodissipative Navier--Stokes equation when $\alpha<\frac{\log(3)}{6\log(2)}$. This argument gives a new proof of finite-time blowup for the dyadic Navier--Stokes equation. Although the threshold for finite-time blowup $\Tilde{\alpha}<\frac{1}{3}$ is the same as that found by Cheskidov \cite{Cheskidov}, a different choice of Lyapunov functional allows for the class of data covered to be slightly more generic.

In \Cref{AppendixModeInteraction}, we will go through some elementary, but also long and tedious vector calculus computations. These are essential to the construction---indeed finding the correct subset of modes with the necessary properties was the core difficulty of the work---but the calculations themselves are left for the appendices to avoid clogging up the body of the paper.
In \Cref{AppendixNonlinearBound}, we will use these computations to prove a key bound on the nonlinearity.
In \Cref{GeometricSection}, we will discuss the geometry of our blowup solutions and the relationship of this blowup with the existing literature on the regularity of the Euler and Navier--Stokes equations, including a precise geometric description of the blowup at the origin. 
In \Cref{AppendixFourierLimitations}, we will discuss some limitations to the method of working on a subset of the frequencies in Fourier space when returning to the full Euler and Navier--Stokes equations.

\section{Further results and discussion}
\label{FurtherResultsSection}

In this section, we will state a number of regularity criteria for the Fourier-restricted Euler and hypodissipative Navier--Stokes equations. We will also discuss the relationship of this work to the previous literature, in particular 
addressing the geometric significance of the blowup Ansatz considered in \Cref{RestrictedEulerBlowupIntro,RestrictedHypoBlowupIntro}.

\subsection{Regularity criteria}

The Fourier-restricted Euler and hypodissipative Navier--Stokes equations have a regularity criterion in terms of the positive part of the middle eigenvalue of the strain matrix. These results are directly analogous to the regularity criteria on $\lambda_2^+$ proven by Neustupa and Penel for the Navier--Stokes equation \cite{NeustupaPenel1}. In fact, for the blowup Anstaz considered in \Cref{RestrictedEulerBlowupIntro,RestrictedHypoBlowupIntro}, we can even prove that this blowup must occur at the origin.

\begin{theorem} \label{ViscousStrainRegCritIntro}
    Suppose $u\in C\left([0,T_{max});
    \dot{H}^1_{\mathcal{M}}\right)$,
is a solution of the Fourier-restricted hypodissipative Navier--Stokes equation, that $\alpha<\frac{\log(3)}{4\log(2)}$,
and suppose $\frac{1}{p}+\frac{3}{2\alpha q}=1, 
\frac{3}{2\alpha}<q\leq +\infty$.
Then for all $0\leq t<T_{max}$,
\begin{equation} 
    \|S(\cdot,t)\|_{L^2}^2
    \leq \|S^0\|_{L^2}^2
    \exp\left(\frac{C_{\alpha,q}}
    {\nu^\frac{p-1}{p^2}}
    \int_0^t
    \left\|\lambda_2^+(\cdot,\tau)\right\|_{L^q}^p
    \diff\tau \right),
\end{equation}
where $C_{\alpha,q}>0$ is an absolute constant independent of $\nu,s$ and $u^0$ depending only on $\alpha$ and $q$; and 
$\lambda_1(x,t)\leq 
\lambda_2(x,t)\leq 
\lambda_3(x,t)$
are the eigenvalues of $S(x,t)$, the strain matrix.
In particular, if $T_{max}<+\infty$, then
\begin{equation}
    \int_0^{T_{max}}
    \left\|\lambda_2^+(\cdot,t)\right\|_{L^q}^p
    \diff t
    =+\infty.
\end{equation}
\end{theorem}

\begin{theorem} \label{EulerStrainRegCritIntro}
     Suppose $u\in C^1\left([0,T_{max});
    \dot{H}^s_{\mathcal{M}}\right),
    s>\frac{5}{2}$
is a solution of the Fourier-restricted Euler equation.
Then for all $0\leq t<T_{max}$,
\begin{equation} 
    \|S(\cdot,t)\|_{L^2}^2
    \leq 
    \left\|S^0\right\|_{L^2}^2
    \exp\left( 2\int_0^t
    \left\|\lambda_2^+(\cdot,\tau)
    \right\|_{L^\infty}
    \diff\tau\right).
\end{equation}
In particular, if $T_{max}<+\infty$, then
\begin{equation}
    \int_0^{T_{max}}
    \left\|\lambda_2^+(\cdot,t)\right\|_{L^\infty}
    \diff t
    =+\infty.
\end{equation}
\end{theorem}

\begin{theorem} \label{PsiSupRegCritIntro}
    Suppose $u\in C\left([0,T_{max});
    \dot{H}^\frac{\log(3)}{2\log(2)}_{\mathcal{M}}
    \right)$,
is an odd, permutation symmetric, $\sigma$-mirror symmetrtic solution of the Fourier-restricted Euler or hypodissipative Navier--Stokes equation.
Then for all $0\leq t<T_{max}$,
\begin{equation} 
    \|\psi(t)\|_{\mathcal{H}^1}^2
    \leq 
    \left\|\psi^0\right\|_{\mathcal{H}^1}^2
    \exp\left(4\sqrt{2}\pi \int_0^t
    \sup_{n\in\mathbb{Z}^+}
    \left(\sqrt{3}\right)^n \psi_n(\tau)
    \diff\tau\right).
\end{equation}
In particular, if $T_{max}<+\infty$, then
\begin{equation}
    \int_0^{T_{max}}
    \sup_{n\in\mathbb{Z}^+}
    \left(\sqrt{3}\right)^n \psi_n(t)
    \diff t
    =
    +\infty.
\end{equation}
\end{theorem}

\begin{theorem} \label{StrainOriginBlowupThmIntro}
    Suppose $u\in C\left([0,T_{max});
    \dot{H}^s_{\mathcal{M}}\right), 
    s>\frac{5}{2}$
is an odd, permutation symmetric, $\sigma$-mirror symmetric solution of the Fourier-restricted hypodissipative Navier--Stokes equation or Fourier-restricted Euler equation.
Further suppose that for all $n\in\mathbb{Z}^+$, 
we have $\psi_n(0)\geq 0$.
Then for all $0\leq t<T_{max}$,
\begin{equation}
    \nabla u(\Vec{0},t)
    =
    \lambda(t)
    \left(\begin{array}{ccc}
         0 & -1 & -1  \\
         -1 & 0 & -1 \\
         -1 & -1 & 0
    \end{array}\right),
\end{equation}
where
\begin{equation}
    \lambda(t)=
    12\sqrt{2}\pi \sum_{n=0}^{+\infty}
    \psi_n(t)
    \frac{\left(\sqrt{3}\right)^n}
    {\left(1+\frac{2}{3}
    \left(\frac{3}{4}\right)^n
    \right)^\frac{1}{2}} 
    \geq 0
    \end{equation}

    Furthermore, for all $0\leq t<T_{max}$,
\begin{equation}
    \|\psi(t)\|_{\mathcal{H}^1}^2
    \leq 
    \left\|\psi^0\right\|_{\mathcal{H}^1}^2
    \exp\left(\frac{\sqrt{5}}{3\sqrt{3}}
    \int_0^t \lambda(\tau) 
    \diff\tau\right).
    \end{equation}
    In particular, if $T_{max}<+\infty$,
    then
    \begin{equation}
    \int_0^{T_{max}}\lambda(t) \diff t
    =
    +\infty.
    \end{equation}
    \end{theorem}

\begin{remark}
    We should note that this last regularity criterion can be seen in terms of scaling as an analogue of the Beale-Kato-Majda criterion, because it implies that $\nabla u$ must blowup in $L^1_t L^\infty_x$ in order for a solution to blowup in finite-time.
    The classic Beale-Kato-Majda regularity criteria \cite{BKM} requires the vorticity to blowup in $L^1_t L^\infty_x$, but Kato and Ponce \cite{KatoPonce} proved the analogous result for the strain, showing that if a smooth solution of the Euler equation blows up in finite-time $T_{max}<+\infty$, then
    \begin{equation}
    \int_0^{T_{max}} \|S(\cdot,t)\|_{L^\infty}
    \diff t=+\infty.
    \end{equation}
    Because the velocity gradient is symmetric at the origin, \Cref{StrainOriginBlowupThmIntro} is analogous to this latter result, although
    it is substantially stronger, as it guarantees blowup at a specific point, $x=\Vec{0}$, and it provides geometric information about the structure of the stagnation point blowup at the origin.

    We do not state the standard Beale-Kato-Majda criterion in terms of the vorticity in this section, because we already have something much stronger in terms of scaling in the local wellposedness theory.
    We have shown that for the Fourier-restricted Euler equation, 
    if $T_{max}<+\infty$, then
    \begin{equation}
    \|u(\cdot,t)\|_{\dot{H}^\frac{\log(3)}{2\log(2)}}
    \geq
    \frac{1}{\mathcal{C}_* (T_{max}-t)},
    \end{equation}
    which clearly implies that
    \begin{equation}
    \int_0^{T_{max}}
    \|u(\cdot,t)\|_{\dot{H}^{\frac{\log(3)}{2\log(2)}}} \diff t
    =+\infty.
    \end{equation}
    On the torus, we know that
    \begin{equation}
    (2\pi)^{1-\frac{\log(3)}{2\log(2)}}
    \|u\|_{\dot{H}^{\frac{\log(3)}{2\log(2)}}}
    \leq 
    \|u\|_{\dot{H}^1}
    =
    \|\omega\|_{L^2}
    \leq
    \|\omega\|_{L^\infty},
    \end{equation}
    and so we can see that 
    \begin{equation}
    \int_0^{T_{max}}
    \|\omega(\cdot,t)\|_{L^\infty} \diff t
    =+\infty.
    \end{equation}
    The dyadic structure of the Fourier-restricted Euler equation means that we have regularity criteria that are much stronger than the Beale-Kato-Majda criterion in terms of scaling.
\end{remark}

\subsection{Relationship to previous literature}
\label{PreviousLit}

The blowup results in this paper will be substantially based around the blowup results for the Friedlander-Katz-Pavlovi\'c model, the dyadic Euler and Navier--Stokes equations.
This model equation can be reduced to an infinite system of ODEs, where for all $n\in\mathbb{Z}^+$
\begin{equation}
    \partial_t u_n=-c_1\nu\lambda^{2n\Tilde{\alpha}} u_n
    +c_2\lambda^n u_{n-1}^2
    -c_2\lambda^{n+1} u_n u_{n+1},
\end{equation}
with $c_1,c_2>0$, the nonlinearity parameter $\lambda>1$, the viscosity $\nu\geq 0$, the degree of dissipation $\Tilde{\alpha}> 0$, and with $u_{-1}:=0$ by convention.
Note that this equation is often taken with indices in $\mathbb{N}$, starting with $u_1$, but re-indexing to start with $u_0$ is not a problem. Also, it is standard to fix the constants $c_1=c_2=1$, but it is easy to see this is equivalent to the general case $c_1,c_2>0$ by a simple rescaling of time and viscosity. The inviscid case, where $\nu=0$, is the dyadic Euler equation, while the viscous case, $\nu>0$, is the dyadic Navier--Stokes equation. It is simple to check that if the factors $\beta_n,\mu_n$ are dropped, the infinite system of ODEs in \Cref{RestrictedEulerBlowupIntro,RestrictedHypoBlowupIntro} are exactly the dyadic Euler and dyadic Navier--Stokes equations respectively. Because the factors $\beta_n$ and $\mu_n$ increase from $\frac{\sqrt{2}}{\sqrt{3}}$ and $\frac{\sqrt{3}}{\sqrt{5}}$ respectively to $1$ exponentially fast, the blowup dynamics of the dyadic model can be exploited to prove finite-time blowup for the Fourier-restricted Euler and hypodissipative Navier--Stokes equations when considering solutions that are odd, permutation-symmetric, and $\sigma$-mirror symmetric.

Friedlander and Pavlovi\'c proved finite-time blowup for the inviscid case \cite{FriedlanderPavlovic}, and independently around the same time Katz and Pavlovi\'c proved finite-time blowup for the inviscid case and for the viscous case when
$\Tilde{\alpha}<\frac{1}{4}$. Cheskidov refined this analysis in \cite{Cheskidov}, proving finite-time blowup for $\Tilde{\alpha}<\frac{1}{3}$, and global regularity for $\Tilde{\alpha}\geq \frac{1}{2}$.
Additionally, Barbato, Morandin, and Romito \cite{BarbatoMorandinRomito} also proved a global regularity result for nonnegative solutions in the range $\frac{2}{5}\leq \Tilde{\alpha}<\frac{1}{2}$ when $\lambda=2^{\frac{1}{\Tilde{\alpha}}}$, beyond the energy threshold. This result, however, cannot be adapted straightforwardly to the Fourier-restricted hypodissipative Navier--Stokes equation, because any generalization would apply only to the Fourier-restricted equation in the symmetry class where we prove blowup. In the general case, the dynamics are much more complicated, and the dynamical systems argument from \cite{BarbatoMorandinRomito} will not apply.

In this paper, we prove finite-time blowup for the Fourier-restricted hypodissipative Navier--Stokes equation
when $\Tilde{\alpha}<\frac{1}{3}$, which corresponds to $\alpha<\frac{\log(3)}{6\log(2)}$,
and global regularity when 
$\Tilde{\alpha}\geq \frac{1}{2}$,
which corresponds to 
$\alpha\geq\frac{\log(3)}{4\log(2)}$.
For the proof of global regularity, we will closely follow the methods of Cheskidov in \cite{Cheskidov}. For the proof of finite-time blowup, we will use a different Lyapunov functional that gives finite-time blowup for a somewhat broader class of initial data than the Lyapunov functional in
\cite{Cheskidov}, but with the same threshold for blowup, $\Tilde{\alpha}<\frac{1}{3}$. In particular, we can relax the requirement from \cite{Cheskidov} that the $u_n(0)\geq 0$, for all $n \in\mathbb{Z}^+$, giving the following result.

\begin{theorem} \label{DyadicNSblowupIntro}
    Suppose $u\in C\left([0,T_{max};\mathcal{H}^1\right)$
    is a solution of the dyadic Navier--Stokes equation, and that for some $0<\alpha<s<\frac{1}{3}$,
    \begin{equation}
    H_s(0)=\sum_{n=0}^\infty
    \lambda^{-sn}u_n(0)
    >C_{\alpha,s}\nu,
    \end{equation}
    where
    \begin{equation}
    C_{\alpha,s}
    =
    \frac{2}{\left(1-\lambda^{-(1+s)}
    \right)^\frac{1}{2}
    \left(\lambda^{1-s}
    -\lambda^{\frac{1+s}{2}}\right)
    \left(1-\lambda^{
    -(1+s-4\alpha)}\right)^\frac{1}{2}}.
    \end{equation}
    Then this solution blows up in finite-time
    with
    \begin{equation}
    T_{max}<\frac{1}{\kappa_s H_s(0)},
    \end{equation}
    where
    \begin{equation}
    \kappa_s
    =
    \frac{1}{2}\left(\lambda^{1-s}
    -\lambda^{\frac{1+s}{2}}\right)
    \left(1-\lambda^{-(1+s)}
    \right).
    \end{equation}
\end{theorem}

It is important to note that these methods adapted from the analysis of the dyadic model will only be available once the class of symmetries---odd, permutation symmetric, $\sigma$-mirror symmetric---has been used to dramatically simplify the dynamics Fourier-restricted Euler and hypodissipative Navier--Stokes equations, which in the general case are much more complicated than the dyadic Euler/Navier--Stokes equations. Without these symmetries, there would be 24 equations for each frequency shell $\mathcal{M}_n$, rather than 1.

The factor of $\frac{\log(3)}{2\log(2)}$ in our results comes from a partial depletion of nonlinearity at high modes due to anisotropy. Heuristically based on scaling we would expect that
\begin{equation}
    (u\cdot\nabla)u=
    \nabla\cdot(u\otimes u)
    \sim |\nabla| u^2.
\end{equation}
Instead, at high frequencies we find that
\begin{equation}
    (u\cdot\nabla)u=
    \nabla\cdot(u\otimes u)
    \sim |\nabla|^\frac{\log(3)}{\log(4)} u^2.
\end{equation}
This is because the nonlinearity grows at a rate of $3^m$, while the frequencies grow at a rate of $4^m$.
This can be seen for all the interactions from the computations in \Cref{AppendixModeInteraction}, but we will briefly consider one of the computations here.
Consider two modes in the shell $\mathcal{P}[k^m]$:
\begin{align}
    u&=i v^{k^m}e^{2\pi i k^m\cdot x} \\
    \Tilde{u}&=
    i v^{P_{12}(k^m)}e^{2\pi i P_{12}(k^m)\cdot x}.
\end{align}
Then we have 
\begin{equation}
    (\Tilde{u}\cdot\nabla)u+(u\cdot\nabla)\Tilde{u}
    =
    -i\left((k^m\cdot 
    P_{12}\left(v^{k^m}\right)\right)
    \left(v^{k^m}+P_{12}\left(v^{k^m}\right)\right)
    e^{2\pi i h^m\cdot x},
\end{equation}
using the fact that $h^m=k^m+P_{12}(k^m)$,
that $P_{12}(v^{k^m})=v^{P_{12}(v^{k^m})}$, and that
\begin{equation}
    P_{12}\left(k^m\right)\cdot v^{k^m}
    =
    k^m\cdot P_{12}\left(v^{k^m}\right).
\end{equation}
It would appear that the nonlinearity grows like $|k^m|$, but there is an additional cancellation.
The divergence free constraint implies that
$P_{12}(k^m)\cdot P_{12}(v^{k^m})=0$,
and so we can see that
\begin{equation}
    (\Tilde{u}\cdot\nabla)u+(u\cdot\nabla)\Tilde{u}
    =
    -i\left(\left(k^m-P_{12}\left(k^m\right)\right)
    \cdot P_{12}\left(v^{k^m}\right)\right)
    \left(v^{k^m}+P_{12}\left(v^{k^m}\right)\right)
    e^{2\pi i h^m\cdot x}.
\end{equation}
Observe that
\begin{equation}
    k^m-P_{12}(k^m)=
    3^m\left(
    \begin{array}{c}
         1 \\ -1  \\ 0
    \end{array}\right),
\end{equation}
and so
\begin{equation}
    |k^m-P_{12}(k^m)| = \sqrt{2} \left(3^m\right),
\end{equation}
but asymptotically as $m \to +\infty$,
\begin{equation}
    |k^m| \sim \sqrt{3} \left(4^m\right),
\end{equation}
and the partial depletion of nonlinearity is apparent.

The fact that this depletion of nonlinearity is due to anisotropy can be observed from the fact that $k^m,h^m,j^m$ each converge to conically to $\spn(\sigma)$. In particular, we will show in \Cref{DynamicsSection} that
\begin{equation}
    \lim_{m\to+\infty}
    \frac{\sigma\cdot k^m}{|\sigma||k^m|}=1,
\end{equation}
and likewise for $h^m$ and $j^m$.
This implies that the angle between $\sigma$ and $k^m,h^m,j^m$ shrinks to zero as $m\to +\infty$, and so the frequency basis is becoming highly anisotropic.

\begin{remark}
    Note that because anisotropy grows larger at higher frequencies, the Fourier-restricted hypodissipative Navier--Stokes equation does not have any scale invariance.
    The full hypodissipative Navier--Stokes equation has the scaling invariance
    \begin{equation}
    u^\lambda(x,t)=
    \lambda^{2\alpha-1} u
    \left(\lambda x, \lambda^{2\alpha}t\right).
    \end{equation}
    If $u$ is a solution of the hypodissipative Navier--Stokes equation on the whole space, then $u^\lambda$ is also a solution for all $\lambda> 0$.
    If $u$ is a solution of the hypodissipative Navier--Stokes equation on the torus, then $u^\lambda$ is also a solution for all $\lambda\in \mathbb{N}$.
    This scale invariance is broken for the Fourier-restricted hypodissipative Navier--Stokes equation due to anisotropy; the angle between $k^m,h^m,j^m$ and $\sigma$ goes to zero as $m\to \infty$,
    which implies that large frequencies cannot be obtained by rescaling small frequencies.
\end{remark}

\begin{remark}
    There are also significant similarities between the Fourier-restricted hypodissipative Navier--Stokes equation and the Fourier averaged Navier--Stokes equation introduced by Tao in \cite{TaoModifiedNS}.
    Both of these models involve the replacement of the nonlinear term from the (hypodissipative) Navier--Stokes equation
    \begin{equation}
    B(u,u)=\mathbb{P}_{df}((u\cdot\nabla) u),
    \end{equation}
    with an adjusted nonlinear term $\Tilde{B}(u,u)$. The nonlinear term in Tao's model equation has a key bound that is identical to the full Navier--Stokes nonlinearity
    \begin{equation}
        \left\|\Tilde{B}(u,u)\right\|_{L^2}
        \leq 
        \|u\|_{L^4} \|\nabla u\|_{L^4}.
    \end{equation}
    This bound also holds for the nonlinear term in the restricted model equation introduced in this paper, but for the restricted model, this bound is far from being sharp.
    For the nonlinearity in the Fourier-restricted hypodissipative Navier--Stokes equation, we have the much stronger bound
    \begin{equation}
    \left\|\Tilde{B}(u,u)\right\|_{L^2}
        \leq 
    C\|u\|_{\dot{H}^\frac{\log(3)}{4\log(2)}}^2,
    \end{equation}
    and it is this bound that means, in both the viscous and inviscid case, we can consider much rougher data for the restricted model equations, and that we have regularity criteria that are much stronger than the Beale-Kato-Majda criterion in terms of scaling.
    Crucially, the nonlinear terms in both model equations are skew-symmetric with
    \begin{equation}
        \left<\Tilde{B}(u,u),u\right>=0,
    \end{equation}
    just as for the actual Navier--Stokes equation,
    which results in an energy equality.
    
    The essential difference is that the model introduced by Tao in \cite{TaoModifiedNS} alters the structure of the self-advection nonlinearity $(u\cdot\nabla)u$, which is very important both physically and mathematically.
    The model equations considered in this paper maintain the self-advection 
    structure by taking the adjusted nonlinear term $\Tilde{B}$ to be given by
    \begin{equation}
        \Tilde{B}(u,u)=
        \mathbb{P}_\mathcal{M}((u\cdot\nabla)u).
    \end{equation}
    Because only the projection has been altered, the identities for enstrophy growth are the same in the hypodissipative Navier--Stokes equation and the Fourier-restricted hypodissipative Navier--Stokes equation. Both have the following equivalent identities:
    \begin{align}
        \frac{\diff}{\diff t}
        \frac{1}{2}\|\nabla u(\cdot,t)\|_{L^2}^2
        &=
        -\nu\|\nabla u\|_{\dot{H}^\alpha}^2
        -\left<-\Delta u,(u\cdot\nabla)u\right> \\
        \frac{\diff}{\diff t}
        \frac{1}{2}\|\omega(\cdot,t)\|_{L^2}^2
        &=
        -\nu\|\omega\|_{\dot{H}^\alpha}^2
        +\left<S,\omega\otimes\omega\right> \\
        \frac{\diff}{\diff t}
        \|S(\cdot,t)\|_{L^2}^2
        &= \label{EnstrophyIdentityIntro}
        -2\nu\|S\|_{\dot{H}^\alpha}^2
        -4\int\det(S),
    \end{align}
and analogous statements hold for the Euler and Fourier-restricted Euler equations with $\nu=0$. This does not hold for Tao's model equation in \cite{TaoModifiedNS}, which breaks the structure of vortex stretching.

    Because the growth of enstrophy is a key feature of the Navier--Stokes regularity problem, this is an important advance on the Fourier averaged model in \cite{TaoModifiedNS}.
    This is the first (hypodissipative) Navier--Stokes model equation that exhibits finite-time blowup and respects both the energy equality and the identity for enstrophy growth. The author previously proved finite-time blowup for the strain self-amplification model equation, which respects the identities for enstrophy growth above, but not the energy equality \cite{MillerStrainModel}. Furthermore, the structure of the nonlinearity was altered much more drastically for the strain self-amplification model equation than for the restricted model equations in this paper.
    
    The main weakness of our blowup result for the Fourier-restricted, hypodissipative Navier--Stokes equation relative to Tao's blowup result for the Fourier averaged Navier--Stokes equation is that we only prove blowup for the Fourier-restricted equation when $\alpha<\frac{\log(3)}{6\log(2)}\approx .264$,
    whereas the Fourier averaged Navier--Stokes equation exhibits blowup with the standard Laplacian dissipation term, corresponding to $\alpha=1$.
    Compared with the model equation introduced by Tao in \cite{TaoModifiedNS},
    the Fourier-restricted hypodissipative Navier--Stokes equation is arguably closer to the full Navier--Stokes equation in the sense of fluid mechanics, but farther from the full Navier--Stokes equation in the sense of harmonic analysis.
\end{remark}

\begin{remark}
    Campolina and Mailybaev previously considered a model equation for the Euler and Navier--Stokes equations restricting the Fourier modes to a logarithmic lattice \cites{CampolinaMailybaevPhysLett,CampolinaMailybaevNon}; 
    their work represents the first attempt at studying the finite-time blowup problem for the incompressible fluid equations by restricting the velocity to certain Fourier modes.
    In particular, they consider Fourier modes supported on $\Lambda^3$, where
    \begin{equation}
    \Lambda =\left\{0, \pm 1, 
    \pm \lambda 
    \pm \lambda^2, \pm \lambda^3, ...\right\},
    \end{equation}
    and $\lambda>1$.
    Note that $\lambda$ can only take certain values in order to obtain interactions between the Fourier modes: $\lambda=2$ is the only possible integer value; the golden mean and the plastic number are also considered in \cites{CampolinaMailybaevNon}, although in these cases the resulting velocity field is not periodic. 
    This has certain similarities in structure with the model taken under consideration here, because $\mathcal{M}$ also exhibits exponential growth of frequencies, although the lattice $\mathcal{M}$ has no scale invariance due to anisotropy, while the lattice $\Lambda^3$ has the scale invariance
    $\lambda^n \Lambda^3 \subset \Lambda^3$, for all $n\in\mathbb{N}$.

    There are also more crucial differences between the models. In order to exploit the existing blowup results for the dyadic model for Euler and Navier--Stokes, it is necessary that a frequency interact with a frequency at the same level in order to produce a frequency approximately double, which yields the dyadic structure. The incompressibility condition is a barrier to achieving this, because it guarantees that no Fourier mode can interact with itself in the $(u\cdot\nabla)u$ nonlinearity.
    We get around this difficulty by considering shells of frequencies with the same magnitude and a permutation invariance. This allows for frequencies to interact with other frequencies in the same shell creating a frequency in the next shell. For permutation and $\sigma$-mirror symmetric solutions, this can be treated as the self-interaction of each frequency shell, even if each individual frequency has no self-interaction. 
    For the Euler equation on log lattices, this would correspond to $\lambda=2$; however, in this case ``the spacing factor $\lambda=2$ does not provide a reliable model for the
blowup study, because the incompressibility condition together with a small number of triad
interactions cause degeneracies in coupling of different modes" \cites{CampolinaMailybaevNon}.
    For example, if we take a frequency with all positive components $k=(2^a,2^b,2^c)$,
    then the only interaction with another frequency with all positive coefficients would be
    $k+k=2k$, and this interaction cannot produce anything from the nonlinearity $(u\cdot\nabla)u$,
    due to the incompressibility constraint $k\cdot\hat{u}(k)=0$.

    It is worth pointing out that Waleffe \cite{Waleffe} showed that the 1D Burgers equation on the one dimensional log lattice $k_n=2^n$ reduces exactly to the dyadic model (this was also observed by Campolina and Mailybaev \cite{CampolinaMailybaevNon}). 
    This is similar in spirit to the reduction of the Fourier-restricted Euler and hypodissipative Navier--Stokes equations to an infinite system of ODEs very similar to the dyadic model in \Cref{RestrictedEulerBlowupIntro,RestrictedHypoBlowupIntro}, albeit in a dramatically simpler setting.
    All of the technical difficulties in the construction of blowup for the Fourier-restricted Euler/Navier--Stokes model equation on the lattice $\mathcal{M}$ in this sense come from the difficulty of getting a Burgers-shock type of blowup for the Fourier coefficients while maintaining the incompressibility constraint.

    The importance of permutation symmetry in our construction is that it allows the blowup via the self-interaction of Fourier modes from the scalar dyadic model to be lifted to the $(u\cdot\nabla)u$ nonlinearity of the Euler and Navier--Stokes equations. In particular, if $\nabla\cdot u=0$, then for all $k\in\mathbb{Z}^3, k\cdot\hat{u}(k)=0$.
    This implies that if
    \begin{equation}
        v(x)=\hat{u}(k)e^{2\pi i k\cdot x},
    \end{equation}
    then
    \begin{equation}
        (v\cdot\nabla)v=0,
    \end{equation}
    and so there is no self-interaction in the nonlinearity of a single Fourier mode. This is also true of a single sine wave.

    The self-interaction of the $n$-th order mode to produce the $n+1$-th order mode is what drives blowup for the dyadic model, so we need a way to get around this lack of a self-interaction term in the nonlinearity while retaining both the incompressibility constraint and the $(u\cdot\nabla)u$ nonlinearity. The key is that for a permutation-symmetric vector field, the Fourier mode $\hat{u}(2,1,0)$ fully determines the Fourier modes at all permutations of $(2,1,0)$. There is an interaction between the modes $\hat{u}(2,1,0)$ and $\hat{u}(2,0,1)$, for example, and we can treat this as a self-interaction, because the $(2,1,0)$ mode is interacting with a different mode that it completely determines by permutation symmetry, producing the next order Fourier mode $\hat{u}(4,1,1)$. Without making use of permutation symmetry, the dyadic model could be derived rigourously in terms of a Fourier restriction only from the 1D Burgers equation \cites{Waleffe,CampolinaMailybaevNon}---to frequencies that are powers of two in this case---not for the 3D Euler equation, so this is a significant step forward.
The derivation of the dyadic Euler and Navier--Stokes equations from a vector valued model in \cites{FriedlanderPavlovic,KatzPavlovic} is based on a nonlinearity that mimics the $(u\cdot\nabla)u$ nonlinearity in terms of the interaction of Littlewood-Paley coefficients heuristically in terms of scaling, but the material derivative of the Littlewood-Paley decomposition is not used directly. To use the language of \cite{KatzPavlovic}, the nonlinearity in the dyadic Navier--Stokes equation is not based on computing $(u^Q\cdot\nabla) u^{Q'}$, where $u^Q$ and $u^{Q'}$ are elements of an orthonormal, Littlewood-Paley dyadic decomposition.
By contrast, the whole basis of the nonlinearity for the Fourier-restricted model in this paper is computing $(u^k\cdot\nabla) u^{k'}$, where $u^k,u^{k'}$ are Fourier modes with $k,k'\in \mathcal{M}$. These computations are made in \Cref{AppendixModeInteraction}, and underlie all of the proofs in \Cref{DynamicsSection,BlowupSection}.

     Taking $\lambda$ to be the golden mean and the plastic number, Campolina and Mailybaev provide robust numerical evidence for the finite-time blowup of the Euler equation on a log lattice \cite{CampolinaMailybaevNon}, although in both of these cases the finite-time blowup problem remains open from the viewpoint of rigourous analysis. 
    Interestingly, Figure 4 in \cite{CampolinaMailybaevNon} suggests that the $\sigma$-axis in frequency space ($k_1=k_2=k_3$)
    seems to play the same role in the energy cascade for the Navier--Stokes equation on log lattices that it does for the Fourier-restricted hypodissipative Navier--Stokes equation considered in this paper. This is quite natural, because log lattices are symmetric with respect to permutation, but not with respect to generic rotations, making the $\sigma$-axis a special invariant.

  I would like to thank Prof. Miguel Bustamante for bringing the references \cites{CampolinaMailybaevNon,CampolinaMailybaevPhysLett} to my attention.
\end{remark}

\begin{remark}
    We have already noted that the
    regularity criteria in \Cref{ViscousStrainRegCritIntro,EulerStrainRegCritIntro} are the direct analogues of a regularity criterion proven by Neustupa and Penel \cite{NeustupaPenel1} for the Navier--Stokes equation.
    This regularity criterion is a direct motivation for the present work. The Anstaz for blowup is specifically chosen to yield
    an eigenvalue structure of $-2\lambda,\lambda,\lambda$ for the strain matrix at the origin.
    This is exactly what the \Cref{StrainOriginBlowupThmIntro} tells us occurs, because the matrix
    \begin{equation}
    \lambda
    \left(\begin{array}{ccc}
         0 & -1 & -1  \\
         -1 & 0 & -1 \\
         -1 & -1 & 0
    \end{array}\right),
\end{equation}
has eigenvalues $-2\lambda,\lambda,\lambda$, where axis $\spn(\sigma)$ is the eigenspace corresponding to the eigenvalue $-2\lambda$ and the plane $\spn(\sigma)^\perp$ is the eigenspace corresponding to the multiplicity two eigenvalue $\lambda$.
Our analysis permits us to understand blowup at the origin (including rates) very precisely, because $\nabla u$ has a Fourier cosine series, which of course simplifies dramatically at the origin, where each cosine function has value one.
\end{remark}

\begin{remark}
    The blowup Anstaz considered in this paper also has a similar geometric structure to the blowup Anstaz used by Elgindi to prove finite-time blowup for $C^{1,\alpha}$ solutions of the Euler equation \cite{Elgindi} and by Elgindi, Ghoul, and Masmoudi to prove finite-time blowup for $L^2 \cap C^{1,\alpha}$ solutions of the Euler equation \cite{ElgindiGhoulMasmoudi}. In both cases, there is stagnation point blowup at the origin, with $u(\Vec{0},t),\omega(\Vec{0},t)=0$, and with the strain matrix yielding axial compression and planar stretching at the origin. In their case, the Anstaz is axisymmetric about the $x_3$-axis, and the $x_3$-axis is the axis of compression. In our case, the Anstaz is permutation symmetric, and the axis of compression is the $\sigma$-axis.

    One crucial difference is that the symmetry considered in this paper is discrete, and does not reduce the dimension of the solution, whereas axisymmetry is a continuous symmetry that does reduce the dimension of the solution, which only depends on $r=\sqrt{x_1^2+x_2^2}$ and $z=x_3$ when the vector field is expressed in cylindrical coordinates.
    The blowup argument for $C^{1,\alpha}$ solutions of the Euler equation that are axisymmetric and swirl-free cannot possibly be extended to smooth solutions of the Euler equation, because Ukhovskii and Yudovich proved, in one of the classical results in the incompressible fluid mechanics literature, that smooth axisymmetric, swirl-free solutions of the Euler equation cannot form singularities in finite-time \cite{Yudovich}. The key to global regularity is the fact that $\frac{\omega_\theta}{r}$ is transported by the flow, which is related to the reduction of dimension. The fact that odd symmetry, permutation symmetry, and $\sigma$-mirror symmetry form a discrete symmetry group, with no corresponding reduction in dimension, means there is no such barrier for the blowup of smooth solutions, although of course the analysis for the full Euler and hypodissipative Navier--Stokes equations is dramatically more complicated than for the model equations considered in this paper, in which the nonlinearity is simplified drastically by the dyadic structure imposed by $\mathbb{P}_\mathcal{M}$. 
\end{remark}

\begin{remark}
    There is previous work on the Euler equation under discrete symmetries that has direct bearing on the problem of finite-time blowup for odd, permutation symmetric, $\sigma$-mirror symmetric solutions of Euler equation raised in \Cref{FullEulerOpenQuestion}.
    Elgindi and Jeong studied the Euler equation under octahedral symmetry \cite{ElgindiJeongOctahedral}. In this paper, they considered solutions of the Euler equation that are permutation symmetric and mirror symmetric about the planes $x_1=0, x_2=0, x_3=0$. Note that these three mirror symmetries combined imply the solution is odd.
Elgindi and Jeong prove finite-time blowup for solutions of the Euler equation on the infinite tetrahedron $x_1>x_2>x_3>0$, which is directly related to the group of symmetries considered \cite{ElgindiJeongOctahedral}. These solutions have $C^\alpha$ vorticity, but when they are extended to the whole space via symmetry, these solutions have bounded, but not continuous, vorticity. This means that, unlike the $C^{1,\alpha}$ blowup solutions considered by Elgindi in \cite{Elgindi}, these solutions are too rough on the whole space to be in a class where strong local wellposedness is known.
 There is a straightforward geometric reason, when considering permutation symmetric solutions, that oddness and mirror symmetry about the plane $x_1+x_2+x_3=0$ is a more likely symmetry class to form a stagnation point singularity at the origin than mirror symmetries for each of the coordinate planes. A smooth solution with the symmetries considered in \cite{ElgindiJeongOctahedral} must satisfy $\nabla u(\Vec{0},t)=0$; this makes proving blowup at the origin---which is the vertex of the tetrahedron corresponding to the symmetry group and therefore the natural candidate---much more difficult.
\end{remark}

\begin{remark}
    While the large majority of blowup results for model equations for the Euler and Navier--Stokes equations do not have a direct bearing on a possible Ansatz for finite-time blowup for the full Euler or Navier--Stokes equations, there are a series of papers on model equations for axisymmetric flows with swirl by Hou and a number of collaborators that have a direct relationship to a specific blowup configuration for the full Euler equation. Hou and Luo conjectured that smooth solutions of the axisymmetric Euler equations with swirl on a periodic cylinder could develop finite-time singularities at the boundary. In particular, they considered an odd ``tornado"-type blowup, where the swirl goes one direction in when $z>0$ and the opposite direction when $z<0$, and a singularity forms in finite-time at $r=1, z=0$, and provided substantial numerical evidence that such a blowup actually occurs \cite{HouLuoNumerical}. 
    
    Hou and Luo also proposed a one dimensional model equation that models the behaviour of this class of solutions at the $r=1$ boundary \cite{HouLuoModel}. Choi et al. proved finite-time blowup for this model equation \cite{ChoiEtAl}, providing further evidence of for this blowup scenario for the axisymmetric, 3D Euler equations with swirl on a periodic cylinder. 
    Most recently, Chen and Hou proved finite-time blowup for smooth solutions of the full 3D Euler equation based around this configuration \cite{HouChenEulerBlowup}.
    Among a large number of other results, too numerous to discuss here, it also bears mentioning that Hou and Lei proposed a model equation for the axisymmetric Euler equation with swirl where the effect of advection is neglected \cite{HouLei}, and Hou et al. proved that smooth solutions of this model equation can blowup in finite-time \cite{HouEtAl}.
    For a more complete list of references related to that line of research, see the references in \cite{HouChenEulerBlowup}.
\end{remark}

\section{Definitions and preliminaries}
\label{Prelims}

\begin{definition}
    For all $s\geq 0$, we will define the Hilbert space on the torus by
    \begin{equation}
    H^s\left(\mathbb{T}^3\right)
    =
    \left\{f\in L^2\left(\mathbb{T}^3\right):
    \|f\|_{H^s}<+\infty\right\},
    \end{equation}
    where the  $H^s\left(\mathbb{T}^3\right)$ norm is given by
    \begin{equation}
    \|f\|_{H^s}
    =
    \left(\sum_{k\in\mathbb{Z}^3}
    (1+4\pi^2|k|^2)^s |\hat{f}(k)|^2\right)^\frac{1}{2}
    =
    \left\|(1-\Delta)^\frac{s}{2} f\right\|_{L^2}.
    \end{equation}
    For all $s\geq 0$, we will define the mean-free Hilbert space on the torus by
    \begin{equation}
    \dot{H}^s\left(\mathbb{T}^3\right)
    =
    \left\{f\in H^s\left(\mathbb{T}^3\right):
    \int_{\mathbb{T}^3}f(x)\diff x=0 \right\},
    \end{equation}
    where the homogeneous $\dot{H}^s\left(\mathbb{T}^3\right)$ norm associated with this space is given by
    \begin{equation}
    \|f\|_{\dot{H}^s}
    =
    \left(\sum_{\substack{k\in\mathbb{Z}^3 \\ k\neq 0}}
    (4\pi^2|k|^2)^s |\hat{f}(k)|^2\right)^\frac{1}{2}
    =
    \left\|(-\Delta)^\frac{s}{2} f\right\|_{L^2}.
    \end{equation}
\end{definition}

\begin{remark}
    Note that the mean-free condition is a standard condition (sometimes described as drift-free) in the analysis of the Euler and Navier--Stokes equations on the torus, as it is convenient for inverting the Laplacian and for the Sobolev embedding
    $L^{q*} \hookrightarrow \dot{H}^{s}$, where $s>0$ and
    $\frac{1}{q*}=\frac{1}{2}-\frac{s}{3}$.
    We will therefore build this condition into our constraint space. The dyadic structure of the restricted model equations means that the results will be more convenient to prove in the homogeneous $\dot{H}^s$ norm, but it is fine to interchange these norms as convenient, because they are equivalent for mean-free functions.
    \end{remark}

    \begin{proposition} \label{SobolevEquivalentProp}
        Fix $s\geq 0$. For all $f\in \dot{H}^s\left(\mathbb{T}^3\right)$
        \begin{equation}
        \|f\|_{\dot{H}^s}
        \leq 
        \|f\|_{H^s}
        \leq 
        \left(1+\frac{1}{4\pi^2}\right)^\frac{s}{2}
        \|f\|_{\dot{H}^s}.
        \end{equation}
    \end{proposition}

    \begin{proof}
    Observe that for all $k\in\mathbb{Z}^3, k\neq 0$,
    we have $|k|\geq 1$, and therefore
    \begin{equation}
    1\leq 
    \frac{\left(1+4\pi^2|k|^2\right)^s}
    {\left(4\pi^2 |k|^2\right)^s}
    \leq
    \left(1+\frac{1}{4\pi^2}\right)^s.
    \end{equation}
    Therefore, we may conclude that
    \begin{equation}
    \sum_{\substack{k\in\mathbb{Z}^3 \\ k\neq 0}}
    (4\pi^2|k|^2)^s |\hat{f}(k)|^2
    \leq 
    \sum_{\substack{k\in\mathbb{Z}^3 \\ k\neq 0}}
    (1+4\pi^2|k|^2)^s |\hat{f}(k)|^2
    \leq 
    \left(1+\frac{1}{4\pi^2}\right)^s
    \sum_{\substack{k\in\mathbb{Z}^3 \\ k\neq 0}}
    (4\pi^2|k|^2)^s |\hat{f}(k)|^2,
    \end{equation}
    and this completes the proof.
    \end{proof}

    We will define the fractional Laplacian in terms of its Fourier symbol by
    \begin{equation}
        \mathcal{F}\left((-\Delta)^s f\right)(k)
        =\left(4\pi^2 |k|^2\right)^s \hat{f}(k).
    \end{equation}
    Note that this implies that 
    \begin{equation}
        (-\Delta)^s f(x)
        =
        \sum_{k\in\mathbb{Z}^3}
        \left(4\pi^2 |k|^2\right)^s \hat{f}(k)
        e^{2\pi i k\cdot x}.
    \end{equation}
    Also note that this is only well defined for $s<0$ when $\hat{f}(0)=0$.

\begin{proposition} \label{HomogenousSobolevProp}
    For all $0\leq s'< s$, we have the continuous embedding $\dot{H}^s \hookrightarrow \dot{H}^{s'}$,
    and in particular for all $f\in \dot{H}^s$,
    \begin{equation}
    \|f\|_{\dot{H}^{s'}}\leq \frac{1}{(2\pi)^{s-s'}}
    \|f\|_{\dot{H}^s}.
    \end{equation}
\end{proposition}

\begin{proof}
    It suffices to prove the inequality, and then the embedding must be continuous by definition.
    Using the fact that $\hat{f}(0)=0$, we can see that $|k|^2\geq 1$, for all
    $k\in\supp\left(\hat{f}\right)$, and therefore
    \begin{align}
    \|f\|_{\dot{H}^{s'}}^2
    &=
    (4\pi^2)^{s'}
    \sum_{\substack{k\in\mathbb{Z}^3 \\ k\neq 0}}
    (|k|^2)^{s'}|\hat{f}(k)|^2 \\
    &\leq 
    (4\pi^2)^{s'}
   \sum_{\substack{k\in\mathbb{Z}^3 \\ k\neq 0}}
    (|k|^2)^{s}|\hat{f}(k)|^2 \\
    &=
    (4\pi^2)^{s'-s} \|f\|_{\dot{H}^s}^2.
    \end{align}
    This completes the proof.
\end{proof}

Throughout we will use the shorthand
\begin{equation}
C_T H^s_x=
C\left([0,T];
H^s\left(\mathbb{T}^3\right)\right),
\end{equation}
and likewise
\begin{equation}
C^1_T H^s_x=
C^1\left([0,T];
H^s\left(\mathbb{T}^3\right)\right),
\end{equation}
where there spaces have the ordinary norms
\begin{equation}
    \|u\|_{C_T H^s_x}
    =
    \max_{0\leq t\leq T}\|u(\cdot,t)\|_{H^s},
\end{equation}
and
\begin{equation}
    \|u\|_{C^1_T H^s_x}
    =
    \max\left(\|u\|_{C_T H^s_x},
    \|\partial_t u\|_{C_T H^s_x}\right).
\end{equation}
We will use the shorthand $C_T \dot{H}^s_x$ and $C^1_T \dot{H}^s_x$ analogously for the mean-free Sobolev space and the corresponding homogeneous norm.

We will define the space $C^k\left(\mathbb{T}^3\right)$ by
\begin{equation}
C^k\left(\mathbb{T}^3\right)
=\left\{f\in C\left(\mathbb{T}^3\right):
D^j f\in C\left(\mathbb{T}^3\right), 
\text{for all } 0\leq j\leq k
\right\},
\end{equation}
where the $C^k$ norm is given by 
\begin{equation}
    \|f\|_{C^k}=\max_{0\leq j\leq k}
    \max_{x\in\mathbb{T}^3}
    |D^j f(x)|,
\end{equation}
and likewise for $C^k\left([0,T]\right)$ and
$C^k\left([0,T]\times \mathbb{T}^3\right)$.
We will also define spaces of smooth functions by 
\begin{align}
    C^\infty &=\bigcap_{k\in \mathbb{N}} C^k \\
    H^\infty &= \bigcap_{s\geq 0} H^s. 
\end{align}

\begin{remark}
    For all $k\in \mathbb{N}$ we have the continuous embedding $C^k\left(\mathbb{T}^3\right) \hookrightarrow H^k\left(\mathbb{T}^3\right)$, where this embedding follows immediately from the fact that
    \begin{equation}
    \|D^j f\|_{L^2\left(\mathbb{T}^3\right)}
    \leq 
    \|D^j f\|_{L^\infty\left(\mathbb{T}^3\right)}.
    \end{equation}
    We also have the continuous embedding 
    $H^s\left(\mathbb{T}^3\right)
    \hookrightarrow C^k\left(\mathbb{T}^3\right)$, 
    for all $s>k+\frac{3}{2}$, by Sobolev embedding.
    As a result of these two embeddings, $C^\infty\left(\mathbb{T}^3\right)=
    H^\infty\left(\mathbb{T}^3\right)$, and furthermore these spaces have the same topology, with the equivalent convergence conditions: i) $f_n \to f$ as $n\to +\infty$ if, for all $k\in \mathbb{N}$, $f_n \to f$ in $C^k$; ii) $f_n \to f$ as $n\to +\infty$, if, for all $s \geq 0$, $f_n \to f$ in $H^s$. 
    \end{remark}

We define the space $\mathcal{H}^s$, a dyadic analogue of the standard Hilbert space, following the conventions of 
\cites{FriedlanderPavlovic,KatzPavlovic,Cheskidov}.

\begin{definition} \label{DyadicHilbertIntro}
    For all $s\in\mathbb{R}$, the $\mathcal{H}^s$ norm will be given by
    \begin{equation}
    \|\psi\|_{\mathcal{H}^s}^2
    =\sum_{n=0}^{+\infty}
    \left(\sqrt{3}\right)^{2sn}\psi_n^2.
    \end{equation}
\end{definition}

\begin{remark}
    Note that this definition implies that
    $\psi\in\mathcal{H}^s
    \left(\mathbb{Z}^+\right)$
    if and only if
    $\left(\left(\sqrt{3}\right)^{sn}\psi_n
    \right)_{n\in\mathbb{Z}^+}
    \in L^2\left(\mathbb{Z}^+\right)$.
    We will see that for odd, permutation symmetric, $\sigma$-mirror symmetric vector fields, $u\in \dot{H}^s_\mathcal{M}$,
    the $\dot{H}^s$ norms of $u$ and the 
    $\mathcal{H}^{\Tilde{s}}$ norms of the Fourier coefficients $\psi$ are equivalent, where
    \begin{equation}
    \Tilde{s}=\frac{2\log(2)}{\log(3)}s.
    \end{equation}
\end{remark}

We will denote projections onto Hilbert spaces by $\mathbb{P}$, where $\mathbb{P}_\mathcal{M}$ is the projection onto $\dot{H}^s_\mathcal{M}$ and $\mathbb{P}_{df}$ is the projection onto $\dot{H}^s_{df}$, the space of divergence free vector fields. We have already defined $\dot{H}^s_\mathcal{M}$, but now we will define the space $\dot{H}^s_{df}$.

\begin{definition}
    For all $s\geq 0$ and for all
    $u\in \dot{H}^s\left(\mathbb{T}^3;\mathbb{R}^3\right)$,
    we will say that 
    $u\in \dot{H}^s_{df}
    \left(\mathbb{T}^3\right)$ if,
    for all $k\in \mathbb{Z}^3$,
    \begin{equation}
        k\cdot \hat{u}(k)=0.
    \end{equation}
\end{definition}

\begin{definition}
    We will say that a vector field $u\in H^s\left(\mathbb{T}^3;\mathbb{R}^3\right)$
    is odd if for all $x\in\mathbb{T}^3$
    \begin{equation}
    u(-x)=-u(x),
    \end{equation}
    and we will say that a vector field $u\in H^s\left(\mathbb{T}^3;\mathbb{R}^3\right)$
    is component-wise odd if each component $u_i$ is odd in $x_i$ and even in $x_j$ where $i\neq j$.
    Note that every component-wise odd vector field is odd, but the converse is not true.
\end{definition}

\begin{remark}
    The blowup results considered in this paper involve vector fields that are odd, but not component-wise odd. We have introduced the term component-wise odd for clarity, because this condition is sometimes used as a definition of oddness in the context of incompressible fluids. For example, the geometric setting of component-wise odd vector fields is used in Elgindi's proof of finite-time blowup for $C^{1,\alpha}$ solutions of the three dimensional Euler equation \cite{Elgindi} and by Iftime, Sideris, and Gamblin in their study of the growth of the support of vorticity for solutions of the two dimensional Euler equation \cite{IftimieSiderisGamblin}.
\end{remark}

Finally, we will go over some miscellaneous notation.
We will define the positive part of a function by
\begin{equation}
    f^+(x)=\max(f(x),0).
\end{equation}
For a divergence free vector field $u\in H^1_{df}$, we will take the strain matrix $S$ to be the symmetric gradient
\begin{equation}
    S_{ij}=\frac{1}{2}\left(
    \partial_i u_j+\partial_j u_i\right).
\end{equation}
We will refer denote the permutation that swaps the $ij$ entries of a vector and leaves the other entries alone by $P_{ij}$.
We will denote the projection onto the span of a vector $k\in \mathbb{R}^3, k\neq 0$, by
\begin{equation}
    P_{k}(v)=\frac{k\cdot v}{|k|^2}k,
\end{equation}
and the projection onto the orthogonal compliment by
\begin{equation}
    P_{k}^\perp(v)=v-\frac{k\cdot v}{|k|^2}k.
\end{equation}
In some very long calculations we will use $*$ to indicate multiplication, because the computations would be unreadable otherwise. We also use $*$ to denote convolution in other places, as is standard, but the difference will always be clear from context.

\section{Local and global wellposedness} \label{WellPosedSection}

In this section, we will prove local well posedness for the Fourier-restricted Euler and hypodissipative Navier--Stokes equations.
We will begin by defining a bilinear operator related to our nonlinearity and stating an important technical bound whose proof will be left to \Cref{AppendixNonlinearBound}.

\begin{lemma} \label{BilinearBoundLemma}
    Define the bilinear operator $B:
    \dot{H}^{\frac{s}{2}+\frac{\log(3)}{4\log(2)}}_\mathcal{M}
    \times
    \dot{H}^{\frac{s}{2}+\frac{\log(3)}{4\log(2)}}_\mathcal{M}
    \to 
    \dot{H}^s_\mathcal{M}$, for any $s\geq 0$ by
    \begin{equation}
    B(u,w)=-\frac{1}{2}\mathbb{P}_\mathcal{M}
    ((u\cdot\nabla)w+(w\cdot\nabla)u).
    \end{equation}
    This bilinear operator satisfies the bound
    \begin{equation}
    \|B(u,w)\|_{\dot{H}^s}
    \leq \mathcal{C}_s
    \|u\|_{\dot{H}^{\frac{s}{2}
    +\frac{\log(3)}{4\log(2)}}}
    \|w\|_{\dot{H}^{\frac{s}{2}
    +\frac{\log(3)}{4\log(2)}}}
    \end{equation}
    where $\mathcal{C}_s>0$ is a constant depending only on $s$.
\end{lemma}

\begin{remark}
    Throughout this section $\mathcal{C}_s$ will refer to the constant in \Cref{BilinearBoundLemma}, and we will take 
    \begin{equation}
    \mathcal{C}_*=\mathcal{C}_\frac{\log(3)}{2\log(2)}.
    \end{equation}
    Note that the Fourier-restricted Euler equation can be expressed in terms of $B$ by
    \begin{equation} \label{OdeB}
    \partial_t u= B(u,u).
    \end{equation}
\end{remark}

\begin{theorem} \label{RestrictedEulerExistenceThm}
    For all $u^0\in \dot{H}^\frac{\log(3)}
    {2\log(2)}_{\mathcal{M}}$, there exists a unique solution of the Fourier-restricted Euler equation $u\in C^1\left([0,T_{max});
    \dot{H}^\frac{\log(3)}
    {2\log(2)}_{\mathcal{M}}\right)$,
    where
    \begin{equation}
    T_{max}\geq \frac{1}{\mathcal{C}_*
    \left\|u^0\right\|_{\dot{H}
    ^\frac{\log(3)}{2\log(2)}}},
    \end{equation}
    and if $T_{max}<+\infty$, 
    then for all $0\leq t<T_{max}$,
    \begin{equation}
    \|u(\cdot,t)\|_{\dot{H}^\frac{\log(3)}{2\log(2)}}
    \geq 
    \frac{1}{\mathcal{C}_*(T_{max}-t)}.
    \end{equation}
\end{theorem}

\begin{proof}
    The proof will follow the classic Picard iteration method.
    We begin by defining the map
    $Q: C\left([0,T];\dot{H}^\frac{\log(3)}
    {2\log(2)}_\mathcal{M}\right)
    \to C\left([0,T];\dot{H}^\frac{\log(3)}
    {2\log(2)}_\mathcal{M}\right)$ 
    by
    \begin{equation}
    Q[u](\cdot,t)=
    u^0
    +\int_0^t B[u,u](\cdot,\tau)
    \diff\tau,
    \end{equation}
    where we fix
    \begin{equation}
    T<\frac{1}{4 C_* \left\|u^0
    \right\|_{\dot{H}^\frac{\log(3)}{2\log(2)}}}.
    \end{equation}
    This is map is useful, because
    \begin{equation}
    u=Q[u]
    \end{equation}
    expresses the Fourier-restricted Euler equation as a integral equation, and so existence reduces to finding a fixed point of the map $Q$. Applying \Cref{BilinearBoundLemma}, we find that
    \begin{equation}
    \left\|Q[u]\right\|_{C_T\dot{H}^\frac{\log(3)}
    {2\log(2)}_x}
    \leq 
    \left\|u^0\right\|_{
    \dot{H}^\frac{\log(3)}{2\log(2)}}
    +C_* T \left\|u\right\|_{C_T
    \dot{H}^\frac{\log(3)}{2\log(2)}_x}^2.
    \end{equation}
    Now let
    \begin{equation}
    X=\left\{u\in C\left([0,T];
    \dot{H}^\frac{\log(3)}
    {2\log(2)}_\mathcal{M}\right):
    \|u\|_{C_T\dot{H}^\frac{\log(3)}
    {2\log(2)}_x}\leq
    2\left\|u^0\right\|_{\dot{H}^
    \frac{\log(3)}{2\log(2)}}\right\}.
    \end{equation}
    Observe that for all $u\in X$,
    \begin{align}
    \left\|Q[u]\right\|_{C_T\dot{H}^\frac{\log(3)}
    {2\log(2)}_x}
    &\leq 
    \left\|u^0\right\|_{
    \dot{H}^\frac{\log(3)}{2\log(2)}}
    +4C_* T \left\|u^0\right\|_{
    \dot{H}^\frac{\log(3)}{2\log(2)}}^2 \\
    &\leq 
    2\left\|u^0\right\|_{
    \dot{H}^\frac{\log(3)}{2\log(2)}},
    \end{align}
    and so $u\in X$ implies $Q[u]\in X$, and 
    $Q:X\to X$.

    Likewise, observe that for all $u,w\in X$,
    we have
    \begin{align}
    Q[u](\cdot,t)-Q[w](\cdot,t)
    &=
    \int_0^t B[u,u](\cdot,\tau)
    -B[w,w](\cdot,\tau) \diff\tau \\
    &=
    \int_0^t B[u+w,u-w](\cdot,\tau)
    \diff\tau.
    \end{align}
    Again applying \Cref{BilinearBoundLemma},
    and letting $r=4\mathcal{C}_* T 
    \left\|u^0\right\|_{
    \dot{H}^\frac{\log(3)}{2\log(2)}}<1$, we find that
    \begin{align}
    \left\|Q[u]\right\|_{
    C_T\dot{H}^\frac{\log(3)}{2\log(2)}_x}
    &\leq 
    \mathcal{C}_* T
    \left\|u+w\right\|_{
    C_T\dot{H}^\frac{\log(3)}{2\log(2)}_x}
    \left\|u-w\right\|_{
    C_T\dot{H}^\frac{\log(3)}{2\log(2)}_x} \\
    &\leq 
    4\mathcal{C}_* T 
    \left\|u^0\right\|_{
    \dot{H}^\frac{\log(3)}{2\log(2)}}
    \left\|u-w\right\|_{
    C_T\dot{H}^\frac{\log(3)}{2\log(2)}_x} \\
    &=
    r \left\|u-w\right\|_{
    C_T\dot{H}^\frac{\log(3)}{2\log(2)}_x}.
    \end{align}
    Therefore, $Q:X\to X$ is a contraction mapping, and so by the Banach fixed point theorem, there exists a unique fixed point $u\in X$ such that
    \begin{equation}
        Q[u]=u.
    \end{equation}

    By the fundamental theorem of calculus, we can see that $\partial_t u=B(u,u)$, and so $\partial_t u\in C\left([0,T];\dot{H}^\frac{\log(3)}
    {2\log(2)}_\mathcal{M}\right)$, and consequently
    $u\in C^1\left([0,T];\dot{H}^\frac{\log(3)}
    {2\log(2)}_\mathcal{M}\right)$.
    Note that uniqueness in $X$, combined with continuity-in-time guarantees that this solution is unique in $C^1\left([0,T];\dot{H}^\frac{\log(3)}
    {2\log(2)}_\mathcal{M}\right)$.
    Recalling that $T<\frac{1}{4 C_* \left\|u^0
    \right\|_{\dot{H}^\frac{\log(3)}{2\log(2)}}}$
    is arbitrary, we have now shown that there is a unique solution of the Fourier-restricted Euler equation $u\in C^1\left([0,T_{max});
    \dot{H}^\frac{\log(3)}
    {2\log(2)}_{\mathcal{M}}\right)$,
    where
    \begin{equation}
    T_{max}\geq \frac{1}{4\mathcal{C}_*
    \left\|u^0\right\|_{\dot{H}
    ^\frac{\log(3)}{2\log(2)}}}.
    \end{equation}

    Now we will improve the constant in this bound.
    Note that if $T_{max}<+\infty$,
    then 
    \begin{equation}
    \lim_{t\to T_{max}} 
    \|u(\cdot,t)\|_{\dot{H}^\frac
    {\log(3)}{2\log(2)}}
    =
    +\infty,
    \end{equation}
    because otherwise the solution could be extended beyond $T_{max}$. 
    Observe that for all $0\leq t<T_{max}$,
    \begin{align}
    \frac{\diff}{\diff t}
    \|u(\cdot,t)\|_{\dot{H}^\frac{
    \log(3)}{2\log(2)}}^2
    &=
    2\left<u,B(u,u)\right>_{\dot{H}^\frac
    {\log(3)}{2\log(2)}} \\
    &\leq 
    2\|u\|_{\dot{H}^\frac{
    \log(3)}{2\log(2)}}
    \|B(u,u)\|_{\dot{H}^\frac{
    \log(3)}{2\log(2)}} \\
    &\leq 
    2\mathcal{C}_*
    \|u\|_{\dot{H}^\frac{
    \log(3)}{2\log(2)}}^3,
    \end{align}
    and therefore
    \begin{equation}
    \frac{\diff}{\diff t}
    \|u(\cdot,t)\|_{\dot{H}^\frac{
    \log(3)}{2\log(2)}}
    \leq \mathcal{C}_*
    \|u(\cdot,t)\|_{\dot{H}^\frac{
    \log(3)}{2\log(2)}}^2.
    \end{equation}
    Integrating this differential inequality, we find that for all $0\leq t<T_{max}$
    \begin{equation}
    \|u(\cdot,t)\|_{\dot{H}^\frac{
    \log(3)}{2\log(2)}}
    \leq 
    \frac{\left\|u^0(\cdot,t)\right\|_{
    \dot{H}^\frac{\log(3)}{2\log(2)}}}
    {1-\mathcal{C}_*
    \left\|u^0(\cdot,t)\right\|_{
    \dot{H}^\frac{\log(3)}{2\log(2)}}t},
    \end{equation}
    and so clearly
    \begin{equation}
    T_{max}\geq \frac{1}{\mathcal{C}_*
    \left\|u^0\right\|_{\dot{H}
    ^\frac{\log(3)}{2\log(2)}}}.
    \end{equation}
    Treating $u(\cdot,t)$ as initial data, and applying this estimate, we find that for all $0\leq t<T_{max}$,
    \begin{equation}
    T_{max}-t
    \geq 
    \frac{1}{\mathcal{C}_*
    \|u(\cdot,t)\|_{\dot{H}^\frac{
    \log(3)}{2\log(2)}}},
    \end{equation}
    and this completes the proof.
\end{proof}

\begin{proposition} \label{RestrictedEulerEnergyProp}
    Suppose $u\in C^1\left([0,T_{max});
    \dot{H}^\frac{\log(3)}
    {2\log(2)}_{\mathcal{M}}\right)$
    is a solution of the Fourier-restricted Euler equation.
    Then for all $0\leq t<T_{max}$,
    \begin{equation}
    \|u(\cdot,t)\|_{L^2}^2
    =
    \left\|u^0\right\|_{L^2}^2.
    \end{equation}
\end{proposition}

\begin{proof}
First we compute that 
\begin{equation}
\frac{\diff}{\diff t}
    \frac{1}{2}\|u\|_{L^2}^2
    =
    -\left<\mathbb{P}_\mathcal{M}((u\cdot\nabla)u),u\right>
\end{equation}
Observe that by \Cref{BilinearBoundLemma} we have $\mathbb{P}_\mathcal{M}((u\cdot\nabla)u)\in \dot{H}^\frac{\log(3)}{2\log(2)}_\mathcal{M}$, and so the inner product is well defined. The result follows from the divergence free constraint, but because we don't have sufficient regularity to integrate by parts, we will prove the bound by Fourier truncation.
Define the Fourier truncation operator
\begin{equation}
J_N(u)=\sum_{|k|\leq N}
\hat{u}(k) e^{2\pi ik\cdot x}.
\end{equation}
Observe that $J_N(u) \to u$ and $P_{\mathcal{M}}\left(J_N(u)\cdot\nabla J_N(u)\right) \to 
P_{\mathcal{M}}(u\cdot\nabla u)$
in $\dot{H}^\frac{\log(3)}{2\log(2)}$,
as $N\to \infty$,
and so
\begin{equation}
\lim_{N\to +\infty}
-\left<\mathbb{P}_\mathcal{M}((J_N u\cdot\nabla)J_N u),J_N u\right>
=
-\left<\mathbb{P}_\mathcal{M}((u\cdot\nabla)u),u\right>.
\end{equation}
Therefore, it suffices to prove that for all $N\in\mathbb{N}$,
\begin{equation}
    -\left<\mathbb{P}_\mathcal{M}((J_N u\cdot\nabla)J_N u),J_N u\right>
    =0.
\end{equation}
Observe that $J_N(u)\in \dot{H}^\infty_\mathcal{M}$ is smooth and divergence free, so dropping the projection and integrating by parts, we find that
\begin{align}
-\left<\mathbb{P}_\mathcal{M}
((J_N u\cdot\nabla)J_N u),J_N u\right>
&=
-\left<(J_N u\cdot\nabla)J_N u,J_N u\right> \\
&=
\left<J_N u,(J_N u\cdot\nabla)J_N u\right> \\
&=
0,
\end{align}
and this completes the proof.
\end{proof}

\begin{theorem}
    For all $u^0\in \dot{H}^s_{\mathcal{M}}, 
    s>\frac{\log(3)}{2\log(2)}$, there exists a unique solution of the Fourier-restricted Euler equation $u\in C^1\left([0,T_{max});
    \dot{H}^s_{\mathcal{M}}\right)$,
    where
    \begin{equation}
    T_{max}\geq \frac{1}{\mathcal{C}_*
    \left\|u^0\right\|_{\dot{H}
    ^\frac{\log(3)}{2\log(2)}}},
    \end{equation}
    and if $T_{max}<+\infty$, 
    then for all $0\leq t<T_{max}$,
    \begin{equation}
    \|u(\cdot,t)\|_{\dot{H}^\frac{\log(3)}{2\log(2)}}
    \geq 
    \frac{1}{\mathcal{C}_*(T_{max}-t)}.
    \end{equation}
\end{theorem}

\begin{proof}
    First observe that for all 
    $s>\frac{\log(3)}{2\log(2)}$,
    we have
    \begin{equation}
    \frac{s}{2}+\frac{\log(3)}{4\log(2)}
    <s,
    \end{equation}
    and so by \Cref{BilinearBoundLemma} and \Cref{HomogenousSobolevProp}, we can see that
    \begin{align}
    \|B(u,w)\|_{\dot{H}^s}
    &\leq 
    \mathcal{C}_s 
    \|u\|_{\dot{H}^{
    \frac{s}{2}+\frac{\log(3)}{4\log(2)}}}
    \|w\|_{\dot{H}^{
    \frac{s}{2}+\frac{\log(3)}{4\log(2)}}} \\
    &\leq 
    C_s \|u\|_{\dot{H}^s} \|w\|_{\dot{H}^s}.
    \end{align}
    Therefore, proceeding exactly as in \Cref{RestrictedEulerExistenceThm},
    we can see there exists a strong solution of the Fourier-restricted Euler equation 
    $u\in C^1\left([0,T_{max});
    \dot{H}^s_{\mathcal{M}}\right)$,
    where
    $T_{max}\geq \frac{1}{C_s\left\|u^0\right\|_{
    \dot{H}^s}}$.
    Because we clearly have $u^0\in \dot{H}^\frac{\log(3)}{2\log(2)}_\mathcal{M}$,
    we can apply \Cref{RestrictedEulerExistenceThm},
    and find that 
    there exists a strong solution of the Fourier-restricted Euler equation 
    $u\in C^1\left([0,T^*);
    \dot{H}^\frac{\log(3)}{2\log(2)}
    _{\mathcal{M}}\right)$,
    where
    $T^*\geq \frac{1}{\mathcal{C}_*\left\|u^0\right\|_{
    \dot{H}^\frac{\log(3)}{2\log(2)}}}$, 
    and $T^*\geq T_{max}$, and these solutions agree up to $T_{max}$ by uniqueness.
    It remains only to show that $T_{max}=T^*$.
    Compute that for all $0\leq t<T_{max}$,
    \begin{align}
    \frac{\diff}{\diff t}
    \|u\|_{\dot{H}^s}^2
    &=
    2\left<B(u,u),u\right>_{\dot{H}^s} \\
    &\leq 
    2\mathcal{C}_s \|u\|_{\dot{H}^s}
    \|u\|_{\dot{H}^{\frac{s}{2}
    +\frac{\log(3)}{4\log(2)}}}^2 \\
    &=
    2\mathcal{C}_s \|u\|_{\dot{H}^s}^2
    \|u\|_{\dot{H}^{\frac{\log(3)}{2\log(2)}}},
    \end{align}
    interpolating between $\dot{H}^\frac{
    \log(3)}{2\log(2)}$ and $\dot{H}^s$.
    Applying Gr\"onwall's inequality, we find 
    that for all $0\leq t<T_{max}$,
    \begin{equation}
    \|u(\cdot,t)\|_{\dot{H}^s}^2
    \leq
    \left\|u^0(\cdot,t)\right\|_{\dot{H}^s}^2
    \exp\left(2\mathcal{C}_s
    \int_0^t\|u(\cdot,\tau)\|_{\dot{H}^
    \frac{\log(3)}{2\log(2)}}
    \diff\tau\right).
    \end{equation}
    This implies that $T_{max}=T_*$,
    which completes the proof.
\end{proof}

\begin{remark}
    We have now proven \Cref{RestrictedEulerExistenceThmIntro}, which was broken into multiple pieces for ease of reading. The purpose of the above theorem is that the $\dot{H}^\frac{\log(3)}{2\log(2)}$ norm controls regularity even for smoother solutions, which gives us a concise definition of blowup. For any strong solution---of arbitrarily high regularity---there is blowup at $T_{max}<+\infty$ if and only if
    \begin{equation}
    \lim_{t \to T_{max}}
    \|u(\cdot,t)\|_{\dot{H}^\frac{\log(3)}{2\log(2)}}
    =+\infty.
    \end{equation}
    We also have a strong local stability result in $\dot{H}^\frac{\log(3)}{2\log(2)}$.
\end{remark}

\begin{theorem} \label{StabilityRestrictedEuler}
    Suppose $u,w \in C^1\left([0,T];
    \dot{H}^\frac{\log(3)}{2\log(2)}_{\mathcal{M}}\right)$,
    are solutions of the Fourier-restricted Euler equation.
    Then for all $0\leq t\leq T$,
    \begin{equation}
    \|(u-w)(\cdot,t)
    \|_{\dot{H}^\frac{\log(3)}{2\log(2)}}^2
    \leq 
    \left\|u^0-w^0\right\|_{\dot{H}^
    \frac{\log(3)}{2\log(2)}}^2
    \exp\left(2\mathcal{C}_*\int_0^t
    \|(u+w)(\cdot,t)\|_{\dot{H}^\frac{\log(3)}{2\log(2)}}
    \diff\tau\right).
    \end{equation}
\end{theorem}

\begin{proof}
Observe that
\begin{align}
    \partial_t(u-w)
    &=
    B(u,u)-B(w,w) \\
    &=
    B(u+w,u-w).
\end{align}
Therefore we may compute that
\begin{align}
    \frac{\diff}{\diff t}
    \|(u-w)(\cdot,t)\|_{\dot{H}^s}^2
    &=
    2\Big<u-w,B(u+w,u-w)\Big>_{
    \dot{H}^\frac{\log(3)}{2\log(2)}} \\
    &\leq 
    2\mathcal{C}_* 
    \|u+w\|_{\dot{H}^\frac{\log(3)}{2\log(2)}}
    \|u-w\|_{\dot{H}^\frac{\log(3)}{2\log(2)}}^2.
\end{align}
Applying Gr\"onwall's inequality, this completes the proof.
\end{proof}

\subsection{Viscous Case}
In this section, we will develop the local wellposedness theory for the Fourier-restricted hypodissipative Navier--Stokes equation, proving \Cref{RestrictedHypoExistenceThmIntro}, which will be broken into several pieces for ease of reading.

\begin{proposition} \label{SmoothingProp}
    Suppose $s\geq 0$ and $r,\alpha,t>0$.
    Then for all $u\in \dot{H}^s$
    \begin{equation}
    \left\|e^{-t (-\Delta)^\alpha}
    u\right\|_{\dot{H}^{s+r}}
    \leq 
    \frac{C_{r,\alpha}}
    {t^\frac{r}{2\alpha}}
    \|u\|_{\dot{H}^s},
    \end{equation}
    where
    \begin{equation}
    C_{r,\alpha}
    =
    \left(\frac{r}{2\alpha}
    \right)^\frac{r}{2\alpha}
    e^{-\frac{r}{2\alpha}}.
    \end{equation}
\end{proposition}

\begin{proof}
Begin by recalling that
\begin{equation}
    \widehat{e^{-t(-\Delta)^\alpha}u}(k)
    =e^{-\left(4\pi^2|k|^2\right)^\alpha t} 
    \hat{u}(k).
\end{equation}
    Next compute that
    \begin{align}
    \left\|e^{-t (-\Delta)^\alpha}
    u\right\|_{\dot{H}^{s+r}}
    &=
    \sum_{\substack{k\in\mathbb{Z}^3 \\
    k\neq 0}}
    \left(4\pi^2 |k|^2\right)^{s+r}
    e^{-\left(4\pi^2|k|^2\right)^\alpha 2t}
    \left|\hat{u}(k)\right|^2 \\
    &=
    \sum_{\substack{k\in\mathbb{Z}^3 \\
    k\neq 0}}
    \left(4\pi^2 |k|^2\right)^r
    e^{-\left(4\pi^2|k|^2\right)^\alpha 2t}
    \left(4\pi^2 |k|^2\right)^s
    \left|\hat{u}(k)\right|^2 \\
    &\leq 
    \left(\sup_{\substack{k\in\mathbb{Z}^3 \\
    k\neq 0}}
    \left(4\pi^2 |k|^2\right)^r
    e^{-\left(4\pi^2|k|^2\right)^\alpha 2t}
    \right)
    \sum_{\substack{k\in\mathbb{Z}^3 \\
    k\neq 0}}
    \left(4\pi^2 |k|^2\right)^s
    \left|\hat{u}(k)\right|^2 \\
    &\leq 
    \frac{1}{(2t)^\frac{r}{\alpha}}
    \left(\sup_{\tau>0}
    \tau^\frac{r}{\alpha}e^{-\tau}\right)
    \|u\|_{\dot{H}^s}^2,
    \end{align}
    where we have made the substitution $\tau= 
    \left(4\pi^2|k|^2\right)^\alpha 2t$.
    Observe that for all $a>0$,
    \begin{equation}
    \sup_{\tau>0}
    \tau^\frac{r}{a}e^{-\tau}
    = 
    a^a e^{-a},
    \end{equation}
    and therefore
    \begin{equation}
    \left\|e^{-t (-\Delta)^\alpha}
    u\right\|_{\dot{H}^{s+r}}^2
    \leq 
    \frac{\left(
    \left(\frac{r}{\alpha}\right)^{
    \left(\frac{r}{\alpha}\right)}
    e^{-\frac{r}{\alpha}}\right)}
    {(2t)^\frac{r}{\alpha}}
    \|u\|_{\dot{H}^s}^2,
    \end{equation}
    which completes the proof.
\end{proof}

\begin{definition} \label{MildDefinition}
    We will say that $u\in C\left([0,T];
    \dot{H}^s_\mathcal{M}\right)$, 
    where $s>\frac{\log(3)}{2\log(2)}-2\alpha$ and $s\geq 0$, is a mild solution of the Fourier-restricted, hypodissipative Navier--Stokes equation, if for all $0<t\leq T$,
    \begin{equation}
    u(\cdot,t)=e^{-\nu t(-\Delta)^\alpha}u^0
    -\int_0^t e^{-\nu(t-\tau)(-\Delta)^\alpha}
    B(u,u)(\cdot,\tau).
    \end{equation}
\end{definition}

\begin{theorem} \label{ExistenceRestrictedNS}
    Suppose $u^0\in \dot{H}^s_\mathcal{M}$, where
    $\frac{\log(3)}{2\log(2)}-2\alpha
    <s\leq \frac{\log(3)}{2\log(2)}$ and $s\geq 0$.
    Then there exists a unique mild solution of the Fourier-restricted, hypodissipative Navier--Stokes equation, $u\in C\left([0,T_{max});
    \dot{H}^s_{\mathcal{M}}\right)$,
    where
    \begin{equation}
    T_{max}\geq 
    \frac{\nu^{\frac{1}{\rho}-1}}
    {\left(C_{s,\alpha}\left\|u^0\right\|_{\dot{H}^s}
    \right)^\frac{1}{\rho}}
    \end{equation}
    with $C_{s,\alpha}>0$ is an absolute constant independent of $u^0$ and $\nu$, and
    \begin{equation}
    \rho=1-\frac{\frac{\log(3)}{2\log(2)}-s}{2\alpha}.
    \end{equation}
    Furthermore, if $T_{max}<+\infty$, 
    then for all $0\leq t<T_{max}$,
    \begin{equation}
    \left\|u(\cdot,t)\right\|_{\dot{H}^s}
    \geq 
    \frac{\nu^{1-\rho}}{C_{s,\alpha}
    \left(T_{max}-t\right)^\rho}.
    \end{equation}
\end{theorem}

\begin{proof}
    We begin by fixing
    \begin{equation}
    T<\frac{\nu^{\frac{1}{\rho}-1}}
    {\left(C_{s,\alpha}\left\|u^0\right\|_{\dot{H}^s}
    \right)^\frac{1}{\rho}},
    \end{equation}
    noting that 
    \begin{equation}
        T^\rho
        <
        \frac{\nu^{1-\rho}}{C_{s,\alpha}
        \left\|u^0\right\|_{\dot{H}^s}},
    \end{equation}
    We define the map
    $Q: C\left([0,T];\dot{H}^s_\mathcal{M}\right)
    \to C\left([0,T];\dot{H}^s_\mathcal{M}\right)$ 
    by
    \begin{equation}
    Q[u](\cdot,t)=
    e^{-\nu t(-\Delta)^\alpha}u^0
    +\int_0^t e^{-\nu(t-\tau)(-\Delta)^\alpha}
    B[u,u](\cdot,\tau) \diff\tau.
    \end{equation}
    Note that $u\in C\left([0,T];
    \dot{H}^s_\mathcal{M}\right)$ is a mild solution
    if 
    \begin{equation}
        u=Q[u],
    \end{equation}
    so the existence of a mild solution immediately reduces to a fixed point argument.
    Now define
    \begin{equation}
    b=\frac{\frac{\log(3)}{2\log(2)}-s}
    {2\alpha},
    \end{equation}
    and observe that $0\leq b<1$.
    Note that $\rho=1-b$, so we can also check that $0<\rho\leq 1$.
    Applying \Cref{SmoothingProp} and \Cref{BilinearBoundLemma}, we find that
    for all $0< t\leq T$
    \begin{align}
    \|Q[u](\cdot,t)\|_{\dot{H}^s}
    &\leq
    \left\|u^0\right\|_{\dot{H}^s}
    +\int_0^t \left\|
    e^{-\nu(t-\tau)(-\Delta)^\alpha}
    B(u,u)(\cdot,\tau)\right\|_{\dot{H}^s}
    \diff\tau \\
    &\leq 
    \left\|u^0\right\|_{\dot{H}^s}
    +C\int_0^t \frac{1}{\nu^b(t-\tau)^b}
    \|B(u,u)(\cdot,\tau)\|_{\dot{H}^{s-2\alpha b}} 
    \diff\tau \\
    &=
    \left\|u^0\right\|_{\dot{H}^s}
    +C\int_0^t \frac{1}{\nu^b(t-\tau)^b}
    \|B(u,u)(\cdot,\tau)\|_{\dot{H}^{
    2s-\frac{\log(3)}{2\log(2)}}} 
    \diff\tau\\
    &\leq 
    \left\|u^0\right\|_{\dot{H}^s}
    +C\int_0^t \frac{1}{\nu^b(t-\tau)^b}
    \|u(\cdot,\tau)\|_{\dot{H}^s}^2 
    \diff\tau\\
    &\leq 
    \left\|u^0\right\|_{\dot{H}^s}
    +C\|u\|_{C_T \dot{H}^s_x}
    \int_0^t \frac{1}{\nu^b(t-\tau)^b}  
    \diff\tau \\
    &=
    \left\|u^0\right\|_{\dot{H}^s}
    +\frac{C_{s,\alpha} t^{1-b}}{4\nu^b}
    \|u\|_{C_T \dot{H}^s_x}^2.
    \end{align}
    Therefore, we can clearly see that
    \begin{equation} \label{KeyMildBound1}
    \|Q[u]\|_{C_T \dot{H}^s_x}
    \leq 
    \left\|u^0\right\|_{\dot{H}^s}
    +\frac{C_{s,\alpha} T^\rho}{4\nu^{1-\rho}}
    \|u\|_{C_T \dot{H}^s_x}^2.
    \end{equation}

    Similarly to the inviscid case, let
    \begin{equation}
    X=\left\{u\in C\left([0,T];
    \dot{H}^s_\mathcal{M}\right):
    \|u\|_{C_T\dot{H}^s_x}\leq
    2\left\|u^0\right\|_{\dot{H}^s}\right\}.
    \end{equation}
    Recall that by definition 
    \begin{equation}
    \frac{C_{s,\alpha} T^\rho}{\nu^{1-\rho}}
    <
    \frac{1}{\left\|u^0\right\|_{\dot{H}^s}},
    \end{equation}
    so we can apply \eqref{KeyMildBound1} and find that if 
    \begin{equation}
    \|u\|_{C_T \dot{H}^s_x}\leq
    2\left\|u^0\right\|_{\dot{H}^s}, 
    \end{equation}
    then
    \begin{equation}
    \|Q[u]\|_{C_T \dot{H}^s_x}\leq
    2\left\|u^0\right\|_{\dot{H}^s}.
    \end{equation}
    This implies that $Q:X \to X$.

    Next observe that for all $u,w \in X$,
    \begin{equation}
    Q[u](\cdot,t)-Q[w](\cdot,t)
    =
    \int_0^t e^{-\nu(t-\tau)(-\Delta)^\alpha}
    B(u+w,u-w)(\cdot,\tau) \diff\tau.
    \end{equation}
    Therefore, using the same bounds as in \eqref{KeyMildBound1}, we can compute that
    for all $u,w\in X$,
    \begin{align}
    \|Q[u](\cdot,t)-Q[w](\cdot,t)\|_{\dot{H}^s}
    &\leq
    \int_0^t \left\|
    e^{-\nu(t-\tau)(-\Delta)^\alpha}
    B(u+w,u-w)(\cdot,\tau)\right\|_{\dot{H}^s}
    \diff\tau \\
    &\leq 
    C\int_0^t \frac{1}{\nu^b(t-\tau)^b}
    \|B(u+w,u-w)(\cdot,\tau)\|_{\dot{H}^{
    2s-\frac{\log(3)}{2\log(2)}}} 
    \diff\tau\\
    &\leq 
    C\int_0^t \frac{1}{\nu^b(t-\tau)^b}
    \|(u+w)(\cdot,\tau)\|_{\dot{H}^s}
    \|(u-w)(\cdot,\tau)\|_{\dot{H}^s}
    \diff\tau\\
    &\leq 
    C \|u+w\|_{C_T \dot{H}^s_x}
    \|u-w\|_{C_T \dot{H}^s_x}
    \int_0^t \frac{1}{\nu^b(t-\tau)^b}
    \diff\tau\\
    &=
    \frac{C_{s,\alpha} t^\rho}{4\nu^b}
    \|u+w\|_{C_T \dot{H}^s_x}
    \|u-w\|_{C_T \dot{H}^s_x} \\
    &\leq 
    \frac{C_{s,\alpha} t^\rho}{\nu^b}
    \left\|u^0\right\|_{C_T \dot{H}^s_x}
    \|u-w\|_{C_T \dot{H}^s_x},
    \end{align}
    We may therefore conclude that for all
    $u,w\in X$
    \begin{equation}
    \left\|Q[u]-Q[w]\right\|_{C_T \dot{H}^s_x}
    \leq 
    \frac{C_{s,\alpha} T^\rho}{\nu^b}
    \left\|u^0\right\|_{C_T \dot{H}^s_x}
    \|u-w\|_{C_T \dot{H}^s_x}.
    \end{equation}
    Recall that by hypothesis, 
    \begin{equation}
    \frac{C_{s,\alpha} T^\rho}{\nu^b}
    \left\|u^0\right\|_{C_T \dot{H}^s_x}
    <1,
    \end{equation}
    and so by the $Q$ is a contraction on $X$, and by the Banach fixed point theorem, there exists a unique $u\in X$, such that
    \begin{equation}
    Q[u]=u.
    \end{equation}
    Combined with continuity-in-time, this tells us there exists a unique mild solution of the Fourier-restricted hypodissipative Navier--Stokes equation
    $C\left([0,T];\dot{H}^s_\mathcal{M}\right)$.
    
    We have now proven existence locally in time, and uniqueness up until blowup. The proof holds for any
    \begin{equation}
    T<\frac{\nu^{\frac{1}{\rho}-1}}
    {\left(C_{s,\alpha}\left\|u^0\right\|_{\dot{H}^s}
    \right)^\frac{1}{\rho}},
    \end{equation}
    and so clearly
    \begin{equation}
    T_{max}\geq 
    \frac{\nu^{\frac{1}{\rho}-1}}
    {\left(C_{s,\alpha}\left\|u^0\right\|_{\dot{H}^s}
    \right)^\frac{1}{\rho}}.
    \end{equation}
    Likewise, applying the semigroup method at any time $0<t<T_{max}$, to continue the solution, we can see that
    \begin{equation}
    T_{max}-t\geq 
    \frac{\nu^{\frac{1}{\rho}-1}}
    {\left(C_{s,\alpha}\|u(\cdot,t)\|_{\dot{H}^s}
    \right)^\frac{1}{\rho}},
    \end{equation}
    and so if $T_{max}<+\infty$, 
    then for all $0\leq t<T_{max}$,
    \begin{equation}
    \left\|u(\cdot,t)\right\|_{\dot{H}^s}
    \geq 
    \frac{\nu^{1-\rho}}{C_{s,\alpha}
    \left(T_{max}-t\right)^\rho},
    \end{equation}
    and this completes the proof.
\end{proof}

\begin{theorem}
    Suppose $u\in C\left([0,T_{max});
    \dot{H}^s_\mathcal{M}\right)$, where 
    $\frac{\log(3)}{2\log(2)}-2\alpha
    <s\leq \frac{\log(3)}{2\log(2)}$ and $s\geq 0$, is a mild solution of the Fourier-restricted, hypodissipative Navier--Stokes equation.
    Then for all positive times up until blowup, the solution is smooth with
    $u\in C^\infty\left((0,T_{max});C^\infty
    \left(\mathbb{T}^3\right)\right)$,
    or equivalently
    $u\in C^\infty\left((0,T_{max})\times 
    \mathbb{T}^3)\right)$.
\end{theorem}

\begin{proof}
    We will start by showing that if $u\in C\left((0,T_{max});\dot{H}^{s'}
    _\mathcal{M}\right)$
    then for all $m<s'+2\alpha-\frac{\log(3)}{2\log(2)}$, we have
    $u\in C\left((0,T_{max});\dot{H}^{s'+m}
    _\mathcal{M}\right)$.
    Fix $0<\epsilon<\frac{T_{max}}{2}$ and $0<m<
    s+2\alpha-\frac{\log(3)}{2\log(2)}$. 
    Clearly we then have that 
    $u\in C\left([\epsilon,T_{max}-\epsilon]
    ;\dot{H}^{s'}
    _\mathcal{M}\right)$.
    We also know by definition that for all 
    $\epsilon<t<T_{max}-\epsilon$,
    \begin{equation}
    u(\cdot,t)
    = e^{-\nu(t-\epsilon)(-\Delta)^\alpha}
    u^0
    +\int_0^{t-\epsilon} 
    e^{-\nu(t-\epsilon-\tau)(-\Delta)^\alpha}
    B(u,u)(\cdot,\tau) \diff\tau.
    \end{equation}
    Applying \Cref{SmoothingProp},
    we find that for all 
    $\epsilon<t<T_{max}-\epsilon$,
    \begin{align}
    \|u(\cdot,t)\|_{\dot{H}^{s'+m}}
    &\leq 
    \frac{C\|u(\cdot,\epsilon)\|_{\dot{H}^s}}
    {(t-\epsilon)^\frac{m}{2\alpha}}
    +C\int_0^{t-\epsilon}
    \frac{1}{(t-\epsilon-\tau)^\frac
    {m+\frac{\log(3)}{2\log(2)}-s}{2\alpha}}
    \|B(u,u)(\cdot,\tau)\|_{\dot{H}^{2s-
    \frac{\log(3)}{2\log(2)}}} \diff\tau \\
    &\leq 
    \frac{C\|u(\cdot,\epsilon)\|_{\dot{H}^s}}
    {(t-\epsilon)^\frac{m}{2\alpha}}
    +C\int_0^{t-\epsilon}
    \frac{1}{(t-\epsilon-\tau)^\frac
    {m+\frac{\log(3)}{2\log(2)}-s}{2\alpha}}
    \|u(\cdot,\tau)\|_{\dot{H}^s} 
    \diff\tau \\
    &\leq 
    \frac{C\|u(\cdot,\epsilon)\|_{\dot{H}^s}}
    {(t-\epsilon)^\frac{m}{2\alpha}}
    +C \left(\sup_{\epsilon\leq \tau \leq 
    T_{max}-\epsilon}
    \|u(\cdot,\tau)\|_{\dot{H}^s} \right)
    (t-\epsilon)^{1-\frac
    {m+\frac{\log(3)}{2\log(2)}-s}{2\alpha}}.
    \end{align}

    We have now bounded the $H^{s'+m}$ norm for all
    $\epsilon<t<T_{max}-\epsilon$, and we may conclude that for all $0<\epsilon<\frac{T_{max}}{2}$
    $u\in C\left((\epsilon,T_{max}-\epsilon);
    \dot{H}^{s'+m}_\mathcal{M}\right)$.
    Recalling that $\epsilon>0$ can be taken arbitrarily small, we may then conclude
    $u\in C\left((0,T_{max});
    \dot{H}^{s'+m}_\mathcal{M}\right)$.
    Iterating this smoothing effect, we can prove by induction that 
    $u\in C\left((0,T_{max});
    \dot{H}^\infty\right)$,
    and equivalently
    $u\in C\left((0,T_{max});
    C^\infty\left(\mathbb{T}^3\right)\right)$.
    Observe that for all $j\in\mathbb{Z}^+$,
    \begin{equation}
    \partial_t^{j+1}u
    =
    -\nu(-\Delta)^\alpha \partial_t^j u
    -\mathbb{P}_{\mathcal{M}}\left(
    \partial_t^j\nabla\cdot(u\otimes u)\right).
    \end{equation}
    Applying the product rule $j$ times, we can therefore see that if $u \in C^j\left((0,T_{max});
    C^\infty\left(\mathbb{T}^3\right)\right)$,
    then $u\in C^{j+1}\left((0,T_{max});
    C^\infty\left(\mathbb{T}^3\right)\right)$.
    By induction we may conclude that
    $u\in C^\infty\left((0,T_{max});C^\infty
    \left(\mathbb{T}^3\right)\right)$,
    and equivalently
    $u\in C^\infty\left((0,T_{max})\times 
    \mathbb{T}^3)\right)$,
    which completes the proof.
\end{proof}

\begin{proposition}
   For all $u^0\in \dot{H}^\frac{\log(3)}
    {2\log(2)}_{\mathcal{M}}$, we have the lower bound on the time of existence
    \begin{equation}
    T_{max}\geq \frac{1}{\mathcal{C}_*
    \left\|u^0\right\|_{\dot{H}
    ^\frac{\log(3)}{2\log(2)}}},
    \end{equation}
    and if $T_{max}<+\infty$, 
    then for all $0\leq t<T_{max}$,
    \begin{equation}
    \|u(\cdot,t)\|_{\dot{H}^\frac{\log(3)}{2\log(2)}}
    \geq 
    \frac{1}{\mathcal{C}_*(T_{max}-t)}.
    \end{equation} 
\end{proposition}

\begin{remark}
    It is somewhat surprising that this bound is the same as the inviscid case, but we should note that when $s=\frac{\log(3)}{2\log(2)}$, then $\rho=1$, and so the bound on the $T_{max}$ from \Cref{ExistenceRestrictedNS} is independent of the viscosity $\nu$, and this really quite natural.
\end{remark}

\begin{proof}
    We already know there exists a solution of the Fourier-restricted, hypodissipative Navier--Stokes equation, $u\in C\left([0,T_{max});
    \dot{H}^\frac{\log(3)}{2\log(2)}_{
    \mathcal{M}}\right)$.
    Next we compute that for all $0< t<T_{max}$,
    \begin{align}
    \frac{\diff}{\diff t}
    \|u(\cdot,t)\|_{
    \dot{H}^\frac{\log(3)}{2\log(2)}}^2
    &=
    -2\nu\|u\|_{\dot{H}^{\frac{\log(3)}{2\log(2)}
    +\alpha}}^2
    +\left<B(u,u),u\right>_{
    \dot{H}^\frac{\log(3)}{2\log(2)}} \\
    &\leq 
    2\mathcal{C}_* \|u\|_{\dot{H}^\frac{\log(3)}{2\log(2)}}^3,
    \end{align}
    neglecting the dissipation term that will not enhance our bounds.
    We can then conclude that
    \begin{equation}
    \frac{\diff}{\diff t}
    \|u(\cdot,t)\|_{
    \dot{H}^\frac{\log(3)}{2\log(2)}}
    \leq 
    \mathcal{C}_* \|u\|_{\dot{H}^\frac{\log(3)}{2\log(2)}}^2.
    \end{equation}
    Integrating this differential inequality we find that for all $0\leq t<T_{max}$,
    \begin{equation}
    \|u(\cdot,t)\|_{
    \dot{H}^\frac{\log(3)}{2\log(2)}}
    \leq 
    \frac{\left\|u^0\right\|_{
    \dot{H}^\frac{\log(3)}{2\log(2)}}}
    {1-\mathcal{C}_* \left\|u^0\right\|_{
    \dot{H}^\frac{\log(3)}{2\log(2)}}t}.
    \end{equation}
    We know that if $T_{max}<+\infty$, then
    \begin{equation}
    \lim_{t\to T_{max}}\|u(\cdot,t)
    \|_{\dot{H}^\frac{\log(3)}{2\log(2)}}
    =
    +\infty,
    \end{equation}
    and therefore,
    \begin{equation}
    T_{max}\geq \frac{1}{\mathcal{C}_*
    \left\|u^0\right\|_{\dot{H}
    ^\frac{\log(3)}{2\log(2)}}}.
    \end{equation}
    Consequently, if $T_{max}<+\infty$, 
    then for all $0\leq t<T_{max}$,
    \begin{equation}
    \|u(\cdot,t)\|_{\dot{H}^\frac{\log(3)}{2\log(2)}}
    \geq 
    \frac{1}{\mathcal{C}_*(T_{max}-t)}.
    \end{equation} 
\end{proof}

\begin{proposition}
    Suppose $u\in C\left([0,T_{max});
    \dot{H}^s_\mathcal{M}\right)$, where 
    $\frac{\log(3)}{2\log(2)}-2\alpha
    <s\leq \frac{\log(3)}{2\log(2)}$ and $s\geq 0$, is a mild solution of the Fourier-restricted, hypodissipative Navier--Stokes equation.
    Then for all $0< t<T_{max}$,
    \begin{equation}
    \frac{1}{2}\|u(\cdot,t)\|_{L^2}^2
    +\nu\int_0^t\|u(\cdot,\tau)\|_{\dot{H}^\alpha}^2
    \diff\tau
    =
    \frac{1}{2}\left\|u^0\right\|_{L^2}^2.
    \end{equation}
\end{proposition}

\begin{proof}
    We can see that for all $0<t<T_{max}$,
    \begin{equation}
    \frac{\diff}{\diff t}
    \frac{1}{2}\|u(\cdot,t)\|_{L^2}^2
    =-\nu\left<(-\Delta)^\alpha u,u\right>
    -\left<\mathbb{P}_\mathcal{M}
    ((u\cdot\nabla)u),u\right>.
    \end{equation}
    Because $u(\cdot,t)\in C^\infty$ for all $0<t<T_{max}$, we have sufficient regularity to integrate by parts using the divergence free constraint, and conclude that
\begin{align}
    -\left<\mathbb{P}_\mathcal{M}((u\cdot\nabla)u),u\right> 
    &=
    -\left<(u\cdot\nabla)u,u\right> \\
    &=
    \left<u,(u\cdot\nabla)u\right> \\
    &=
    0.
\end{align}
    Therefore, for all $0<t<T_{max}$, we have
    \begin{equation}
    \frac{\diff}{\diff t}
    \frac{1}{2}\|u(\cdot,t)\|_{L^2}^2
    =-\nu \|u\|_{\dot{H}^\alpha}^2,
    \end{equation}
    and integrating this differential equation completes the proof.
\end{proof}

\begin{remark}
    We have now proven all of the local wellposedness results in \Cref{RestrictedHypoExistenceThmIntro}, which were broken up into multiple pieces to make the proofs more readable. Before moving on to the global wellposedness theory, we will prove a stability result in 
$C_T\dot{H}^\frac{\log(3)}{2\log(2)}_x
\cap
L^2_T\dot{H}^{\frac{\log(3)}{2\log(2)}+\alpha}_x$.
\end{remark}

\begin{theorem} \label{StabilityRestrictedNS}
    Suppose $u,w \in C\left([0,T];
    \dot{H}^\frac{\log(3)}{2\log(2)}_{\mathcal{M}}\right)$,
    are solutions of the Fourier-restricted hypodissipative Navier--Stokes equation.
    Then for all $0\leq t\leq T$,
    \begin{multline}
    \|(u-w)(\cdot,t)\|_{\dot{H}^\frac{\log(3)}{2\log(2)}}^2
    +2\nu\int_0^t\|(u-w)(\cdot,\tau)\|_{\dot{H}^{
    \frac{\log(3)}{2\log(2)}+\alpha}}^2
    \diff \tau
    \leq \\
    \left\|u^0-w^0\right\|_{\dot{H}^
    \frac{\log(3)}{2\log(2)}}^2
    \exp\left(2\mathcal{C}_*\int_0^t
    \|(u+w)(\cdot,t)\|_{\dot{H}^\frac{\log(3)}{2\log(2)}}
    \diff\tau\right).
    \end{multline}
\end{theorem}

\begin{proof}
    We proceed as in the proof of \Cref{StabilityRestrictedEuler}, 
    observing that
    \begin{equation}
    \partial_t(u-w)
    +\nu(-\Delta)^\alpha (u-w)
    =B(u+w,u-w),
    \end{equation}
    and letting
    \begin{equation}
    G(t)= \|(u-w)(\cdot,t)\|_{
    \dot{H}^\frac{\log(3)}{2\log(2)}}^2
    +2\nu\int_0^t \|(u-w)\|_{
    \dot{H}^{\frac{\log(3)}{2\log(2)}+\alpha}}^2.
    \end{equation}
    Now we compute for all $0<t\leq T$,
    \begin{align}
    \frac{\diff}{\diff t} G(t)
    &=
    \frac{\diff}{\diff t}\|(u-w)(\cdot,t)\|_{
    \dot{H}^\frac{\log(3)}{2\log(2)}}^2
    +2\nu \|(u-w)(\cdot,t)\|_{
    \dot{H}^{\frac{\log(3)}{2\log(2)}+\alpha}}^2 \\
    &=
    2\Big<B(u+w,u-w),u-w\Big> \\
    &\leq 
    2C_* \|u+w\|_{
    \dot{H}^\frac{\log(3)}{2\log(2)}}
    \|u-w\|_{
    \dot{H}^\frac{\log(3)}{2\log(2)}}^2 \\
    &\leq 
    2C_* \|u+w\|_{
    \dot{H}^\frac{\log(3)}{2\log(2)}} G.
    \end{align}
    Applying Gr\"onwall's inequality, this completes the proof.
\end{proof}

\subsection{Global wellposedness}
In this section we will deal with global wellposedness for the Fourier-restricted hypodissipative Navier--Stokes equation: for generic initial data when $\alpha\geq 
\frac{\log(3)}{4\log(2)}$, 
and for small initial data when $0<\alpha<\frac{\log(3)}{4\log(2)}$.
We begin by proving \Cref{GwpThmIntro}, which is restated for the readers convenience.

\begin{theorem} \label{GwpThm}
    Fix the degree of dissipation $\alpha\geq \frac{\log(3)}{4\log(2)}$, and the viscosity $\nu>0$.
    If $\alpha=\frac{\log(3)}{4\log(2)}$, suppose $u^0\in \dot{H}^s_\mathcal{M}$ with $s>0$.
    If $\alpha>\frac{\log(3)}{4\log(2)}$, suppose $u^0\in \dot{H}^s_\mathcal{M}$, with $s\geq 0$.
    Then there exists a unique, global smooth solution of the Fourier-restricted hypodissipative Navier--Stokes equation $u\in C\left([0,+\infty);\dot{H}^s_\mathcal{M}\right)$,
    and for all $0\leq t<+\infty$,
    \begin{equation}
    \|u(\cdot,t)\|_{\dot{H}^s}^2
    \leq
    \left\|u^0\right\|_{\dot{H}^s}^2
    \exp\left(\frac{\mathcal{C}_s^2}{4(2\pi)^{
    4\alpha-\frac{\log(3)}{2\log(2)}}}
    \frac{\left\|u^0\right\|_{L^2}^2}{\nu^2}
    \right).
    \end{equation}
\end{theorem}

\begin{proof}
    We have already shown that there exists a unique mild solution $u\in C\left([0,T_{max});\dot{H}^s_\mathcal{M}\right)$.
    Next we compute that for all $0<t<T_{max}$,
    \begin{equation}
    \frac{\diff}{\diff t}
    \|u(\cdot,t)\|_{\dot{H}^s}^2
    =
    -2\nu\|u\|_{\dot{H}^{s+\alpha}}^2
    +2\Big<B(u,u),u\Big>_{\dot{H}^s}.
    \end{equation}
    Applying \Cref{BilinearBoundLemma}, interpolating between Hilbert spaces, and applying Young's inequality, we find that
    \begin{align}
    \Big<B(u,u),u\Big>_{\dot{H}^s}
    &\leq 
    \mathcal{C}_s \|u\|_{\dot{H}^\frac{s}{2}+
    \frac{\log(3)}{4\log(2)}}^2
    \|u\|_{\dot{H}^s} \\
    &\leq 
    \mathcal{C}_s \|u\|_{
    \dot{H}^{\frac{\log(3)}{2\log(2)}-\alpha}}
    \|u\|_{\dot{H}^s}
    \|u\|_{\dot{H}^{s+\alpha}} \\
    &\leq 
    \frac{\mathcal{C}_s^2}{4\nu}
    \|u\|_{\dot{H}^{\frac{\log(3)}{2\log(2)}-\alpha}}^2
    \|u\|_{\dot{H}^s}^2
    +\nu\|u\|_{\dot{H}^{s+\alpha}}^2.
    \end{align}
    Therefore, we may conclude that
    \begin{equation}
    \frac{\diff}{\diff t}
    \|u(\cdot,t)\|_{\dot{H}^s}^2
    \leq 
    \frac{\mathcal{C}_s^2}{2\nu}
    \|u\|_{\dot{H}^{\frac{\log(3)}{2\log(2)}-\alpha}}^2
    \|u\|_{\dot{H}^s}^2.
    \end{equation}
    
    Next recall that $\alpha\geq 
    \frac{\log(3)}{4\log(2)}$,
    and so 
    \begin{equation}
    \frac{\log(3)}{2\log(2)}-\alpha \leq\alpha.
    \end{equation}
    Applying \Cref{HomogenousSobolevProp}, we can see that
    \begin{equation}
    \|u\|_{\dot{H}^{\frac{\log(3)}{2\log(2)}-\alpha}}
    \leq
    \frac{1}{(2\pi)^{2\alpha-\frac{\log(3)}{2\log(2)}}}
    \|u\|_{\dot{H}^\alpha},
    \end{equation}
    and so
    \begin{equation}
    \frac{\diff}{\diff t}
    \|u(\cdot,t)\|_{\dot{H}^s}^2
    \leq 
    \frac{\mathcal{C}_s^2}{2(2\pi)^{
    4\alpha-\frac{\log(3)}{2\log(2)}}\nu}
    \|u\|_{\dot{H}^\alpha}^2
    \|u\|_{\dot{H}^s}^2.
    \end{equation}
    Applying Gr\"onwall's inequality and the energy equality, we conclude that
    for all $0< t<T_{max}$,
    \begin{align}
    \|u(\cdot,t)\|_{\dot{H}^s}^2
    &\leq 
    \left\|u^0\right\|_{\dot{H}^s}^2
    \exp\left(\frac{\mathcal{C}_s^2}{2(2\pi)^{
    4\alpha-\frac{\log(3)}{2\log(2)}}\nu}
    \int_0^t
    \|u(\cdot,\tau)\|_{\dot{H}^\alpha}^2
    \diff\tau\right)  \\
    &\leq 
    \left\|u^0\right\|_{\dot{H}^s}^2
    \exp\left(\frac{\mathcal{C}_s^2}{4(2\pi)^{
    4\alpha-\frac{\log(3)}{2\log(2)}}}
    \frac{\left\|u^0\right\|_{L^2}^2}{\nu^2}
    \right).
    \end{align}
    Note that we have already established that if $T_{max}<+\infty$, then
    \begin{equation}
    \lim_{t\to T_{max}}
    \|u(\cdot,t)\|_{\dot{H}^s}^2
    =+\infty,
    \end{equation}
    so this implies that $T_{max}=+\infty$, and this completes the proof.
\end{proof}

We will now prove the small data global wellposedness result \Cref{SmallDataGwpIntro}, but first we will prove a key lemma.

\begin{lemma} \label{SmallDataLemma}
    Suppose $u\in C\left([0,T_{max});
    \dot{H}^s_\mathcal{M}\right)$,
    is a smooth solution of the Fourier-restricted 
    hypodissipative Navier--Stokes equation.
    Then for all $0<t<T_{max}$, and for all $s'\geq 0, s'\geq \frac{\log(3)}{2\log(2)}-2\alpha$,
    \begin{equation}
    \frac{\diff}{\diff t} \frac{1}{2}
    \|u(\cdot,t)\|_{\dot{H}^{s'}}^2
    \leq 
    -\nu\|u\|_{\dot{H}^{s'+\alpha}}^2
    +\mathcal{C}_{s'}
    \|u\|_{\dot{H}^{
    \frac{\log(3)}{2\log(2)}-2\alpha}}
    \|u\|_{\dot{H}^{s'+\alpha}}^2
    \end{equation}
\end{lemma}

\begin{proof}
    Applying \Cref{BilinearBoundLemma} and interpolating between Hilbert spaces, we find that for all $0< t<T_{max}$ and for all $s'>0$
    \begin{align}
    \frac{\diff}{\diff t} \frac{1}{2}
    \|u(\cdot,t)\|_{\dot{H}^{s'}}^2
    &\leq 
    -\nu\|u\|_{\dot{H}^{s'+\alpha}}^2
    +\mathcal{C}_{s'}
    \|u\|_{\dot{H}^{
    \frac{s'}{2}+\frac{\log(3)}{4\log(2)}}}^2
    \|u\|_{\dot{H}^{s'}} \\
    &\leq 
    -\nu\|u\|_{\dot{H}^{s'+\alpha}}^2
    +\mathcal{C}_{s'}
    \|u\|_{\dot{H}^{
    \frac{\log(3)}{2\log(2)}-2\alpha}}
    \|u\|_{\dot{H}^{s'+\alpha}}^2,
    \end{align}
    which completes the proof.
\end{proof}

\begin{theorem} \label{SmallDataGwp}
    Suppose $u^0\in \dot{H}^s_{\mathcal{M}}$, where $s>\frac{\log(3)}{2\log(2)}-2\alpha$ and $0<\alpha<\frac{\log(3)}{4\log(2)}$. Further suppose that 
    \begin{equation}
    \left\|u^0\right\|_{\dot{H}^{\frac{\log(3)}{2\log(2)}-2\alpha}}
    <
    \frac{\nu}
    {\max\left(
    \mathcal{C}_{\frac{\log(3)}{4\log(2)}-2\alpha},
    \mathcal{C}_s\right)}.
    \end{equation}
    Then there exists a global smooth solution of the Fourier-restricted hypodissipative Navier--Stokes equation
    $u\in C\left([0,+\infty);\dot{H}^s_\mathcal{M}\right)$,
    and for all $0< t<+\infty$,
    \begin{align}
    \|u(\cdot,t)\|_{
    \dot{H}^{\frac{\log(3)}{2\log(2)}-2\alpha}}
    &\leq
    \left\|u^0\right\|_{
    \dot{H}^{\frac{\log(3)}{2\log(2)}-2\alpha}} \\
    \|u(\cdot,t)\|_{\dot{H}^s}
    &\leq
    \left\|u^0\right\|_{\dot{H}^s}.
    \end{align}
\end{theorem}

\begin{proof}
    We have already proven that there exists a local mild solution $u\in C\left([0,T_{max});\dot{H}^s_\mathcal{M}\right)$.
    We can see from \Cref{SmallDataLemma}, 
    that for all $0<t<T_{max}$,
    \begin{equation}
    \frac{\diff}{\diff t} \frac{1}{2}
    \|u(\cdot,t)\|_{\dot{H}^{
    \frac{\log(3)}{2\log(2)}-2\alpha}}^2
    \leq 
    -\nu\|u\|_{\dot{H}^{
    \frac{\log(3)}{2\log(2)}-\alpha}}^2
    +\mathcal{C}_{\frac{\log(3)}{2\log(2)}-2\alpha}
    \|u\|_{\dot{H}^{
    \frac{\log(3)}{2\log(2)}-2\alpha}}
    \|u\|_{\dot{H}^{
    \frac{\log(3)}{2\log(2)}-\alpha}}^2.
    \end{equation}
    This clearly implies that if 
    \begin{equation}
    \|u(\cdot,t)\|_{\dot{H}^{
    \frac{\log(3)}{2\log(2)}-2\alpha}}
    <
    \frac{\nu}
    {\mathcal{C}_{\frac{\log(3)}{2\log(2)}-2\alpha}},
    \end{equation}
    then
    \begin{equation}
    \frac{\diff}{\diff t}
    \|u(\cdot,t)\|_{\dot{H}^{
    \frac{\log(3)}{2\log(2)}-2\alpha}}
    \leq 0.
    \end{equation}
    Using the fact that
    \begin{equation}
    \left\|u^0\right\|_{\dot{H}^{
    \frac{\log(3)}{2\log(2)}-2\alpha}}
    <
    \frac{\nu}
    {\mathcal{C}_{\frac{\log(3)}{2\log(2)}-2\alpha}},
    \end{equation}
    we can conclude that for all $0< t<T_{max}$,
    \begin{equation}
    \|u(\cdot,t)\|_{\dot{H}^{
    \frac{\log(3)}{2\log(2)}-2\alpha}}
    \leq 
    \left\|u^0\right\|_{\dot{H}^{
    \frac{\log(3)}{2\log(2)}-2\alpha}}.
    \end{equation}

    Now we will control the $\dot{H}^s$ norm, also using \Cref{SmallDataLemma}.
    Observe that for all $0<t<T_{max}$,
    \begin{align}
    \mathcal{C}_s
    \|u(\cdot,t)\|_{\dot{H}^{
    \frac{\log(3)}{2\log(2)}-2\alpha}}
    &\leq 
    \mathcal{C}_s
    \left\|u^0\right\|_{\dot{H}^{
    \frac{\log(3)}{2\log(2)}-2\alpha}} \\
    &\leq 
    \nu,
    \end{align}
    and therefore, for all $0<t<T_{max}$
    \begin{align}
    \frac{\diff}{\diff t} \frac{1}{2}
    \|u(\cdot,t)\|_{\dot{H}^s}^2
    &\leq 
    -\nu\|u\|_{\dot{H}^{s+\alpha}}^2
    +\mathcal{C}_s
    \|u\|_{\dot{H}^{
    \frac{\log(3)}{2\log(2)}-2\alpha}}
    \|u\|_{\dot{H}^{s+\alpha}}^2 \\
    &\leq 
    0.
    \end{align}
    Therefore we may conclude that for all $0<t<T_{max}$,
    \begin{equation}
    \|u(\cdot,t)\|_{\dot{H}^s}
    \leq
    \left\|u^0\right\|_{\dot{H}^s},
    \end{equation}
    which also implies that $T_{max}=+\infty$
    and completes the proof.
    
\end{proof}

\section{Permutation symmetry} \label{PermutationSection}

In this section, we will conisder permutation symmetry for both the Fourier-restricted Euler and hypodissipative Navier--Stokes equations, as well as for the full Euler and Navier--Stokes equations.

    \begin{lemma} \label{DotLemma}
    For all $v,w\in\mathbb{R}^3$ and for all $P\in\mathcal{P}_3$,
    \begin{equation}
        P(v)\cdot w
        =
        v\cdot P^{-1}(w).
    \end{equation}
    \end{lemma}

    \begin{proof}
    Note that every permutation is invertible and its inverse is again a permutation.
    Using the fact that permutations are bijective, we let $j=P(i)$ and compute that
    \begin{align}
    P(v)\cdot w
    &=
    \sum_{i=1}^3
    v_{P(i)} w_i \\
    &=
    \sum_{j=1}^3
    v_{j}
    w_{P^{-1}(j)} \\
    &=
    v\cdot P^{-1}(w).
    \end{align}
    \end{proof}

    \begin{proposition} \label{MaterialDerivativePermute}
         For all $u\in \dot{H}^s\left(\mathbb{T}^3\right), s\geq 1$,
        \begin{equation}
            (u^P\cdot\nabla) u^P
            =\big((u\cdot\nabla)u\big)^P
        \end{equation}
    In particular, if $u$ is permutation symmetric, then $(u\cdot\nabla) u$ is permutation symmetric.
    \end{proposition}

    \begin{proof}
        First we will note that
        \begin{equation}
        u^P\cdot\nabla =\sum_{i=1}^3 u_{P(i)}\partial_i.
        \end{equation}
        Therefore, for any differentiable function $f$,
        \begin{equation}
            (u^P\cdot\nabla) f(P^{-1}x)
            =
            \sum_{i,j=1}^3
            u_{P(i)}(y)\frac{\partial f(y)}{\partial y_j}
            \frac{\partial y_j}{\partial x_i},
        \end{equation}
        by the chain rule, where $y=P^{-1}(x)$.
        Note that
        \begin{equation}
            y_j=x_{P^{-1}(j)},
        \end{equation}
        and so
        \begin{equation}
        \frac{\partial y_j}{\partial x_i}
        =
        \begin{cases}
            1, &i=P^{-1}(j) \\
            0, & \text{otherwise}
        \end{cases},
        \end{equation}
        or equivalently
                \begin{equation}
        \frac{\partial y_j}{\partial x_i}
        =
        \begin{cases}
            1, &j=P(i) \\
            0, & \text{otherwise}
        \end{cases}.
        \end{equation}
        Therefore we may conclude that
        \begin{align}
            (u^P\cdot\nabla) f(P^{-1}x)
            &=
            \sum_{i,j=1}^3
            (\delta_{P(i)j} u_{P(i)}\partial_j f)(y) \\
            &=
            \sum_{j=1}^3 (u_j\partial_j f)(y) \\
            &=
            (u\cdot\nabla f)(P^{-1}x).
        \end{align}
        Taking $f$ as the various components of $u^P$ we then find that
        \begin{align}
        ((u^P\cdot\nabla)u^P)(x)
        &=
        ((u\cdot\nabla)Pu)(P^{-1}x))\\
        &=
        P((u\cdot\nabla) u)(P^{-1}x) \\
        &=
        ((u\cdot\nabla)u)^P(x),
        \end{align}
        and this completes the proof.
    \end{proof}

    \begin{remark}
        It is classical that the material derivative and Helmholtz projection are preserved by transformations in $O(3)$, which includes permutations, but these details are essential to the development of the model, and so are included for completeness.
    \end{remark}

    \begin{proposition} \label{FourierPermute}
    For all $u\in \dot{H}^s\left(\mathbb{T}^3\right), s\geq 0$,
    and for all permutations $P\in\mathcal{P}_3$, the Fourier transform satisfies
    \begin{align}
    \mathcal{F}\left(u^P\right) \label{PhysicalPermutation}
    &=
    (\hat{u})^P \\
    \mathcal{F}^{-1}(\hat{u}^P) \label{FourierPermutation}
    &=
    u^{P}.
    \end{align}
    In particular, $u$ is permutation symmetric if and only if $\hat{u}$ is permutation symmetric.
    \end{proposition}

    \begin{proof}
        Recall that by definition 
        $\hat{u}^P(k)=P\hat{u}(P^{-1}k)$.
        Therefore we can compute that
        \begin{align}
        \mathcal{F}\left(u^P\right)(k)
        &=
        \int_{\mathbb{T}^3}
        u^P(x)e^{-2\pi i k\cdot x} \diff x \\
        &=
        \int_{\mathbb{T}^3}
        Pu(P^{-1}x)e^{-2\pi i k\cdot x} 
        \diff x \\
        &=
        \int_{\mathbb{T}^3}
        Pu(y)e^{-2\pi i k\cdot P(y)}  
        \diff y \\
        &=
        \int_{\mathbb{T}^3}
        Pu(y)e^{-2\pi i P^{-1}(k)\cdot y}  
        \diff y \\
        &=
        P\hat{u}(P^{-1}k) \\
        &=
        \hat{u}^P(k),
        \end{align}
        where we have taken the change of variables $y=P^{-1}x$, and we have applied \Cref{DotLemma}.
        This completes the proof of the identity \eqref{PhysicalPermutation}.

        We will now prove the identity \eqref{FourierPermutation}.
        We will compute that
        \begin{align}
    \mathcal{F}^{-1}(\hat{u}^P)(x)
    &=
    \sum_{k\in\mathbb{Z}^3}
    \hat{u}^P(k) e^{2\pi ik\cdot x} \\
    &=
    \sum_{k\in\mathbb{Z}^3}
    P\hat{u}(P^{-1}k) e^{2\pi ik\cdot x} \\
    &=
    \sum_{h\in\mathbb{Z}^3}
    P\hat{u}(h) e^{2\pi i P(h)\cdot x} \\
    &=
    P\sum_{h\in\mathbb{Z}^3}
    \hat{u}(h) e^{2\pi i h\cdot P^{-1}(x)} \\
    &=
    Pu(P^{-1}x) \\
    &=
    u^P(x),
        \end{align}
    where we have made the substitution $h=P^{-1}(k)$ and used the Fourier series representation.
    \end{proof}

    \begin{proposition} \label{HelmholtzPermute}
    For all $w\in \dot{H}^s\left(\mathbb{T}^3\right), s\geq 0$,
    and for all permutations $P\in\mathcal{P}_3$
        \begin{equation}
            \mathbb{P}_{df}\left(w^P\right)
            =
            \left(\mathbb{P}_{df}(w)\right)^P.
        \end{equation}
        Note that this implies that $\dot{H}^s_{df}$ is preserved under permutations of vector fields.
    \end{proposition}

    \begin{proof}
    Recall that the Helmholtz projection can be expressed in Fourier space by
    \begin{equation}
        \mathcal{F}(\mathbb{P}_{df}(w))(k)
        =
        \left(Id-\frac{k\otimes k}{|k|^2}\right)\hat{w}(k)
        =\hat{w}(k)-\frac{(k\cdot\hat{w}(k))k}{|k|^2}.
    \end{equation}
    Therefore we can compute that
    \begin{align}
        \mathcal{F}\left(\mathbb{P}_{df}
        \left(w^P\right)\right)(k)
        &=
        \hat{w}^P(k)
        -\frac{(k\cdot\hat{w}^P(k))k}{|k|^2} \\
        &=
        P\hat{w}(P^{-1}k)
        -\frac{(k\cdot P\hat{w}(P^{-1}k))k}{|k|^2} \\
        &=
        P\hat{w}(P^{-1}k)
        -\frac{(P^{-1}(k)\cdot \hat{w}(P^{-1}k))k}{|k|^2},
    \end{align}
    where we have applied \Cref{FourierPermute} and \Cref{DotLemma}.
    Making the substitution $h=P^{-1}(k)$, we find
    \begin{align}
       \mathcal{F}(\mathbb{P}_{df}(w^P))(k)
        &= 
        P\hat{w}(h)
        -\frac{h\cdot \hat{w}(h))P(h)}{|h|^2} \\
        &=
        P\left(\hat{w}(h)
        -\frac{h\cdot \hat{w}(h))h}{|h|^2}
        \right) \\
        &=
        P \mathcal{F}\left( \mathbb{P}_{df}(w)
        \right)(P^{-1}k) \\
        &=
        \mathcal{F}\left( \mathbb{P}_{df}(w)
        \right)^P(k) \\
        &=
        \mathcal{F}\left( \mathbb{P}_{df}(w)^P
        \right)(k),
    \end{align}
    where we have applied \Cref{FourierPermute}.
    Taking the inverse Fourier transform, we find that
    \begin{equation}
        \mathbb{P}_{df}\left(w^P\right)
        =
        \left(\mathbb{P}_{df}(w)\right)^P.
    \end{equation}

    Note in particular this implies that if 
    \begin{equation}
        w=\mathbb{P}_{df}(w),
    \end{equation}
    then
    \begin{equation}
       w^P=\mathbb{P}_{df}\left(w^P\right). 
    \end{equation}
    Therefore, for all permutations $P\in\mathcal{P}_3$, if $w\in \dot{H}^s_{df}$, then $w^P \in \dot{H}^s_{df}$.
    \end{proof}

    \begin{proposition} \label{mPermute}
        For all $w\in \dot{H}^s\left(\mathbb{T}^3\right), s\geq 0$,
        and for all permutations $P\in\mathcal{P}_3$,
        \begin{equation}
            \mathbb{P}_{\mathcal{M}}
            \left(w^P\right)
            =
            \left(\mathbb{P}_{\mathcal{M}}
            (w)\right)^P.
        \end{equation}
        Note that this implies that 
        $\dot{H}^s_{\mathcal{M}}$ is preserved under permutations of vector fields.
    \end{proposition}

    \begin{proof}
        Observe from the construction of $\dot{H}^s_\mathcal{M}$,
        that the projection $\mathbb{P}_\mathcal{M}$ can be expressed in Fourier space by
        \begin{equation}
        \mathcal{F}(\mathbb{P}_\mathcal{M}(w))(k)
        =
        \left(v^k\cdot\hat{w}(k)\right)v^k
        \mathds{1}_\mathcal{M}(k).
        \end{equation}
        Therefore we can compute that
        \begin{align}
    \mathcal{F}\left(\mathbb{P}_\mathcal{M}
    \left(w^P\right)\right)(k) 
    &=
    \left(v^k\cdot\hat{w}^P(k)\right)v^k
    \mathds{1}_\mathcal{M}(k)\\
    &=
    \left(v^k\cdot P\hat{w}(P^{-1}k)\right)v^k
    \mathds{1}_\mathcal{M}(k)\\
    &=
    \left(P^{-1}(v^k)\cdot \hat{w}(P^{-1}k)\right)v^k
    \mathds{1}_\mathcal{M}(k)\\
    &=
    \left((v^{P^{-1}k})\cdot \hat{w}(P^{-1}k)\right)v^k
    \mathds{1}_\mathcal{M}(k),
        \end{align}
        where we have used the fact that 
        $P^{-1}(v^k)=v^{P^{-1}(k)}$, by the construction of $v^k$.
        Letting $h=P^{-1}(k)$, we can compute that
    \begin{align}
    \mathcal{F}\left(\mathbb{P}_\mathcal{M}
    \left(w^P\right)\right)(k) 
    &=
    \left(v^h\cdot\hat{w}(h)\right) v^{P(h)}
    \mathds{1}_\mathcal{M}(P(h))\\
    &=
    \left(v^h \cdot\hat{w}(h)\right)P(v^{h})
    \mathds{1}_\mathcal{M}(P(h))\\
    &=
    P\mathcal{F}\left(\mathbb{P}_\mathcal{M}(w)
    \right)(P^{-1}(k)) \\
    &=
    \left(\mathcal{F} \mathbb{P}_\mathcal{M}(w)
    \right)^P(k) \\
    &=
    \mathcal{F}\left( \mathbb{P}_\mathcal{M}(w)^P
    \right)(k),
    \end{align}
    just as in \Cref{HelmholtzPermute}.
    Taking the inverse Fourier transform, we can see that
    \begin{equation}
        \mathbb{P}_{\mathcal{M}}
            \left(w^P\right)
            =
            \left(\mathbb{P}_{\mathcal{M}}
            (w)\right)^P.
    \end{equation}
    Note in particular that this implies that if
    \begin{equation}
        w=\mathbb{P}_\mathcal{M}(w),
    \end{equation}
    then
    \begin{equation}
        w^P=\mathbb{P}_\mathcal{M}\left(w^P\right).
    \end{equation}
    Therefore, if $w\in \dot{H}^s_{\mathcal{M}}$, then
    $w^P\in \dot{H}^s_{\mathcal{M}}$.
    \end{proof}

    \begin{proposition} \label{EulerInvariantProp}
    Suppose $u\in C\left([0,T_{max});
    \dot{H}^s_{df}\right)
    \cap
    C^1\left([0,T_{max});
    \dot{H}^{s-1}_{df}\right), s>\frac{5}{2}$, is a solution of the Euler equation. Then $u^P$ is also a solution of the Euler equation
    for any permutation $P\in\mathcal{P}_3$.
    \end{proposition}

\begin{proof}
    Applying \Cref{MaterialDerivativePermute,HelmholtzPermute}
    we can see that 
    \begin{equation}
    \mathbb{P}_{df}\left(\left(u^P\cdot\nabla
    \right)u^P\right)
    =
    \mathbb{P}_{df}\left((u\cdot\nabla u)^P\right)
    =
    \left(\mathbb{P}_{df}(u\cdot\nabla)u \right)^P.
    \end{equation}
    Therefore we can see that
    \begin{equation}
    \partial_t u^P+ 
    \mathbb{P}_{df}\left(\left(u^P\cdot\nabla
    \right)u^P\right)
    =
    \left(\partial_t u +\mathbb{P}_{df}
    ((u\cdot\nabla)u)\right)^P
    =0,
    \end{equation}
    and this completes the proof.
\end{proof}

    \begin{proposition} \label{EulerPermuteDynamicsProp}
    Suppose $u^0\in \dot{H}^s_{df}, s>\frac{5}{2}$ is permutation symmetric. Then the solution of the 
    Euler equation 
    $u\in C\left([0,T_{max});
    \dot{H}^s_{df}\right)
    \cap
    C^1\left([0,T_{max});
    \dot{H}^{s-1}_{df}\right)$
    is also permutation symmetric for all $0\leq t<T_{max}$.
    \end{proposition}

\begin{proof}
    We need to show that for all $P\in\mathcal{P}_3$,
    and for all $0\leq t<T_{max}$
    \begin{equation}
        u(\cdot,t)=u^P(\cdot,t).
    \end{equation}
    Fix $P\in\mathcal{P}_3$, and let $v=u^P$.
    We know from \Cref{EulerInvariantProp} that $v$ is also a solution of the Euler equation, and
    we know by hypothesis that for all $v^0=u^0$.
    Therefore, uniqueness immediately implies that
    for all $0\leq t<T_{max}$,
    \begin{equation}
        v(\cdot,t)=u(\cdot,t),
    \end{equation}
    and this completes the proof.
\end{proof}

    \begin{proposition} 
    \label{RestrictedPermuteInvariantProp}
    Suppose $u\in 
    C^1\left([0,T_{max});
    \dot{H}^\frac{\log(3)}{2\log(2)}_{\mathcal{M}}\right)$
    is a solution of the Fourier-restricted Euler equation. Then $u^P$ is also a solution of Fourier-restricted Euler equation for any permutation $P\in\mathcal{P}_3$.
    \end{proposition}

\begin{proof}
    Applying \Cref{MaterialDerivativePermute,mPermute}
    we can see that 
    \begin{equation}
    \mathbb{P}_\mathcal{M}\left(\left(u^P\cdot\nabla
    \right)u^P\right)
    =
    \mathbb{P}_\mathcal{M}
    \left((u\cdot\nabla u)^P \right)
    =
    \left(\mathbb{P}_\mathcal{M}(u\cdot\nabla)u \right)^P.
    \end{equation}
    Therefore we can see that
    \begin{equation}
    \partial_t u^P+ 
    \mathbb{P}_\mathcal{M}\left(\left(u^P\cdot\nabla
    \right)u^P\right)
    =
    \left(\partial_t u +\mathbb{P}_\mathcal{M}
    ((u\cdot\nabla)u)\right)^P
    =0,
    \end{equation}
    and this completes the proof.
\end{proof}

    \begin{proposition} \label{RestrictedPermuteDynamicsProp}
    Suppose $u^0\in \dot{H}^\frac{\log(3)}{2\log(2)}_{\mathcal{M}}$ is permutation symmetric. Then the solution of the 
    Fourier-restricted Euler equation 
    $u\in C^1\left([0,T_{max});
    \dot{H}^\frac{\log(3)}{2\log(2)}
    _{\mathcal{M}}\right)$
    is also permutation symmetric for all $0\leq t<T_{max}$.
    \end{proposition}

    \begin{proof}
    We need to show that for all $P\in\mathcal{P}_3$,
    and for all $0\leq t<T_{max}$
    \begin{equation}
        u(\cdot,t)=u^P(\cdot,t).
    \end{equation}
    Fix $P\in\mathcal{P}_3$, and let $v=u^P$.
    We know from \Cref{RestrictedPermuteInvariantProp} that $v$ is also a solution of the Fourier-restricted Euler equation, and
    we know by hypothesis that $v^0=u^0$.
    Therefore, uniqueness immediately implies that for all $0\leq t<T_{max}$
    \begin{equation}
        v(\cdot,t)=u(\cdot,t),
    \end{equation}
    and this completes the proof.
\end{proof}

\begin{lemma} \label{LaplacePermute}
    For all $w\in H^{2s}\left(
    \mathbb{T}^3;\mathbb{R}^3\right)$
    and for all permutations $P\in\mathcal{P}_3$,
    \begin{equation}
    (-\Delta)^s w^P=\left((-\Delta)^s w)\right)^P.
    \end{equation}
\end{lemma}

\begin{proof}
    In Fourier space, the proof is a simple exercise left to the reader, and all comes down to the fact that
    \begin{equation}
        |k|=|P(k)|.
    \end{equation}
\end{proof}

    \begin{proposition} 
    \label{RestrictedHypoPermuteInvariantProp}
    Suppose $u\in C\left([0,T_{max});
    \dot{H}^s_{\mathcal{M}}\right), s>\frac{\log(3)}{2\log(2)}-2\alpha, s\geq 0$,
    is a solution of the Fourier-restricted, hypodissipative Navier--Stokes equation. Then $u^P$ is also a solution of Fourier-restricted, hypodissipative Navier--Stokes equation for any permutation $P\in\mathcal{P}_3$.
    \end{proposition}

    \begin{proof}
    We have already shown that
    \begin{equation}
    \mathbb{P}_\mathcal{M}
    \left((u^P\cdot\nabla)u^P\right)
    =
    \left(\mathbb{P}_\mathcal{M}
    (u\cdot\nabla)u\right)^P,
    \end{equation}
    and so applying \Cref{LaplacePermute}, we find that
    \begin{equation}
    \partial_t u^P+\nu(-\Delta)^\alpha u^P
    +\mathbb{P}_\mathcal{M}
    \left((u^P\cdot\nabla)u^P\right)
    =
    \left(\partial_t u+\nu(-\Delta)^\alpha u
    +\mathbb{P}_\mathcal{M}
    (u\cdot\nabla)u\right)^P
    =0,
    \end{equation}
    and this completes the proof.
    \end{proof}

    \begin{proposition} \label{RestrictedHypoPermuteDynamicsProp}
    Suppose $u^0\in \dot{H}^s_{\mathcal{M}}$ is permutation symmetric, where $s>\frac{\log(3)}{2\log(2)}-2\alpha, s\geq 0$. Then the solution of the 
    Fourier-restricted, hypodissipative Navier--Stokes equation 
    $u\in C\left([0,T_{max});
    \dot{H}^s_{\mathcal{M}}\right)$
    is also permutation symmetric for all $0\leq t<T_{max}$.
    \end{proposition}

    \begin{proof}
    This result follows immediately from uniqueness and \Cref{RestrictedHypoPermuteInvariantProp}
    exactly as in the proof of \Cref{RestrictedPermuteDynamicsProp}.
    \end{proof}

    \begin{remark}
        It is classical that the odd subspace is preserved by the dynamics of the Euler equation, Navier--Stokes equation, and hypodissipative Navier--Stokes equation.
        We will show that this is also true for the Fourier-restricted Euler equation, and the Fourier-restricted, hypodissipative Navier--Stokes equation.
    \end{remark}

    \begin{proposition} \label{RestrictedOddDynamicsProp}
        Suppose $u^0\in \dot{H}^\frac{\log(3)}{2\log(2)}_{\mathcal{M}}$ is odd. Then the solution of the 
        Fourier-restricted Euler equation 
    $u\in C^1\left([0,T_{max});
    \dot{H}^\frac{\log(3)}{2\log(2)}_{\mathcal{M}}\right)$
    is also odd for all $0\leq t<T_{max}$.
    \end{proposition}

    \begin{proof}
        Let $v(x,t)=-u(-x,t)$. It is straightforward to check that $v$ is also a solution of the Fourier-restricted Euler equation, in particular because $-\mathcal{M}=\mathcal{M}$. Note that $v^0=u^0$, so by uniqueness
    \begin{equation}
        v(x,t)=u(x,t),
    \end{equation}
    for all $x\in\mathbb{T}^3$
    and for all $0\leq t<T_{max}$.
    \end{proof}

    \begin{proposition} \label{RestrictedHypoOddDynamicsProp}
        Suppose $u^0\in \dot{H}^s_{\mathcal{M}}$ is odd, where $s>\frac{\log(3)}{2\log(2)}-2\alpha, s\geq 0$. Then the solution of the 
        Fourier-restricted hypodissipative Navier--Stokes equation,
        $u\in C\left([0,T_{max});
    \dot{H}^s_{\mathcal{M}}\right)$
    is also odd for all $0\leq t<T_{max}$.
    \end{proposition}

    \begin{proof}
        Again let $v(x,t)=-u(-x,t)$. It is again straightforward to check that $v$ is also a solution of the Fourier-restricted, hypodissipative Navier--Stokes equation, and therefore by uniqueness we can conclude that oddness must be preserved dynamically in time.
    \end{proof}

\section{Dynamics of the Fourier-restricted model equation} \label{DynamicsSection}

In this section, we will study the dynamics of odd, permutation symmetric solutions of the Fourier-restricted Euler and hypodissipative Navier--Stokes equations. We will reduce the dynamics of these equations to an infinite system of ODEs that has a similar structure to the dyadic Euler/Navier--Stokes equations. In particular, we will show that solutions with these symmetries must satisfy the infinite system of ODES given in \cref{EulerODEintro,NavierStokesODEintro}.

Recall from \Cref{bunchofdefs}, that
for all $m\in\mathbb{Z}^+$
\begin{align}
    k^m&=2^{2m}\sigma
    +3^m \left(\begin{array}{c}
         1  \\
         0 \\ 
         -1 
    \end{array}\right) \\
    h^m&=2^{2m+1}\sigma
    +3^m \left(\begin{array}{c}
         1  \\
         1 \\ 
         -2 
    \end{array}\right) \\
    j^m&=2^{2m+1}\sigma
    +3^m \left(\begin{array}{c}
         2  \\
         -1 \\ 
         -1 
    \end{array}\right).
\end{align}
We will now state some useful identities involving the canonical elements of $\mathcal{M}$, which are straightforward computations left to the reader.
\begin{proposition} \label{CanonicalComputations}
    For all $m\in\mathbb{Z}^+,$
    \begin{align}
        \sigma \cdot k^m
        &= 3*2^{2m} \\
        \sigma \cdot h^m
        &= 3*2^{2m+1} \\
        \sigma \cdot j^m
        &= 3*2^{2m+1},
    \end{align}
    and we also have
    \begin{align}
        \left|k^m\right|^2
        &=3*2^{4m}+2*3^{2m} \\
        \left|h^m\right|^2
        &=3*2^{4m+2}+2*3^{2m+1} \\
        \left|j^m\right|^2
        &=3*2^{4m+2}+2*3^{2m+1}.
    \end{align}
    Note that this can be equivalently stated as, for all $n\in\mathbb{Z}^+$ and for all $k\in\mathcal{M}^+_n$,
    \begin{align}
    \sigma\cdot k &= 3* 2^n \\
    |k|^2 &=
    3*2^{2n}+2*3^n.
    \end{align}
\end{proposition}

We have already defined for each $k\in \mathcal{M},$ the vector
\begin{equation}
    v^k=\frac{P_{k}^\perp(\sigma)}
    {|P_{k}^\perp(\sigma)|},
\end{equation}
which gives the direction of the Fourier transform at each frequency. Now we will explicitly compute these vectors $v^{k^m},v^{h^m},v^{j^m}$, for our canonical frequencies,
but first we will show that our our frequencies are highly anisotropic, in that $k^m,h^m,j^m$ all converge conically to the $\spn(\sigma)$ exponentially fast.

\begin{definition}
    We will say that a sequence 
    $\left\{w^m\right\}_{m\in\mathbb{N}}
    \subset\mathbb{R}^3$ 
    converges conically to $\spn^+(v)$, where $v\in\mathbb{R}^3$ can be any vector $v\neq 0$, if
    \begin{equation}
        \lim_{m\to +\infty}
        \frac{w^m\cdot v}{|w^m||v|}
        =1.
    \end{equation}
    Likewise, we will define an $\epsilon$-conical neighborhood as the set
    \begin{equation}
        C_\epsilon(\spn^+(v))
        =
        \left\{w\in \mathbb{R}^3:
        (v\cdot w)>(1-\epsilon)|v||w|
        \right\}.
    \end{equation}
\end{definition}

\begin{proposition}
    For all $m\in\mathbb{Z}^+$,
    \begin{equation}
        \frac{\sigma\cdot k^m}
        {|\sigma||k^m|}
        =
        \frac{1}
        {\left(1+\frac{2}{3}\left(\frac{3}{4}\right)^{2m}\right)^\frac{1}{2}},
    \end{equation}
    and
    \begin{equation}
    \frac{\sigma\cdot h^m}
    {|\sigma||h^m|}
    =
    \frac{\sigma\cdot j^m}
    {|\sigma||j^m|}
    =
    \frac{1}{\left(1+\frac{1}{2}\left(\frac{3}{4}\right)^{2m}\right)^\frac{1}{2}}.
    \end{equation}
    We can therefore conclude that our canonical frequencies $k^m,h^m,j^m$---and by symmetry all of their permutations---converge conically to $\spn^+(\sigma)$ exponentially fast 
    as $m\to +\infty$.
\end{proposition}

\begin{proof}
    This result follows immediately from \Cref{CanonicalComputations}.
\end{proof}

\begin{proposition} \label{BasisComputations}
For all $m\in\mathbb{Z}^+$,
    \begin{align}
    v^{k^m}
    &=
    \frac{1}{\left(4*3^{4m+1}
    +2^{4m+1}*3^{2m+2}
    \right)^\frac{1}{2}}
    \left(
    2*3^{2m}\sigma
    -2^{2m}*3^{m+1}
        \left(\begin{array}{c}
             1  \\
             0 \\
             -1
        \end{array}\right)
        \right) \\
    v^{h^m}
    &=
    \frac{1}{\left(4*3^{4m+3}
    +2^{4m+3}*3^{2m+3}
    \right)^\frac{1}{2}}
        \left(
    2*3^{2m+1}\sigma
    -2^{2m+1}*3^{m+1}
    \left(
    \begin{array}{c}
         1  \\
         1 \\
         -2
    \end{array}
    \right)
    \right) \\
    v^{j^m}
    &=
    \frac{1}{\left(4*3^{4m+3}
    +2^{4m+3}*3^{2m+3}
    \right)^\frac{1}{2}}
        \left(
    2*3^{2m+1}\sigma
    -2^{2m+1}*3^{m+1}
    \left(
    \begin{array}{c}
         2  \\
         -1 \\
         -1
    \end{array}
    \right)
    \right).
    \end{align}
\end{proposition}

\begin{proof}
Beginning, with $k^m$, we compute that
\begin{align}
    P_{k^m}^\perp(\sigma)
    &=
    \sigma-\frac{\sigma\cdot k^m}{\left|k^m\right|^2}k^m \\
    &=
    \sigma-
    \frac{3*2^{2m}}{3*2^{4m}+2*3^{2m}}
    \left(2^{2m}\sigma+3^m
    \left(\begin{array}{c}
             1  \\
             0 \\
             -1
        \end{array}\right)\right) \\
    &=
    \frac{1}{3*2^{4m}+2*3^{2m}}
    \left(
    2*3^{2m}\sigma
    -2^{2m}*3^{m+1}
        \left(\begin{array}{c}
             1  \\
             0 \\
             -1
        \end{array}\right)
        \right).
    \end{align}
    Renormalizing we find that
    \begin{equation}
    \frac{P_{k^m}^\perp(\sigma)}
    {\left|P_{k^m}^\perp
    (\sigma)\right|}
    =
    \frac{1}{\left(4*3^{4m+1}
    +2^{4m+1}*3^{2m+2}
    \right)^\frac{1}{2}}
    \left(
    2*3^{2m}\sigma
    -2^{2m}*3^{m+1}
        \left(\begin{array}{c}
             1  \\
             0 \\
             -1
        \end{array}\right)
        \right).
\end{equation}

Now turning to $h^m$, we compute that
\begin{align}
    P_{h^m}^\perp(\sigma)
    &=
    \sigma-\frac{\sigma\cdot h^m}
    {\left|h^m\right|^2}h^m \\
    &=
    \sigma
    -
    \frac{3*2^{2m+1}}
    {3*2^{4m+2}+2*3^{2m+1}}
    \left(2^{2m+1}\sigma+3^m\left(
    \begin{array}{c}
         1  \\
         1 \\
         -2
    \end{array}
    \right)\right) \\
    &=
    \frac{1}{3*2^{4m+2}+2*3^{2m+1}}
    \left(
    2*3^{2m+1}\sigma
    -2^{2m+1}*3^{m+1}
    \left(
    \begin{array}{c}
         1  \\
         1 \\
         -2
    \end{array}
    \right)
    \right).
\end{align}
Renormalzing, we find that
\begin{equation}
    \frac{P_{h^m}^\perp(\sigma)}
    {\left|P_{h^m}^\perp
    (\sigma)\right|}
    =
    \frac{1}{\left(4*3^{4m+3}
    +2^{4m+3}*3^{2m+3}
    \right)^\frac{1}{2}}
        \left(
    2*3^{2m+1}\sigma
    -2^{2m+1}*3^{m+1}
    \left(
    \begin{array}{c}
         1  \\
         1 \\
         -2
    \end{array}
    \right)
    \right).
\end{equation}

Finally turning to $j^m$, we compute that
\begin{align}
    P_{j^m}^\perp(\sigma)
    &=
    \sigma-\frac{\sigma\cdot j^m}
    {\left|j^m\right|^2}j^m \\
    &=
    \sigma
    -
    \frac{3*2^{2m+1}}
    {3*2^{4m+2}+2*3^{2m+1}}
    \left(2^{2m+1}\sigma+3^m\left(
    \begin{array}{c}
         2  \\
         -1 \\
         -1
    \end{array}
    \right)\right) \\
    &=
    \frac{1}{3*2^{4m+2}+2*3^{2m+1}}
    \left(
    2*3^{2m+1}\sigma
    -2^{2m+1}*3^{m+1}
    \left(
    \begin{array}{c}
         2  \\
         -1 \\
         -1
    \end{array}
    \right)
    \right).
\end{align}
Renormalzing, we find that
\begin{equation}
    \frac{P_{j^m}^\perp(\sigma)}
    {\left|P_{j^m}^\perp
    (\sigma)\right|}
    =
    \frac{1}{\left(4*3^{4m+3}
    +2^{4m+3}*3^{2m+3}
    \right)^\frac{1}{2}}
        \left(
    2*3^{2m+1}\sigma
    -2^{2m+1}*3^{m+1}
    \left(
    \begin{array}{c}
         2  \\
         -1 \\
         -1
    \end{array}
    \right)
    \right).
\end{equation}
This completes the proof.
    \end{proof}

\begin{proposition} \label{BasisProp}
    Our basis vectors are permutation symmetric, with for all permutations $P\in\mathcal{P}_3$, and all $k\in \mathcal{M}$,
    \begin{equation}
        v^{P(k)}=P(v^k).
    \end{equation}
Our basis vectors are also even, with for all $k\in\mathcal{M}$, 
    \begin{equation}
        v^{-k}=v^k
    \end{equation}
\end{proposition}

\begin{proof}
    For all $k\in\mathcal{M}$, let 
    \begin{equation}
        w^k=\sigma -\frac{\sigma\cdot k}{|k|^2}k,
    \end{equation}
    noting that
    \begin{equation}
        v^k=\frac{w^k}{|w^k|}.
    \end{equation}
    It is immediate to observe that 
    \begin{align}
        w^{P(k)}
        &=
        \sigma -\frac{\sigma\cdot P(k)}
        {|P(k)|^2}P(k)\\
        &=
        \sigma -\frac{\sigma\cdot k}
        {|k|^2}P(k) \\
        &=
        P(w^k).
    \end{align}
    Clearly $|P(w^k)|=|w^k|$, because permuting the entries will not change the magnitude of a vector, and so
    \begin{equation}
        v^{P(k)}=P(v^k).
    \end{equation}
    Likewise we can see that
    \begin{equation}
        w^{-k}=w^k,
    \end{equation}
    and again normalizing, this completes the proof.
\end{proof}

\begin{proposition}  \label{DyadicPowers}
    For all $k\in \mathcal{M}_n^+$,
    \begin{equation} \label{SigmaIdentity}
    \sigma\cdot k= 3*2^n.
    \end{equation}
    Furthermore, if $k=h+j$, where $h,j\in \mathcal{M}^+$,
    then $h,j\in \mathcal{M}_{n-1}^+$.
    If  $k=h-j$, where $h,j\in\mathcal{M}^+$,
    then $h\in\mathcal{M}_{n+1}^+, j\in\mathcal{M}_n^+$. Note that 
    $-j\in\mathcal{M}_n^-$
    
    Likewise, for all $k\in \mathcal{M}_n^-$,
    \begin{equation}
    \sigma\cdot k= -3*2^n.
    \end{equation}
    Furthermore, if $k=h+j$, where $h,j\in\mathcal{M}^-$, then $h,j\in\mathcal{M}_{n-1}^-$. If $k=h+j$ where $h\in\mathcal{M}^-,j\in\mathcal{M}^+$,
    then $h\in \mathcal{M}_{n+1}^-, j\in\mathcal{M}^n_+$.
\end{proposition}

\begin{proof}
It is only necessary to deal with the case where $k\in\mathcal{M}^+_n$, because the negative case is exactly analogous, simply with all the signs flipped.
The identity \eqref{SigmaIdentity} was already proven in \Cref{CanonicalComputations},
and the result follows from \eqref{SigmaIdentity} and the fact that $2^n+2^r=2^s$, if and only if $r=n$ and $s=n+1$; and
    $2^n-2^r=2^s$ if and only if
    $r=s$ and $n=s+1$.
\end{proof}

\begin{remark}
    Note that this result can be stated equivalently in terms of the permutations of $k^m,h^m,j^m$.
    If $a,b,a+b \in \mathcal{M}^+$, then either
    $a,b\in \mathcal{P}\left[k^m\right]$
    and $a+b\in \mathcal{P}\left[h^m\right]
    \cup
    \mathcal{P}\left[j^m\right]$;
    or
    $a,b\in
    \mathcal{P}\left[h^m\right]
    \cup
    \mathcal{P}\left[j^m\right]$
    and 
    $a+b\in \mathcal{P}\left[k^{m+1}\right]$.
    Furthermore if $a,b,a-b \in \mathcal{M}^+$, then either
    $a\in \mathcal{P}\left[k^{m+1}\right]$
    and $b,a-b\in
    \mathcal{P}\left[h^m\right]
    \cup
    \mathcal{P}\left[j^m\right]$;
    or 
    $a\in \mathcal{P}\left[h^m\right]
    \cup
    \mathcal{P}\left[j^m\right]$
    and $b,a-b \in \mathcal{P}\left[k^m\right]$.
\end{remark}

\begin{remark}
    Note that \Cref{DyadicPowers} implies that $\mathcal{M}^+$ can be decomposed into levels
    \begin{align}
        \mathcal{M}^+
        &=
        \mathcal{M}^+_0 \cup
        \mathcal{M}^+_1 \cup
        \mathcal{M}^+_2 \cup
        \mathcal{M}^+_3 \cup
        \mathcal{M}^+_4 \cup
        \mathcal{M}^+_5 \cup... \\
        &=
        \left\{
        \mathcal{P}\left[k^0\right],
        \mathcal{P}\left[h^0\right]
        \cup
        \mathcal{P}\left[j^0\right],
        \mathcal{P}\left[k^1\right],
        \mathcal{P}\left[h^1\right]
        \cup
        \mathcal{P}\left[j^1\right],
        \mathcal{P}\left[k^2\right],
        \mathcal{P}\left[h^2\right]
        \cup
        \mathcal{P}\left[j^2\right],
        ...
        \right\},
    \end{align}
    such that when we take the sum and difference of Fourier modes due to the convolution in Fourier space from the quadratic nonlinearity $(u\cdot\nabla)u$, we only get dyadic interactions involving the current level, and the level above and below, which dramatically simplifies the system in Fourier space.
\end{remark}

\begin{proposition} \label{DyadicSumDiff}
    The canonical elements of $M^+$ can be expressed as elements of
    $\mathcal{M}^+ +\mathcal{M}^+$ in only the following ways:
    for all $m\geq 1$,
    \begin{equation} \label{eq1}
        k^m=h^{m-1}+j^{m-1};
    \end{equation}
    for all $m\geq 0$,
    \begin{align}
        h^m= \label{eq2}
        k^m+P_{12}\left(k^m\right) \\
        j^m= \label{eq3}
        k^m+P_{23}\left(k^m\right).
    \end{align}
    Furthermore, the canonical elements of $M^+$ can be expressed as elements of
    $\mathcal{M}^+ -\mathcal{M}^+$ in only the following ways:
    for all $m\geq 0$,
    \begin{align}
        k^m
        &= \label{eq4}
        h^{m}-
        P_{12}\left(k^m\right) \\
        k^m
        &=  \label{eq5}
        j^{m}-
        P_{23}\left(k^m\right) \\
        h^m
        &= \label{eq6}
        k^{m+1}-j^m \\
        h^m
        &= \label{eq7}
        P_{12}\left(k^{m+1}\right)
        -P_{12}\left(j^m\right) \\
        j^m
        &= \label{eq8}
        k^{m+1}-h^m \\
        j^m
        &= \label{eq9}
        P_{23}\left(k^{m+1}\right)
        -P_{23}\left(h^m\right).
    \end{align}
\end{proposition}

\begin{proof}
    First we will deal with the sum case. We begin by observing that \Cref{DyadicPowers} implies that if
    $a+b=k^m$, for $m\in \mathbb{N}$,
    then $a,b\in 
    \mathcal{P}[h^{m-1}]
    \cup\mathcal{P}[j^{m-1}]$.
    The vector $\sigma=(1,1,1)$ is invariant under all permutations, and clearly 
    \begin{equation}
        2^{2m-1}\sigma+2^{2m-1}\sigma
        =2^{2m}\sigma,
    \end{equation}
    so it suffices to observe that
    \begin{equation}
        \left(
        \begin{array}{c}
              3 \\
               0 \\
               -3 
        \end{array}
        \right)
        =
        \left(
        \begin{array}{c}
              1 \\
               1 \\
               -2 
        \end{array}
        \right)
        +\left(
        \begin{array}{c}
              2 \\
               -1 \\
               -1 
        \end{array}
        \right),
    \end{equation}
    is the only way to express $(3,0,-3)$ as a sum of permutations of $(1,1,-2)$ and $(2,-1,-1)$.
    This immediately implies $a=h^{m-1}, b=j^{m-1}$ or vice versa.
    Next we observe that if $a+b=h^m$,
    then \Cref{DyadicPowers} implies that $a,b\in \mathcal{P}[k^m]$.
    It suffices to observe that 
        \begin{equation}
        \left(
        \begin{array}{c}
              2 \\
               -1 \\
               -1 
        \end{array}
        \right)
        =
        \left(
        \begin{array}{c}
              1 \\
               0 \\
               -1 
        \end{array}
        \right)
        +\left(
        \begin{array}{c}
              1 \\
               -1 \\
               0 
        \end{array}
        \right),
    \end{equation}
    is the only way to express 
    $(2,-1,-1)$ as a sum of permutations of $(1,0,-1)$.
    This implies that $a=k^m, 
    b=P_{23}(k^m)$,
    or vice versa.
    The proof for $j^m$ is entirely analogous, and is left to the reader.

    Now we can consider the differences, expressing elements of $\mathcal{M}^+$ as elements of $\mathcal{M}^+-\mathcal{M}^-$.
    Simply by subtracting across the correct term, we can derive \eqref{eq4} from \eqref{eq2}; can derive \eqref{eq5} from \eqref{eq3};
    and can derive \eqref{eq6} and \eqref{eq8} from \eqref{eq1}.
    Observing that $h^m=P_{12}(h^m)$, we can apply $P_{12}$ to \eqref{eq6} to obtain \eqref{eq7}. Likewise observing that $j^m=P_{23}(j^m)$, we can apply $P_{23}$ to \eqref{eq8}
    to obtain \eqref{eq9}.

    It remains to show that these differences are in fact the only ways to express the canonical elements of $\mathcal{M}^+$ as elements of $\mathcal{M}^+ -\mathcal{M}^+$.
    For this it suffices to observe the following:
    \begin{equation}
        \left(
        \begin{array}{c}
             1  \\
             0 \\
             -1
        \end{array}
        \right)
        =
        \left(
        \begin{array}{c}
             2  \\
             -1 \\
             -1
        \end{array}
        \right)
        -\left(
        \begin{array}{c}
             1  \\
             -1 \\
             0
        \end{array}
        \right),
    \end{equation}
    is the only ways to express $(1,0,-1)$ as a difference of permutations of $(2,-1,-1)$ and $(1,0,-1)$;
    \begin{equation}
        \left(
        \begin{array}{c}
             1  \\
             0 \\
             -1
        \end{array}
        \right)
        =
        \left(
        \begin{array}{c}
             1  \\
             1 \\
             -2
        \end{array}
        \right)
        -\left(
        \begin{array}{c}
             0  \\
             1 \\
             -1
        \end{array}
        \right),
    \end{equation}
    is the only way to express $(1,0,-1)$ as a difference of permutations of $(1,1,-2)$ and $(1,0,-1)$;
    \begin{align}
        \left(
        \begin{array}{c}
             2  \\
             -1 \\
             -1
        \end{array}
        \right)
        &=
        \left(
        \begin{array}{c}
             3  \\
             0 \\
             -3
        \end{array}
        \right)
        -\left(
        \begin{array}{c}
             1  \\
             1 \\
             -2
        \end{array}
        \right) \\
        &=
        \left(
        \begin{array}{c}
             3  \\
             -3 \\
             0
        \end{array}
        \right)
        -\left(
        \begin{array}{c}
             1  \\
             -2 \\
             1
        \end{array}
        \right)
    \end{align}
    are the only ways to express $(2,-1,-1)$ as a difference of permutations of $(3,0,-3)$ and $(1,1,-2)$;
    \begin{align}
        \left(
        \begin{array}{c}
             1  \\
             1 \\
             -2
        \end{array}
        \right)
        &=
        \left(
        \begin{array}{c}
             3  \\
             0 \\
             -3
        \end{array}
        \right)
        -\left(
        \begin{array}{c}
             2  \\
             -1 \\
             -1
        \end{array}
        \right) \\
        &=
        \left(
        \begin{array}{c}
             0  \\
             3 \\
             -3
        \end{array}
        \right)
        -\left(
        \begin{array}{c}
             -1  \\
            2 \\
             -1
        \end{array}
        \right)
    \end{align}
    are the only ways to express $(1,1,-2)$ as a difference of permutations of $(3,0,-3)$ and $(2,-1,-1)$.
    This completes the proof.
\end{proof}

\begin{proposition} \label{ComputeTheFuckingModesProp}
    We can express the bilinear term on the canonical frequencies for the interactions of the modes listed above as follows. Let
    \begin{align}
        a_m &= 
        \frac{\sqrt{6}\pi}
    {\left(1+\frac{1}{2}
    \left( \frac{3}{4}
    \right)^{2m}\right)^\frac{1}{2}}
    3^m
    \\
    b_m &= \frac{\sqrt{2}\pi}{\left(1+\frac{3}{8}
    \left(\frac{3}{4}\right)^{2m}
    \right)^\frac{1}{2}} 3^{m+1},
    \end{align} 
    and fix $m\in\mathbb{Z}^+$.
    \begin{enumerate}
    
        \item 
    Let $u$ and $\Tilde{u}$ be given by
\begin{align}
    u&=
    iv^{k^m}e^{2\pi i k^m\cdot x} \\
    \Tilde{u}
    &=
    iP_{12}(v^{k^m})
    e^{2\pi i P_{12}(k^m)\cdot x}.
    \end{align}
Then the bilinear term in the restricted model is given by
\begin{equation}
    \mathbb{P}_{M}((\Tilde{u}\cdot\nabla) u
    +(u\cdot\nabla)\Tilde{u})
    =
    -a_m i v^{h^m}e^{2\pi i h^m\cdot x}.
\end{equation}

\item 
Let $u$ and $\Tilde{u}$ be given by
\begin{align}
    u&=
    iv^{k^m}e^{2\pi i k^m\cdot x} \\
    \Tilde{u}
    &=
    iP_{23}(v^{k^m})
    e^{2\pi i P_{23}(k^m)\cdot x}.
    \end{align}
Then the bilinear term in the restricted model is given by
\begin{equation}
    \mathbb{P}_{\mathcal{M}}((\Tilde{u}\cdot\nabla) u
    +(u\cdot\nabla)\Tilde{u})
    =
    -a_m i v^{j^m}e^{2\pi i j^m\cdot x}.
\end{equation}

\item
Let $u$ and $\Tilde{u}$ be given by
\begin{align}
    u&=
    iv^{h^m}e^{2\pi i h^m\cdot x} \\
    \Tilde{u}
    &=
    iv^{j^m}
    e^{2\pi i j^m\cdot x}.
    \end{align}
Then the bilinear term in the restricted model is given by
\begin{equation}
    \mathbb{P}_{\mathcal{M}}((\Tilde{u}\cdot\nabla) u
    +(u\cdot\nabla)\Tilde{u})
    =
    -b_m i v^{k^{m+1}}e^{2\pi i k^{m+1}\cdot x}.
\end{equation}

\item 
Let $u$ and $\Tilde{u}$ be given by
\begin{align}
    u&=
    iv^{h^m}e^{2\pi i h^m\cdot x} \\
    \Tilde{u}
    &=
    -iP_{12}(v^{k^m})
    e^{-2\pi i P_{12}(k^m)\cdot x}.
    \end{align}
Then the bilinear term in the restricted model is given by
\begin{equation}
    \mathbb{P}_{\mathcal{M}}((\Tilde{u}\cdot\nabla) u
    +(u\cdot\nabla)\Tilde{u})
    =
    \frac{a_m}{2}i v^{k^m}e^{2\pi i k^m\cdot x}.
\end{equation}

\item
Let $u$ and $\Tilde{u}$ be given by
\begin{align}
    u&=
    iv^{j^m}e^{2\pi i j^m\cdot x} \\
    \Tilde{u}
    &=
    -iP_{23}(v^{k^m})
    e^{-2\pi i P_{23}(k^m)\cdot x}.
    \end{align}
Then the bilinear term in the restricted model is given by
\begin{equation}
    \mathbb{P}_{\mathcal{M}}((\Tilde{u}\cdot\nabla) u
    +(u\cdot\nabla)\Tilde{u})
    =
    \frac{a_m}{2}i v^{k^m}e^{2\pi i k^m\cdot x}.
\end{equation}

\item 
Let $u$ and $\Tilde{u}$ be given by
\begin{align}
    u&=
    iv^{k^{m+1}}e^{2\pi i k^{m+1}\cdot x} \\
    \Tilde{u}
    &=
    -iv^{j^m}
    e^{-2\pi i j^m\cdot x}.
    \end{align}
Then the bilinear term in the restricted model is given by
\begin{equation}
    \mathbb{P}_{\mathcal{M}}((\Tilde{u}\cdot\nabla) u
    +(u\cdot\nabla)\Tilde{u})
    =
    \frac{b_m}{2} i v^{h^m}e^{2\pi i h^m\cdot x}.
\end{equation}

\item 
Let $u$ and $\Tilde{u}$ be given by
\begin{align}
    u&=
    iP_{12}(v^{k^{m+1}})e^{2\pi i 
    P_{12}(k^{m+1})\cdot x} \\
    \Tilde{u}
    &=
    -iP_{12}(v^{j^m})
    e^{-2\pi i P_{12}(j^m)\cdot x}.
    \end{align}
Then the bilinear term in the restricted model is given by
\begin{equation}
    \mathbb{P}_{\mathcal{M}}((\Tilde{u}\cdot\nabla) u
    +(u\cdot\nabla)\Tilde{u})
    =
    \frac{b_m}{2} i v^{h^m}e^{2\pi i h^m\cdot x}.
\end{equation}

\item 
Let $u$ and $\Tilde{u}$ be given by
\begin{align}
    u&=
    iv^{k^{m+1}}e^{2\pi i k^{m+1}\cdot x} \\
    \Tilde{u}
    &=
    -iv^{h^m}
    e^{-2\pi i h^m\cdot x}.
    \end{align}
Then the bilinear term in the restricted model is given by
\begin{equation}
    \mathbb{P}_{\mathcal{M}}((\Tilde{u}\cdot\nabla) u
    +(u\cdot\nabla)\Tilde{u})
    =
    \frac{b_m}{2} i v^{j^m}e^{2\pi i j^m\cdot x}.
\end{equation}

\item 
Let $u$ and $\Tilde{u}$ be given by
\begin{align}
    u&=
    iv^{P_{23}(k^{m+1})}
    e^{2\pi i P_{23}(k^{m+1})\cdot x} \\
    \Tilde{u}
    &=
    -iP_{23}(v^{h^m})
    e^{-2\pi i P_{23}(h^m)\cdot x}.
    \end{align}
Then the bilinear term in the restricted model is given by
\begin{equation}
    \mathbb{P}_{\mathcal{M}}((\Tilde{u}\cdot\nabla) u
    +(u\cdot\nabla)\Tilde{u})
    =
    \frac{b_m}{2} i v^{j^m}e^{2\pi i j^m\cdot x}.
\end{equation}

    \end{enumerate}    
\end{proposition}

\begin{remark}
    We will prove this proposition in \Cref{AppendixModeInteraction}. The proof is just a series of elementary, but rather tedious, multi-variable calculus computations. Note that we choose to express the modes with a multiple of $i$ in front of the positive modes---and $-i$ for the negative modes---because our blowup Ansatz is odd. 
\end{remark}

\begin{proposition} \label{PermutationSymmetricCategorizationProp}
    Suppose $u\in \dot{H}^\frac{\log(3)}{2\log(2)}_{\mathcal{M}}$
    is an odd, permutation-symmetric vector field.
    For all $m\in\mathbb{Z}^+$, define the scalars $\phi_m,\eta_m,\zeta_m\in \mathbb{R}$, by
    \begin{align}
        \phi_m &= -i v^{k^m} \cdot \hat{u}(k^m) \\
        \eta_m &= -i v^{h^m} \cdot \hat{u}(h^m) \\
        \zeta_m &= -i v^{j^m} \cdot \hat{u}(j^m).
    \end{align}
    Then for all permutations $P\in \mathcal{P}_3$,
    \begin{align}
        \hat{u}(P(k^m))&=
        i \phi_m v^{P(k^m)} \\
        \hat{u}(P(h^m))&=
        i \eta_m v^{P(h^m)} \\
        \hat{u}(P(j^m))&=
        i \zeta_m v^{P(j^m)},
    \end{align}
    and 
    \begin{align}
        \hat{u}(-P(k^m))&=
        -i \phi_m v^{P(k^m)} \\
        \hat{u}(-P(h^m))&=
        -i \eta_m v^{P(h^m)} \\
        \hat{u}(-P(j^m))&=
        -i \zeta_m v^{P(j^m)}
    \end{align}
    Furthermore, $u$ can then be expressed in terms of the scalars $\phi_m,\eta_m,\zeta_m$ as a Fourier series using complex exponentials,
    \begin{multline}
    u(x)=\sum_{m=0}^{+\infty}\bigg(
    i\phi_m\sum_{k\in \mathcal{P}[k^m]} v^k
    \left(e^{2\pi i k\cdot x}
    -e^{-2\pi i k\cdot x}\right) 
    +i\eta_m\sum_{h\in \mathcal{P}[h^m]} v^h
    \left(e^{2\pi i h\cdot x}
    -e^{-2\pi i h\cdot x}\right) \\
    +i\zeta_m\sum_{j\in \mathcal{P}[j^m]} v^j
    \left(e^{2\pi i j\cdot x}
    -e^{-2\pi i j\cdot x}\right)
    \bigg),
    \end{multline}
    or using sines,
    \begin{multline}
    u(x)=-2\sum_{m=0}^{+\infty}\bigg(
    \phi_m\sum_{k\in \mathcal{P}[k^m]} v^k
    \sin(2\pi k\cdot x)
    +\eta_m\sum_{h\in \mathcal{P}[h^m]} v^h
    \sin(2\pi h\cdot x) \\
    +\zeta_m\sum_{j\in \mathcal{P}[j^m]} v^j
    \sin(2\pi j\cdot x)
    \bigg).
    \end{multline}
\end{proposition}

\begin{proof}
    We will begin by observing that because $u$ is odd, and is in the constraint space $\dot{H}^s_\mathcal{M}$,
    we must have
    \begin{equation}
        \hat{u}(k^m)=ic_{k^m} v^{k^m},
    \end{equation}
    and so we may compute that
    \begin{align}
    \phi_m&= -i v^{k^m}\cdot \hat{u}(k^m) \\
    &= 
    c_{k^m}.
    \end{align}
    We can see from this that $\phi_m\in\mathbb{R}$ and that
    \begin{equation}
        \hat{u}(k^m)=i\phi_m v^{k^m},
    \end{equation}
    By hypothesis, $u$ is permutation symmetric, and so $\hat{u}$ must also be permutation symmetric by \Cref{FourierPermute}. Applying \Cref{BasisProp}, we can see that for all $P\in\mathcal{P}_3$
    \begin{align}
        \hat{u}(P(k^m))
        &=
        P(\hat{u}(k^m)) \\
        &=
        P(i\phi_m v^{k^m}) \\
        &=
        i\phi_m v^{P(k^m)}.
    \end{align}
    Likewise, because $u$---and therefore $\hat{u}$---is odd, we can conclude
    \begin{align}
        \hat{u}(-P(k^m))
        &=
        -P(\hat{u}(k^m)) \\
        &=
        -i\phi_m v^{P(k^m)}.
    \end{align}
    The computations for $h^m$ and $j^m$ are entirely analogous and are left to the reader.

    It remains only to express $u$ in terms of the scalars $\phi_m,\eta_m,\zeta_m$.
    Observe that the scalars $\phi_m,\eta_m,\zeta_m$, completely determine the Fourier transform of the vector field, $\hat{u}$, because
    \begin{equation}
        \supp(\hat{u})\subset \mathcal{M}
        =\bigcup_{m=0}^\infty
        \Big(\pm\mathcal{P}[k^m]
        \cup \pm\mathcal{P}[h^m]
        \cup \pm\mathcal{P}[j^m]\Big).
    \end{equation}
    Therefore we have
    \begin{align}
    u(x)
    &=
    \sum_{k\in\mathcal{M}}
    \hat{u}(k) e^{2\pi i k\cdot x} \\
    &=
    \sum_{k\in\mathcal{M}^+}\hat{u}(k) 
    \left(e^{2\pi i k\cdot x}
    -e^{-2\pi i k\cdot x}\right),
    \end{align}
    and plugging into our identities for $\hat{u}(k)$, the result follows.
\end{proof}

\begin{theorem} \label{GeneralRestrictedEulerODE}
    Suppose $u^0\in \dot{H}^\frac{\log(3)}{2\log(2)}_\mathcal{M}$ is odd and permutation-symmetric. Then the solution of the Fourier-restricted Euler equation 
    $u\in C^1\left([0,T_{max});
    \dot{H}^\frac{\log(3)}{2\log(2)}
    _{\mathcal{M}}\right)$
    with initial data $u^0$ is also odd and permutation symmetric for all $0\leq t<T_{max}$.
    This implies the solution can be expressed as
    \begin{multline}
    u(x,t)=-2\sum_{m=0}^{+\infty}\bigg(
    \phi_m(t)\sum_{k\in \mathcal{P}[k^m]} v^k
    \sin(2\pi k\cdot x)
    +\eta_m(t)\sum_{h\in \mathcal{P}[h^m]} v^h
    \sin(2\pi h\cdot x) \\
    +\zeta_m(t)\sum_{j\in \mathcal{P}[j^m]} v^j
    \sin(2\pi j\cdot x)
    \bigg).
    \end{multline}
    Furthermore, for all $0\leq t<T_{max}$, the scalars $\phi_m,\eta_m,\zeta_m$ satisfy the 
    infinite system of ODEs:
    \begin{equation}
        \partial_t \phi_0=
        -\frac{a_0}{2}\phi_0\eta_0
        -\frac{a_0}{2}\phi_0\zeta_0,
    \end{equation}
    and for all $m\in \mathbb{N}$,
    \begin{equation} \label{kMode}
        \partial_t\phi_m
        =
        b_{m-1}\eta_{m-1}\zeta_{m-1}
        -\frac{a_m}{2}\phi_m\eta_m
        -\frac{a_m}{2}\phi_m\zeta_m,
    \end{equation}
    and for all $m\in \mathbb{Z}^+$,
    \begin{align}
        \partial_t\eta_m  \label{hMode}
        &=
        a_m\phi_m^2
        -b_m \zeta_m \phi_{m+1} \\
        \partial_t \zeta_m  \label{jMode}
        &=
        a_m\phi_m^2
        -b_m \eta_m \phi_{m+1},
    \end{align}
    where for all $m\in \mathbb{Z}^+$,
    \begin{align}
        a_m &= 
        \frac{\sqrt{2}\pi}
    {\left(1+\frac{1}{2}
    \left( \frac{3}{4}
    \right)^{2m}\right)^\frac{1}{2}}
    3^{m+\frac{1}{2}}
    \\
    b_m &= \frac{\sqrt{2}\pi}{\left(1+\frac{1}{2}
    \left(\frac{3}{4}\right)^{2m+1}
    \right)^\frac{1}{2}} 3^{m+1},
    \end{align}
    as in \Cref{ComputeTheFuckingModesProp}.
    Note that we will take
    \begin{equation}
        b_{-1},\eta_{-1},\zeta_{-1}:=0
    \end{equation} 
    by convention, so that \eqref{kMode} holds for $m=0$ as well.
\end{theorem}

\begin{proof}
We have already shown in \Cref{RestrictedPermuteDynamicsProp,RestrictedOddDynamicsProp} that if the initial data is permutation-symmetric and odd, then this is preserved by the dynamics, so clearly for all $0\leq t<T_{max}$, we know that $u(\cdot,t)$ is odd and permutation symmetric.
\Cref{PermutationSymmetricCategorizationProp} then immediately implies that
    \begin{multline}
    u(x,t)=\sum_{m=0}^{+\infty}\bigg(
    i\phi_m(t)\sum_{k\in \mathcal{P}[k^m]} v^k
    \left(e^{2\pi i k\cdot x}
    -e^{-2\pi i k\cdot x}\right) 
    +i\eta_m(t)\sum_{h\in \mathcal{P}[h^m]} v^h
    \left(e^{2\pi i h\cdot x}
    -e^{-2\pi i h\cdot x}\right) \\
    +i\zeta_m(t)\sum_{j\in \mathcal{P}[j^m]} v^j
    \left(e^{2\pi i j\cdot x}
    -e^{-2\pi i j\cdot x}\right)
    \bigg).
    \end{multline}
The sine series formulation may be clearer, but we leave the series in complex exponential form because it will make it easier to compute the bilinear term that way.

\Cref{DyadicSumDiff} then implies that the only interactions for the canonical frequencies from the nonlinear term are the ones computed in \Cref{ComputeTheFuckingModesProp}, yielding
\begin{multline}
    -\mathbb{P}_{\mathcal{M}}
    ((u\cdot\nabla)u)=\sum_{m=0}^{+\infty}\bigg(
    i(b_{m-1}\eta_{m-1}\zeta_{m-1}
        -\frac{a_m}{2}\phi_m\eta_m
        -\frac{a_m}{2}\phi_m\zeta_m)
        \sum_{k\in \mathcal{P}[k^m]} v^k
    \left(e^{2\pi i k\cdot x}
    -e^{-2\pi i k\cdot x}\right) \\
    +i(a_m\phi_m^2-b_m \zeta_m\phi_{m+1})
        \sum_{h\in \mathcal{P}[h^m]} v^h
    \left(e^{2\pi i h\cdot x}
    -e^{-2\pi i h\cdot x}\right) \\
    +i(a_m\phi_m^2-b_m \eta_m\phi_{m+1})
    \sum_{j\in \mathcal{P}[j^m]} v^j
    \left(e^{2\pi i j\cdot x}
    -e^{-2\pi i j\cdot x}\right)
    \bigg).
    \end{multline}
More specifically, applying \Cref{ComputeTheFuckingModesProp}, we can see that the computations 3, 4, and 5 yield \eqref{kMode}; the computations 1,6, and 7 yield \eqref{hMode};
and the computations 2, 8, and 9 yield \eqref{jMode}.
These equations simply express
\begin{equation}
    \partial_t u=-\mathbb{P}_{\mathcal{M}}((u\cdot\nabla) u),
\end{equation}
in terms of the Fourier series coefficients.
The fact that 
$\left\{e^{2\pi ik\cdot x}
\right\}_{k\in\mathcal{M}}$,
provides a complete basis of $\dot{H}^s_{\mathcal{M}}$ means that two vector fields in $\dot{H}^s_{\mathcal{M}}$ are equal if and only if their Fourier transforms are equal at each frequency $k\in\mathcal{M}$, so this completes the proof.
Note that we have changed the expressions for $a_m$ and $b_m$ from \Cref{ComputeTheFuckingModesProp} in order to symmetrize the expressions, but the values of the constants are unchanged.
\end{proof}

\begin{definition}
    We will say that an odd, permutation-symmetric vector field $u\in \dot{H}^\frac{\log(3)}{2\log(2)}
    _{\mathcal{M}}$
    has hj-parity if for all $m\in \mathbb{Z}^+$
    \begin{equation}
        \eta_m=\zeta_m,
    \end{equation}
    with $\eta_m,\zeta_m$ taken as in \Cref{PermutationSymmetricCategorizationProp}.
\end{definition}

\begin{proposition} \label{hjProp}
    Suppose initial data $u^0 \in \dot{H}^\frac{\log(3)}{2\log(2)}_{\mathcal{M}}$
    is an odd, permutation-symmetric vector field with hj-parity.
    Then the solution of the Fourier-restricted Euler equation 
    $u\in C^1\left([0,T_{max});
    \dot{H}^\frac{\log(3)}{2\log(2)}
    _{\mathcal{M}}\right)$,
    also has hj-parity for all $0\leq t<T_{max}$.
\end{proposition}

\begin{remark}
    We will in fact prove a slightly stronger result. We will show that if
    $u\in C^1\left([0,T_{max});
    \dot{H}^\frac{\log(3)}{2\log(2)}
    _{\mathcal{M}}\right)$
    is a strong solution of the Fourier-restricted Euler equation, then for all $m\in\mathbb{Z}^+$ and for all $0\leq t<T_{max}$,
    \begin{equation}
        \eta_m(t)-\zeta_m(t)
        =
        \left(\eta_m(0)-\zeta_m(0)\right)
        \exp\left(-b_m\int_0^t
        \phi_{m+1}(\tau)\diff\tau\right),
    \end{equation}
    and in particular if $\eta_m(0)=\zeta_m(0)$, then for all $0\leq t<T_{max}$,
\begin{equation}
    \eta_m(t)=\zeta_m(t).
    \end{equation}
\end{remark}

\begin{proof}
    Let $\rho_m(t)=\eta_m(t)-\zeta_m(t)$.
    Plugging into \Cref{GeneralRestrictedEulerODE},
    we can see that for all $0\leq t<T_{max}$,
    \begin{align}
\partial_t \rho_m(t)
&=
-b_m\zeta_m\phi_{m+1}+b_m\eta_m\phi_{m+1} \\
&=
-b_m\phi_{m+1}\rho_m,
    \end{align}
and integrating this differential equation completes the proof.
\end{proof}

\begin{proposition} \label{RestrictedEulerPartiyODE}
    Suppose $u\in C^1\left([0,T_{max});
    \dot{H}^\frac{\log(3)}{2\log(2)}
    _{\mathcal{M}}\right)$
    is an odd, permutation symmetric solution of the Fourier-restricted Euler equation with hj-parity.
    For all $m\in\mathbb{Z}^+$, let
    \begin{equation}
    \psi_{2m}:=\phi_m
    \end{equation}
    and let
    \begin{equation}
    \psi_{2m+1}:=\eta_m=\zeta_m,
    \end{equation}
    Then for all $0\leq t<T_{max}$, and for all $n\in\mathbb{Z}^+$,
    \begin{equation}
        \partial_t\psi_n
        =\sqrt{2}\pi \beta_{n-1}
        \left(\sqrt{3}\right)^n
        \psi_{n-1}^2
        -\sqrt{2}\pi \beta_{n}
        \left(\sqrt{3}\right)^{n+1}
        \psi_{n}\psi_{n+1},
    \end{equation}
    where for all $n\in\mathbb{Z}^+$
    \begin{equation}
        \beta_n=\frac{1}{\left(
        1+\frac{1}{2}
        \left(\frac{3}{4}\right)^n
        \right)^\frac{1}{2}},
    \end{equation}
    and by convention 
    \begin{equation}
        \beta_{-1},\psi_{-1}:=0.
    \end{equation}
\end{proposition}

\begin{proof}
    In order to prove the result for all $n\in\mathbb{Z}^+$, we will need to deal with two separate cases, $n=2m$ and $n=2m+1$.
    Let $n=2m$.
    Then from \Cref{GeneralRestrictedEulerODE},
    we can see that for all $0\leq t<T_{max}$,
    \begin{equation}
    \partial_t\psi_n
    =
    b_{m-1}^2\psi_{n-1}^2
    -a_m \psi_n\psi_{n+1}
    \end{equation}
Recall that
\begin{align}
    b_{m-1} 
    &=
    \frac{\sqrt{2}\pi}{\left(1+\frac{1}{2}
    \left(\frac{3}{4}\right)^{2m-1}
    \right)^\frac{1}{2}} 3^m \\
    &=
    \frac{\sqrt{2}\pi}{\left(1+\frac{1}{2}
    \left(\frac{3}{4}\right)^{n-1}
    \right)^\frac{1}{2}} 3^\frac{n}{2},
\end{align}
and that
\begin{align}
    a_m &= 
    \frac{\sqrt{2}\pi}
    {\left(1+\frac{1}{2}
    \left( \frac{3}{4}
    \right)^{2m}\right)^\frac{1}{2}}
    3^{m+\frac{1}{2}} \\
    &=
    \frac{\sqrt{2}\pi}
    {\left(1+\frac{1}{2}
    \left( \frac{3}{4}
    \right)^{n}\right)^\frac{1}{2}}
    3^\frac{n+1}{2}.
\end{align}
Therefore, by \Cref{GeneralRestrictedEulerODE},
we can see that for all even $n\in \mathbb{Z}^+$,
\begin{equation}
        \partial_t\psi_n
        =\sqrt{2}\pi \beta_{n-1}
        \left(\sqrt{3}\right)^n
        \psi_{n-1}^2
        -\sqrt{2}\pi \beta_{n}
        \left(\sqrt{3}\right)^{n+1}
        \psi_{n}\psi_{n+1}.
    \end{equation}

Now let $n=2m+1$.
We can see from \Cref{GeneralRestrictedEulerODE}, that for all $0\leq t<T_{max}$,
\begin{equation}
    \partial_t\psi_n
    =
    a_m\psi_{n-1}^2
    -b_m\psi_n\psi_{n+1}.
\end{equation}
Recall that
\begin{align}
    a_m &= 
    \frac{\sqrt{2}\pi}
    {\left(1+\frac{1}{2}
    \left( \frac{3}{4}
    \right)^{2m}\right)^\frac{1}{2}}
    3^{m+\frac{1}{2}} \\
    &=
    \frac{\sqrt{2}\pi}
    {\left(1+\frac{1}{2}
    \left( \frac{3}{4}
    \right)^{n-1}\right)^\frac{1}{2}}
    3^{\frac{n}{2}}
\end{align}
and that
\begin{align}
    b_m 
    &= 
    \frac{\sqrt{2}\pi}{\left(1+\frac{1}{2}
    \left(\frac{3}{4}\right)^{2m+1}
    \right)^\frac{1}{2}} 3^{m+1} \\
    &=
    \frac{\sqrt{2}\pi}{\left(1+\frac{1}{2}
    \left(\frac{3}{4}\right)^{n}
    \right)^\frac{1}{2}} 3^{\frac{n+1}{2}}
\end{align}

Therefore, by \Cref{GeneralRestrictedEulerODE},
we can see that for all odd $n\in \mathbb{Z}^+$,
\begin{equation}
        \partial_t\psi_n
        =\sqrt{2}\pi \beta_{n-1}
        \left(\sqrt{3}\right)^n
        \psi_{n-1}^2
        -\sqrt{2}\pi \beta_{n}
        \left(\sqrt{3}\right)^{n+1}
        \psi_{n}\psi_{n+1},
    \end{equation}
and this completes the proof.
\end{proof}

Now that we have reduced the Fourier-restricted Euler equation to a system of ODEs similar to the dyadic Euler equation, we will consider an analogous reduction of the Fourier-restricted hypodissipative Navier--Stokes equation to the dyadic (hypodissipative) Navier--Stokes equation.

\begin{theorem} \label{GeneralRestrictedHypoODE}
    Suppose $u^0\in \dot{H}^\frac{\log(3)}{2\log(2)}_\mathcal{M}$ is odd and permutation-symmetric. Then the solution of the Fourier-restricted hypodissipative Navier--Stokes equation 
    $u\in C\left([0,T_{max});
    \dot{H}^\frac{\log(3)}{2\log(2)}
    _{\mathcal{M}}\right)$
    with initial data $u^0$ is also odd and permutation symmetric for all $0\leq t<T_{max}$.
    This implies the solution can be expressed as
    \begin{multline}
    u(x,t)=-2\sum_{m=0}^{+\infty}\bigg(
    \phi_m(t)\sum_{k\in \mathcal{P}[k^m]} v^k
    \sin(2\pi k\cdot x)
    +\eta_m(t)\sum_{h\in \mathcal{P}[h^m]} v^h
    \sin(2\pi h\cdot x) \\
    +\zeta_m(t)\sum_{j\in \mathcal{P}[j^m]} v^j
    \sin(2\pi j\cdot x)
    \bigg).
    \end{multline}
    Furthermore, for all $0< t<T_{max}$, the scalars $\phi_m,\eta_m,\zeta_m$ satisfy the 
    infinite system of ODEs:
    \begin{equation}
        \partial_t \phi_0=
        -(20\pi^2)^\alpha \nu \phi_0
        -\frac{a_0}{2}\phi_0\eta_0
        -\frac{a_0}{2}\phi_0\zeta_0,
    \end{equation}
    and for all $m\in \mathbb{N}$:
    \begin{equation} \label{kModeHypo}
        \partial_t\phi_m
        =
        -(4\pi^2)^\alpha \nu
        \left(3*2^{4m}+2*3^{2m}\right)^\alpha \phi_m
        +b_{m-1}\eta_{m-1}\zeta_{m-1}
        -\frac{a_m}{2}\phi_m\eta_m
        -\frac{a_m}{2}\phi_m\zeta_m,
    \end{equation}
    and for all $m\in \mathbb{Z}^+$,
    \begin{align}
        \partial_t\eta_m  \label{hModeHypo}
        &=
        -(4\pi^2)^\alpha \nu
        \left(3*2^{4m+2}+2*3^{2m+1}\right)^\alpha \eta_m
        +a_m\phi_m^2
        -b_m \zeta_m \phi_{m+1} \\
        \partial_t \zeta_m  \label{jModeHypo}
        &=
        -(4\pi^2)^\alpha \nu
        \left(3*2^{4m+2}+2*3^{2m+1}\right)^\alpha \zeta_m
        +a_m\phi_m^2
        -b_m \eta_m \phi_{m+1},
    \end{align}
    where for all $m\in \mathbb{Z}^+$,
    \begin{align}
        a_m &= 
        \frac{\sqrt{2}\pi}
    {\left(1+\frac{1}{2}
    \left( \frac{3}{4}
    \right)^{2m}\right)^\frac{1}{2}}
    3^{m+\frac{1}{2}}
    \\
    b_m &= \frac{\sqrt{2}\pi}{\left(1+\frac{1}{2}
    \left(\frac{3}{4}\right)^{2m+1}
    \right)^\frac{1}{2}} 3^{m+1},
    \end{align}
    as in \Cref{ComputeTheFuckingModesProp}.
    Note that we will take
    \begin{equation}
        b_{-1},\eta_{-1},\zeta_{-1}:=0
    \end{equation} 
    by convention, so that \eqref{kModeHypo} holds for $m=0$ as well.
\end{theorem}

\begin{proof}
We have already shown in \Cref{RestrictedHypoPermuteDynamicsProp,RestrictedHypoOddDynamicsProp} that if the initial data is permutation-symmetric and odd, then this is preserved by the dynamics, so clearly for all $0\leq t<T_{max}$, we know that $u(\cdot,t)$ is odd an permutation symmetric.
\Cref{PermutationSymmetricCategorizationProp} then immediately implies that
    \begin{multline}
    u(x,t)=\sum_{m=0}^{+\infty}\bigg(
    i\phi_m(t)\sum_{k\in \mathcal{P}[k^m]} v^k
    \left(e^{2\pi i k\cdot x}
    -e^{-2\pi i k\cdot x}\right) 
    +i\eta_m(t)\sum_{h\in \mathcal{P}[h^m]} v^h
    \left(e^{2\pi i h\cdot x}
    -e^{-2\pi i h\cdot x}\right) \\
    +i\zeta_m(t)\sum_{j\in \mathcal{P}[j^m]} v^j
    \left(e^{2\pi i j\cdot x}
    -e^{-2\pi i j\cdot x}\right)
    \bigg),
    \end{multline}
or equivalently the sine series stated in the theorem.
We already have shown in the proof of \Cref{GeneralRestrictedEulerODE} that 
\begin{multline}
    -\mathbb{P}_{\mathcal{M}}
    ((u\cdot\nabla)u)=\sum_{m=0}^{+\infty}\bigg(
    i(b_{m-1}\eta_{m-1}\zeta_{m-1}
        -\frac{a_m}{2}\phi_m\eta_m
        -\frac{a_m}{2}\phi_m\zeta_m)
        \sum_{k\in \mathcal{P}[k^m]} v^k
    \left(e^{2\pi i k\cdot x}
    -e^{-2\pi i k\cdot x}\right) \\
    +i(a_m\phi_m^2-b_m \zeta_m\phi_{m+1})
        \sum_{h\in \mathcal{P}[h^m]} v^h
    \left(e^{2\pi i h\cdot x}
    -e^{-2\pi i h\cdot x}\right) \\
    +i(a_m\phi_m^2-b_m \eta_m\phi_{m+1})
    \sum_{j\in \mathcal{P}[j^m]} v^j
    \left(e^{2\pi i j\cdot x}
    -e^{-2\pi i j\cdot x}\right)
    \bigg).
    \end{multline}
    Therefore, it suffices to observe that
    \begin{multline}
    (-\Delta)^\alpha u(x,t)=
    \sum_{m=0}^{+\infty}\bigg(
    i(4\pi^2)^\alpha
    \left(3*2^{4m}+2*3^{2m}\right)^\alpha \phi_m
    \sum_{k\in \mathcal{P}[k^m]} v^k
    \left(e^{2\pi i k\cdot x}
    -e^{-2\pi i k\cdot x}\right) \\
    +i(4\pi^2)^\alpha
    \left(3*2^{4m+2}+2*3^{2m+1}\right)^\alpha 
    \eta_m
    \sum_{h\in \mathcal{P}[h^m]} v^h
    \left(e^{2\pi i h\cdot x}
    -e^{-2\pi i h\cdot x}\right) \\
    +i(4\pi^2)^\alpha
    \left(3*2^{4m+2}+2*3^{2m+1}\right)^\alpha 
    \zeta_m
    \sum_{j\in \mathcal{P}[j^m]} v^j
    \left(e^{2\pi i j\cdot x}
    -e^{-2\pi i j\cdot x}\right)
    \bigg).
    \end{multline}
    Then we can see that \eqref{kModeHypo},\eqref{hModeHypo},\eqref{jModeHypo} express the PDE
    \begin{equation}
        \partial_t u
        =
        -\nu(-\Delta)^\alpha u
        -\mathbb{P}_{\mathcal{M}}((u\cdot\nabla)u),
    \end{equation}
    in terms of Fourier coefficients.
    Again using the fact that 
    $\left\{e^{2\pi i k\cdot x}\right\}_{k\in \mathcal{M}}$ is a complete basis for $\dot{H}^s_\mathcal{M}$,
    this implies the infinite system of ODEs
\eqref{kModeHypo},\eqref{hModeHypo},
\eqref{jModeHypo} is equivalent to the Fourier-restricted hypodissipative Navier--Stokes equation for odd, permutation symmetric solutions.
    \end{proof}

\begin{proposition} \label{hjPropHypo}
    Suppose $u^0 \in \dot{H}^\frac{\log(3)}{2\log(2)}_{\mathcal{M}}$ is an odd, permutation-symmetric vector field with hj-parity.
    Then $u\in C\left([0,T_{max});
    \dot{H}^\frac{\log(3)}{2\log(2)}_{\mathcal{M}}\right)$,
    the solution of the Fourier-restricted hypodissipative Navier--Stokes equation 
    with initial data $u^0$,
    also has hj-parity for all $0\leq t<T_{max}$.
\end{proposition}

\begin{remark}
    We will in fact prove a slightly stronger result. We will show that
    for all $m\in\mathbb{Z}^+$ and for all $0\leq t<T_{max}$,
    \begin{multline}
        \eta_m(t)-\zeta_m(t)
        =
        \left(\eta_m(0)-\zeta_m(0)\right)
        \exp\Bigg(-(4\pi^2)^\alpha \nu
        \left(3*2^{4m+2}+2*3^{2m+1}\right)^\alpha t \\
        -b_m\int_0^t
        \phi_{m+1}(\tau)\diff\tau\Bigg),
    \end{multline}
    and in particular if $\eta_m(0)=\zeta_m(0)$, then for all $0\leq t<T_{max}$,
\begin{equation}
    \eta_m(t)=\zeta_m(t).
    \end{equation}
\end{remark}

\begin{proof}
    Let $\rho_m(t)=\eta_m(t)-\zeta_m(t)$.
    Plugging into \Cref{GeneralRestrictedHypoODE},
    we can see that for all $0< t<T_{max}$,
    \begin{align}
\partial_t \rho_m(t)
&=
 -(4\pi^2)^\alpha \nu
\left(3*2^{4m+2}+2*3^{2m+1}\right)^\alpha
(\eta_m-\zeta_m)
-b_m\zeta_m\phi_{m+1}+b_m\eta_m\phi_{m+1} \\
&=
 -(4\pi^2)^\alpha \nu
\left(3*2^{4m+2}+2*3^{2m+1}\right)^\alpha
\rho_m
-b_m\phi_{m+1}\rho_m,
    \end{align}
and integrating this differential equation completes the proof.
\end{proof}

\begin{proposition} \label{RestrictedHypoPartiyODE}
    Suppose $u\in C\left([0,T_{max});
    \dot{H}^\frac{\log(3)}{2\log(2)}
    _{\mathcal{M}}\right)$
    is an odd, permutation symmetric solution of the Fourier-restricted, hypodissipative Navier--Stokes equation with hj-parity.
    For all $m\in\mathbb{Z}^+$, let
    \begin{equation}
    \psi_{2m}:=\phi_m
    \end{equation}
    and let
    \begin{equation}
    \psi_{2m+1}:=\eta_m=\zeta_m,
    \end{equation}
    Then for all $0< t<T_{max}$, and for all $n\in\mathbb{Z}^+$,
    \begin{equation}
        \partial_t \psi_n
        =
        -(12\pi^2)^\alpha\mu_n^\alpha \nu
        \left(\sqrt{3}\right)^{2\Tilde{\alpha}n}
        \psi_n
        +\sqrt{2}\pi\beta_{n-1}
        \left(\sqrt{3}\right)^{n}
        \psi_{n-1}^2
        -\sqrt{2}\pi\beta_n
        \left(\sqrt{3}\right)^{n+1}
        \psi_n \psi_{n+1},
    \end{equation}
    where
    \begin{equation}
        \Tilde{\alpha}=\frac{2\log(2)}{\log(3)}\alpha
    \end{equation}
    and for all $n\in\mathbb{Z}^+$,
    \begin{align}
        \beta_n &=
        \frac{1}{\left(1+\frac{1}{2}
        \left(\frac{3}{4}\right)^n
        \right)^\frac{1}{2}} \\
        \mu_n &=
        \frac{1}{\left(1+\frac{2}{3}
        \left(\frac{3}{4}\right)^n
        \right)^\frac{1}{2}}
    \end{align}
    and by convention
    \begin{equation}
        \psi_{-1},\beta_{-1}:=0.
    \end{equation}
\end{proposition}

\begin{proof}
    The proof that the nonlinearity takes the form
    \begin{equation}
        \sqrt{2}\pi\beta_{n-1}
        \left(\sqrt{3}\right)^{n}
        \psi_{n-1}^2
        -\sqrt{2}\pi\beta_n
        \left(\sqrt{3}\right)^{n+1}
        \psi_n \psi_{n+1},
    \end{equation}
    is precisely the same as in \Cref{RestrictedEulerPartiyODE},
    so we will not repeat the details here.
    It suffices to prove that the dissipative term has the form
    \begin{equation}
        -(12\pi^2)^\alpha\mu_n^\alpha \nu
        \left(\sqrt{3}\right)^{2\Tilde{\alpha}n}
        \psi_n.
    \end{equation}
    First we will consider the even case. Let $n=2m$.
    Then we can see that for all $m\in\mathbb{Z}^+$,
    \begin{align}
        -(4\pi^2)^\alpha \nu
        \left(3*2^{4m}+2*3^{2m}\right)^\alpha \phi_m
        &=
        -(4\pi^2)^\alpha \nu
    \left(3*2^{2n}+2*3^{n}\right)^\alpha \psi_{n} \\ 
    &=
    -(12\pi^2)^\alpha \nu
    \left(1+\frac{2}{3}
    \left(\frac{3}{4}\right)^n\right)^\alpha
    2^{2n\alpha} \psi_n \\
    &=
    -(12\pi^2)^\alpha \nu
    \mu_n^\alpha
    2^{2n\alpha} \psi_n.
     \end{align}   
    Next we will consider the odd case. Let $n=2m+1$.
    Then we can see that for all $m\in\mathbb{Z}^+$,
    \begin{align}
    -(4\pi^2)^\alpha \nu
    \left(3*2^{4m+2}+2*3^{2m+1}\right)^\alpha \zeta_m
    &=
    -(4\pi^2)^\alpha \nu
    \left(3*2^{2n}+2*3^{n}\right)^\alpha \psi_n \\
    &=
    -(12\pi^2)^\alpha \nu
    \left(1+\frac{2}{3}
    \left(\frac{3}{4}\right)^n\right)^\alpha
    2^{2n\alpha} \psi_n \\
    &=
    -(12\pi^2)^\alpha \nu
    \mu_n^\alpha
    2^{2n\alpha} \psi_n,
    \end{align}
    and likewise for $\eta$.

    It now remains only to show that
    \begin{equation}
        \left(\sqrt{3}\right)^{\Tilde{\alpha}}
        =
        2^\alpha.
    \end{equation}
    Recall that
    \begin{equation}
    \log(2)\alpha=\frac{1}{2}\log(3)\Tilde{\alpha},
    \end{equation}
    and so
    \begin{equation}
        \log\left(2^\alpha\right)
        =
        \log\left(\left(\sqrt{3}
        \right)^{\Tilde{\alpha}}\right),
    \end{equation}
and taking the exponential of both sides, we find 
\begin{equation}
2^\alpha=\left(\sqrt{3}\right)^{\Tilde{\alpha}}.
\end{equation}
\end{proof}

Note that the positivity of the Fourier coefficients $\psi_n$ is preserved by the dynamics, which we will show now.

\begin{proposition} \label{HypoPositivity}
    Suppose $u\in C\left([0,T_{max});
    \dot{H}^\frac{\log(3)}{2\log(2)}
    _{\mathcal{M}}\right)$ 
    is an odd, permutation-symmetric solution of the Fourier-restricted, hypodissipative Navier--Stokes equation with hj-parity.
    Then for all $n\in\mathbb{Z}^+$,
    and for all $0\leq t<T_{max}$,
    \begin{equation}
        \psi_n(t)\geq \psi_n(0)
        \exp\left(-(12\pi^2)^\alpha\mu_n^\alpha \nu
        \left(\sqrt{3}\right)^{2\Tilde{\alpha}n} t
        -\sqrt{2}\pi\beta_n\left(\sqrt{3}\right)^{n+1}
        \int_0^t \psi_{n+1}(\tau)\diff\tau \right).
    \end{equation}
    Note in particular, that if $\psi_n(0)\geq 0$, then for all $0\leq t<T_{max}$, $\psi_n(t)\geq 0$, and 
    $\psi_n(0)> 0$, then for all $0\leq t<T_{max}$, $\psi_n(t)> 0$.
\end{proposition}

\begin{proof}
    We can see from \Cref{RestrictedHypoPartiyODE}, that
    for all $0<t<T_{max}$, and for all $n\in\mathbb{Z}^+$,
    \begin{equation}
        \partial_t \psi_n
        \geq 
        -(12\pi^2)^\alpha\mu_n^\alpha \nu
        \left(\sqrt{3}\right)^{2\Tilde{\alpha}n}
        \psi_n
        -\sqrt{2}\pi\beta_n\left(\sqrt{3}\right)^{n+1}
        \psi_n \psi_{n+1}.
    \end{equation}
    Integrating this differential inequality completes the proof.
\end{proof}

\begin{proposition} \label{RestrictedPositivity}
    Suppose $u\in C^1\left([0,T_{max});
    \dot{H}^\frac{\log(3)}{2\log(2)}
    _{\mathcal{M}}\right)$
    is a solution of the Fourier-restricted Euler equation with hj-parity.
    Then for all $n\in\mathbb{Z}^+$,
    and for all $0\leq t<T_{max}$,
    \begin{equation}
        \psi_n(t)\geq \psi_n(0)
        \exp\left(
        -\sqrt{2}\pi\beta_n\left(\sqrt{3}\right)^{n+1}
        \int_0^t \psi_{n+1}(\tau)\diff\tau \right).
    \end{equation}
    Note in particular, that if $\psi_n(0)\geq 0$, then for all $0\leq t<T_{max}$, $\psi_n(t)\geq 0$.
\end{proposition}

\begin{proof}
    We can see from \Cref{RestrictedEulerPartiyODE}, that
    for all $0\leq t<T_{max}$, and for all $n\in\mathbb{Z}^+$,
    \begin{equation}
        \partial_t \psi_n
        \geq 
        -\sqrt{2}\pi\beta_n\left(\sqrt{3}\right)^{n+1}
        \psi_n \psi_{n+1}.
    \end{equation}
    Integrating this differential inequality completes the proof.
\end{proof}

\begin{remark}
    We have now reduced the Fourier-restricted Euler and hypodissipative Navier--Stokes equations to an ODE system fundamentally similar to the dyadic Euler and Navier--Stokes  equations introduced in \cites{FriedlanderPavlovic,KatzPavlovic}---modulo the factor $\beta_n$ (and $\mu_n$ in the viscous case) which does change the dynamics in any fundamental way.
    This is the first part of \Cref{RestrictedEulerBlowupIntro}.
    Note that this reduction only applies to odd, permutation-symmetric solutions with hj-parity. In general, the dynamics of the Fourier-restricted Euler equation are much more complicated, because the general case requires a separate ODE for each of the 12 frequencies in each shell, rather than the one ODE per shell for odd, permutation-symmetric solutions with hj-parity.
    In the next section, we will use this reduction to prove finite-time blowup for the Fourier-restricted Euler equation by using an appropriately chosen Lyapunov functional on the scalar Fourier coefficients $\psi_n$,
    which will complete the proof of \Cref{RestrictedEulerBlowupIntro,RestrictedHypoBlowupIntro}.
\end{remark}

\section{Finite-time blowup} \label{BlowupSection}

In this section, we will prove there exist smooth solutions of the Fourier-restricted Euler and hypodissipative Navier--Stokes equations (with weak enough dissipation) that blowup in finite-time.
We will now complete the proof of \Cref{RestrictedHypoBlowupIntro}, showing finite-time blowup for the Fourier-restricted hypodissipative Navier--Stokes equation when $\alpha<\frac{\log(3)}{6\log(2)}$. The proof will be broken into several pieces. Throughout this section we will follow the notation in \Cref{RestrictedHypoPartiyODE} that and odd, permutation, hj-parity solution of the Fourier-restricted hypodissipative Navier--Stokes equation will be expressed as
\begin{equation}
    u(x,t)=-2\sum_{m=0}^{+\infty}\left(
    \psi_n(t)\sum_{k\in\mathcal{M}_n^+}
    v^k\sin(2\pi k\cdot x)
    \right).
\end{equation}

For this reason it will be more convenient to work with directly the $\mathcal{H}^{\Tilde{s}}$ norms of the coefficients $\psi_n$ than with the $H^s$ norms $u$.
Recall that in \Cref{DyadicHilbertIntro} we have defined the space $\mathcal{H}^s$ as a dyadic analogue of the standard Hilbert space, following the conventions of 
\cites{FriedlanderPavlovic,KatzPavlovic,Cheskidov} by
    \begin{equation}
    \|\psi\|_{\mathcal{H}^s}^2
    =\sum_{n=0}^{+\infty}
    \left(\sqrt{3}\right)^{2sn}\psi_n^2.
    \end{equation}

\begin{remark}
    We will note here that for odd, permutation symmetric vector fields with hj-parity, $u\in \dot{H}^{s}_\mathcal{M}$, the norms 
    $\|u\|_{\dot{H}^s}$ and 
    $\|\psi\|_{\mathcal{H}^{\Tilde{s}}}$ are equivalent, where 
    $\Tilde{s}=\frac{2\log(2)}{\log(3)}s$.
    \end{remark}

    \begin{proposition} \label{DyadicComparableProp}
    Fix $s\in\mathbb{R}$ and let 
    $\Tilde{s}=\frac{2\log(2)}{\log(3)}s$.
    For all vector fields $u\in \dot{H}^{s}_\mathcal{M}$,
    odd, permutation symmetric, with hj-parity,
    \begin{equation}
        12\left(12\pi^2\right)^s
        \|\psi\|_{\mathcal{H}^{\Tilde{s}}}^2
        \leq
        \|u\|_{\dot{H}^s}^2
        \leq 
        12\left(20\pi^2\right)^s
        \|\psi\|_{\mathcal{H}^{\Tilde{s}}}^2.
    \end{equation}
    \end{proposition}

    \begin{proof}
    We can see as above that
    \begin{align}
    \frac{1}{12}\|u\|_{\dot{H}^s}^2
    &=
    \sum_{n=0}^{+\infty}
    (4\pi^2)^s
    \left(3*2^{2n}+2*3^{n}\right)^s \psi_n^2 \\
    &=
    \sum_{n=0}^{+\infty}
    \left(12\pi^2\right)^s
    \left(1+\frac{2}{3}
    \left(\frac{3}{4}\right)^n\right)^s
    2^{2ns} \psi_n^2 \\
    &=
    \left(12\pi^2\right)^s\sum_{n=0}^{+\infty}
    \left(1+\frac{2}{3}
    \left(\frac{3}{4}\right)^n\right)^s
    \left(\sqrt{3}\right)^{2n\Tilde{s}} \psi_n^2.
    \end{align}
    Note that the factor of $\frac{1}{12}$ comes from the fact that each $\psi_n$ corresponds to twelve distinct frequencies: six in $\mathcal{M}_n^+$ and six in $\mathcal{M}_n^-$.
    Next observe that
    for all $n\in\mathbb{Z}^+$
    \begin{equation}
    1 \leq 
    \left(1+\frac{2}{3}
    \left(\frac{3}{4}\right)^n\right)^s
    \leq 
    \left(\frac{5}{3}\right)^s,
    \end{equation}
    and so
    \begin{equation}
        (12\pi^2)^s
        \|\psi\|_{\mathcal{H}^{\Tilde{s}}}^2
        \leq
        \frac{1}{12}\|u\|_{\dot{H}^s}^2
        \leq 
        \left(\frac{5}{3}\right)^s
        (12\pi^2)^s
        \|\psi\|_{\mathcal{H}^{\Tilde{s}}}^2.
    \end{equation}
    Rearranging terms, this completes the proof.
    \end{proof}

    \begin{corollary} \label{DyadicHilbertBlowupRateCor}
    Suppose $u\in C\left([0,T_{max});\dot{H}^
    \frac{\log(3)}{2\log(2)}_\mathcal{M}\right)$
    is a solution of the Fourier-restricted Euler or hypodissipative Navier--Stokes equation.
    If $T_{max}<+\infty$, then
    \begin{equation}
    \|\psi(t)\|_{\mathcal{H}^1}
    \geq
    \frac{1}{\sqrt{12}(20\pi^2)^\frac{\log(3)}{4\log(2)}
    \mathcal{C}_*(T_{max}-t)}.
    \end{equation}
    \end{corollary}

    \begin{proof}
    This follows immediately from \Cref{DyadicComparableProp} and the wellposedness theory in 
    $\dot{H}^\frac{\log(3)}{2\log(2)}_\mathcal{M}$ developed in \Cref{WellPosedSection}.
    \end{proof}

    \begin{remark}
    We should note that \Cref{DyadicComparableProp} implies that a solution $u\in C^1_T \dot{H}^\frac{\log(3)}{2\log(2)}_x$ of the Fourier-restricted Euler or hypodissipative Navier--Stokes equations in our symmetry class is equivalent to a solution
    $\psi \in C^1_T \mathcal{H}^1$ satisfying the infinite system of ODEs in \Cref{RestrictedEulerPartiyODE,RestrictedHypoPartiyODE}.
    \end{remark}

\begin{proposition} \label{BlowupIdProp}
    Suppose $u\in C\left([0,T_{max});
    \dot{H}^\frac{\log(3)}{2\log(2)}
    _{\mathcal{M}}\right)$ 
    is an odd, permutation-symmetric solution of the Fourier-restricted, hypodissipative Navier--Stokes equation with hj-parity. Further suppose that $\alpha<\frac{\log(3)}{6\log(2)}$ and that $\Tilde{\alpha}<s<\frac{1}{3}$. 
    Define the Lyapunov functional as
    \begin{equation}
    H_s(t)
    =
    \sum_{n=0}^\infty
    \left(\sqrt{3}\right)^{-sn} \psi_n(t).
    \end{equation}
    Then for all $0\leq t<T_{max}$,
    \begin{multline} 
    \frac{\diff H_s}{\diff t}
    =
    -\left(12 \pi^2\right)^\alpha \nu
    \sum_{n=0}^\infty \mu_n^\alpha 
    \left(\sqrt{3}\right)^{(2\Tilde{\alpha}-s)n} \psi_n
    +\sqrt{2}\pi \left(\sqrt{3}\right)^{1-s}
    \sum_{n=0}^\infty 
    \beta_n\left(\sqrt{3}\right)^{(1-s)n}
    \psi_n^2 \\
    -\sqrt{2}\pi \sum_{n=0}^\infty \beta_n
    \left(\sqrt{3}\right)^{-sn+n+1}\psi_n\psi_{n+1}.
    \end{multline}
\end{proposition}

\begin{proof}
    We can see immediately from \Cref{RestrictedHypoPartiyODE} that
    \begin{multline} 
    \frac{\diff H_s}{\diff t}
    =
    -\left(12 \pi^2\right)^\alpha \nu
    \sum_{n=0}^\infty \mu_n^\alpha 
    \left(\sqrt{3}\right)^{(2\Tilde{\alpha}-s)n} \psi_n
    +\sqrt{2}\pi \sum_{n=1}^\infty \beta_{n-1} 
    \left(\sqrt{3}\right)^{(1-s)n}
    \psi_{n-1}^2 \\
    -\sqrt{2}\pi \sum_{n=0}^\infty \beta_n
    \left(\sqrt{3}\right)^{-sn+n+1}\psi_n\psi_{n+1}.
    \end{multline}
    Then observe that
    \begin{equation}
    \sqrt{2}\pi \sum_{n=1}^\infty \beta_{n-1} 
    \left(\sqrt{3}\right)^{(1-s)n}
    \psi_{n-1}^2
    =
    \sqrt{2}\pi \left(\sqrt{3}\right)^{1-s}
    \sum_{n=0}^\infty 
    \beta_n\left(\sqrt{3}\right)^{(1-s)n}
    \psi_n^2,
    \end{equation}
    and this completes the proof.
\end{proof}

\begin{proposition} \label{PropBoundA}
    Suppose $u\in \dot{H}^\frac{\log(3)}{2\log(2)}
    _{\mathcal{M}}$ 
    is odd, permutation-symmetric, and has hj-parity. 
    Then
    \begin{equation}
    \sqrt{2}\pi \sum_{n=0}^\infty \beta_n 
    \left(\sqrt{3}\right)^{-sn+n+1}\psi_n\psi_{n+1}
    \leq 
    \sqrt{2}\pi \left(\sqrt{3}\right)^{\frac{1+s}{2}}
    \sum_{n=0}^\infty \beta_n\left(\sqrt{3}\right)^{(1-s)n}
    \psi_n^2.
    \end{equation}
\end{proposition}

\begin{proof}
    Observe that
    \begin{equation}
    \sqrt{2}\pi \sum_{n=1}^\infty \beta_{n-1} 
    \left(\sqrt{3}\right)^{(1-s)n}
    =
    \sqrt{2}\pi \left(\sqrt{3}\right)^{1-s}
    \sum_{n=0}^\infty 
    \beta_n\left(\sqrt{3}\right)^{(1-s)n}
    \psi_n^2.
    \end{equation}
    Next observe that
    \begin{equation}
    \sqrt{2}\pi \sum_{n=0}^\infty \beta_n 
    \left(\sqrt{3}\right)^{-sn+n+1}\psi_n\psi_{n+1}
    =
    \sqrt{2}\pi \left(\sqrt{3}\right)^{\frac{1+s}{2}}\sum_{n=0}^\infty 
    \beta_n \left(\sqrt{3}
    \right)^{n-ns+\frac{1}{2}-\frac{s}{2}}\psi_n
    \psi_{n+1},
    \end{equation}
    and apply H\"older's inequality to
    $\beta_n^\frac{1}{2} \left(\sqrt{3}\right)^{\frac{(1-s)n}{2}}\psi_n$
    and $\beta_n^\frac{1}{2} 
    \left(\sqrt{3}\right)^{\frac{(1-s)(n+1)}{2}}
    \psi_{n+1}$ to conclude that
    \begin{multline} \label{BlowupBound1}
    \sqrt{2}\pi \sum_{n=0}^\infty \beta_n 
    \left(\sqrt{3}\right)^{-sn+n+1}\psi_n\psi_{n+1}
    \leq 
    \sqrt{2}\pi \left(\sqrt{3}\right)^{\frac{1+s}{2}}
    \left(\sum_{n=0}^\infty \beta_n
    \left(\sqrt{3}\right)^{(1-s)n}\psi_n^2\right)^\frac{1}{2} \\
    \left(\sum_{n=0}^\infty \beta_n
    \left(\sqrt{3}\right)^{(1-s)(n+1)}\psi_{n+1}^2\right)^\frac{1}{2}.
    \end{multline}
    Using the fact that $\beta_n<\beta_{n+1}$, we can see that
    \begin{align}
    \sum_{n=0}^\infty \beta_n
    \left(\sqrt{3}\right)^{(1-s)(n+1)}\psi_{n+1}^2
    &\leq 
    \sum_{n=0}^\infty \beta_{n+1}
    \left(\sqrt{3}\right)^{(1-s)(n+1)}\psi_{n+1}^2 \\
    &=
    \sum_{n=1}^\infty 
    \beta_n \left(\sqrt{3}\right)^{(1-s)n}
    \psi_n^2 \\
    &\leq 
    \sum_{n=0}^\infty \beta_n\left(\sqrt{3}\right)^{(1-s)n}
    \psi_n^2.
    \end{align}
    Plugging this back into \eqref{BlowupBound1},
    this implies that 
    \begin{equation}
    \sqrt{2}\pi \sum_{n=0}^\infty \beta_n 
    \left(\sqrt{3}\right)^{-sn+n+1}\psi_n\psi_{n+1}
    \leq 
    \sqrt{2}\pi \left(\sqrt{3}\right)^{\frac{1+s}{2}}
    \sum_{n=0}^\infty \beta_n\left(\sqrt{3}\right)^{(1-s)n}
    \psi_n^2.
    \end{equation}
\end{proof}

\begin{proposition} \label{PropBoundB}
    Suppose $u\in \dot{H}^\frac{\log(3)}{2\log(2)}
    _{\mathcal{M}}$ 
    is odd, permutation-symmetric, and has hj-parity, and further suppose that $\alpha<\frac{\log(3)}{6\log(2)}$ and that $\Tilde{\alpha}<s<\frac{1}{3}$. Then
    \begin{equation}
    \sum_{n=0}^\infty \mu_n^\alpha 
    \left(\sqrt{3}\right)^{(2\Tilde{\alpha}-s)n} \psi_n
    \leq 
    \left(\frac{1}{1-\left(\sqrt{3}
    \right)^{-(1+s-4\Tilde{\alpha})}}
    \right)^\frac{1}{2}
    \|\psi\|_{\mathcal{H}^\frac{1-s}{2}}
    \end{equation}
\end{proposition}

\begin{proof}
    We begin by applying Young's inequality, finding that
    \begin{align}
    \sum_{n=0}^\infty \mu_n^\alpha 
    \left(\sqrt{3}\right)^{(2\Tilde{\alpha}-s)n} \psi_n
    &=
    \sum_{n=0}^\infty \mu_n^\alpha 
    \left(\sqrt{3}\right)^{\frac{(1-s)n}{2}} \psi_n
    \left(\sqrt{3}\right)^{-\left(\frac{1}{2}
    +\frac{s}{2}-2\Tilde{\alpha}\right)n} \\
    &\leq 
    \left(\sum_{n=0}^\infty \mu_n^{2\alpha}
    \left(\sqrt{3}\right)^{(1-s)n} 
    \psi_n^2\right)^\frac{1}{2}
    \left(\sum_{n=0}^\infty 
    \left(\sqrt{3}\right)^{-(1+s-4\Tilde{\alpha})n} 
    \right)^\frac{1}{2} \\
    &\leq 
     \left(\sum_{n=0}^\infty
    \left(\sqrt{3}\right)^{(1-s)n} 
    \psi_n^2\right)^\frac{1}{2}
    \left(\frac{1}{1-\left(\sqrt{3}
    \right)^{-(1+s-4\Tilde{\alpha})}}
    \right)^\frac{1}{2}.
    \end{align}
    Note that we have used the fact that $\mu_n<1$,
    and that $0<\Tilde{\alpha}<s<\frac{1}{3}$ implies that $1+s-4\Tilde{\alpha}>0$.
\end{proof}

\begin{lemma} \label{BlowupLemmaA}
    Suppose $u\in C\left([0,T_{max});
    \dot{H}^\frac{\log(3)}{2\log(2)}
    _{\mathcal{M}}\right)$ 
    is an odd, permutation-symmetric solution of the Fourier-restricted, hypodissipative Navier--Stokes equation with hj-parity. Further suppose that $\alpha<\frac{\log(3)}{6\log(2)}$ and that $\Tilde{\alpha}<s<\frac{1}{3}$. 
    Then for all $0\leq t<T_{max}$,
    \begin{equation}
    \frac{\diff H_s}{\diff t}
    \geq 
    -\left(\frac{\left(12\pi^2\right)^\alpha}
    {\left(1-\left(\sqrt{3}
    \right)^{-(1+s-4\Tilde{\alpha})}\right)^\frac{1}{2}}\right)\nu 
    \|\psi\|_{\mathcal{H}^\frac{1-s}{2}}
    +\frac{2\pi}{\sqrt{3}} \left(
    \left(\sqrt{3}\right)^{1-s}
    -\left(\sqrt{3}\right)^\frac{1+s}{2}\right)
    \|\psi\|_{\mathcal{H}^\frac{1-s}{2}}^2.
    \end{equation}
\end{lemma}

\begin{proof}
    Applying \Cref{BlowupIdProp,PropBoundA,PropBoundB}
    We find that
    \begin{multline}
    \frac{\diff H_s}{\diff t}
    \geq 
    -\left(\frac{\left(12\pi^2\right)^\alpha}
    {\left(1-\left(\sqrt{3}
    \right)^{-(1+s-4\Tilde{\alpha})}\right)^\frac{1}{2}}\right)\nu 
    \|\psi\|_{\mathcal{H}^\frac{1-s}{2}} \\
    +\sqrt{2}\pi \left(
    \left(\sqrt{3}\right)^{1-s}
    -\left(\sqrt{3}\right)^\frac{1+s}{2}\right)
    \sum_{n=0}^\infty \beta_n 
    \sqrt{3}^{(1-s)n}\psi_n^2.
    \end{multline}
    Recalling that $\beta_n\geq \frac{\sqrt{2}}{\sqrt{3}}$ for all $n\in\mathbb{Z}^+$, we can conclude that 
    \begin{equation}
    \sum_{n=0}^\infty \beta_n 
    \sqrt{3}^{(1-s)n}\psi_n^2
    \geq 
    \frac{\sqrt{2}}{\sqrt{3}}
    \|\psi\|_{\mathcal{H}^\frac{1-s}{2}}^2
    \end{equation}
    Note that $s<\frac{1}{3}$ implies that 
    $1-s>\frac{1+s}{2}$, and therefore that
    $\left(\sqrt{3}\right)^{1-s}
    -\left(\sqrt{3}\right)^\frac{1+s}{2}>0$.
\end{proof}

\begin{proposition} \label{PropBoundC}
    Suppose $u\in \dot{H}^\frac{\log(3)}{2\log(2)}
    _{\mathcal{M}}$ 
    is a odd, permutation-symmetric, and has hj-parity, and that $s>-1$. Then 
    \begin{equation}
    H_s=\sum_{n=0}^\infty 
    \left(\sqrt{3}\right)^{-sn} \psi_n
    \leq 
    \frac{1}{\left(1-\left(\sqrt{3}\right)^{-(1+s)}\right)^\frac{1}{2}} \|\psi\|_{\mathcal{H}^\frac{1-s}{2}}
    \end{equation}
\end{proposition}

\begin{proof}
    Applying H\"older's inequality we find that 
    \begin{align}
    \sum_{n=0}^\infty 
    \left(\sqrt{3}\right)^{-sn}\psi_n
    &=
    \sum_{n=0}^\infty
    \left(\sqrt{3}\right)^{\left(\frac{1-s}{2}\right)n} \psi_n
    \left(\sqrt{3}\right)^{-\left(\frac{1+s}{2}\right)n}  \\
    &\leq 
    \left(\sum_{n=0}^\infty
    \left(\sqrt{3}\right)^{(1-s)n)} \psi_n^2
    \right)^\frac{1}{2}
    \left(\sum_{n=0}^\infty
    \left(\sqrt{3}\right)^{-(1+s)n}
    \right)^\frac{1}{2}  \\
    &=
    \frac{1}{\left(1-\left(\sqrt{3}\right)^{-(1+s)}\right)^\frac{1}{2}} \|\psi\|_{\mathcal{H}^\frac{1-s}{2}},
    \end{align}
    and this completes the proof.
\end{proof}

\begin{lemma} \label{BlowupLemmaB}
    Suppose $u\in C\left([0,T_{max});
    \dot{H}^\frac{\log(3)}{2\log(2)}
    _{\mathcal{M}}\right)$ 
    is an odd, permutation-symmetric solution of the Fourier-restricted, hypodissipative Navier--Stokes equation with hj-parity. Further suppose that $\alpha<\frac{\log(3)}{6\log(2)}$, and that $0<\Tilde{\alpha}<s<\frac{1}{3}$, and that
    \begin{equation}
    H_s(0)>C_{\alpha,s}\nu,
    \end{equation}
    where 
    \begin{equation}
    C_{\alpha,s }=
    \frac{\sqrt{3}\left(12\pi^2\right)^\alpha}
    {\pi\left(1-\left(\sqrt{3}\right)^{-(1+s-4\Tilde{\alpha})}\right)^\frac{1}{2}
    \left(\left(\sqrt{3}\right)^{1-s}-
    \left(\sqrt{3}\right)^\frac{1+s}{2}\right)
    \left(1-
    \left(\sqrt{3}\right)^{-(1+s)}
    \right)^\frac{1}{2}}.
    \end{equation}
    Then for all $0\leq t<T_{max}$,
    \begin{equation}
    H_s(t)> C_{\alpha,s}\nu,
    \end{equation}
    and furthermore
    \begin{equation}
    \frac{\diff H_s}{\diff t}
    >
    \kappa_s H_s^2,
    \end{equation}
    where 
    \begin{equation}
    \kappa_s
    =
    \frac{\pi}{\sqrt{3}} \left(
    \left(\sqrt{3}\right)^{1-s}
    -\left(\sqrt{3}\right)^\frac{1+s}{2}\right)
    \left(1-\left(\sqrt{3}\right)^{-(1+s)}\right).
    \end{equation}
\end{lemma}

\begin{proof}
    We will begin by showing that $H_s$ is increasing in time.
    Suppose that for some time $0\leq t<T_{max}$, $H_s(t)>C_{\alpha,s}\nu$, 
    then by \Cref{PropBoundC}
    \begin{equation}
    \|\psi(t)\|_{\mathcal{H}^\frac{1-s}{2}}
    >
    \frac{\sqrt{3}\left(12\pi^2\right)^\alpha}
    {\pi\left(1-\left(\sqrt{3}\right)^{-(1+s-4\Tilde{\alpha})}\right)^\frac{1}{2}
    \left(\left(\sqrt{3}\right)^{1-s}-
    \left(\sqrt{3}\right)^\frac{1+s}{2}\right)}
    \nu.
    \end{equation}
    Applying \Cref{BlowupLemmaA}, this implies that 
    \begin{equation}
    \frac{\diff}{\diff t} H_s(t)>0.
    \end{equation}
    By hypothesis $H_s(0)>C_{\alpha,s}\nu$, so we may conclude that for all $0\leq t<T_{max}$,
    \begin{equation}
    H_s(t)>C_{\alpha,s}\nu,
    \end{equation}
    and that 
    \begin{equation}
    \frac{\diff}{\diff t} H_s(t)>0.
    \end{equation}

    Now we require a singular lower bound.
    Again applying \Cref{PropBoundC}, we can see that for all $0<t<T_{max}$,
    \begin{equation}
    \|\psi(t)\|_{\mathcal{H}^\frac{1-s}{2}}
    >
    \frac{\sqrt{3}\left(12\pi^2\right)^\alpha}
    {\pi\left(1-\left(\sqrt{3}\right)^{-(1+s-4\Tilde{\alpha})}\right)^\frac{1}{2}
    \left(\left(\sqrt{3}\right)^{1-s}-
    \left(\sqrt{3}\right)^\frac{1+s}{2}\right)}
    \nu.
    \end{equation}
    This implies that for all $0\leq t<T_{max}$,
    \begin{equation}
    -\left(\frac{\left(12\pi^2\right)^\alpha}
    {\left(1-\left(\sqrt{3}
    \right)^{-(1+s-4\Tilde{\alpha})}\right)^\frac{1}{2}}\right)\nu 
    \|\psi\|_{\mathcal{H}^\frac{1-s}{2}}
    +\frac{\pi}{\sqrt{3}} \left(
    \left(\sqrt{3}\right)^{1-s}
    -\left(\sqrt{3}\right)^\frac{1+s}{2}\right)
    \|\psi\|_{\mathcal{H}^\frac{1-s}{2}}^2
    >0.
    \end{equation}
    Applying \Cref{BlowupLemmaA} and \Cref{PropBoundC},
    we find that for all $0\leq t<T_{max}$,
    \begin{align}
    \frac{\diff H_s}{\diff t}
    &\geq
    \frac{\pi}{\sqrt{3}} \left(
    \left(\sqrt{3}\right)^{1-s}
    -\left(\sqrt{3}\right)^\frac{1+s}{2}\right)
    \|\psi\|_{\mathcal{H}^\frac{1-s}{2}}^2 \\
    &\geq \frac{\pi}{\sqrt{3}} \left(
    \left(\sqrt{3}\right)^{1-s}
    -\left(\sqrt{3}\right)^\frac{1+s}{2}\right)
    \left(1-\left(\sqrt{3}\right)^{-(1+s)}\right)
    H_s^2,
    \end{align}
    and this completes the proof.
\end{proof}

\begin{theorem} \label{RestrictedHypoBlowup}
    Suppose $u\in C\left([0,T_{max});
    \dot{H}^\frac{\log(3)}{2\log(2)}
    _{\mathcal{M}}\right)$ 
    is an odd, permutation-symmetric solution of the Fourier-restricted, hypodissipative Navier--Stokes equation with hj-parity. 
    Further suppose that $\alpha<\frac{\log(3)}{6\log(2)}$, and for some $0<\Tilde{\alpha}<s<\frac{1}{3}$
    \begin{equation}
    H_s(0)=\sum_{n=0}^\infty 
    \left(\sqrt{3}\right)^{-sn}\psi_n(0)
    > 
    C_{\alpha,s} \nu,
    \end{equation}
    where
    \begin{equation}
    C_{\alpha,s }=
    \frac{\sqrt{3}\left(12\pi^2\right)^\alpha}
    {\pi\left(1-\left(\sqrt{3}\right)^{-(1+s-4\Tilde{\alpha})}\right)^\frac{1}{2}
    \left(\left(\sqrt{3}\right)^{1-s}-
    \left(\sqrt{3}\right)^\frac{1+s}{2}\right)
    \left(1-
    \left(\sqrt{3}\right)^{-(1+s)}
    \right)^\frac{1}{2}}.
    \end{equation}
    Then for all $0\leq t<T_{max}$,
    \begin{equation}
    H_s(t)>
    \frac{H_s(0)}
    {1-\kappa_s H_s(0)t},
    \end{equation}
    where
    \begin{equation}
    \kappa_s
    =
    \frac{\pi}{\sqrt{3}} \left(
    \left(\sqrt{3}\right)^{1-s}
    -\left(\sqrt{3}\right)^\frac{1+s}{2}\right)
    \left(1-\left(\sqrt{3}\right)^{-(1+s)}\right).
    \end{equation}
    Note in particular this implies that 
    \begin{equation}
        T_{max}\leq
        \frac{1}
        {\kappa_s H_s(0)}.
    \end{equation}
\end{theorem}

\begin{proof}
    The result follows immediately by integrating the differential inequality in \Cref{BlowupLemmaB}.
\end{proof}

\begin{theorem} \label{RestrictedEulerBlowup}
    Suppose $u\in C\left([0,T_{max});
    \dot{H}^\frac{\log(3)}{2\log(2)}
    _{\mathcal{M}}\right)$ 
    is an odd, permutation-symmetric solution of the Fourier-restricted, hypodissipative Euler equation with hj-parity. If $u^0$ is not identically zero, then this solution blows up in finite-time with
    \begin{equation}
    T_{max}\leq 
    \left(\left(\frac{3^\frac{1}{4}}
    {12\left(3^\frac{1}{4}-1\right)}
    \right)^\frac{1}{2}
    \left\|u^0\right\|_{L^2}
    -H(0) \right)
    \frac{6\sqrt{3}}{\pi\left(
    3^\frac{3}{8}-3^\frac{5}{16}\right)
    \left\|u^0\right\|_{L^2}^2},
    \end{equation}
    where
    \begin{equation}
    H(0)
    =
    \sum_{n=0}^\infty
    \left(\sqrt{3}\right)^{-\frac{n}{4}}
    \psi_n(0).    
    \end{equation}
\end{theorem}

\begin{proof}
    Let $H(t)$ be given by
    \begin{equation}
    H(t)
    =
    \sum_{n=0}^\infty
    \left(\sqrt{3}\right)^{-\frac{n}{4}}
    \psi_n(t).
    \end{equation}
    Applying \Cref{BlowupLemmaA}, with $\nu=0$ and $s=\frac{1}{4}$, we can see that
    \begin{align}
    \frac{\diff H}{\diff t}
    &\geq 
    \frac{2\pi}{\sqrt{3}} \left(
    \left(\sqrt{3}\right)^\frac{3}{4}
    -\left(\sqrt{3}\right)^\frac{5}{8}\right)
    \|\psi\|_{\mathcal{H}^\frac{1-s}{2}}^2 \\
    &\geq 
    \frac{2\pi}{\sqrt{3}} \left(
    \left(\sqrt{3}\right)^\frac{3}{4}
    -\left(\sqrt{3}\right)^\frac{5}{8}\right)
    \|\psi\|_{L^2}^2 \\
    &=
    \frac{\pi}{6\sqrt{3}} \left(
    \left(\sqrt{3}\right)^\frac{3}{4}
    -\left(\sqrt{3}\right)^\frac{5}{8}\right)
    \|u\|_{L^2}^2 \\
    &=
    \frac{\pi}{6\sqrt{3}} \left(
    3^\frac{3}{8}-3^\frac{5}{16}\right)
    \left\|u^0\right\|_{L^2}^2,
    \end{align}
    where we have used the energy equality and the fact that $\|u\|_{L^2}^2=12\|\psi\|_{L^2}^2$.
    Integrating this differential inequality implies that for all $0<t<T_{max}$,
    \begin{equation}
    H(t)
    \geq 
    H(0)+\frac{\pi}{6\sqrt{3}} \left(
    3^\frac{3}{8}-3^\frac{5}{16}\right)
    \left\|u^0\right\|_{L^2}^2 t.
    \end{equation}

    Next we observe that applying H\"older's inequality, we find that for all $0<t<T_{max}$
    \begin{align}
    H(t)
    &= 
    \sum_{n=0}^\infty
    \left(\sqrt{3}\right)^{-\frac{n}{4}}
    \psi_n(t) \\
    & \leq 
    \left(\sum_{n=0}^\infty
    3^{-\frac{n}{4}}\right)^\frac{1}{2}
    \|\psi(t)\|_{L^2} \\
    &=
    \left(\frac{1}{1-\frac{1}{3^\frac{1}{4}}}\right)^\frac{1}{2}
    \frac{1}{\sqrt{12}}
    \|u(\cdot,t)\|_{L^2} \\
    &=
    \left( \frac{3^\frac{1}{4}}
    {3^\frac{1}{4}-1}\right)^\frac{1}{2}
    \frac{1}{\sqrt{12}}
    \left\|u^0\right\|_{L^2}.
    \end{align}
    This implies that for all $0<t<T_{max}$,
    \begin{equation}
    H(0)+\frac{\pi}{6\sqrt{3}} \left(
    3^\frac{3}{8}-3^\frac{5}{16}\right)
    \left\|u^0\right\|_{L^2}^2 t
    \leq
    \left( \frac{3^\frac{1}{4}}
    {12\left(3^\frac{1}{4}-1\right)}
    \right)^\frac{1}{2}
    \left\|u^0\right\|_{L^2},
    \end{equation}
    and therefore
    \begin{equation}
    T_{max}\leq 
    \left(\left(\frac{3^\frac{1}{4}}
    {12\left(3^\frac{1}{4}-1\right)}
    \right)^\frac{1}{2}
    \left\|u^0\right\|_{L^2}
    -H(0) \right)
    \frac{6\sqrt{3}}{\pi\left(
    3^\frac{3}{8}-3^\frac{5}{16}\right)
    \left\|u^0\right\|_{L^2}^2}.
    \end{equation}
\end{proof}

\begin{remark}
    We should note that we have proven \Cref{RestrictedEulerBlowup,RestrictedHypoBlowup}---the blowup results for both the Fourier-restricted Euler and hypodissipative Navier--Stokes equations respectively---for odd, permutation symmetric, hj-parity solutions, while in the introduction the blowup result was stated for odd, permutation symmetric, $\sigma$-mirror symmetric solutions. These two sets of conditions are in fact equivalent, as we will prove in \Cref{MirrorHJequivThm}.
\end{remark}

\subsection{Finite-time blowup for Dyadic Navier--Stokes}

In this section, we will prove \Cref{DyadicNSblowupIntro}, giving a finite time blowup result for the dyadic Navier--Stokes equation when $\alpha<\frac{1}{3}$ that is somewhat more general than the result in \cite{Cheskidov}, albeit still applying in the same range of dissipation.
Recall that the dyadic Navier--Stokes equation is given by
\begin{equation} \label{DyadicNSeqn}
    \partial_t u_n
    =
    -\nu \lambda^{2\alpha n}u_n
    +\lambda^n u_{n-1}^2
    -\lambda^{n+1}u_n u_{n+1},
\end{equation}
for all $n\in\mathbb{Z}^+$ with $\lambda>1$, and where $u_{-1}=0$ by convention. The proof of finite-time blowup will follow the structure of the proof of \Cref{RestrictedHypoBlowup}, and so the details will only be sketched. Note that throughout this section we will take
\begin{equation}
    \|u\|_{\mathcal{H}^s}^2
    =
    \sum_{n=0}^\infty 
    \lambda^{2sn} u_n^2.
\end{equation}
Previously, we only considered the case $\lambda=\sqrt{3}$, as this is the $\lambda$ that corresponds the the Fourier-restricted hypodissipative Navier--Stokes equation.

\begin{proposition} \label{DyadicPropA}
    Suppose $u\in C\left([0,T_{max});\mathcal{H}^1\right)$
    is a solution of the dyadic Navier--Stokes equation and that $0<\alpha<s<\frac{1}{3}$,
    and let
    \begin{equation}
    H_s(t)=\sum_{n=0}^\infty \lambda^{-sn} u_n.
    \end{equation}
    Then for all $0\leq t<T_{max}$,
    \begin{equation}
    \frac{\diff H_s}{\diff t}
    \geq 
    -\frac{\nu}{\left(1-\lambda^{
    -(1+s-4\alpha)}\right)^\frac{1}{2}}
    \|u\|_{\mathcal{H}^\frac{1-s}{2}}
    +\left(\lambda^{1-s}
    -\lambda^{\frac{1+s}{2}}\right)
    \|u\|_{\mathcal{H}^\frac{1-s}{2}}^2
    \end{equation}
\end{proposition}

\begin{proof}
    Differentiating $H_s$ term-by-term using \eqref{DyadicNSeqn}, we find that
    \begin{align}
    \frac{\diff H_s}{\diff t}
    &=
    -\nu \sum_{n=0}^\infty
    \lambda^{(2\alpha-s)n}u_n
    +\sum_{n=1}^\infty \lambda^{(1-s)n} u_{n-1}^2
    -\sum_{n=0}^\infty \lambda^{-sn+n+1}
    u_n u_{n+1} \\
    &=
    -\nu \sum_{n=0}^\infty
    \lambda^{(2\alpha-s)n}u_n
    +\lambda^{1-s}\sum_{n=0}^\infty \lambda^{(1-s)n} u_{n-1}^2
    -\lambda^{\frac{1+s}{2}}\sum_{n=0}^\infty \lambda^{-sn+n+\frac{1}{2}-\frac{s}{2}}
    u_n u_{n+1} \\
    &=
    -\nu \sum_{n=0}^\infty
    \lambda^{(2\alpha-s)n}u_n
    +\lambda^{1-s}\|u\|_{\mathcal{H}^\frac{1-s}{2}}^2
    -\lambda^{\frac{1+s}{2}}\sum_{n=0}^\infty \lambda^{-sn+n+\frac{1}{2}-\frac{s}{2}}
    u_n u_{n+1}.
    \end{align}
    Next obeserve that
    \begin{align}
    \sum_{n=0}^\infty \lambda^{-sn+n+\frac{1}{2}-\frac{s}{2}}
    u_n u_{n+1}
    &\leq 
    \left(\sum_{n=0}^\infty 
    \lambda^{(1-s)n} u_n^2\right)^\frac{1}{2} 
    \left(\sum_{n=0}^\infty 
    \lambda^{(1-s)(n+1)}
    u_{n+1}^2\right)^\frac{1}{2} \\
    &\leq 
    \left(\sum_{n=0}^\infty 
    \lambda^{(1-s)n} u_n^2\right)^\frac{1}{2} 
    \left(\sum_{n=1}^\infty 
    \lambda^{(1-s)n}
    u_n^2\right)^\frac{1}{2} \\
    &\leq 
    \|u\|_{\mathcal{H}^\frac{1-s}{2}}^2.
    \end{align}
    Finally observe that
    \begin{align}
    \sum_{n=0}^\infty
    \lambda^{(2\alpha-s)n}u_n
    &\leq 
    \sum_{n=0}^\infty
    \lambda^{\left(2\alpha-\frac{s}{2}-\frac{1}{2}\right)}
    \lambda^{\left(\frac{1-s}{2}\right)n}
    u_n \\
    &\leq 
    \left(\sum_{n=0}^\infty \lambda^{-(1+s-4\alpha)n}\right)^\frac{1}{2}
    \left(\sum_{n=0}^\infty
    \lambda^{(1-s)n}u_n^2
    \right)^\frac{1}{2} \\
    &=
    \frac{1}{\left(1-\lambda^{
    -(1+s-4\alpha)}\right)^\frac{1}{2}}
    \|u\|_{\mathcal{H}^\frac{1-s}{2}},
    \end{align}
    and this completes the proof. Note that 
    $0<\alpha<s<\frac{1}{3}$ implies that
    $1-s>\frac{1+s}{2}$ and $1+s-4\alpha>0$.
\end{proof}

\begin{proposition} \label{DyadicPropB}
    Suppose $u\in \mathcal{H}^1$ and $0<s<\frac{1}{3}$.
    Then 
    \begin{equation}
    \sum_{n=0}^\infty \lambda^{-sn}u_n
    \leq 
    \frac{1}{\left(1-\lambda^{-(1+s)}
    \right)^\frac{1}{2}}
    \|u\|_{\mathcal{H}^\frac{1-s}{2}}.
    \end{equation}
\end{proposition}

\begin{proof}
    Compute that
    \begin{align}
    \sum_{n=0}^\infty
    \lambda^{-sn}u_n
    &=
    \sum_{n=0}^\infty
    \lambda^{\left(\frac{1-s}{2}\right)n}u_n
    \lambda^{-\left(\frac{1+s}{2}\right)n} \\
    &\leq 
    \left(\sum_{n=0}^\infty
    \lambda^{(1-s)n}u_n^2
    \right)^\frac{1}{2}
    \left(\sum_{n=0}^\infty
    \lambda^{-(1+s)n}
    \right)^\frac{1}{2} \\
    &=
    \frac{1}{\left(1-\lambda^{-(1+s)}
    \right)^\frac{1}{2}}
    \|u\|_{\mathcal{H}^\frac{1-s}{2}}.
    \end{align}
\end{proof}

Now we can prove \Cref{DyadicNSblowupIntro}, which is restated for the reader's convenience.

\begin{theorem} \label{DyadicNSblowup}
    Suppose $u\in C\left([0,T_{max};\mathcal{H}^1\right)$
    is a solution of the dyadic Navier--Stokes equation, and that for some $0<\alpha<s<\frac{1}{3}$,
    \begin{equation}
    H_s(0)=\sum_{n=0}^\infty
    \lambda^{-sn}u_n(0)
    >C_{\alpha,s}\nu,
    \end{equation}
    where
    \begin{equation}
    C_{\alpha,s}
    =
    \frac{2}{\left(1-\lambda^{-(1+s)}
    \right)^\frac{1}{2}
    \left(\lambda^{1-s}
    -\lambda^{\frac{1+s}{2}}\right)
    \left(1-\lambda^{
    -(1+s-4\alpha)}\right)^\frac{1}{2}}.
    \end{equation}
    Then this solution blows up in finite-time
    with
    \begin{equation}
    T_{max}<\frac{1}{\kappa_s H_s(0)},
    \end{equation}
    where
    \begin{equation}
    \kappa_s
    =
    \frac{1}{2}\left(\lambda^{1-s}
    -\lambda^{\frac{1+s}{2}}\right)
    \left(1-\lambda^{-(1+s)}
    \right).
    \end{equation}
\end{theorem}

\begin{proof}
    We begin by showing that $H_s(t)>C_{\alpha,s}\nu$ implies that
    $\frac{\diff}{\diff t}H_s(t)>0$.
    Note that if $H_s(t)>C_{\alpha,s}\nu$, then by \Cref{DyadicPropB} we can see that
    \begin{equation}
    \|u(t)\|_{\mathcal{H}^\frac{1-s}{2}}
    > 
    \frac{2\nu}
    {\left(\lambda^{1-s}
    -\lambda^{\frac{1+s}{2}}\right)
    \left(1-\lambda^{
    -(1+s-4\alpha)}\right)^\frac{1}{2}},
    \end{equation}
    and by \Cref{DyadicPropA}, we can see that
    \begin{align}
    \frac{\diff H_s}{\diff t}
    \geq
    -\frac{\nu}{\left(1-\lambda^{
    -(1+s-4\alpha)}\right)^\frac{1}{2}}
    \|u\|_{\mathcal{H}^\frac{1-s}{2}}
    +\left(\lambda^{1-s}
    -\lambda^{\frac{1+s}{2}}\right)
    \|u\|_{\mathcal{H}^\frac{1-s}{2}}^2
    >0.
    \end{align}
    Because $H_s(0)>C_{\alpha,s}\nu$ by hypothesis, this implies that for all $0\leq t<T_{max}$,
    \begin{equation}
    H_s(t)>C_{\alpha,s}\nu.
    \end{equation}

    This in turn implies that
    for all $0\leq t<T_{max}$,
    \begin{equation}
    \|u(t)\|_{\mathcal{H}^\frac{1-s}{2}}
    > 
    \frac{2\nu}
    {\left(\lambda^{1-s}
    -\lambda^{\frac{1+s}{2}}\right)
    \left(1-\lambda^{
    -(1+s-4\alpha)}\right)^\frac{1}{2}}.
    \end{equation}
    Again applying \Cref{DyadicPropA,DyadicPropB}, we find that
    for all $0\leq t<T_{max}$,
    \begin{align}
    \frac{\diff H_s}{\diff t}
    &\geq 
    -\frac{\nu}{\left(1-\lambda^{
    -(1+s-4\alpha)}\right)^\frac{1}{2}}
    \|u(t)\|_{\mathcal{H}^\frac{1-s}{2}}
    +\left(\lambda^{1-s}
    -\lambda^{\frac{1+s}{2}}\right)
    \|u(t)\|_{\mathcal{H}^\frac{1-s}{2}}^2 \\
    &\geq 
    \frac{1}{2}\left(\lambda^{1-s}
    -\lambda^{\frac{1+s}{2}}\right)
    \|u(t)\|_{\mathcal{H}^\frac{1-s}{2}}^2 \\
    &\geq 
    \frac{1}{2}\left(\lambda^{1-s}
    -\lambda^{\frac{1+s}{2}}\right)
    \left(1-\lambda^{-(1+s)}
    \right)
    H_s(t)^2 \\
    &=
    \kappa_s H_s(t)^2.
    \end{align}
    Integrating this differential inequality, we find that for all $0\leq t<T_{max}$,
    \begin{equation}
    H_s(t)
    \geq 
    \frac{H_s(0)}
    {1-\kappa_s H_s(0)t}.
    \end{equation}
    This implies that
    \begin{equation}
    T_{max}<\frac{1}{\kappa_s H_s(0)},
    \end{equation}
    and this completes the proof.
\end{proof}

\begin{remark}
    The finite-time blowup argument given by Cheskidov in \cite{Cheskidov} is also based on a singular lower bound on a Lyapunov functional, but has the additional restriction that $u_n(0)\geq 0$ for all $n\in\mathbb{Z}^+$, and so the class of data covered by \Cref{DyadicNSblowup} is slightly more broad. Note that we still have a positivity condition, in that if for an initial data $u^0$, we have $H_s(0)>0$ for some $0<\alpha<s<\frac{1}{3}$, then there is finite-time blowup for the solution with sufficiently small viscosity $\nu<\frac{H_s(0)}{C_{\alpha,s}}$. However, the positivity condition is now only that a weighted sum of the $u_n$ must be positive; we do not require that each of the $u_n$ be nonnegative.
\end{remark}

\appendix

\section{Mode interaction computations} \label{AppendixModeInteraction}

In this appendix, we will go through the computations from \Cref{ComputeTheFuckingModesProp}, breaking the nine different mode interactions into nine different propositions for ease of reading.

\begin{proposition} \label{Tedious1}
Fix $m\in\mathbb{Z}^+$, and let $u$ and $\Tilde{u}$ be given by
\begin{align}
    u
    &=
    iv^{k^m}e^{2\pi i k^m\cdot x} \\
    \Tilde{u}
    &=
    iP_{12}(v^{k^m})
    e^{2\pi i P_{12}(k^m)\cdot x}.
    \end{align}
Then the bilinear term in the restricted model is given by
\begin{equation}
    \mathbb{P}_{M}((\Tilde{u}\cdot\nabla) u
    +(u\cdot\nabla)\Tilde{u})
    =
    -a_m i v^{h^m}e^{2\pi i h^m\cdot x},
\end{equation}
where 
\begin{equation}
    a_m=
    \frac{\sqrt{6}\pi}
    {\left(1+\frac{1}{2}
    \left( \frac{3}{4}
    \right)^{2m}\right)^\frac{1}{2}}
    3^m
\end{equation}
\end{proposition}

\begin{proof}
    First observe that
    \begin{align}
    (\Tilde{u}\cdot\nabla)u
    &=
    -2\pi i k^m\cdot P_{12}(v^{k^m}) 
    v^{k^m}
    e^{2\pi i h^m\cdot x} \\
    (u\cdot\nabla)\Tilde{u}
    &=
    -2\pi i P_{12}(k^m)\cdot v^{k^m} 
    P_{12}(v^{k^m})
    e^{2\pi i h^m\cdot x},
    \end{align}
recalling that $h^m=k^m+P_{12}(k^m)$.
Next compute that
\begin{align}
    k^m\cdot P_{12}(v^{k^m})
    &=
    P_{12}(k^m)\cdot v^{k^m} \\
    &=
    \left(2^{2m}\sigma+3^m\left(
    \begin{array}{c}
         1  \\ 0 \\ -1
    \end{array}
    \right)\right)
    \cdot
    \frac{2*3^{2m}\sigma- 2^{2m}*3^{m+1}\left(
    \begin{array}{c}
         0  \\ 1 \\ -1
    \end{array}\right)}
    {\left(4*3^{4m+1}+2^{4m+1}*3^{2m+2}
    \right)^\frac{1}{2}} \\
    &=
    \frac{2^{2m+1}*3^{2m+1}-2^{2m}*3^{2m+1}}
    {\left(4*3^{4m+1}+2^{4m+1}*3^{2m+2}
    \right)^\frac{1}{2}} \\
    &=
     \frac{2^{2m}*3^{2m+1}}
    {\left(4*3^{4m+1}+2^{4m+1}*3^{2m+2}
    \right)^\frac{1}{2}},
\end{align}
and that
\begin{equation}
    v^{k^m}+P_{12}(v^{k^m})
    =
    \frac{4*3^{2m}\sigma- 2^{2m}*3^{m+1}\left(
    \begin{array}{c}
         1  \\ 1 \\ -2
    \end{array}\right)}
    {\left(4*3^{4m+1}+2^{4m+1}*3^{2m+2}
    \right)^\frac{1}{2}}.
\end{equation}
Therefore we can conclude that
\begin{equation*}
    (\Tilde{u}\cdot\nabla)u
    +(u\cdot\nabla)\Tilde{u}
    =
    \frac{-2\pi i*2^{2m}*3^{2m+1}}
    {4*3^{4m+1}+2^{4m+1}*3^{2m+2}}
    \left(
    4*3^{2m}\sigma- 2^{2m}*3^{m+1}\left(
    \begin{array}{c}
         1  \\ 1 \\ -2
    \end{array}\right)
    \right)
    e^{2\pi i h^m\cdot x}.
\end{equation*}
Note that because we only have a single Fourier mode, we can compute the projection onto the constraint space by
\begin{equation}
    \mathbb{P}_{\mathcal{M}}
    ((\Tilde{u}\cdot\nabla)u
    +(u\cdot\nabla)\Tilde{u})
    =
    v^{h^m}\cdot((\Tilde{u}\cdot\nabla)u
    +(u\cdot\nabla)\Tilde{u}) v^{h^m}.
\end{equation}
All that is left is to compute that
\begin{align}
    v^{h^m}\cdot
    \left(
    4*3^{2m}\sigma- 2^{2m}*3^{m+1}\left(
    \begin{array}{c}
         1 \\ 1 \\ -2
    \end{array}\right)
    \right)
    &=
    \frac{2*3^{2m+1}\sigma-
    2^{2m+1}3^{m+1}\left(
    \begin{array}{c}
         1 \\ 1 \\ -2
    \end{array}\right)}
    {\left(4*3^{4m+3}+2^{4m+3}*3^{2m+3}
    \right)^\frac{1}{2}} \\
    \notag &\quad\quad \cdot \left(
    4*3^{2m}\sigma- 2^{2m}*3^{m+1}\left(
    \begin{array}{c}
         1  \\ 1 \\ -2
    \end{array}\right)\right)\\
    &=
    \frac{8*3^{4m+2}+2^{4m+2}*3^{2m+3}}
    {\left(4*3^{4m+3}+2^{4m+3}*3^{2m+3}
    \right)^\frac{1}{2}},
\end{align}
and we may conclude that
\begin{equation}
    \mathbb{P}_{\mathcal{M}}
    ((\Tilde{u}\cdot\nabla)u
    +(u\cdot\nabla)\Tilde{u})
    =
    -i\frac{2\pi *2^{2m}*3^{2m+1}
    \left(8*3^{4m+2}+2^{4m+2}*3^{2m+3}\right)
    v^{h^m} e^{2\pi i h^m\cdot x}}
    {\left(4*3^{4m+1}+2^{4m+1}*3^{2m+2}\right)
    \left(4*3^{4m+3}+2^{4m+3}*3^{2m+3}
    \right)^\frac{1}{2}}.
\end{equation}
We have now shown the result with
\begin{equation}
    a_m=
    \frac{2\pi *2^{2m}*3^{2m+1}
    \left(8*3^{4m+2}+2^{4m+2}*3^{2m+3}\right)}
    {\left(4*3^{4m+1}+2^{4m+1}*3^{2m+2}\right)
    \left(4*3^{4m+3}+2^{4m+3}*3^{2m+3}
    \right)^\frac{1}{2}}.
\end{equation}
To complete the proof, observe that
\begin{equation}
    \frac{8*3^{4m+2}+2^{4m+2}*3^{2m+3}}
    {4*3^{4m+1}+2^{4m+1}*3^{2m+2}}
    =6,
\end{equation}
and so
\begin{align}
    a_m
    &=
    \frac{12\pi*2^{2m}*3^{2m+1}}
    {\left(4*3^{4m+3}+2^{4m+3}*3^{2m+3}
    \right)^\frac{1}{2}} \\
    &=
    \frac{2^{2m+2}*3^{2m+2}\pi}
    {\left(1+2^{-4m-1}*3^{2m}
    \right)^\frac{1}{2}
    2^{2m+3/2}*3^{m+3/2}} \\
    &=
    \frac{\sqrt{6}\pi}
    {\left(1+\frac{1}{2}
    \left( \frac{3}{4}
    \right)^{2m}\right)^\frac{1}{2}}
    3^m.
\end{align}
\end{proof}

\begin{proposition} \label{Tedious2}
Fix $m\in\mathbb{Z}^+$, and let $u$ and $\Tilde{u}$ be given by
\begin{align}
    u&=
    iv^{k^m}e^{2\pi i k^m\cdot x} \\
    \Tilde{u}
    &=
    iP_{23}(v^{k^m})
    e^{2\pi i P_{23}(k^m)\cdot x}.
    \end{align}
Then the bilinear term in the restricted model is given by
\begin{equation}
    \mathbb{P}_{\mathcal{M}}((\Tilde{u}\cdot\nabla) u
    +(u\cdot\nabla)\Tilde{u})
    =
    -a_m i v^{j^m}e^{2\pi i j^m\cdot x},
\end{equation}
where 
\begin{equation}
    a_m=
    \frac{\sqrt{6}\pi}
    {\left(1+\frac{1}{2}
    \left( \frac{3}{4}
    \right)^{2m}\right)^\frac{1}{2}}
    3^m
\end{equation}
\end{proposition}

\begin{proof}
    First observe that
    \begin{align}
    (\Tilde{u}\cdot\nabla)u
    &=
    -2\pi i k^m\cdot P_{23}(v^{k^m}) 
    v^{k^m}
    e^{2\pi i j^m\cdot x} \\
    (u\cdot\nabla)\Tilde{u}
    &=
    -2\pi i P_{23}(k^m)\cdot v^{k^m} 
    P_{23}(v^{k^m})
    e^{2\pi i j^m\cdot x},
    \end{align}
recalling that $j^m=k^m+P_{23}(k^m)$.
Next compute that
\begin{align}
    k^m\cdot P_{23}(v^{k^m})
    &=
    P_{23}(k^m)\cdot v^{k^m} \\
    &=
    \left(2^{2m}\sigma+3^m\left(
    \begin{array}{c}
         1  \\ 0 \\ -1
    \end{array}
    \right)\right)
    \cdot
    \frac{2*3^{2m}\sigma- 2^{2m}*3^{m+1}\left(
    \begin{array}{c}
         1  \\ -1 \\ 0
    \end{array}\right)}
    {\left(4*3^{4m+1}+2^{4m+1}*3^{2m+2}
    \right)^\frac{1}{2}} \\
    &=
    \frac{2^{2m+1}*3^{2m+1}-2^{2m}*3^{2m+1}}
    {\left(4*3^{4m+1}+2^{4m+1}*3^{2m+2}
    \right)^\frac{1}{2}} \\
    &=
     \frac{2^{2m}*3^{2m+1}}
    {\left(4*3^{4m+1}+2^{4m+1}*3^{2m+2}
    \right)^\frac{1}{2}},
\end{align}
and that
\begin{equation}
    v^{k^m}+P_{23}(v^{k^m})
    =
    \frac{4*3^{2m}\sigma- 2^{2m}*3^{m+1}\left(
    \begin{array}{c}
         2  \\ -1 \\ -1
    \end{array}\right)}
    {\left(4*3^{4m+1}+2^{4m+1}*3^{2m+2}
    \right)^\frac{1}{2}}.
\end{equation}
Therefore we can conclude that
\begin{equation*}
    (\Tilde{u}\cdot\nabla)u
    +(u\cdot\nabla)\Tilde{u}
    =
    \frac{-2\pi i*2^{2m}*3^{2m+1}}
    {4*3^{4m+1}+2^{4m+1}*3^{2m+2}}
    \left(
    4*3^{2m}\sigma- 2^{2m}*3^{m+1}\left(
    \begin{array}{c}
         2  \\ -1 \\ -1
    \end{array}\right)
    \right)
    e^{2\pi i j^m\cdot x}.
\end{equation*}
Note that because we only have a single Fourier mode, we can compute the projection onto the constraint space by
\begin{equation}
    \mathbb{P}_{\mathcal{M}}
    ((\Tilde{u}\cdot\nabla)u
    +(u\cdot\nabla)\Tilde{u})
    =
    v^{j^m}\cdot((\Tilde{u}\cdot\nabla)u
    +(u\cdot\nabla)\Tilde{u}) v^{j^m}.
\end{equation}
All that is left is to compute that
\begin{align}
    v^{j^m}\cdot
    \left(
    4*3^{2m}\sigma- 2^{2m}*3^{m+1}\left(
    \begin{array}{c}
         2  \\ -1 \\ -1
    \end{array}\right)
    \right)
    &=
    \frac{2*3^{2m+1}\sigma-
    2^{2m+1}3^{m+1}\left(
    \begin{array}{c}
         2 \\ -1 \\ -1 
    \end{array}\right)}
    {\left(4*3^{4m+3}+2^{4m+3}*3^{2m+3}
    \right)^\frac{1}{2}} \\
    \notag &\quad\quad \cdot \left(
    4*3^{2m}\sigma- 2^{2m}*3^{m+1}\left(
    \begin{array}{c}
         2  \\ -1 \\ -1
    \end{array}\right)\right)\\
    &=
    \frac{8*3^{4m+2}+2^{4m+2}*3^{2m+3}}
    {\left(4*3^{4m+3}+2^{4m+3}*3^{2m+3}
    \right)^\frac{1}{2}},
\end{align}
and we may conclude that
\begin{equation}
    \mathbb{P}_{\mathcal{M}}
    ((\Tilde{u}\cdot\nabla)u
    +(u\cdot\nabla)\Tilde{u})
    =
    -i\frac{2\pi *2^{2m}*3^{2m+1}
    \left(8*3^{4m+2}+2^{4m+2}*3^{2m+3}\right)
    v^{j^m} e^{2\pi i j^m\cdot x}}
    {\left(4*3^{4m+1}+2^{4m+1}*3^{2m+2}\right)
    \left(4*3^{4m+3}+2^{4m+3}*3^{2m+3}
    \right)^\frac{1}{2}}.
\end{equation}
We have now shown the result with
\begin{equation}
    a_m=
    \frac{2\pi *2^{2m}*3^{2m+1}
    \left(8*3^{4m+2}+2^{4m+2}*3^{2m+3}\right)}
    {\left(4*3^{4m+1}+2^{4m+1}*3^{2m+2}\right)
    \left(4*3^{4m+3}+2^{4m+3}*3^{2m+3}
    \right)^\frac{1}{2}}.
\end{equation}
Observe as above that
\begin{equation}
    a_m
    =
    \frac{\sqrt{6}\pi}
    {\left(1+\frac{1}{2}
    \left( \frac{3}{4}
    \right)^{2m}\right)^\frac{1}{2}}
    3^m,
\end{equation}
and this completes the proof.
\end{proof}

\begin{proposition} \label{Tedious3}
Fix $m\in\mathbb{Z}^+$, and let $u$ and $\Tilde{u}$ be given by
\begin{align}
    u&=
    iv^{h^m}e^{2\pi i h^m\cdot x} \\
    \Tilde{u}
    &=
    iv^{j^m}
    e^{2\pi i j^m\cdot x}.
    \end{align}
Then the bilinear term in the restricted model is given by
\begin{equation}
    \mathbb{P}_{\mathcal{M}}((\Tilde{u}\cdot\nabla) u
    +(u\cdot\nabla)\Tilde{u})
    =
    -b_m i v^{k^{m+1}}e^{2\pi i k^{m+1}\cdot x},
\end{equation}
where 
\begin{equation}
    b_m=
    \frac{\sqrt{2}\pi}{\left(1+\frac{3}{8}
    \left(\frac{3}{4}\right)^{2m}
    \right)^\frac{1}{2}} 3^{m+1}.
\end{equation}
\end{proposition}

\begin{proof}
First observe that
\begin{align}
    (\Tilde{u}\cdot\nabla) u
    &=
    -2\pi i (h^m\cdot v^{j^m})
    v^{h^m} 
    e^{2\pi i h^m\cdot x} \\
    (u\cdot\nabla)\Tilde{u}
    &=
    -2\pi i (j^m\cdot v^{h^m})
    v^{j^m} 
    e^{2\pi i h^m\cdot x}
\end{align}
Next we compute that
\begin{align}
    h^m\cdot v^{j^m}
    &=
    \left(2^{2m+1}\sigma
    +3^m\left(
    \begin{array}{c}
         1 \\ 1 \\ -2 
    \end{array}\right)
    \right) \cdot
    \frac{2*3^{2m+1}\sigma
    -2^{2m+1}*3^{m+1}
    \left(\begin{array}{c}
         2 \\ -1 \\ -1
    \end{array}
    \right)}
    {\left(4*3^{4m+3}
    +2^{4m+3}*3^{2m+3}
    \right)^\frac{1}{2}} \\
    &=
    \frac{2^{2m+2}*3^{2m+2}
    -2^{2m+1}*3^{2m+2}}
    {\left(4*3^{4m+3}
    +2^{4m+3}*3^{2m+3}
    \right)^\frac{1}{2}} \\
    &=
    \frac{2^{2m+1}*3^{2m+2}}
    {\left(4*3^{4m+3}
    +2^{4m+3}*3^{2m+3}
    \right)^\frac{1}{2}},
\end{align}
and likewise that
\begin{align}
    j^m\cdot v^{h^m}
    &=
    \left(2^{2m+1}\sigma
    +3^m\left(
    \begin{array}{c}
         2 \\ -1 \\ -1
    \end{array}\right)
    \right) \cdot
    \frac{2*3^{2m+1}\sigma
    -2^{2m+1}*3^{m+1}
    \left(\begin{array}{c}
         1 \\ 1 \\ -2
    \end{array}
    \right)}
    {\left(4*3^{4m+3}
    +2^{4m+3}*3^{2m+3}
    \right)^\frac{1}{2}} \\
    &=
    \frac{2^{2m+2}*3^{2m+2}
    -2^{2m+1}*3^{2m+2}}
    {\left(4*3^{4m+3}
    +2^{4m+3}*3^{2m+3}
    \right)^\frac{1}{2}} \\
    &=
    \frac{2^{2m+1}*3^{2m+2}}
    {\left(4*3^{4m+3}
    +2^{4m+3}*3^{2m+3}
    \right)^\frac{1}{2}}.
\end{align}
Further computing that
\begin{align}
    v^{h^m}+v^{j^m}
    &=
    \frac{2*3^{2m+1}\sigma
    -2^{2m+1}*3^{m+1}
    \left(\begin{array}{c}
         2 \\ -1 \\ -1
    \end{array}
    \right)}
    {\left(4*3^{4m+3}
    +2^{4m+3}*3^{2m+3}
    \right)^\frac{1}{2}}
    +
    \frac{2*3^{2m+1}\sigma
    -2^{2m+1}*3^{m+1}
    \left(\begin{array}{c}
         1 \\ 1 \\ -2
    \end{array}
    \right)}
    {\left(4*3^{4m+3}
    +2^{4m+3}*3^{2m+3}
    \right)^\frac{1}{2}} \\
    &=
    \frac{4*3^{2m+1}\sigma
    -2^{2m+1}*3^{m+1}
    \left(\begin{array}{c}
         3 \\ 0 \\ -3
    \end{array}
    \right)}
    {\left(4*3^{4m+3}
    +2^{4m+3}*3^{2m+3}
    \right)^\frac{1}{2}} \\
    &=
    \frac{4*3^{2m+1}\sigma
    -2^{2m+1}*3^{m+2}
    \left(\begin{array}{c}
         1 \\ 0 \\ -1
    \end{array}
    \right)}
    {\left(4*3^{4m+3}
    +2^{4m+3}*3^{2m+3}
    \right)^\frac{1}{2}}
\end{align}
Putting these together we find that
\begin{equation*}
    (\Tilde{u}\cdot\nabla)u
    +
    (u\cdot\nabla)\Tilde{u}
    =
    -2\pi i
    \frac{2^{2m+1}*3^{2m+2}
    \left(
    4*3^{2m+1}\sigma
    -2^{2m+1}*3^{m+2}
    \left(\begin{array}{c}
         1 \\ 0 \\ -1
    \end{array}
    \right)\right)}
    {4*3^{4m+3}
    +2^{4m+3}*3^{2m+3}}
    e^{2\pi i k^{m+1}\cdot x}.
\end{equation*}
Noting that this term consists of the single Fourier mode $k^{m+1}$, we can see that
\begin{equation}
    \mathbb{P}_{\mathcal{M}}
    ((\Tilde{u}\cdot\nabla)u
    +
    (u\cdot\nabla)\Tilde{u})
    =
    v^{k^{m+1}}\cdot
    ((\Tilde{u}\cdot\nabla)u
    +
    (u\cdot\nabla)\Tilde{u})
    v^{k^{m+1}},
\end{equation}
and compute that
\begin{align}
    v^{k^{m+1}}\cdot
    ((\Tilde{u}\cdot\nabla)u
    +
    (u\cdot\nabla)\Tilde{u})
    &=
    -2\pi i
    e^{2\pi i k^{m+1}\cdot x}
    \frac{2*3^{2m+2}\sigma
    -2^{2m+2}*3^{m+2}
        \left(\begin{array}{c}
             1  \\
             0 \\
             -1
        \end{array}\right)}
    {\left(4*3^{4m+5}
    +2^{4m+5}*3^{2m+4}
    \right)^\frac{1}{2}} 
    \\
    \notag &\quad\quad \cdot 
    \frac{2^{2m+1}*3^{2m+2}
    \left(
    4*3^{2m+1}\sigma
    -2^{2m+1}*3^{m+2}
    \left(\begin{array}{c}
         1 \\ 0 \\ -1
    \end{array}
    \right)\right)}
    {4*3^{4m+3}
    +2^{4m+3}*3^{2m+3}}
    \\
    &=
    \frac{-2^{2m+2}*3^{2m+2}\left(
    8*3^{4m+4}+2^{4m+4}*3^{2m+4}\right)
    \pi i e^{2\pi ik^{m+1}\cdot x}}
    {\left(4*3^{4m+3}
    +2^{4m+3}*3^{2m+3}\right)
    \left(4*3^{4m+5}
    +2^{4m+5}*3^{2m+4}
    \right)^\frac{1}{2}} \\
    &=
    -\frac{2^{2m+3}*3^{2m+3}
    \pi i e^{2\pi ik^{m+1}\cdot x}}
    {\left(4*3^{4m+5}
    +2^{4m+5}*3^{2m+4}
    \right)^\frac{1}{2}}.
\end{align}
Therefore we can see the result holds with
\begin{align}
    b_m
    &=
    \frac{2^{2m+3}*3^{2m+3} \pi}
    {\left(4*3^{4m+5}
    +2^{4m+5}*3^{2m+4}
    \right)^\frac{1}{2}} \\
    &=
    \frac{2^{2m+3}*3^{2m+3} \pi}
    {2^{2m+\frac{5}{2}}*3^{m+2}
    \left(1+2^{-4m-3}*3^{2m+1}
    \right)^\frac{1}{2}} \\
    &=
    \frac{\sqrt{2}\pi}{\left(1+\frac{3}{8}
    \left(\frac{3}{4}\right)^{2m}
    \right)^\frac{1}{2}} 3^{m+1}.
\end{align}
This completes the proof.
\end{proof}

\begin{proposition} \label{Tedious4}
Fix $m\in\mathbb{Z}^+$, and let $u$ and $\Tilde{u}$ be given by
\begin{align}
    u&=
    iv^{h^m}e^{2\pi i h^m\cdot x} \\
    \Tilde{u}
    &=
    -iP_{12}(v^{k^m})
    e^{-2\pi i P_{12}(k^m)\cdot x}.
    \end{align}
Then the bilinear term in the restricted model is given by
\begin{equation}
    \mathbb{P}_{\mathcal{M}}((\Tilde{u}\cdot\nabla) u
    +(u\cdot\nabla)\Tilde{u})
    =
    \frac{a_m}{2}i v^{k^m}e^{2\pi i k^m\cdot x},
\end{equation}
where 
\begin{equation}
    a_m=
    \frac{\sqrt{6}\pi}
    {\left(1+\frac{1}{2}
    \left( \frac{3}{4}
    \right)^{2m}\right)^\frac{1}{2}}
    3^m.
\end{equation}
\end{proposition}

\begin{proof}
We begin by computing that
\begin{align}
    (\Tilde{u}\cdot\nabla)u
    &=
    2\pi i (h^m\cdot P_{12}(v^{k^m}))
    v^{h^m}e^{2\pi ik^m\cdot x} \\
    (u\cdot\nabla)\Tilde{u}
    &=
    -2\pi i (P_{12}(k^m)\cdot v^{h^m})
    P_{12}(v^{k^m})
    e^{2\pi i k^m\cdot x}.
\end{align}
Observe that
\begin{align}
    h^m\cdot P_{12}(v^{k^m})
    &=
    \left(2^{2m+1}\sigma+3^m\left(
    \begin{array}{cc}
         1 \\ 1 \\ -2
    \end{array}\right)\right)
    \cdot
    \frac{2*3^{2m}\sigma-2^{2m}*3^{m+1}
    \left(\begin{array}{c}
         0 \\ 1 \\ -1 
    \end{array}\right)}
    {\left(4*3^{4m+1}
    +2^{4m+1}*3^{2m+2}
    \right)^\frac{1}{2}} \\
    &=
    \frac{2^{2m+2}*3^{2m+1}
    -2^{2m}*3^{2m+2}}
    {\left(4*3^{4m+1}
    +2^{4m+1}*3^{2m+2}
    \right)^\frac{1}{2}} \\
    &=
    \frac{2^{2m}*3^{2m+1}}
    {\left(4*3^{4m+1}
    +2^{4m+1}*3^{2m+2}
    \right)^\frac{1}{2}},
\end{align}
and conclude that
\begin{equation}
    (h^m\cdot P_{12}(v^{k^m}))v^{h^m}
    =
    \frac{2^{2m}*3^{2m+1}
    \left(2*3^{2m+1}\sigma
    -2^{2m+1}*3^{m+1}
    \left(\begin{array}{c}
         1 \\ 1 \\ -2
    \end{array}\right)\right)}
    {\left(4*3^{4m+1}
    +2^{4m+1}*3^{2m+2}
    \right)^\frac{1}{2}
    \left(4*3^{4m+3}
    +2^{4m+3}*3^{2m+3}
    \right)^\frac{1}{2}},
\end{equation}
and therefore that
\begin{equation}
    (\Tilde{u}\cdot\nabla)u
    =
    \frac{2^{2m}*3^{2m+1}
    \left(2*3^{2m+1}\sigma
    -2^{2m+1}*3^{m+1}
    \left(\begin{array}{c}
         1 \\ 1 \\ -2
    \end{array}\right)\right)}
    {\left(4*3^{3m+1}
    +2^{4m+1}*3^{2m+2}
    \right)^\frac{1}{2}
    \left(4*3^{4m+3}
    +2^{4m+3}*3^{2m+3}
    \right)^\frac{1}{2}}
    e^{2\pi i k^m\cdot x}.
\end{equation}
Likewise, observe that
\begin{align}
    P_{12}(k^m)\cdot v^{h^m}
    &=
    \left(2^{2m}\sigma+3^m
    \left(\begin{array}{c}
         0 \\ 1 \\ -1
    \end{array}\right)\right)
    \cdot
    \frac{2*3^{2m+1}\sigma
    -2^{2m+1}*3^{m+1}
    \left(
    \begin{array}{c}
         1  \\
         1 \\
         -2
    \end{array}
    \right)}
    {\left(4*3^{4m+3}
    +2^{4m+3}*3^{2m+3}
    \right)^\frac{1}{2}} \\
    &=
    \frac{2^{2m+1}*3^{2m+2}
    -2^{2m+1}*3^{2m+2}}
    {\left(4*3^{4m+3}
    +2^{4m+3}*3^{2m+3}
    \right)^\frac{1}{2}} \\
    &=
    0,
\end{align}
and conclude that
\begin{equation}
    (u\cdot\nabla)\Tilde{u}=0.
\end{equation}

Because the bilinear term is supported at a single Fourier mode, we can see that
\begin{equation}
    \mathbb{P}_\mathcal{M}
    ((\Tilde{u}\cdot\nabla)u
    +(u\cdot\nabla)\Tilde{u})
    =
    v^{k^m}\cdot
    ((\Tilde{u}\cdot\nabla)u) v^{k^m}
\end{equation}
Next we compute that
\begin{align}
    v^{k^m}\cdot
    (\Tilde{u}\cdot\nabla)u
    &=
    \frac{2*3^{2m}\sigma
    -2^{2m}*3^{m+1}
        \left(\begin{array}{c}
             1  \\
             0 \\
             -1
        \end{array}\right)}
    {\left(4*3^{4m+1}
    +2^{4m+1}*3^{2m+2}
    \right)^\frac{1}{2}} \\
    \notag &\quad\quad \cdot
    \frac{2^{2m}*3^{2m+1}
    \left(2*3^{2m+1}\sigma
    -2^{2m+1}*3^{m+1}
    \left(\begin{array}{c}
         1 \\ 1 \\ -2
    \end{array}\right)\right)
    2\pi i e^{2\pi i k^m\cdot x}}
    {\left(4*3^{4m+1}
    +2^{4m+1}*3^{2m+2}
    \right)^\frac{1}{2}
    \left(4*3^{4m+3}
    +2^{4m+3}*3^{2m+3}
    \right)^\frac{1}{2}} \\
    &=
    \frac{2^{2m+1}*3^{2m+1}\left(
    4*3^{4m+2}+2^{4m+1}*3^{2m+3}
    \right)}
    {\left(4*3^{4m+1}
    +2^{4m+1}*3^{2m+2}\right)
    \left(4*3^{4m+3}
    +2^{4m+3}*3^{2m+3}
    \right)^\frac{1}{2}} 
    i\pi e^{2\pi i k^m\cdot x} \\
    &=
    \frac{2^{2m+1}*3^{2m+2}}
    {\left(4*3^{4m+3}
    +2^{4m+3}*3^{2m+3}
    \right)^\frac{1}{2}}
    i\pi e^{2\pi i k^m\cdot x}\\
    &=
    \frac{2^{2m+1}*3^{2m+2}}
    {2^{2m+\frac{3}{2}}
    3^{m+\frac{3}{2}}
    \left(1+2^{-4m-1}*3^{2m}
    \right)^\frac{1}{2}}
    i\pi e^{2\pi i k^m\cdot x}\\
    &=
    \frac{2^{-\frac{1}{2}}
    *3^\frac{1}{2}}
    {\left(1+\frac{1}{2}
    \left(\frac{3}{4}\right)^{2m}
    \right)^\frac{1}{2}} 3^m
    i\pi e^{2\pi i k^m\cdot x}\\
    &=
    \frac{i}{2}
    \frac{\sqrt{6}\pi}
    {\left(1+\frac{1}{2}
    \left(\frac{3}{4}\right)^{2m}
    \right)^\frac{1}{2}} 3^m
    e^{2\pi i k^m\cdot x}.
\end{align}
Then we have
\begin{equation}
    \mathbb{P}_{\mathcal{M}}((\Tilde{u}\cdot\nabla) u
    +(u\cdot\nabla)\Tilde{u})
    =
    \frac{a_m}{2}i v^{k^m}e^{2\pi i k^m\cdot x},
\end{equation}
and this completes the proof.
\end{proof}

\begin{proposition} \label{Tedious5}
Fix $m\in\mathbb{Z}^+$, and let $u$ and $\Tilde{u}$ be given by
\begin{align}
    u&=
    iv^{j^m}e^{2\pi i j^m\cdot x} \\
    \Tilde{u}
    &=
    -iP_{23}(v^{k^m})
    e^{-2\pi i P_{23}(k^m)\cdot x}.
    \end{align}
Then the bilinear term in the restricted model is given by
\begin{equation}
    \mathbb{P}_{\mathcal{M}}((\Tilde{u}\cdot\nabla) u
    +(u\cdot\nabla)\Tilde{u})
    =
    \frac{a_m}{2}i v^{k^m}e^{2\pi i k^m\cdot x},
\end{equation}
where 
\begin{equation}
    a_m=
    \frac{\sqrt{6}\pi}
    {\left(1+\frac{1}{2}
    \left( \frac{3}{4}
    \right)^{2m}\right)^\frac{1}{2}}
    3^m.
\end{equation}
\end{proposition}

\begin{proof}
We begin by computing that
\begin{align}
    (\Tilde{u}\cdot\nabla)u
    &=
    2\pi i (j^m\cdot P_{23}(v^{k^m}))
    v^{j^m}e^{2\pi ik^m\cdot x} \\
    (u\cdot\nabla)\Tilde{u}
    &=
    -2\pi i (P_{23}(k^m)\cdot v^{j^m})
    P_{23}(v^{k^m})
    e^{2\pi i k^m\cdot x}.
\end{align}
Observe that
\begin{align}
    j^m\cdot P_{23}(v^{k^m})
    &=
    \left(2^{2m+1}\sigma+3^m\left(
    \begin{array}{cc}
         2 \\ -1 \\ -1
    \end{array}\right)\right)
    \cdot
    \frac{2*3^{2m}\sigma-2^{2m}*3^{m+1}
    \left(\begin{array}{c}
         1 \\ -1 \\ 0 
    \end{array}\right)}
    {\left(4*3^{4m+1}
    +2^{4m+1}*3^{2m+2}
    \right)^\frac{1}{2}} \\
    &=
    \frac{2^{2m+2}*3^{2m+1}
    -2^{2m}*3^{2m+2}}
    {\left(4*3^{4m+1}
    +2^{4m+1}*3^{2m+2}
    \right)^\frac{1}{2}} \\
    &=
    \frac{2^{2m}*3^{2m+1}}
    {\left(4*3^{4m+1}
    +2^{4m+1}*3^{2m+2}
    \right)^\frac{1}{2}},
\end{align}
and conclude that
\begin{equation}
    (j^m\cdot P_{23}(v^{k^m}))v^{j^m}
    =
    \frac{2^{2m}*3^{2m+1}
    \left(2*3^{2m+1}\sigma
    -2^{2m+1}*3^{m+1}
    \left(\begin{array}{c}
         2 \\ -1 \\ -1
    \end{array}\right)\right)}
    {\left(4*3^{4m+1}
    +2^{4m+1}*3^{2m+2}
    \right)^\frac{1}{2}
    \left(4*3^{4m+3}
    +2^{4m+3}*3^{2m+3}
    \right)^\frac{1}{2}},
\end{equation}
and therefore that
\begin{equation}
    (\Tilde{u}\cdot\nabla)u
    =
    2\pi i\frac{2^{2m}*3^{2m+1}
    \left(2*3^{2m+1}\sigma
    -2^{2m+1}*3^{m+1}
    \left(\begin{array}{c}
         1 \\ 1 \\ -2
    \end{array}\right)\right)}
    {\left(4*3^{3m+1}
    +2^{4m+1}*3^{2m+2}
    \right)^\frac{1}{2}
    \left(4*3^{4m+3}
    +2^{4m+3}*3^{2m+3}
    \right)^\frac{1}{2}}
    e^{2\pi i k^m\cdot x}.
\end{equation}
Likewise, observe that
\begin{align}
    P_{23}(k^m)\cdot v^{j^m}
    &=
    \left(2^{2m}\sigma+3^m
    \left(\begin{array}{c}
         1 \\ -1 \\ 0
    \end{array}\right)\right)
    \cdot
    \frac{2*3^{2m+1}\sigma
    -2^{2m+1}*3^{m+1}
    \left(
    \begin{array}{c}
         2  \\
         -1 \\
         -1
    \end{array}
    \right)}
    {\left(4*3^{4m+3}
    +2^{4m+3}*3^{2m+3}
    \right)^\frac{1}{2}} \\
    &=
    \frac{2^{2m+1}*3^{2m+2}
    -2^{2m+1}*3^{2m+2}}
    {\left(4*3^{4m+3}
    +2^{4m+3}*3^{2m+3}
    \right)^\frac{1}{2}} \\
    &=
    0,
\end{align}
and conclude that
\begin{equation}
    (u\cdot\nabla)\Tilde{u}=0.
\end{equation}

Because the bilinear term is supported at a single Fourier mode, we can see that
\begin{equation}
    \mathbb{P}_\mathcal{M}
    ((\Tilde{u}\cdot\nabla)u
    +(u\cdot\nabla)\Tilde{u})
    =
    v^{k^m}\cdot
    ((\Tilde{u}\cdot\nabla)u) v^{k^m}
\end{equation}
Next we compute that
\begin{align}
    v^{k^m}\cdot
    (\Tilde{u}\cdot\nabla)u
    &=
    \frac{2*3^{2m}\sigma
    -2^{2m}*3^{m+1}
        \left(\begin{array}{c}
             1  \\
             0 \\
             -1
        \end{array}\right)}
    {\left(4*3^{4m+1}
    +2^{4m+1}*3^{2m+2}
    \right)^\frac{1}{2}} \\
    \notag &\quad\quad \cdot
    \frac{2^{2m}*3^{2m+1}
    \left(2*3^{2m+1}\sigma
    -2^{2m+1}*3^{m+1}
    \left(\begin{array}{c}
         2 \\ -1 \\ -1
    \end{array}\right)\right)
    2\pi i e^{2\pi i k^m\cdot x}}
    {\left(4*3^{4m+1}
    +2^{4m+1}*3^{2m+2}
    \right)^\frac{1}{2}
    \left(4*3^{4m+3}
    +2^{4m+3}*3^{2m+3}
    \right)^\frac{1}{2}} \\
    &=
    \frac{2^{2m+1}*3^{2m+1}\left(
    4*3^{4m+2}+2^{4m+1}*3^{2m+3}
    \right)}
    {\left(4*3^{4m+1}
    +2^{4m+1}*3^{2m+2}\right)
    \left(4*3^{4m+3}
    +2^{4m+3}*3^{2m+3}
    \right)^\frac{1}{2}} 
    i\pi e^{2\pi i k^m\cdot x} \\
    &=
    \frac{2^{2m+1}*3^{2m+2}}
    {\left(4*3^{4m+3}
    +2^{4m+3}*3^{2m+3}
    \right)^\frac{1}{2}}
    i\pi e^{2\pi i k^m\cdot x}\\
    &=
    \frac{2^{2m+1}*3^{2m+2}}
    {2^{2m+\frac{3}{2}}
    3^{m+\frac{3}{2}}
    \left(1+2^{-4m-1}*3^{2m}
    \right)^\frac{1}{2}}
    i\pi e^{2\pi i k^m\cdot x}\\
    &=
    \frac{2^{-\frac{1}{2}}
    *3^\frac{1}{2}}
    {\left(1+\frac{1}{2}
    \left(\frac{3}{4}\right)^{2m}
    \right)^\frac{1}{2}} 3^m
    i\pi e^{2\pi i k^m\cdot x}\\
    &=
    \frac{i}{2}
    \frac{\sqrt{6}\pi}
    {\left(1+\frac{1}{2}
    \left(\frac{3}{4}\right)^{2m}
    \right)^\frac{1}{2}} 3^m
    e^{2\pi i k^m\cdot x}.
\end{align}
Then we have
\begin{equation}
    \mathbb{P}_{\mathcal{M}}((\Tilde{u}\cdot\nabla) u
    +(u\cdot\nabla)\Tilde{u})
    =
    \frac{a_m}{2}i v^{k^m}e^{2\pi i k^m\cdot x},
\end{equation}
and this completes the proof.
\end{proof}

\begin{proposition} \label{Tedious6}
Fix $m\in\mathbb{Z}^+$, and let $u$ and $\Tilde{u}$ be given by
\begin{align}
    u&=
    iv^{k^{m+1}}e^{2\pi i k^{m+1}\cdot x} \\
    \Tilde{u}
    &=
    -iv^{j^m}
    e^{-2\pi i j^m\cdot x}.
    \end{align}
Then the bilinear term in the restricted model is given by
\begin{equation}
    \mathbb{P}_{\mathcal{M}}((\Tilde{u}\cdot\nabla) u
    +(u\cdot\nabla)\Tilde{u})
    =
    \frac{b_m}{2} i v^{h^m}e^{2\pi i h^m\cdot x},
\end{equation}
where 
\begin{equation}
    b_m=
    \frac{\sqrt{2}\pi}{\left(1+\frac{3}{8}
    \left(\frac{3}{4}\right)^{2m}
    \right)^\frac{1}{2}} 3^{m+1}.
\end{equation}
\end{proposition}

\begin{proof}
        First observe that
    \begin{align}
        (\Tilde{u}\cdot\nabla)u
        &=
        2\pi i (k^{m+1}\cdot v^{j^m})
        v^{k^{m+1}}e^{2\pi i h^{m}\cdot x} \\
        (u\cdot\nabla)\Tilde{u}
        &=
        -2\pi i (j^m\cdot v^{k^{m+1}})
        v^{j^m} e^{2\pi i h^{m}\cdot x}.
    \end{align}
Then we compute that
\begin{align}
    k^{m+1}\cdot v^{j^m}
    &=
    \left(2^{2m+2}\sigma+3^{m+1}
    \left(\begin{array}{c}
         1 \\ 0 \\ -1
    \end{array}\right)\right)
    \cdot
    \frac{2*3^{2m+1}\sigma
    -2^{2m+1}*3^{m+1}
    \left(
    \begin{array}{c}
         2  \\
         -1 \\
         -1
    \end{array}
    \right)}
    {\left(4*3^{4m+3}
    +2^{4m+3}*3^{2m+3}
    \right)^\frac{1}{2}} \\
    &=
    \frac{2^{2m+3}*3^{2m+2}
    -2^{2m+1}*3^{2m+3}}
    {\left(4*3^{4m+3}
    +2^{4m+3}*3^{2m+3}
    \right)^\frac{1}{2}} \\
    &=
    \frac{2^{2m+1}*3^{2m+2}}
    {\left(4*3^{4m+3}
    +2^{4m+3}*3^{2m+3}
    \right)^\frac{1}{2}},
\end{align}
and that
\begin{align}
    j^{m}\cdot v^{k^{m+1}}
    &=
    \left(2^{2m+1}\sigma
    +3^m\left(\begin{array}{c}
        2 \\ -1 \\ -1
    \end{array}\right)\right)
    \cdot 
    \frac{2*3^{2m+2}\sigma
    -2^{2m+2}*3^{m+2}
        \left(\begin{array}{c}
             1  \\
             0 \\
             -1
        \end{array}\right)}
    {\left(4*3^{4m+5}
    +2^{4m+5}*3^{2m+4}
    \right)^\frac{1}{2}} \\
    &=
    \frac{2^{2m+2}*3^{2m+3}
    -2^{2m+2}*3^{2m+3}}
    {\left(4*3^{4m+5}
    +2^{4m+5}*3^{2m+4}
    \right)^\frac{1}{2}} \\
    &=0.
\end{align}
Therefore, we may conclude that
\begin{equation}
    (u\cdot\nabla)\Tilde{u}=0,
\end{equation}
and that
\begin{equation}
    (\Tilde{u}\cdot\nabla)u
    =
    2 \pi i
    \frac{2^{2m+1}*3^{2m+2}
    \left(2*3^{2m+2}\sigma
    -2^{2m+2}*3^{m+2}
        \left(\begin{array}{c}
             1  \\
             0 \\
             -1
        \end{array}\right)\right)}
    {\left(4*3^{4m+3}
    +2^{4m+3}*3^{2m+3}
    \right)^\frac{1}{2}
    \left(4*3^{4m+5}
    +2^{4m+5}*3^{2m+4}
    \right)^\frac{1}{2}}
    e^{2\pi ih^m\cdot x}.
\end{equation}
Observing that this is a single Fourier mode we can see that
\begin{equation}   
    \mathbb{P}_\mathcal{M}
    ((\Tilde{u}\cdot\nabla)u)
    =
    v^{h^m}\cdot ((\Tilde{u}\cdot\nabla)u)v^{h^m}.
\end{equation}
Now we can compute that
\begin{align}
    v^{h^m}\cdot(\Tilde{u}\cdot\nabla)u)
    &=
    \pi i e^{2\pi i h^m\cdot x}
    \frac{2^{2m+2}*3^{2m+2}
    \left(2*3^{2m+2}\sigma
    -2^{2m+2}*3^{m+2}
        \left(\begin{array}{c}
             1  \\
             0 \\
             -1
        \end{array}\right)\right)}
    {\left(4*3^{4m+3}
    +2^{4m+3}*3^{2m+3}
    \right)^\frac{1}{2}
    \left(4*3^{4m+5}
    +2^{4m+5}*3^{2m+4}
    \right)^\frac{1}{2}} \\
    \notag &\quad\quad \cdot
    \frac{1}{\left(4*3^{4m+3}
    +2^{4m+3}*3^{2m+3}
    \right)^\frac{1}{2}}
        \left(
    2*3^{2m+1}\sigma
    -2^{2m+1}*3^{m+1}
    \left(
    \begin{array}{c}
         1  \\
         1 \\
         -2
    \end{array}
    \right)
    \right) \\
    &=
    \pi i e^{2\pi i h^m\cdot x}
    \frac{2^{2m+2}*3^{2m+2}
    \left(4*3^{4m+4}+2^{4m+3}*3^{2m+4}
    \right)}
    {{\left(4*3^{4m+3}
    +2^{4m+3}*3^{2m+3}\right)
    \left(4*3^{4m+5}
    +2^{4m+5}*3^{2m+4}
    \right)^\frac{1}{2}}} \\
    &=
    \pi i e^{2\pi i h^m\cdot x}
    \frac{2^{2m+2}*3^{2m+3}}
    {\left(4*3^{4m+5}
    +2^{4m+5}*3^{2m+4}
    \right)^\frac{1}{2}}\\
    &=
    \pi i e^{2\pi i h^m\cdot x}
    \frac{2^{2m+2}*3^{2m+3}}
    {2^{2m+\frac{5}{2}}*3^{m+2}
    \left(1+2^{-4m-3}*3^{2m+1}
    \right)^\frac{1}{2}} \\
    &=
    \frac{i}{2}
    \frac{\sqrt{2}\pi}{\left(1+\frac{3}{8}
    \left(\frac{3}{4}\right)^{2m}
    \right)^\frac{1}{2}} 3^{m+1}
    e^{2\pi i h^m\cdot x}.
\end{align}
This implies that
\begin{equation}
    \mathbb{P}_{\mathcal{M}}((\Tilde{u}\cdot\nabla) u)
    =
    \frac{b_m}{2} i 
    v^{h^m}e^{2\pi i h^m\cdot x},
\end{equation}
and this completes the proof.
\end{proof}

\begin{proposition} \label{Tedious7}
Fix $m\in\mathbb{Z}^+$, and let $u$ and $\Tilde{u}$ be given by
\begin{align}
    u&=
    iP_{12}(v^{k^{m+1}})e^{2\pi i 
    P_{12}(k^{m+1})\cdot x} \\
    \Tilde{u}
    &=
    -iP_{12}(v^{j^m})
    e^{-2\pi i P_{12}(j^m)\cdot x}.
    \end{align}
Then the bilinear term in the restricted model is given by
\begin{equation}
    \mathbb{P}_{\mathcal{M}}((\Tilde{u}\cdot\nabla) u
    +(u\cdot\nabla)\Tilde{u})
    =
    \frac{b_m}{2} i v^{h^m}e^{2\pi i h^m\cdot x},
\end{equation}
where 
\begin{equation}
    b_m=
    \frac{\sqrt{2}\pi}{\left(1+\frac{3}{8}
    \left(\frac{3}{4}\right)^{2m}
    \right)^\frac{1}{2}} 3^{m+1}.
\end{equation}
\end{proposition}

\begin{proof}
    The proof is essentially identical to the proof of \Cref{Tedious6}, with the vector $(1,0,-1)$ replaced with $(0,1,-1)$ and the vector $(2,-1,-1)$ replaced with
    $(-1,2,-1).$ It is left to the reader to check that this change does not impact any of the dot products taken in the proof above.
\end{proof}

\begin{proposition} \label{Tedious8}
Fix $m\in\mathbb{Z}^+$, and let $u$ and $\Tilde{u}$ be given by
\begin{align}
    u&=
    iv^{k^{m+1}}e^{2\pi i k^{m+1}\cdot x} \\
    \Tilde{u}
    &=
    -iv^{h^m}
    e^{-2\pi i h^m\cdot x}.
    \end{align}
Then the bilinear term in the restricted model is given by
\begin{equation}
    \mathbb{P}_{\mathcal{M}}((\Tilde{u}\cdot\nabla) u
    +(u\cdot\nabla)\Tilde{u})
    =
    \frac{b_m}{2} i v^{j^m}e^{2\pi i j^m\cdot x},
\end{equation}
where 
\begin{equation}
    b_m=
    \frac{\sqrt{2}\pi}{\left(1+\frac{3}{8}
    \left(\frac{3}{4}\right)^{2m}
    \right)^\frac{1}{2}} 3^{m+1}.
\end{equation}
\end{proposition}

\begin{proof}
        First observe that
    \begin{align}
        (\Tilde{u}\cdot\nabla)u
        &=
        2\pi i (k^{m+1}\cdot v^{h^m})
        v^{k^{m+1}}e^{2\pi i j^{m}\cdot x} \\
        (u\cdot\nabla)\Tilde{u}
        &=
        -2\pi i (h^m\cdot v^{k^{m+1}})
        v^{h^m} e^{2\pi i j^{m}\cdot x}.
    \end{align}
Then we compute that
\begin{align}
    k^{m+1}\cdot v^{h^m}
    &=
    \left(2^{2m+2}\sigma+3^{m+1}
    \left(\begin{array}{c}
         1 \\ 0 \\ -1
    \end{array}\right)\right)
    \cdot
    \frac{2*3^{2m+1}\sigma
    -2^{2m+1}*3^{m+1}
    \left(
    \begin{array}{c}
         1  \\
         1 \\
         -2
    \end{array}
    \right)}
    {\left(4*3^{4m+3}
    +2^{4m+3}*3^{2m+3}
    \right)^\frac{1}{2}} \\
    &=
    \frac{2^{2m+3}*3^{2m+2}
    -2^{2m+1}*3^{2m+3}}
    {\left(4*3^{4m+3}
    +2^{4m+3}*3^{2m+3}
    \right)^\frac{1}{2}} \\
    &=
    \frac{2^{2m+1}*3^{2m+2}}
    {\left(4*3^{4m+3}
    +2^{4m+3}*3^{2m+3}
    \right)^\frac{1}{2}},
\end{align}
and that
\begin{align}
    h^{m}\cdot v^{k^{m+1}}
    &=
    \left(2^{2m+1}\sigma
    +3^m\left(\begin{array}{c}
        1 \\ 1 \\ -2
    \end{array}\right)\right)
    \cdot 
    \frac{2*3^{2m+2}\sigma
    -2^{2m+2}*3^{m+2}
        \left(\begin{array}{c}
             1  \\
             0 \\
             -1
        \end{array}\right)}
    {\left(4*3^{4m+5}
    +2^{4m+5}*3^{2m+4}
    \right)^\frac{1}{2}} \\
    &=
    \frac{2^{2m+2}*3^{2m+3}
    -2^{2m+2}*3^{2m+3}}
    {\left(4*3^{4m+5}
    +2^{4m+5}*3^{2m+4}
    \right)^\frac{1}{2}} \\
    &=0.
\end{align}
Therefore, we may conclude that
\begin{equation}
    (u\cdot\nabla)\Tilde{u}=0,
\end{equation}
and that
\begin{equation}
    (\Tilde{u}\cdot\nabla)u
    =
    2 \pi i
    \frac{2^{2m+1}*3^{2m+2}
    \left(2*3^{2m+2}\sigma
    -2^{2m+2}*3^{m+2}
        \left(\begin{array}{c}
             1  \\
             0 \\
             -1
        \end{array}\right)\right)}
    {\left(4*3^{4m+3}
    +2^{4m+3}*3^{2m+3}
    \right)^\frac{1}{2}
    \left(4*3^{4m+5}
    +2^{4m+5}*3^{2m+4}
    \right)^\frac{1}{2}}
    e^{2\pi i j^m\cdot x}.
\end{equation}
Observing that this is a single Fourier mode we can see that
\begin{equation}   
    \mathbb{P}_\mathcal{M}
    ((\Tilde{u}\cdot\nabla)u)
    =
    v^{j^m}\cdot ((\Tilde{u}\cdot\nabla)u)v^{j^m}.
\end{equation}
Now we can compute that
\begin{align}
    v^{j^m}\cdot(\Tilde{u}\cdot\nabla)u)
    &=
    \pi i e^{2\pi i j^m\cdot x}
    \frac{2^{2m+2}*3^{2m+2}
    \left(2*3^{2m+2}\sigma
    -2^{2m+2}*3^{m+2}
        \left(\begin{array}{c}
             1  \\
             0 \\
             -1
        \end{array}\right)\right)}
    {\left(4*3^{4m+3}
    +2^{4m+3}*3^{2m+3}
    \right)^\frac{1}{2}
    \left(4*3^{4m+5}
    +2^{4m+5}*3^{2m+4}
    \right)^\frac{1}{2}} \\
    \notag &\quad\quad \cdot
    \frac{1}{\left(4*3^{4m+3}
    +2^{4m+3}*3^{2m+3}
    \right)^\frac{1}{2}}
        \left(
    2*3^{2m+1}\sigma
    -2^{2m+1}*3^{m+1}
    \left(
    \begin{array}{c}
         2  \\
         -1 \\
         -1
    \end{array}
    \right)
    \right) \\
    &=
    \pi i e^{2\pi i j^m\cdot x}
    \frac{2^{2m+2}*3^{2m+2}
    \left(4*3^{4m+4}+2^{4m+3}*3^{2m+4}
    \right)}
    {{\left(4*3^{4m+3}
    +2^{4m+3}*3^{2m+3}\right)
    \left(4*3^{4m+5}
    +2^{4m+5}*3^{2m+4}
    \right)^\frac{1}{2}}} \\
    &=
    \pi i e^{2\pi i j^m\cdot x}
    \frac{2^{2m+2}*3^{2m+3}}
    {\left(4*3^{4m+5}
    +2^{4m+5}*3^{2m+4}
    \right)^\frac{1}{2}}\\
    &=
    \pi i e^{2\pi i j^m\cdot x}
    \frac{2^{2m+2}*3^{2m+3}}
    {2^{2m+\frac{5}{2}}*3^{m+2}
    \left(1+2^{-4m-3}*3^{2m+1}
    \right)^\frac{1}{2}} \\
    &=
    \frac{i}{2}
    \frac{\sqrt{2}\pi}{\left(1+\frac{3}{8}
    \left(\frac{3}{4}\right)^{2m}
    \right)^\frac{1}{2}} 3^{m+1}
    e^{2\pi i j^m\cdot x}.
\end{align}
This implies that
\begin{equation}
    \mathbb{P}_{\mathcal{M}}((\Tilde{u}\cdot\nabla) u)
    =
    \frac{b_m}{2} i 
    v^{j^m}e^{2\pi i j^m\cdot x},
\end{equation}
and this completes the proof.
\end{proof}

\begin{proposition} \label{Tedious9}
Fix $m\in\mathbb{Z}^+$, and let $u$ and $\Tilde{u}$ be given by
\begin{align}
    u&=
    iv^{P_{23}(k^{m+1})}
    e^{2\pi i P_{23}(k^{m+1})\cdot x} \\
    \Tilde{u}
    &=
    -iP_{23}(v^{h^m})
    e^{-2\pi i P_{23}(h^m)\cdot x}.
    \end{align}
Then the bilinear term in the restricted model is given by
\begin{equation}
    \mathbb{P}_{\mathcal{M}}((\Tilde{u}\cdot\nabla) u
    +(u\cdot\nabla)\Tilde{u})
    =
    \frac{b_m}{2} i v^{j^m}e^{2\pi i j^m\cdot x},
\end{equation}
where 
\begin{equation}
    b_m=
    \frac{\sqrt{2}\pi}{\left(1+\frac{3}{8}
    \left(\frac{3}{4}\right)^{2m}
    \right)^\frac{1}{2}} 3^{m+1}.
\end{equation}
\end{proposition}

\begin{proof}
    The proof is essentially identical to the proof of \Cref{Tedious8}, with the vector $(1,0,-1)$ replaced with $(1,-1,0)$ and the vector $(1,1,-2)$ replaced with
    $(1,-2,1).$ It is left to the reader to check that this change does not impact any of the dot products taken in the proof above.
\end{proof}

\section{Bilinear operator} \label{AppendixNonlinearBound}

The purpose of this appendix is to prove \Cref{BilinearBoundLemma}, which was left to this appendix because it is rather tedious and essentially technical.

\begin{definition}
    For all $u\in \dot{H}^s_\mathcal{M}$, 
    and for all $k\in \mathcal{M}$, let
    \begin{equation}
    u^k=\hat{u}(k)e^{2\pi i k\cdot x}.
    \end{equation}
    This notation will be convenient for proving bounds term by term in the Fourier series expression
    \begin{equation}
    u=\sum_{k\in\mathcal{M}} u^k.
    \end{equation}
\end{definition}

\begin{lemma} \label{L2Sobolev}
    For all $u\in \dot{H}^s_{\mathcal{M}}$, and for all $k\in\mathcal{M}$,
    \begin{equation}
    \left\|u^k\right\|_{\dot{H}^s}
    =
    (2\pi |k|)^s
    \left\|u^k\right\|_{L^2}
    \end{equation}
\end{lemma}

\begin{proof}
    This follows immediately from the definitions of $u^k$ and the $\dot{H}^s$ norm.
\end{proof}

\begin{definition}
    For all $u,w\in \dot{H}^s_{\mathcal{M}},$ define the bilinear maps 
    $B^{+,+},B^{+,-},B^{-,+},B^{-,-}$ by
    \begin{align}
    B^{+,+}(u,w)
    &=
    \sum_{j,k\in\mathcal{M}^+}
    B\left(u^j,w^k\right) \\
    B^{+,-}(u,w)
    &=
    \sum_{\substack{j\in\mathcal{M}^+ \\
    k\in\mathcal{M}^-}}
    B\left(u^j,w^k\right) \\
    B^{-,+}(u,w)
    &=
    \sum_{\substack{j\in\mathcal{M}^- \\
    k\in\mathcal{M}^+}}
    B\left(u^j,w^k\right) \\
    B^{-,-}(u,w)
    &=
    \sum_{j,k\in\mathcal{M}^-}
    B\left(u^j,w^k\right)
    \end{align}
\end{definition}

\begin{remark}
    Recall that we have defined $B$ as
    \begin{equation}
    B(u,w)=-\frac{1}{2}\mathbb{P}_\mathcal{M}
    ((u\cdot\nabla)w+(w\cdot\nabla)u).
    \end{equation}
    It is simple to see that 
    \begin{equation}
    B(u,w)=\sum_{j,k\in\mathcal{M}}
    B\left(u^j,w^k\right).
    \end{equation}
    It then can be seen by splitting $\mathcal{M}$ into positive and negative frequencies that,
    \begin{equation}
    B=B^{+,+}+B^{+,-}+B^{-,+}+B^{-,-}.
    \end{equation}
\end{remark}

\begin{proposition} \label{DiagonalizationProp}
    For all $u,w\in \dot{H}^s_{\mathcal{M}}$,
    \begin{align}
    B^{+,+}(u,w) \label{++}
    &=
    \sum_{n=0}^{+\infty}
    \sum_{j,k\in\mathcal{M}^+_n}
    B\left(u^j,w^k\right) \\
    B^{+,-}(u,w) \label{+-}
    &=
    \sum_{n=0}^{+\infty}
    \sum_{j\in\mathcal{M}_n^+}\sum_{k\in \mathcal{M}_{n-1}^-\cup\mathcal{M}_{n+1}^-}
    B\left(u^j,w^k\right) \\
    B^{-,+}(u,w) \label{-+}
    &=
    \sum_{n=0}^{+\infty} 
    \sum_{j\in\mathcal{M}_n^-}\sum_{k\in \mathcal{M}_{n-1}^+
    \cup\mathcal{M}_{n+1}^+}
    B\left(u^j,w^k\right) \\
    B^{-,-}(u,w)  \label{--}
    &=
    \sum_{n=0}^{+\infty} 
    \sum_{j,k\in\mathcal{M}^-_n}
    B\left(u^j,w^k\right)
    \end{align}
\end{proposition}

\begin{proof}
    We can see from \Cref{DyadicPowers} that if $h,j\in\mathcal{M}^+$ and $h+j\in\mathcal{M}$, then the frequencies are in the same shell $h,j\in\mathcal{M}_n^+$ for some $n\in\mathbb{Z}^+$.
    If $k+j\notin \mathcal{M}$, then
    \begin{equation}
    B\left(u^j,w^k\right)=0,
    \end{equation}
    and it therefore follows that
    \begin{align}
    B^{+,+}(u,w)
    &=
    \sum_{j,k\in\mathcal{M}^+}
    B\left(u^j,w^k\right) \\
    &=
    \sum_{n=0}^{+\infty}
    \sum_{j,k\in\mathcal{M}^+_n}
    B\left(u^j,w^k\right).
    \end{align}
    The proof for \eqref{--} is exactly analogous, but with negative signs.
    
    We can likewise see from \Cref{DyadicPowers} that if $h\in\mathcal{M}_n^+, j\in\mathcal{M}^-$, and $h+j\in\mathcal{M}$, then either 
    $j\in\mathcal{M}_{n-1}^-$ or 
    $j\in\mathcal{M}_{n-1}^-$, and therefore
\begin{equation}
    B^{+,-}(u,w) 
    =
    \sum_{n=0}^{+\infty}
    \sum_{j\in\mathcal{M}_n^+}\sum_{k\in \mathcal{M}_{n-1}^-\cup\mathcal{M}_{n+1}^-}
    B\left(u^j,w^k\right) 
\end{equation}    
    Note that $B^{+,-}$ and $B^{-,+}$ are not each symmetric, but that $B^{+,-}+B^{-,+}$ is symmetric, because
    \begin{equation}
    B^{-,+}(u,w)=B^{+,-}(w,u),
    \end{equation}
    and therefore \eqref{-+} follows from \eqref{+-},
    completing the proof.
    \end{proof}

    \begin{lemma} \label{LemmaCascadeUpBound}
     For all $u,w\in \dot{H}^{\frac{s}{2}+
     \frac{3}{4\log(2)}}_\mathcal{M}$, and for all $j,k\in\mathcal{M}_n^+$,
     \begin{equation}
    \left\|B\left(u^j,w^k\right)
    \right\|_{\dot{H}^s}
    \leq C
    \left\|u^j\right\|_{\dot{H}^{\frac{s}{2}+\frac{\log(3)}{4\log(2)}}}
    \left\|w^k\right\|_{\dot{H}^{\frac{s}{2}+\frac{\log(3)}{4\log(2)}}}
     \end{equation}
     where $C>0$ is an absolute constant independent of $u,w, j$ and $k$.
    \end{lemma}

\begin{proof}
We begin by observing that 
\begin{equation}
    \left\|B\left(u^j,w^k\right)
    \right\|_{L^2}
    \leq 
    C\left(\sqrt{3}\right)^n
    \left\|u^j\right\|_{L^2}
    \left\|w^k\right\|_{L^2}.
\end{equation}
This follows from \Cref{Tedious1,Tedious2} if $n$ is even, and from \Cref{Tedious3} if $n$ is odd. All of the other interactions of two frequencies in $\mathcal{M}_n^+$ can be written as permutations of these cases, and hence have the same bounds.
Applying \Cref{L2Sobolev}, we then find that
\begin{align}
    \left\|B\left(u^j,w^k\right)
    \right\|_{\dot{H}^s}
    &=
    (2\pi)^s |j+k|^s 
    \left\|B\left(u^j,w^k\right)
    \right\|_{L^2} \\
    &\leq
    C \left(\sqrt{3}\right)^n
    \frac{|j+k|^s}{|k|^{\frac{s}{2}
    +\frac{\log(3)}{4\log(2)}}
    |j|^{\frac{s}{2}
    +\frac{\log(3)}{4\log(2)}}}
    \left\|u^j\right\|_{\dot{H}^{\frac{s}{2}+\frac{\log(3)}{4\log(2)}}}
    \left\|w^k\right\|_{\dot{H}^{\frac{s}{2}+\frac{\log(3)}{4\log(2)}}}.
\end{align}
Using the fact that $j+k\in\mathcal{M}_{n+1}^+$, and recalling from \Cref{CanonicalComputations} that
\begin{align}
    |j|, |k|&=\sqrt{3*2^{2n}+2*3^n} \\
    |j+k| &= \sqrt{3*2^{2n+2}+2*3^{n+1}},
\end{align}
we compute that
\begin{align}
\frac{|j+k|^s}{|k|^{\frac{s}{2}
    +\frac{3}{4\log(2)}}
    |j|^{\frac{s}{2}
    +\frac{\log(3)}{4\log(2)}}}
    &=
    \frac{\left(3*2^{2n+2}+2*3^{n+1}
    \right)^{\frac{s}{2}}}
    {\left(3*2^{2n}+2*3^n
    \right)^{\frac{s}{2}+\frac{\log(3)}{4\log(2)}}} \\
    &=
    \frac{\left(18+6\left(\frac{3}{4}\right)^n
    \right)^\frac{s}{2}}
    {\left(3+2*\left(\frac{3}{4}\right)^n
    \right)^{\frac{s}{2}+\frac{\log(3)}{4\log(2)}}}
    2^{-\frac{\log(3)}{2\log(2)}n} \\
    &=
    \frac{\left(18+6\left(\frac{3}{4}\right)^n
    \right)^\frac{s}{2}}
    {\left(3+2*\left(\frac{3}{4}\right)^n
    \right)^{\frac{s}{2}+\frac{\log(3)}{4\log(2)}}}
    \left(\sqrt{3}\right)^{-n} \\
    &\leq 
    \frac{C}{\left(\sqrt{3}\right)^n},
\end{align}
and this completes the proof.
\end{proof}

\begin{proposition} \label{Bound++Prop}
     For all $u,w\in \dot{H}^{\frac{s}{2}+
     \frac{3}{4\log(2)}}_\mathcal{M}$,
    \begin{equation}
    \left\|B^{+,+}(u,w)\right\|_{\dot{H}^s}
    \leq C
    \|u\|_{\dot{H}^{\frac{s}
    {2}+\frac{\log(3)}{4\log(2)}}}\|w\|_{\dot{H}^{\frac{s}
    {2}+\frac{\log(3)}{4\log(2)}}},
    \end{equation}
    where $C>0$ is an absolute constant independent of $u$ and $w$.
\end{proposition}

\begin{proof}
    Applying \Cref{DiagonalizationProp} and \Cref{LemmaCascadeUpBound}, we find that
    \begin{align}
    \left\|B^{+,+}(u,w)\right\|_{\dot{H}^s}
    &\leq
    \sum_{n=0}^{+\infty}
    \sum_{j,k\in\mathcal{M}^+_n}
    \left\|B\left(u^j,w^k\right)
    \right\|_{\dot{H}^s} \\
    &\leq 
    C \sum_{n=0}^{+\infty}
    \sum_{j,k\in\mathcal{M}^+_n}
    \left\|u^j\right\|_{\dot{H}^{\frac{s}{2}+\frac{\log(3)}{4\log(2)}}}
    \left\|w^k\right\|_{\dot{H}^{\frac{s}{2}+\frac{\log(3)}{4\log(2)}}} \\
    &\leq 
    C \sum_{n=0}^{+\infty}
    \left(\sum_{j\in\mathcal{M}^+_n}
    \left\|u^j\right\|_{\dot{H}^{\frac{s}{2}+\frac{\log(3)}{4\log(2)}}}^2
    \right)^\frac{1}{2}
    \left(\sum_{k\in\mathcal{M}^+_n}
    \left\|w^k\right\|_{\dot{H}^{\frac{s}{2}+\frac{\log(3)}{4\log(2)}}}^2
    \right)^\frac{1}{2}.
    \end{align}
    Note that in this last step we have used that $L^1$ and $L^2$ norms (and in fact any norm), are equivalent on a finite set, and that each shell $\mathcal{M}_n^+$ has six elements. This is where the dyadic structure is used, as this bound would not be available if there were an unbounded number of interactions for each shell.
    Now we apply H\"older's inequality to the sequences
    \begin{equation*}
    \left(\sum_{j\in\mathcal{M}^+_n}
    \left\|u^j\right\|_{\dot{H}^{\frac{s}{2}+\frac{\log(3)}{4\log(2)}}}^2
    \right)^\frac{1}{2}_{n\in\mathbb{Z}^+}
    \text{and}
    \left(\sum_{j,k\in\mathcal{M}^+_n}
    \left\|w^k\right\|_{\dot{H}^{\frac{s}{2}+\frac{\log(3)}{4\log(2)}}}^2
    \right)^\frac{1}{2}_{n\in\mathbb{Z}^+},
    \end{equation*}
    and concude that
    \begin{align}
    \left\|B^{+,+}(u,w)\right\|_{\dot{H}^s}
    &\leq
    C \left(\sum_{n=0}^{+\infty}
    \sum_{j\in\mathcal{M}^+_n}
    \left\|u^j\right\|_{\dot{H}^{\frac{s}{2}+\frac{\log(3)}{4\log(2)}}}^2
    \right)^\frac{1}{2}
    \left(\sum_{n=0}^{+\infty}
    \sum_{k\in\mathcal{M}^+_n}
    \left\|w^k\right\|_{\dot{H}^{\frac{s}{2}+\frac{\log(3)}{4\log(2)}}}^2
    \right)^\frac{1}{2} \\
    &=
    C \|u\|_{\dot{H}^{\frac{s}
    {2}+\frac{\log(3)}{4\log(2)}}}
    \|w\|_{\dot{H}^{\frac{s}
    {2}+\frac{\log(3)}{4\log(2)}}},
    \end{align}
    which completes the proof.
\end{proof}

\begin{lemma} \label{LemmaCascadeOutBound}
    For all $u,w\in \dot{H}^{\frac{s}{2}+
     \frac{3}{4\log(2)}}_\mathcal{M}$, and for all $j\in\mathcal{M}_{n+1}^+, k\in\mathcal{M}_n^-$,
     \begin{equation}
    \left\|B\left(u^j,w^k\right)
    \right\|_{\dot{H}^s}
    \leq C
    \left\|u^j\right\|_{\dot{H}^{\frac{s}{2}+\frac{\log(3)}{4\log(2)}}}
    \left\|w^k\right\|_{\dot{H}^{\frac{s}{2}+\frac{\log(3)}{4\log(2)}}}
     \end{equation}
     where $C>0$ is an absolute constant independent of $u,w,j$ and $k$.
\end{lemma}

\begin{proof}
    The proof will follow largely similar arguments to \Cref{LemmaCascadeUpBound}.
    Begin by observing that 
\begin{equation}
    \left\|B\left(u^j,w^k\right)
    \right\|_{L^2}
    \leq 
    C\left(\sqrt{3}\right)^n
    \left\|u^j\right\|_{L^2}
    \left\|w^k\right\|_{L^2}.
\end{equation}
This follows from \Cref{Tedious4,Tedious5} if $n$ is even, and from 
\Cref{Tedious6,Tedious7,Tedious8,Tedious9}
if $n$ is odd. 
All of the other interactions between frequencies in $\mathcal{M}_{n+1}^+$ and $\mathcal{M}_n^-$ can be written as permutations of these cases, and hence have the same bounds.
Applying \Cref{L2Sobolev}, we then find that
\begin{align}
    \left\|B\left(u^j,w^k\right)
    \right\|_{\dot{H}^s}
    &=
    (2\pi)^s |j+k|^s 
    \left\|B\left(u^j,w^k\right)
    \right\|_{L^2} \\
    &\leq
    C \left(\sqrt{3}\right)^n
    \frac{|j+k|^s}{|k|^{\frac{s}{2}
    +\frac{\log(3)}{4\log(2)}}
    |j|^{\frac{s}{2}
    +\frac{\log(3)}{4\log(2)}}}
    \left\|u^j\right\|_{\dot{H}^{\frac{s}{2}+\frac{\log(3)}{4\log(2)}}}
    \left\|w^k\right\|_{\dot{H}^{\frac{s}{2}+\frac{\log(3)}{4\log(2)}}}.
\end{align}
Using the fact that $j+k\in\mathcal{M}_{n}^+$, and again recalling from \Cref{CanonicalComputations} that
\begin{align}
    |k|, |j+k|&=\sqrt{3*2^{2n}+2*3^n} \\
    |j| &= \sqrt{3*2^{2n+2}+2*3^{n+1}},
\end{align}
we compute that
\begin{align}
\frac{|j+k|^s}{|k|^{\frac{s}{2}
    +\frac{3}{4\log(2)}}
    |j|^{\frac{s}{2}
    +\frac{\log(3)}{4\log(2)}}}
    &=
    \frac{\left(3*2^{2n}+2*3^{n}
    \right)^{\frac{s}{2}}}
    {\left(3*2^{2n}+2*3^n
    \right)^{\frac{s}{4}+\frac{\log(3)}{8\log(2)}}
    \left(3*2^{2n+2}+2*3^{n+1}
    \right)^{\frac{s}{4}
    +\frac{\log(3)}{8\log(2)}}} \\
    &\leq 
    \frac{C}{2^{\frac{\log(3)}{2\log(2)}n}} \\
    &=
    \frac{C}{\left(\sqrt{3}\right)^n},
\end{align}
and this completes the proof.
\end{proof}

\begin{proposition} \label{Bound+-Prop}
     For all $u,w\in \dot{H}^{\frac{s}{2}+
     \frac{3}{4\log(2)}}_\mathcal{M}$,
    \begin{equation}
    \left\|B^{+,-}(u,w)\right\|_{\dot{H}^s}
    \leq C
    \|u\|_{\dot{H}^{\frac{s}
    {2}+\frac{\log(3)}{4\log(2)}}}\|w\|_{\dot{H}^{\frac{s}
    {2}+\frac{\log(3)}{4\log(2)}}},
    \end{equation}
    where $C>0$ is an absolute constant independent of $u$ and $w$.
\end{proposition}

\begin{proof}
    We begin by decomposing the bilinear operator $B^{+,-}=B^{+,-}_a+B^{+,-}_b$, where
    \begin{align}
    B^{+,-}_a(u,w) 
    &=
    \sum_{n=0}^{+\infty}
    \sum_{j\in\mathcal{M}_n^+}
    \sum_{k\in \mathcal{M}_{n-1}^-}
    B\left(u^j,w^k\right) \\
    B^{+,-}_b(u,w) 
    &=
    \sum_{n=0}^{+\infty}
    \sum_{j\in\mathcal{M}_n^+}
    \sum_{k\in \mathcal{M}_{n+1}^-}
    B\left(u^j,w^k\right).
    \end{align}
    Noting that $\mathcal{M}_{-1}^-=\emptyset$ by convention---as $\mathcal{M}_0$ is the lowest order frequency shell---we can shift the index in $B^{+,-}_a$,
    finding that
    \begin{equation}
    B^{+,-}_a(u,w) 
    =
    \sum_{n=0}^{+\infty}
    \sum_{j\in\mathcal{M}_{n+1}^+}
    \sum_{k\in \mathcal{M}_n^-}
    B\left(u^j,w^k\right).
    \end{equation}
    Applying \Cref{LemmaCascadeOutBound}, we find that
    \begin{align}
    \left\|B^{+,-}_a(u,w)\right\|_{\dot{H}^s}
    &\leq 
    \sum_{n=0}^{+\infty}
    \sum_{j\in\mathcal{M}_{n+1}^+}
    \sum_{k\in \mathcal{M}_n^-}
    \left\|B\left(u^j,w^k\right)
    \right\|_{\dot{H}^s} \\
    &\leq 
    C
    \sum_{n=0}^{+\infty}
    \left(\sum_{j\in\mathcal{M}_{n+1}^+}
    \left\|u^j\right\|_{\dot{H}^{\frac{s}
    {2}+\frac{\log(3)}{4\log(2)}}}\right)
    \left(\sum_{k\in \mathcal{M}_n^-}
    \left\|w^k\right\|_{\dot{H}^{\frac{s}
    {2}+\frac{\log(3)}{4\log(2)}}}\right) \\
    &\leq 
    C
    \sum_{n=0}^{+\infty}
    \left(\sum_{j\in\mathcal{M}_{n+1}^+}
    \left\|u^j\right\|_{\dot{H}^{\frac{s}
    {2}+\frac{\log(3)}{4\log(2)}}}^2
    \right)^\frac{1}{2}
    \left(\sum_{k\in \mathcal{M}_n^-}
    \left\|w^k\right\|_{\dot{H}^{\frac{s}
    {2}+\frac{\log(3)}{4\log(2)}}}^2
    \right)^\frac{1}{2} \\
    &\leq 
    C
    \left(\sum_{n=0}^{+\infty}
    \sum_{j\in\mathcal{M}_{n+1}^+}
    \left\|u^j\right\|_{\dot{H}^{\frac{s}
    {2}+\frac{\log(3)}{4\log(2)}}}^2
    \right)^\frac{1}{2}
    \left(\sum_{n=0}^{+\infty}
    \sum_{k\in \mathcal{M}_n^-}
    \left\|w^k\right\|_{\dot{H}^{\frac{s}
    {2}+\frac{\log(3)}{4\log(2)}}}^2
    \right)^\frac{1}{2} \\
    &\leq C
    \|u\|_{\dot{H}^{\frac{s}
    {2}+\frac{\log(3)}{4\log(2)}}}
    \left\|w\right\|_{\dot{H}^{\frac{s}
    {2}+\frac{\log(3)}{4\log(2)}}},
    \end{align}
    where we have again used the fact that each frequency shell is finite to control $L^1$ norms over a shell by $L^2$ norms, and applied H\"older's inequality to sequences in $L^2\left(\mathbb{Z}^+\right)$.

    To complete the proof we need the bound
    \begin{equation}
    \left\|B^{+,-}_b(u,w)\right\|_{\dot{H}^s}
    \leq 
    C\|u\|_{\dot{H}^{\frac{s}
    {2}+\frac{\log(3)}{4\log(2)}}}
    \left\|w\right\|_{\dot{H}^{\frac{s}
    {2}+\frac{\log(3)}{4\log(2)}}},
    \end{equation}
    and proof of this bound is exactly analogous to the $B^{+,-}_a$ case. In \Cref{AppendixModeInteraction}, we computed the interactions of the shell $\mathcal{M}_{n+1}^+$ and $\mathcal{M}_n^-$, and because our blowup Ansatz is odd, we can obtain the interactions of $\mathcal{M}_{n+1}^-$ and $\mathcal{M}_n^+$ by symmetry, without writing out the expressions. The computations for $\mathcal{M}_{n+1}^-$ and $\mathcal{M}_n^+$, however, are exactly the same with reversed signs and so the same bounds will hold.
\end{proof}

\begin{proposition} \label{Bound-+Prop}
     For all $u,w\in \dot{H}^{\frac{s}{2}+
     \frac{3}{4\log(2)}}_\mathcal{M}$,
    \begin{equation}
    \left\|B^{-,+}(u,w)\right\|_{\dot{H}^s}
    \leq C
    \|u\|_{\dot{H}^{\frac{s}
    {2}+\frac{\log(3)}{4\log(2)}}}\|w\|_{\dot{H}^{\frac{s}
    {2}+\frac{\log(3)}{4\log(2)}}},
    \end{equation}
    where $C>0$ is an absolute constant independent of $u$ and $w$.
\end{proposition}

\begin{proof}
    Recall that 
    \begin{equation}
    B^{-,+}(u,w)
    =
    B^{+,-}(w,u),
    \end{equation}
    and the result follows immediately from 
    \Cref{Bound+-Prop}.
\end{proof}

\begin{proposition} \label{Bound--Prop}
     For all $u,w\in \dot{H}^{\frac{s}{2}+
     \frac{3}{4\log(2)}}_\mathcal{M}$,
    \begin{equation}
    \left\|B^{-,-}(u,w)\right\|_{\dot{H}^s}
    \leq C
    \|u\|_{\dot{H}^{\frac{s}
    {2}+\frac{\log(3)}{4\log(2)}}}\|w\|_{\dot{H}^{\frac{s}
    {2}+\frac{\log(3)}{4\log(2)}}},
    \end{equation}
    where $C>0$ is an absolute constant independent of $u$ and $w$.
\end{proposition}

\begin{proof}
    The proof of this bound is entirely analogous to the proof of \Cref{Bound++Prop}. In \Cref{AppendixModeInteraction}, we only compute the interactions between frequencies in $\mathcal{M}_n^+$, but not the frequencies in $\mathcal{M}_n^-$, which we get automatically by symmetry for odd solutions.
    The computations for interactions in $\mathcal{M}_n^-$ are exactly analogous with negative signs, so the same bounds will hold.
\end{proof}

\begin{remark}
    Putting together \Cref{Bound++Prop,Bound+-Prop,Bound-+Prop,Bound--Prop}, \Cref{BilinearBoundLemma} follows.
\end{remark}

\section{Geometric motivation for the blowup Ansatz}
\label{GeometricSection}

One motivation for the blowup Ansatz considered in this paper is a geometric constraint on the blowup of smooth solutions of the Navier--Stokes equation proven by Neustupa and Penel \cite{NeustupaPenel1}. They proved that if a solution to the Navier--Stokes equation blows up in finite-time $T_{max}<+\infty$, then for all 
$\frac{3}{2}<q\leq +\infty, 
\frac{2}{p}+\frac{3}{q}=2$,
\begin{equation} \label{StrainRegCritNS}
    \int_0^{T_{max}}
    \|\lambda_2^+(\cdot,t)\|_{L^q}^p
    \diff t=+\infty,
\end{equation}
where $\lambda_1\leq \lambda_2\leq \lambda_3$ are the eigenvalues of the strain matrix, which is the symmetric part of the gradient of the velocity
\begin{equation}
    S_{ij}=\frac{1}{2}\left(
    \partial_i u_j+\partial_j u_i\right).
\end{equation}
An alternative proof was provided by the author in \cite{MillerStrain} by making use of the evolution equation for the strain.

\begin{remark}
    This regularity criterion is scale critical, and provides important information about the geometric structure of any possible finite-time blowup for smooth solutions of the Navier--Stokes equation. Because the strain must be trace free with
    \begin{equation}
        \tr(S)
        =
        \nabla\cdot u
        =
        \lambda_1+\lambda_2+\lambda_3
        =0,
    \end{equation}
    we can see that $\lambda_2$ is always the smallest eigenvalue in absolute value. 
    While the velocity describes how a fluid is advected, and the vorticity describes how it is rotated, the strain describes how a parcel of fluid is deformed.
    
    Two positive eigenvalues corresponds to planar stretching and axial compression, while two negative eigenvalues corresponds to axial stretching and planar compression.
    The regularity criteria on the positive part of the intermediate eigenvalue implies that blowup requires unbounded planar stretching (as measured in scale critical $L^p_t L^q_x$ spaces).
    In physical terms, this means deformations that take spheres to pancakes promote blowup, while deformations that take spheres to rods undermine singularity formation.

    This regularity criterion is based on estimates controlling enstrophy growth. This is directly related to the Fourier-restricted hypodissipative Navier--Stokes equation considered in this paper, because this model equation shares the same identities for enstrophy growth as the full hypodissipative Navier--Stokes equation.
    We will state these identities now. For details see \cites{NeustupaPenel1,MillerStrain}.
\end{remark}

\begin{proposition} \label{IsometryProp}
    Suppose $u\in \dot{H}^s_{df}, s\geq 1$.
    Then for all $0\leq s'\leq s-1$,
    \begin{equation}
    \|\nabla u\|_{\dot{H}^{s'}}^2=\|\omega\|_{\dot{H}^{s'}}^2
    =2\|S\|_{\dot{H}^{s'}}^2
    \end{equation}
\end{proposition}

\begin{proposition} \label{VortexStretchProp}
    For all $u\in \dot{H}^2_{df}$,
    \begin{equation}
    -\left<(u\cdot\nabla)u,-\Delta u\right>
    =\left<S,\omega\otimes\omega\right>
    =-4\int_{\mathbb{T}^3}\det(S).
    \end{equation}
\end{proposition}

These two propositions yield the following equivalent identities for enstrophy growth for solutions of the Navier--Stokes equation:
\begin{align}
    \frac{\diff}{\diff t}
    \frac{1}{2}\|\nabla u(\cdot,t)\|_{L^2}^2
    &=-\nu\|-\Delta u\|_{L^2}^2
    -\left<(u\cdot\nabla)u,-\Delta u\right> \\
    \frac{\diff}{\diff t}
    \frac{1}{2}\|\omega(\cdot,t)\|_{L^2}^2
    &=-\nu \|\nabla \omega\|_{L^2}^2
    +\left<S,\omega\otimes\omega\right> \\
    \frac{\diff}{\diff t}
    \|S(\cdot,t)\|_{L^2}^2
    &=-2\nu \|\nabla S\|_{L^2}^2
    -4\int_{\mathbb{T}^3}\det(S).
\end{align}
From the last identity we can show that 
\begin{equation}
    \frac{\diff}{\diff t}\|S(\cdot,t)\|_{L^2}^2
    \leq -2\nu \|\nabla S\|_{L^2}^2
    +2\int_{\mathbb{T}^3}\lambda_2^+ |S|^2,
\end{equation}
and the regularity criterion \eqref{StrainRegCritNS} follows by a Gr\"onwall estimate on enstrophy growth after applying H\"older's inequality, the Sobolev inequality, and Young's inequality to find that 
\begin{equation}
    \frac{\diff}{\diff t}\|S(\cdot,t)\|_{L^2}^2
    \leq C_q
    \|\lambda_2^+\|_{L^q}^p\|S\|_{L^2}^2.
\end{equation}

These estimates are of interest to us, because they follow from the structure of the nonlinearity $(u\cdot\nabla)u$.
In particular, these identities also will hold for the Fourier-restricted, hypodissipative Navier--Stokes equation, with the appropriate adjustment of the dissipative term based on the degree of dissipation $\alpha$.

\begin{proposition} \label{HypoEnstrophyProp}
    Suppose $u\in C^1\left([0,T_{max});
    \dot{H}^1_{\mathcal{M}}\right)$
    is a solution of the Fourier-restricted, hypodissipative Navier--Stokes equation.
    Then for all $0< t< T_{max}$,
    the evolution of enstrophy can be expressed equivalently by
    \begin{align}
    \frac{\diff}{\diff t} \label{HypoEntVelo}
    \frac{1}{2}\|\nabla u(\cdot,t)\|_{L^2}^2
    &=-\nu\|\nabla u\|_{\dot{H}^\alpha}^2
    -\left<(u\cdot\nabla)u,-\Delta u\right> \\
    \frac{\diff}{\diff t}
    \frac{1}{2}\|\omega(\cdot,t)\|_{L^2}^2
    &=-\nu \|\omega\|_{\dot{H}^\alpha}^2 \label{HypoEntVort}
    +\left<S,\omega\otimes\omega\right> \\
    \frac{\diff}{\diff t}
    \|S(\cdot,t)\|_{L^2}^2 \label{HypoEntStrain}
    &=-2\nu \|S\|_{\dot{H}^\alpha}^2
    -4\int_{\mathbb{T}^3}\det(S).
\end{align}
In particular, we also have the upper bound
\begin{equation} \label{EntUpperBound}
    \frac{\diff}{\diff t}
    \|S(\cdot,t)\|_{L^2}^2
    \leq 
    -2\nu \|S\|_{\dot{H}^\alpha}^2
    +2\int_{\mathbb{T}^3}\lambda_2^+ |S|^2.
\end{equation}
\end{proposition}

\begin{proof}
    Directly from the evolution equation, we compute that for all
    $0<t<T_{max}$,
    \begin{align}
    \frac{\diff}{\diff t}
    \frac{1}{2}\|\nabla u(\cdot,t)\|_{L^2}^2
    &=
    -\left<-\Delta u, -\nu(-\Delta)^\alpha u
    +\mathbb{P}_\mathcal{M}((u\cdot\nabla)u)
    \right> \\
    &=
    -\nu \|\nabla u\|_{\dot{H}^\alpha}^2
    -\left<-\Delta u, (u\cdot\nabla)u\right>.
    \end{align}
    Note that these expressions are all well defined because $u\in C^\infty\left((0,T_{max});
    C^\infty\left(\mathbb{T}^3\right)\right)$.
    The identities \eqref{HypoEntVort} and \eqref{HypoEntStrain} then follow from \Cref{IsometryProp,VortexStretchProp}.
    It then remains only to prove \eqref{EntUpperBound}.
    First we observe that $\lambda_1\leq 0$, because if $0<\lambda_1\leq \lambda_2\leq \lambda_3$, then
    \begin{equation}
        \tr(S)=\lambda_1+\lambda_2+\lambda_3>0,
    \end{equation}
    which contradicts the divergence free constraint.
    Likewise, we can see that $\lambda_3\geq 0$ is also guaranteed by the divergence free constraint.
    This implies that $-\lambda_1\lambda_3\geq 0$,
    and so we may conclude that
    \begin{align}
        -\det(S)
        &=
        (-\lambda_1\lambda_3)\lambda_2 \\
        &\leq 
        (-\lambda_1\lambda_3)\lambda_2^+ \\
        &\leq \frac{1}{2}
        (\lambda_1^2+\lambda_3^2) \lambda_2^+ \\
        &\leq \frac{1}{2}|S|^2 \lambda_2^+,
    \end{align}
    and this completes the proof.
\end{proof}

We also have an analogous result for the Fourier-restricted Euler equation. The proof is identical except that there is no dissipation term, and so is omitted.

\begin{proposition} \label{RestrictedEnstrophyProp}
    Suppose $u\in C^1\left([0,T_{max});
    \dot{H}^2_{\mathcal{M}}\right)$
    is a solution of the Fourier-restricted Euler equation.
    Then for all $0\leq t< T_{max}$,
    the evolution of enstrophy can be expressed equivalently by
    \begin{align}
    \frac{\diff}{\diff t}
    \frac{1}{2}\|\nabla u(\cdot,t)\|_{L^2}^2
    &=
    -\left<(u\cdot\nabla)u,-\Delta u\right> \\
    \frac{\diff}{\diff t}
    \frac{1}{2}\|\omega(\cdot,t)\|_{L^2}^2
    &=
    \left<S,\omega\otimes\omega\right> \\
    \frac{\diff}{\diff t}
    \|S(\cdot,t)\|_{L^2}^2
    &=
    -4\int_{\mathbb{T}^3}\det(S).
\end{align}
In particular, we also have the upper bound
\begin{equation}
    \frac{\diff}{\diff t}
    \|S(\cdot,t)\|_{L^2}^2
    \leq 
    2\int_{\mathbb{T}^3}\lambda_2^+ |S|^2
\end{equation}
\end{proposition}

\begin{remark}
    Note that the $\dot{H}^2$ regularity is required to apply \Cref{VortexStretchProp}.
\end{remark}

We will also have regularity criteria for the Fourier-restricted Euler and hypodissipative Navier--Stokes equations. We will now prove \Cref{ViscousStrainRegCritIntro,EulerStrainRegCritIntro}, which are restated for the reader's convenience.

\begin{theorem} \label{ViscousStrainRegCrit}
    Suppose $u\in C\left([0,T_{max});
    \dot{H}^1_{\mathcal{M}}\right)$
is a solution of the Fourier-restricted hypodissipative Navier--Stokes equation, and that $\alpha<\frac{\log(3)}{4\log(2)}$.
Suppose $\frac{1}{p}+\frac{3}{2\alpha q}=1, 
\frac{3}{2\alpha}<q\leq +\infty$.
Then for all $0\leq t<T_{max}$,
\begin{equation} \label{StrainRegCritBound}
    \|S(\cdot,t)\|_{L^2}^2
    \leq \left\|S^0\right\|_{L^2}^2
    \exp\left(\frac{C_{\alpha,q}}{\nu^\frac{p-1}{p^2}}
    \int_0^t
    \left\|\lambda_2^+(\cdot,\tau)\right\|_{L^q}^p
    \diff\tau \right),
\end{equation}
where $C_{\alpha,q}>0$ is an absolute constant independent of $\nu$ and $u^0$ depending only on $\alpha$ and $q$.
In particular, if $T_{max}<+\infty$, then
\begin{equation}
    \int_0^{T_{max}}
    \left\|\lambda_2^+(\cdot,t)\right\|_{L^q}^p
    \diff t
    =+\infty.
\end{equation}
Note that in the case $q=+\infty$, we have for all 
$0\leq t<T_{max}$,
\begin{equation} 
    \|S(\cdot,t)\|_{L^2}^2
    \leq 
    \left\|S^0\right\|_{L^2}^2
    \exp\left( 2\int_0^t
    \left\|\lambda_2^+(\cdot,\tau)
    \right\|_{L^\infty}
    \diff\tau\right).
\end{equation}
\end{theorem}

\begin{proof}
    Using the isometry in \Cref{IsometryProp} and the local wellposedness theory from \Cref{RestrictedHypoExistenceThmIntro},
    we can see that if 
    $T_{max}<+\infty$, then
    \begin{equation}
        \lim_{t\to T_{max}}\|S(\cdot,t)
        \|_{L^2}
        =+\infty.
    \end{equation}
    Therefore, it suffices to prove the bound \eqref{StrainRegCritBound}.
    Recall from \Cref{HypoEnstrophyProp} that
    for all $t>0$,
    \begin{equation}
    \frac{\diff}{\diff t}
    \|S(\cdot,t)\|_{L^2}^2
    \leq -2\nu \|S\|_{\dot{H}^\alpha}^2
    +2\int_{\mathbb{T}^3}\lambda_2^+ |S|^2.
\end{equation}

    First we will consider the case $q=+\infty,p=1$.
    Applying H\"older's inequality we can see that
    \begin{equation}
    \frac{\diff}{\diff t}
    \|S(\cdot,t)\|_{L^2}^2
    \leq
    2\|\lambda_2^+\|_{L^\infty}\|S\|_{L^2}^2.
    \end{equation}
    Applying Gr\"onwall's inequality, we find that
\begin{equation} 
    \|S(\cdot,t)\|_{L^2}^2
    \leq 
    \left\|S^0\right\|_{L^2}^2
    \exp\left( 2\int_0^t
    \left\|\lambda_2^+(\cdot,\tau)
    \right\|_{L^\infty}
    \diff\tau\right),
\end{equation}
    and this completes the proof for the case $q=+\infty$.

    Now consider the case $\frac{3}{2\alpha}<q<+\infty$.
    Let $\frac{2}{r}+\frac{1}{q}=1$, and let 
    $\frac{1}{r}=\frac{1}{2}-\frac{s}{3}$.
    Applying H\"older's inequality and the Sobolev inequality, we find that
    \begin{align}
    \frac{\diff}{\diff t}
    \|S(\cdot,t)\|_{L^2}^2
    &\leq 
    -2\nu \|S\|_{\dot{H}^\alpha}^2
    +2\|\lambda_2^+\|_{L^q}\|S\|_{L^r}^2 \\
    &\leq 
    -2\nu \|S\|_{\dot{H}^\alpha}^2
    +C\|\lambda_2^+\|_{L^q}\|S\|_{\dot{H}^s}^2 
    \end{align}
    Note that $\frac{1}{q}=1-\frac{2}{r}=\frac{2s}{3}$,
    and so 
    \begin{equation}
        s=\frac{3}{2q}.
    \end{equation}
    We know that $\frac{3}{2\alpha}<q<+\infty$, so this clearly implies $0<s<\alpha$.
    Let $b=\frac{s}{\alpha}$ and, noting that $s=(1-b)0+b\alpha$, observe that by the interpolation inequality for Sobolev spaces,
    \begin{equation}
    \|S\|_{\dot{H}^s}\leq 
    \|S\|_{L^2}^{1-b}\|S\|_{\dot{H}^\alpha}^b.
    \end{equation}

    Plugging in this inequality we find that
    \begin{equation}
        \frac{\diff}{\diff t}
    \|S(\cdot,t)\|_{L^2}^2
    \leq 
    -2\nu \|S\|_{\dot{H}^\alpha}^2
    +C\|\lambda_2^+\|_{L^q}\|S\|_{L^2}^{2-2b}\|S\|_{\dot{H}^\alpha}^{2b}.
    \end{equation}
    Note that
    \begin{align}
    \frac{1}{p}
    &=
    1-\frac{3}{2\alpha q} \\
    &=
    1-\frac{s}{\alpha} \\
    &=
    1-b,
    \end{align}
    and so we can see that
    \begin{equation}
    \frac{\diff}{\diff t}
    \|S(\cdot,t)\|_{L^2}^2
    \leq 
    -2\nu \|S\|_{\dot{H}^\alpha}^2
    +C\|\lambda_2^+\|_{L^q}\|S\|_{L^2}^\frac{2}{p}
    \|S\|_{\dot{H}^\alpha}^{2b}.
    \end{equation}
    Finally as we can see that $\frac{1}{p}+b=1$,
    we can apply Young's inequality with exponents $p$ and $\frac{1}{b}$, and find that
    \begin{equation}
    \nu^{-b}
    \|\lambda_2^+\|_{L^q}\|S\|_{L^2}^\frac{2}{p}
    \nu^b\|S\|_{\dot{H}^\alpha}^{2b}
    \leq 
    2\nu\|S\|_{\dot{H}^\alpha}^2
    +C\nu^{-\frac{b}{p}}
    \|\lambda_2^+\|_{L^q}^p\|S\|_{L^2}^2.
    \end{equation}
    Observe that $\frac{b}{p}=\frac{p-1}{p^2}$,
    and we can see that for all $0<t<T_{max}$,
    \begin{equation}
    \frac{\diff}{\diff t}
    \|S(\cdot,t)\|_{L^2}^2
    \leq
    \frac{C_{\alpha,q}}{\nu^\frac{p-1}{p^2}}
    \|\lambda_2^+\|_{L^q}^p\|S\|_{L^2}^2.
    \end{equation}
    Applying Gr\"onwall's inequality, this completes the proof.
\end{proof}

\begin{theorem} \label{EulerStrainRegCrit}
     Suppose $u\in C\left([0,T_{max});
    \dot{H}^s_{\mathcal{M}}\right), s>\frac{5}{2}$,
is a solution of the Fourier-restricted Euler equation.
Then for all $0\leq t<T_{max}$,
\begin{equation} \label{EulerStrainRegCritBound}
    \|S(\cdot,t)\|_{L^2}^2
    \leq 
    \left\|S^0\right\|_{L^2}^2
    \exp\left( 2\int_0^t
    \left\|\lambda_2^+(\cdot,\tau)
    \right\|_{L^\infty}
    \diff\tau\right).
\end{equation}
In particular, if $T_{max}<+\infty$, then
\begin{equation}
    \int_0^{T_{max}}
    \left\|\lambda_2^+(\cdot,t)\right\|_{L^\infty}
    \diff t
    =+\infty.
\end{equation}
\end{theorem}

\begin{proof}
    Using the isometry in \Cref{IsometryProp} and the local wellposedness theory from \Cref{RestrictedEulerExistenceThmIntro},
    we can see that if 
    $T_{max}<+\infty$, then
    \begin{equation}
        \lim_{t\to T_{max}}\|S(\cdot,t)
        \|_{L^2}
        =+\infty.
    \end{equation}
    Therefore, it suffices to prove the bound \eqref{EulerStrainRegCritBound}.
    Recalling the bound in \Cref{RestrictedEnstrophyProp}, and applying H\"older's inequality, we find that for all $0\leq t<T_{max}$
    \begin{align}
    \frac{\diff}{\diff t}
    \|S(\cdot,t)\|_{L^2}^2
    &\leq 
    2\int_{\mathbb{T}^3}\lambda_2^+ |S|^2 \\
    &\leq 
    2\|\lambda_2^+\|_{L^\infty}\|S\|_{L^2}^2.
\end{align}
Applying Gr\"onwall's inequality, this completes the proof. Note that we require $u\in H^s$ where $s>\frac{5}{2}$, because this guarantees that $S\in L^\infty$, and therefore its eigenvalues are as well.
\end{proof}

\begin{remark}
    We should note that the full Euler and hypodissipative Navier--Stokes equations also have the bounds \eqref{StrainRegCritBound} and \eqref{EulerStrainRegCritBound}.
This was shown by Chae in \cite{ChaeStrain} for the Euler equation, although the methods from Neustupa and Penel \cite{NeustupaPenel1} cover this case, even if the inviscid result wasn't considered explicitly in their paper, and the hypodissipative case easily follows by the same methods.
However, for the full hypodissipative Navier--Stokes equations, $\|\nabla u\|_{C_T L^2_x}$ is supercritical for $\alpha<\frac{3}{4}$, so these bounds 
 do not control regularity for the Euler equation or the hypodissipative Navier--Stokes equation when $\alpha<\frac{3}{4}$. It is a specific feature of the dyadic structure that enstrophy controls regularity for the Fourier-restricted hypodissipative Navier--Stokes equation when $\alpha$ is small, and in fact even in the inviscid case as well.
\end{remark}

\subsection{Permutation symmetric strains}

Along the $\sigma$-axis, we can use permutation symmetry to show the velocity gradient and the velocity must have a particular structure.
This both allows us to give a precise description of the gradient blowup at the origin for the singular solutions of the Fourier-restricted Euler and hypodissipative Navier--Stokes equations, and also suggests a whole family of initial data to consider as candidates for finite-time blowup for the full Euler equation.

\begin{proposition} \label{PermuteGradientProp}
    Suppose $u\in \dot{H}^s_{df}\left(\mathbb{T}^3;
    \mathbb{R}^3\right)$ is permutation-symmetric. Then
    for all $-\frac{1}{2}\leq a\leq \frac{1}{2}$,
    \begin{equation}
    \nabla u(a,a,a)
    =
    \partial_1 u_2(a,a,a)
    \left(\begin{array}{ccc}
         0 & 1 & 1  \\
        1 & 0 & 1   \\
        1 & 1 & 0
    \end{array}\right).
    \end{equation}
\end{proposition}

\begin{proof}
    This amounts to showing that for all $i\neq j$,
    \begin{equation}
        \partial_i u_j(a,a,a)
        =
         \partial_1 u_2(a,a,a),
    \end{equation}
    and for all $1\leq i\leq 3$,
    \begin{equation}
    \partial_iu_i(a,a,a)=0.
    \end{equation}
    Letting $x=(a,a,a)$, we can immediately see that for all $P\in\mathcal{P}_3$,
    \begin{equation}
        P(x)=x.
    \end{equation}
    Fix any $i\neq j$.
    Take the permutation $P\in\mathcal{P}_3$ such that
    $P(i)=1$ and $P(j)=2$.
    Then we can see that
    \begin{align}
    \partial_i u_j (a,a,a)
    &=
    \partial_i u^P_j (a,a,a) \\
    &=
    \partial_1 u_2(a,a,a),
    \end{align}
    where we have used the fact that $u^P=u$ because $u$ is permutation symmetric. 

    Now we have shown the identity holds for the off diagonal terms; it remains to show that it holds for the diagonal terms.
    For $i=2,3$, we compute $\partial_iu_i$ using the swap permutation $P_{1i}$, finding that
    \begin{align}
        \partial_i u_i(a,a,a)
        &=
        \partial_i u^{P_{1i}}_i(a,a,a) \\
        &=
        \partial_1 u_1(a,a,a).
    \end{align}
    Using the divergence free constraint, we can see that
    \begin{align}
        \nabla\cdot u(a,a,a)
        =
        3\partial_1u_1(a,a,a)
        =
        0,
    \end{align}
    and so for all $1\leq i\leq 3$,
    \begin{equation}
    \partial_i u_i(a,a,a)=0.
    \end{equation}
    This completes the proof.
\end{proof}

\begin{proposition}
    Suppose $u\in \dot{H}^s_{df}\left(\mathbb{T}^3;
    \mathbb{R}^3\right)$ is permutation-symmetric. Then
    for all $-\frac{1}{2}\leq a\leq \frac{1}{2}$,
    \begin{equation}
    u_1(a,a,a)=u_2(a,a,a)=u_3(a,a,a),
    \end{equation}
    and so
    \begin{equation}
    u(a,a,a)=u_1(a,a,a)\sigma.
    \end{equation}
\end{proposition}

\begin{proof}
    This follows immediately from permutation symmetry.
\end{proof}

\begin{lemma} \label{MatrixLemma}
    The matrix 
    \begin{equation}
        M=
         \left(\begin{array}{ccc}
         0 & -1 & -1  \\
        -1 & 0 & -1   \\
        -1 & -1 & 0
    \end{array}\right)
    \end{equation}
    has eigenvalues $-2,1,1$. The vector $w=\frac{\sigma}{\sqrt{3}}$ is the unit eigenvector corresponding to the eigenvalue $\lambda=-2$. Any vector $|v|=1, \sigma\cdot v=0$, is a unit eigenvector corresponding to the two dimensional eigenspace of the eigenvalue $\lambda=1$.
\end{lemma}

\begin{proof}
    First observe that 
    \begin{align}
    M
    &=
    I_3-\sigma\otimes\sigma \\
    &=
    I_3-3w\otimes w.
    \end{align}
    Therefore, we clearly have
    \begin{equation}
        Mw=-2w.
    \end{equation}
    Likewise, if $\sigma\cdot v=0$, then clearly
    \begin{equation}
        Mv=v.
    \end{equation}
\end{proof}

\begin{proposition} \label{MatrixProp}
    For all $k\in \mathbb{Z}^3, k\neq 0$ and for all $v\in \mathbb{R}^3, k\cdot v=0$, let 
    \begin{equation}
    B^{k,v}(x)=-\sum_{P\in \mathcal{P}_3}
    P(v)\sin(2\pi P(k)\cdot x).
    \end{equation}
    Then $B^{k,v} \in C^\infty$ is odd, permutation-symmetric, and divergence free. Furthermore
    \begin{equation}
    \nabla B^{k,v}(\Vec{0})
    =
    2\pi(\sigma\cdot k)(\sigma\cdot v)
    \left(\begin{array}{ccc}
         0 & -1 & -1  \\
        -1 & 0 & -1   \\
        -1 & -1 & 0
    \end{array}\right),
    \end{equation}
    and has eigenvalues $-2\lambda,\lambda,\lambda$, where
    \begin{equation}
        \lambda=2\pi(\sigma\cdot k)(\sigma\cdot v).
    \end{equation}
\end{proposition}

\begin{proof}
    It is obvious that $B^{k,v}$ is odd, because it is composed of sine waves. To see that $B^{k,v}$ is divergence free, observe that
    \begin{align}
    \nabla\cdot B^{k,v}(x)
    &=
    -\sum_{P\in \mathcal{P}_3}
    2\pi P(k)\cdot P(v)\cos(2\pi P(k)\cdot x) \\
    &=
    -2\pi k\cdot v
    \sum_{P\in \mathcal{P}_3}
    \cos(2\pi P(k)\cdot x) \\
    &=
    0,
    \end{align}
    where we have used the fact that
    \begin{equation}
    P(k)\cdot P(v) =k\cdot v.
    \end{equation}
    Note that 
    \begin{equation}
    P(k)\cdot x=k\cdot P^{-1}(x),
    \end{equation}
    and so we have 
    \begin{equation}
    B^{k,v}(x)=-\sum_{P\in \mathcal{P}_3}
    P(v)\sin(2\pi k\cdot P^{-1}(x)).
    \end{equation}
    Therefore, we can clearly see that $B^{k,v}$ is permutation symmetric by construction.
    
    Finally, applying \Cref{PermuteGradientProp}, we can see that
    \begin{equation}
        \nabla B^{k,v}(\Vec{0})=
        -\partial_1 B^{k,v}_2(\Vec{0})
    \left(\begin{array}{ccc}
         0 & -1 & -1  \\
        -1 & 0 & -1   \\
        -1 & -1 & 0
    \end{array}\right),
    \end{equation}
    Therefore it suffices to show that
    \begin{equation}
    \partial_1 B^{k,v}_2(\Vec{0})
    =-2\pi (\sigma\cdot k)(\sigma\cdot v).
    \end{equation}
    Computing the sum over permutations, we find that
    \begin{align}
        \partial_1 B^{k,v}_2(x)
        &=
        -2\pi \sum_{P\in \mathcal{P}_3}
    P(k)_1P(v)_2
    \cos(2\pi P(k)\cdot x) \\
    &=
    -2\pi \sum_{P\in \mathcal{P}_3}
    k_{P(1)}v_{P(2)}
    \cos(2\pi P(k)\cdot x).
    \end{align}
    Therefore plugging in $x=0$, we find
    \begin{align}
\partial_1 B^{k,v}_2(\Vec{0})
&=
-2\pi \sum_{P\in \mathcal{P}_3}
k_{P(1)}v_{P(2)} \\
&=
-2\pi \sum_{i\neq j} k_{i}v_{j} \\
&= 
-2\pi \sum_{i\neq j} k_{i}v_{j} 
-2\pi \sum_{m=1}^3 k_m v_m \\
&=
-2\pi \sum_{1\leq i,j\leq 3} k_i v_j \\
&=
-2\pi (k_1+k_2+k_3)(v_1+v_2+v_3) \\
&=
-2\pi (\sigma\cdot k)(\sigma \cdot v),
    \end{align}
    where we have used that there for each $i\neq j$ there is a unique permutation $P\in\mathcal{P}_3$ such that
    $P(1)=i, P(2)=j$, and the fact that $k\cdot v=0$.
    Applying the eigenvalue/eigenvector results in \Cref{MatrixLemma} completes the proof.
\end{proof}

\begin{remark}
     We can see that any odd, permutation-symmetric vector field 
    $u\in \dot{H}^s_{df}$ can be written in the form
    \begin{equation}
    u=\sum_{k} B^{k,u^k},
    \end{equation}
    where for all $k\in\mathbb{T}^3, k\cdot u^k=0$.
    This must be the case because if $u$ is odd, it must have a Fourier sine series, and if $u$ is permutation symmetric, then its Fourier series must be permutation symmetric, and therefore decomposable into sums of six grouped terms of the form $B^{k,u^k}$.
    Note that in this case
    \begin{equation}
    \nabla u(\Vec{0})=\lambda
    \left(\begin{array}{ccc}
         0 & -1 & -1  \\
        -1 & 0 & -1   \\
        -1 & -1 & 0
    \end{array}\right),
    \end{equation}
    where
    \begin{equation}
    \lambda=2\pi \sum_{k} 
    (\sigma\cdot k)(\sigma\cdot u^k).
    \end{equation}
    
    The gradient of the velocity is symmetric at the origin; this implies that $\omega(\Vec{0})=0$, and the strain matrix satisfies
    \begin{equation}
    S(\Vec{0})=\lambda
    \left(\begin{array}{ccc}
         0 & -1 & -1  \\
        -1 & 0 & -1   \\
        -1 & -1 & 0
    \end{array}\right).
    \end{equation}
    We can clearly see that the strain matrix has eigenvalues $-2\lambda,\lambda,\lambda$, so at the origin we have the eigenvalue structure that tends to generate blowup $\lambda_2=\lambda_3=\lambda$.
    This is exactly what the regularity criterion on $\lambda_2^+$ suggests is the most singular scenario for the Navier--Stokes equation.
    Note that we are deliberately vague about how we sum over frequencies $k$, because there is the issue of avoiding double counting frequencies. This issue is addressed in our analysis of the Fourier-restricted Euler and hypodissipative Navier--Stokes equations by taking our canonical frequencies with descending components.
\end{remark}

\begin{remark}
    If we want to make $\lambda$ as large as possible that means making $\sigma\cdot u^k$ as large as possible when $\sigma\cdot k>0$. We would like to simply make $u^k$ a positive multiple of $\sigma$, but we cannot do this, because that would violate the condition
    $k\cdot u^k=0$. The most singular scenario, in terms of planar stretching at the origin, will be when
    \begin{equation}
    u^k=c P_{k}^\perp(\sigma),
    \end{equation}
    for some $c>0$. The largest amount of planar stretching we can get is when $u^k$ is a positive multiple of $\sigma$ projected onto the orthogonal complement of $\spn(k)$.
    
    We will note that this is exactly the Ansatz we use to prove finite-time blowup for solutions of the Fourier-restricted Euler and hypodissipative Navier--Stokes equations, 
    \begin{multline}
        u(x,t)=-2\sum_{m=0}^{+\infty}\Bigg(
        \psi_{2m}(t)\sum_{k\in\mathcal{P}[k^{m}]}
        v^{k}\sin{(2\pi k\cdot x)} \\
        +\psi_{2m+1}(t) \left(
        \sum_{h\in\mathcal{P}[h^{m}]}
        v^{h}\sin{(2\pi h\cdot x)}
        +\sum_{j\in\mathcal{P}[j^{m}]}
        v^{j}\sin{(2\pi j\cdot x)}
        \right)\Bigg),
    \end{multline}
    and that this structure was the motivation for both the Ansatz and the constraint space.
    The constraint space is constructed in terms of sums of permutations, so that permutation symmetric solutions will reduce to something very similar to the dyadic Euler and Navier--Stokes equations, and the constraint restricting the Fourier amplitudes to be scalar multiples of $P_k^\perp(\sigma)$ is motivated by exactly this geometric constraint on singularity formation from the regularity criterion on $\lambda_2^+$ proven by Neustupa and Penel \cite{NeustupaPenel1}, a variant of which we have proven for the restricted model equation equations in \Cref{ViscousStrainRegCrit,EulerStrainRegCrit}.
\end{remark}

\begin{remark} \label{FullEulerRemark}
    The choice of the frequencies $\mathcal{M}^+$ for the Fourier sine series is motivated by the fact that it lead to dynamics for the Fourier-restricted model equations with useful features related to the dyadic model. As we will see in \Cref{AppendixFourierLimitations}, it will not be possible to reduce the dynamics in this way for the full Euler or Navier--Stokes equations, which means there is no reason to restrict to an Ansatz with only these frequencies when studying the full Euler equation.

    A more general Ansatz to study for the full Euler equation involves initial data of the form
    \begin{equation}
    u^0(x)=-\sum_{\sigma\cdot k>0} 
    \gamma_k P_k^\perp(\sigma) \sin(2\pi k\cdot x),
    \end{equation}
    where $\gamma$ is permutation symmetric in the sense that for all $P\in\mathcal{P}_3$ and $k\in\mathbb{Z}^3$, 
    \begin{equation}
        \gamma_k=\gamma_{P(k)}.
    \end{equation}
    Note that the vorticity in this case can be expressed as the cosine series
    \begin{equation}
    \omega^0(x)
    =
    2\pi \sum_{\sigma\cdot k>0} 
    \gamma_k (\sigma\times k) \cos(2\pi k\cdot x),
    \end{equation}
    where we have used the fact that 
    $k\times P_k^\perp(\sigma)=
    k\times \sigma$.
    There is an enormous amount of flexibility within this class of initial data, so there are a great number of potential permutation-symmetric vortex structures that could be investigated in future work.
\end{remark}

\subsection{Unbounded planar 
stretching at the origin}

In this subsection, we will give a precise description of the blowup at the origin for the blowup solutions of the Fourier-restricted Euler and hypodissipative Navier--Stokes equations 
described by \Cref{RestrictedEulerBlowupIntro,RestrictedHypoBlowupIntro}. The purpose of this section will be to prove \Cref{PsiSupRegCritIntro,StrainOriginBlowupThmIntro}, which will be restated for the reader's convenience immediately after a proof of a key proposition describing the gradient of the flow at the origin for the blowup Anstaz. 

\begin{proposition} \label{OriginStrainFourierProp}
    Suppose $u\in \dot{H}^s_{\mathcal{M}},
    s>\frac{5}{2}$, is odd, permutation-symmetric, with hj-parity, and is therefore given by
    \begin{equation}
    u(x)=-2\sum_{n=0}^{+\infty}\psi_n
    \left(\sum_{k\in\mathcal{M}_n^+}
    v^{k}\sin{(2\pi k\cdot x)}\right).
    \end{equation}
    Then 
    \begin{equation}
    \nabla u(\Vec{0})=
    \lambda \left(\begin{array}{ccc}
         0 & -1 & -1  \\
         -1 & 0 & -1 \\
         -1 & -1 & 0
    \end{array}\right),
    \end{equation}
    where 
    \begin{equation}
    \lambda=
    12\sqrt{2}\pi \sum_{n=0}^{+\infty}
    \psi_n 
    \frac{\left(\sqrt{3}\right)^n}
    {\left(1+\frac{2}{3}
    \left(\frac{3}{4}\right)^n
    \right)^\frac{1}{2}}
    \end{equation}
\end{proposition}

\begin{proof}
    First we will note that from the definition of $\mathcal{M}_n^+$, that $u$ can be expressed as
        \begin{multline}
        u(x)=-2\sum_{m=0}^{+\infty}\Bigg(
        \psi_{2m}\sum_{k\in\mathcal{P}[k^{m}]}
        v^{k}\sin{(2\pi k\cdot x)} \\
        +\psi_{2m+1} \left(
        \sum_{h\in\mathcal{P}[h^{m}]}
        v^{h}\sin{(2\pi h\cdot x)}
        +\sum_{j\in\mathcal{P}[j^{m}]}
        v^{j}\sin{(2\pi j\cdot x)}
        \right)\Bigg),
    \end{multline}
    and that therefore
    \begin{equation}
    u=\sum_{m=0}^{+\infty} \left(
    2\psi_{2m}B^{k^m,v^{k^m}}
    +\psi_{2m+1}\left(
    B^{h^m,v^{h^m}}
    +B^{j^m,v^{j^m}}\right)\right),
    \end{equation}
    where $B^{k,v}$ is defined as in \Cref{MatrixProp}.
    Note we do not have a factor of two with the terms $B^{h^m,v^{h^m}}$ and $B^{j^m,v^{j^m}}$,
    because both of these vectors have a repeated component, so when summing over permutations of the vector, we only have three terms each, rather than six, which absorbs the factor of two.

    Applying \Cref{MatrixProp} to each term in the series, we can see that 
    \begin{equation}
    \nabla u(\Vec{0})=
    \lambda \left(\begin{array}{ccc}
         0 & -1 & -1  \\
         -1 & 0 & -1 \\
         -1 & -1 & 0
    \end{array}\right),
    \end{equation}
    where 
    \begin{equation}
       \lambda
        =
        2\pi \sum_{m=0}^{+\infty}
    2\psi_{2m}(\sigma\cdot k^m)(\sigma\cdot v^{k^m})
    +\psi_{2m+1}\left(
    (\sigma\cdot h^m)(\sigma\cdot v^{h^m})
    +(\sigma\cdot j^m)(\sigma\cdot v^{j^m})
    \right).
    \end{equation}
    Recall from \Cref{CanonicalComputations} that
    For all $m\in\mathbb{Z}^+,$
    \begin{align}
        \sigma \cdot k^m
        &= 3*2^{2m} \\
        \sigma \cdot h^m
        &= 3*2^{2m+1} \\
        \sigma \cdot j^m
        &= 3*2^{2m+1},
    \end{align}
    and from \Cref{BasisComputations} that
    For all $m\in\mathbb{Z}^+$,
    \begin{align}
    \sigma\cdot v^{k^m}
    &=
    \frac{2*3^{2m+1}}{\left(4*3^{4m+1}
    +2^{4m+1}*3^{2m+2}
    \right)^\frac{1}{2}} 
    \\
    \sigma\cdot v^{h^m}
    &=
    \frac{2*3^{2m+2}}{\left(4*3^{4m+3}
    +2^{4m+3}*3^{2m+3}
    \right)^\frac{1}{2}} 
    \\
    \sigma\cdot v^{j^m}
    &=
    \frac{2*3^{2m+2}}{\left(4*3^{4m+3}
    +2^{4m+3}*3^{2m+3}
    \right)^\frac{1}{2}}.
    \end{align}

    Then we can conclude that
    \begin{multline}
    \lambda
    =
    4\pi \sum_{m=0}^{+\infty}
    \psi_{2m}\left(\frac{2^{2m+1}*3^{2m+2}}
    {\left(4*3^{4m+1}
    +2^{4m+1}*3^{2m+2}
    \right)^\frac{1}{2}}\right) \\
    +\psi_{2m+1}\left(\frac{2^{2m+2}*3^{2m+3}}{\left(4*3^{4m+3}
    +2^{4m+3}*3^{2m+3}
    \right)^\frac{1}{2}}\right).
    \end{multline}
    For $n$ even, let $n=2m$; for $n$ odd, let $n=2m+1$. Making this substitution, 
    we can see that
    \begin{align}
    \lambda
    &=
    4\pi \sum_{n=0}^{+\infty}
    \psi_n
    \left(\frac{2^{n+1}*3^{n+2}}
    {\left(4*3^{2n+1}
    +2^{2n+1}*3^{n+2}
    \right)^\frac{1}{2}}\right)\\
    &=
    4\pi \sum_{n=0}^{+\infty}
    \psi_n 
    \frac{3\sqrt{2}\left(\sqrt{3}\right)^n}
    {\left(1+\frac{2}{3}
    \left(\frac{3}{4}\right)^n
    \right)^\frac{1}{2}},
    \end{align}
    and this completes the proof.
\end{proof}

\begin{theorem} \label{PsiSupRegCrit}
    Suppose $u\in C\left([0,T_{max});
    \dot{H}^\frac{\log(3)}{2\log(2)}_{\mathcal{M}}
    \right)$,
is an odd, permutation symmetric solution of the Fourier-restricted Euler or hypodissipative Navier--Stokes equation with hj-parity.
Then for all $0\leq t<T_{max}$,
\begin{equation} \label{PsiSupBound}
    \|\psi(t)\|_{\mathcal{H}^1}^2
    \leq 
    \left\|\psi^0\right\|_{\mathcal{H}^1}^2
    \exp\left(4\sqrt{2}\pi \int_0^t
    \sup_{n\in\mathbb{Z}^+}
    \left(\sqrt{3}\right)^n \psi_n(\tau)
    \diff\tau\right).
\end{equation}
In particular, if $T_{max}<+\infty$, then
\begin{equation}
    \int_0^{T_{max}}
    \sup_{n\in\mathbb{Z}^+}
    \left(\sqrt{3}\right)^n \psi_n(t)
    \diff t
    =
    +\infty.
\end{equation}
\end{theorem}

\begin{proof}
    We can see from \Cref{DyadicHilbertBlowupRateCor},
    that if $T_{max}<+\infty$, then
    \begin{equation}
    \lim_{t\to T_{max}}
    \|\psi(t)\|_{\mathcal{H}^1}^2
    =
    +\infty,
    \end{equation}
    so it suffices to prove the bound \eqref{PsiSupBound}.
    Using the system of ODEs for $\psi$,
    we begin by computing that for all $0<t<T_{max}$,
    \begin{multline}
    \frac{\diff}{\diff t}
    \|\psi(t)\|_{\mathcal{H}^1}^2
    =
    -2(12\pi^2)^\alpha\nu 
    \sum_{n=0}^{+\infty} \mu_n^\alpha
\left(\sqrt{3}\right)^{2(1+\Tilde{\alpha})n}
 \psi_n^2 +2\sqrt{2}\pi \sum_{n=1}^{+\infty}
	\beta_{n-1} \left(\sqrt{3}\right)^{3n}
        \psi_{n-1}^2 \psi_n \\
	-2\sqrt{2}\pi \sum_{n=0}^{+\infty}
	\beta_{n} \left(\sqrt{3}\right)^{3n+1}\psi_{n}^2 \psi_{n+1},
	\end{multline}
 where $\nu=0$ for the inviscid case.
 
    We will not make use of the dissipation; rather, in the viscous case, will simply drop this negative term. Our estimates will be on the nonlinear term,
    which we will denote by
    \begin{equation}
    NLT=
    2\sqrt{2}\pi \sum_{n=1}^{+\infty}
	\beta_{n-1} \left(\sqrt{3}\right)^{3n}
    \psi_{n-1}^2 \psi_n
	-2\sqrt{2}\pi \sum_{n=0}^{+\infty}
	\beta_{n} \left(\sqrt{3}\right)^{3n+1}
 \psi_{n}^2 \psi_{n+1}.
    \end{equation}
	Changing the index $n \to n+1$ in the first sum,
    observe that
	\begin{align}
	NLT
	&=
	2\sqrt{2}\pi \sum_{n=0}^{+\infty}
	\beta_n \left( \left(\sqrt{3}
 \right)^{3n+3}
 -\left(\sqrt{3}\right)^{3n+1}\right) 
 \psi_{n}^2 \psi_{n+1} \\
	&=
	2\sqrt{2}\pi \sum_{n=0}^{+\infty}
	\beta_n (3-1) \left(\sqrt{3}\right)^{2n} 
 \psi_n^2 
 \left(\sqrt{3}\right)^{n+1}
 \psi_{n+1}  \\
 &=
 4\sqrt{2}\pi 
 \sum_{n=0}^{+\infty}
	\beta_n \left(\sqrt{3}\right)^{2n} 
 \psi_n^2 
 \left(\sqrt{3}\right)^{n+1}
 \psi_{n+1}  \\
	&\leq 
	4\sqrt{2}\pi
 \left(\sup_{n\in\mathbb{Z}^+}
 \left(\sqrt{3}\right)^{(n+1)} \psi_{n+1} \right)
	\sum_{n=0}^{+\infty} \beta_n
 \left(\sqrt{3}\right)^{2n} \psi_n^2 \\
	&\leq  
	4\sqrt{2}\pi
    \left(\sup_{n\in\mathbb{Z}^+}
 \left(\sqrt{3}\right)^{n} \psi_n \right)
    \|\psi\|_{\mathcal{H}^1}^2.
	\end{align}
Therefore, we can see that for all $0<t<T_{max}$,
\begin{equation}
    \frac{\diff}{\diff t}
    \|\psi(t)\|_{\mathcal{H}^1}^2
    \leq 
    4\sqrt{2}\pi
    \left(\sup_{n\in\mathbb{Z}^+}
 \left(\sqrt{3}\right)^{n} \psi_n \right)
    \|\psi\|_{\mathcal{H}^1}^2,
\end{equation}
and, applying Gr\"onwall's inequality, this completes the proof.
\end{proof}

\begin{theorem} \label{StrainOriginBlowupThm}
    Suppose $u\in C\left([0,T_{max});
    \dot{H}^s_{\mathcal{M}}\right), s>\frac{5}{2}$
is an odd, permutation symmetric, hj-parity solution of the Fourier-restricted hypodissipative Navier--Stokes equation or Fourier-restricted Euler equation.
Further suppose that for all $n\in\mathbb{Z}^+$, 
we have $\psi_n(0)\geq 0$.
Then for all $0\leq t<T_{max}$,
\begin{equation}
    \nabla u(\Vec{0},t)
    =
    \lambda(t)
    \left(\begin{array}{ccc}
         0 & -1 & -1  \\
         -1 & 0 & -1 \\
         -1 & -1 & 0
    \end{array}\right),
\end{equation}
where
\begin{equation}
    \lambda(t)=
    12\sqrt{2}\pi \sum_{n=0}^{+\infty}
    \psi_n(t)
    \frac{\left(\sqrt{3}\right)^n}
    {\left(1+\frac{2}{3}
    \left(\frac{3}{4}\right)^n
    \right)^\frac{1}{2}} 
    \geq 0
    \end{equation}

    Furthermore, for all $0\leq t<T_{max}$,
\begin{equation}
    \|\psi(t)\|_{\mathcal{H}^1}^2
    \leq 
    \left\|\psi^0\right\|_{\mathcal{H}^1}^2
    \exp\left(\frac{\sqrt{5}}{3\sqrt{3}}
    \int_0^t \lambda(\tau) 
    \diff\tau\right).
    \end{equation}
    In particular, if $T_{max}<+\infty$,
    then
    \begin{equation}
    \int_0^{T_{max}}\lambda(t) \diff t
    =
    +\infty.
    \end{equation}
    \end{theorem}

\begin{proof}
    Observe that we have already shown that the non-negativity of the $\psi_n(t)$ is preserved by the dynamics, so clearly for all $0\leq t<T_{max}$ and for all $n\in\mathbb{Z}^+$, we have $\psi_n(t)\geq 0$.
    Applying \Cref{OriginStrainFourierProp},
    we can then see that for all $0\leq t<T_{max}$,
    \begin{equation}
    \nabla u(\Vec{0},t)
    =
    \lambda(t)
    \left(\begin{array}{ccc}
         0 & -1 & -1  \\
         -1 & 0 & -1 \\
         -1 & -1 & 0
    \end{array}\right),
\end{equation}
where
\begin{equation}
    \lambda(t)=
    12\sqrt{2}\pi \sum_{n=0}^{+\infty}
    \psi_n(t)
    \frac{\left(\sqrt{3}\right)^n}
    {\left(1+\frac{2}{3}
    \left(\frac{3}{4}\right)^n
    \right)^\frac{1}{2}} 
    \geq 0.
    \end{equation}
    
    It remains to prove the regularity criterion.
    Observe that for all $n\in\mathbb{Z}^+$,
    \begin{equation}
    \frac{1}{\left(1+\frac{2}{3}
    \left(\frac{3}{4}\right)^n
    \right)^\frac{1}{2}} 
    \geq
    \frac{\sqrt{3}}{\sqrt{5}}.
    \end{equation}
    Therefore, we can compute that 
    \begin{align}
    \lambda(t)
    &\geq
    \frac{12\sqrt{6}\pi}{\sqrt{5}}
    \sum_{n=0}^{+\infty}
    \left(\sqrt{3}\right)^n\psi_n(t) \\
    &\geq 
    \frac{12\sqrt{6}\pi}{\sqrt{5}}
    \left(\sup_{n\in\mathbb{Z}^+}
    \left(\sqrt{3}\right)^n\psi_n(t)\right),
    \end{align}
    using the fact that for all $n\in\mathbb{Z}^+$, and for all $0\leq t<T_{max}$, we have
    $\psi_n(t)\geq 0$, so there is no cancellation, and the sum must be larger than each of its terms.
   The result then follows from \Cref{PsiSupRegCrit}.
\end{proof}

\begin{remark}
    We can clearly see that $\nabla u(\Vec{0},t)$ is symmetric, and so
    \begin{align}
    S(\Vec{0},t) &= \lambda(t)
    \left(\begin{array}{ccc}
         0 & -1 & -1  \\
         -1 & 0 & -1 \\
         -1 & -1 & 0
    \end{array}\right) \\
    \omega(\Vec{0},t) &= 0.
    \end{align}
    We have just seen that this implies that $S(\Vec{0},t)$ has eigenvalues $-2\lambda(t),\lambda(t),\lambda(t)$, with axial compression along the $\sigma$-axis, and planar stretching in the plane perpendicular to $\sigma$,
    and so the blowup of the strain at the origin has exactly the most singular structure in terms of the regularity criterion involving $\lambda_2^+ \in L^1_t L^\infty_x$. See \cites{NeustupaPenel1,MillerStrain} for further discussion of geometric considerations.
\end{remark}

\subsection{Mirror symmetry} \label{MirrorSubsection}

In this section, we will show the equivalence of $\sigma$-mirror symmetry and hj-parity for odd, permutation symmetric vector fields in the constraint space $\dot{H}^s_\mathcal{M}$.
Our finite-time blowup results in the introduction were stated in terms of odd, permutation, symmetric, and $\sigma$-mirror symmetric solutions, but in the \Cref{DynamicsSection,BlowupSection} were proven in terms of odd, permutation-symmetric, hj-parity solutions, so it is necessary to show that the conditions are equivalent. Recall the definition of $\sigma$-mirror symmetry from the introduction.

\begin{definition}
    Let $M_\sigma=
    I_3-\frac{2}{3}\sigma\otimes\sigma$.
    We will say that a vector field $u\in 
    H^s\left(\mathbb{T}^3;\mathbb{R}^3\right), 
    s\geq 0$ has $\sigma$-mirror symmetry if
    \begin{equation}
    u^{M_\sigma}=u.
    \end{equation}
    Note that $M_\sigma=M_\sigma^{-1}$, so $u$ has $\sigma$-mirror symmetry if and only if for all $x\in \mathbb{T}^3$,
    \begin{equation}
    u(x)=M_\sigma u(M_\sigma x).
    \end{equation}
\end{definition}

\begin{remark}
    Note that $M_\sigma$ does not preserve the set $\left[-\frac{1}{2},\frac{1}{2}\right]^3$.
    The transform can still be defined, because $u$ is defined on all of $\mathbb{R}^3$ by periodicity, but in general $u^{M_\sigma}$ will not be periodic in $x_1,x_2,x_3$. Only for specific vector fields will $u^{M_\sigma}$ still be a function on the torus. Of course, if $u$ has $\sigma$-mirror symmetry, than clearly $u^{M_\sigma}$ will also be a function on the torus.

    Note that in general 
    \begin{equation}
    u^{M_\sigma}(x)
    =
    \sum_{k\in\mathbb{Z}^3}
    M_\sigma \hat{u}(k) 
    e^{2\pi i (M_\sigma k)\cdot x},
    \end{equation}
    and that
    \begin{equation}
    M_\sigma k= k-\frac{2}{3}(\sigma\cdot k)\sigma.
    \end{equation}
    Therefore, we can see that $u^{M\sigma}\in H^s\left(\mathbb{T}^3;\mathbb{R}^3\right)$---that is to say, is a function of the torus---if and only if $\frac{\sigma\cdot k}{3}\in\mathbb{Z}$
    for all $k\in\supp\left(\hat{u}\right)$.
    It is simple to observe that this is the case for all $u\in H^s_\mathcal{M}$.
\end{remark}

\begin{lemma}
    Suppose $u\in H^s\left(\mathbb{T}^3;\mathbb{R}^3\right)$ is odd.
    Then $u$ has $\sigma$-mirror symmetry if and only if
    \begin{equation}
    u^{-M_\sigma}=u
    \end{equation}
\end{lemma}

\begin{proof}
    This comes from composing the odd symmetry with the $\sigma$-mirror symmetry.
    Suppose $u$ is odd and $\sigma$-mirror symmetric. Then for all $x\in \mathbb{T}^3$,
    \begin{equation}
    u(x)
    =-u(-x)
    =-M_\sigma u (-M_\sigma x),
    \end{equation}
    and so $u^{-M_\sigma}=u.$
    Likewise suppose that $u$ is odd and 
    $u=u^{-M_\sigma}$.
    Then
    \begin{equation}
    u(x)=-M_\sigma u(-M_\sigma x)
    =M_\sigma u(M_\sigma x),
    \end{equation}
    and so $u=u^{M_\sigma}$.
\end{proof}

\begin{theorem} \label{MirrorHJequivThm}
    Suppose $u\in \dot{H}^s_\mathcal{M}, s\geq 0$ is odd and permutation symmetric. Then $u$ has $\sigma$-mirror symmetry if and only if $u$ has hj-parity.
\end{theorem}

\begin{proof}
    Recall from \Cref{PermutationSymmetricCategorizationProp} that
    \begin{multline}
    u(x)=-2\sum_{m=0}^{+\infty}\bigg(
    \phi_m\sum_{k\in \mathcal{P}[k^m]} v^k
    \sin(2\pi k\cdot x)
    +\eta_m\sum_{h\in \mathcal{P}[h^m]} v^h
    \sin(2\pi h\cdot x) \\
    +\zeta_m\sum_{j\in \mathcal{P}[j^m]} v^j
    \sin(2\pi j\cdot x)
    \bigg),
    \end{multline}
    which can also be expressed as
    \begin{multline} 
    u(x)
    =
    -\sum_{m=0}^{+\infty}\sum_{P\in\mathcal{P}_3}
    \bigg( 2\phi_m v^{P(k^m)}
    \sin(2\pi P(k^m)\cdot x) \\
    +\eta_m v^{P(h^m)}
    \sin(2\pi P(h^m)\cdot x)
    +\zeta_m v^{P(j^m)}
    \sin(2\pi P(j^m)\cdot x)\bigg).
    \end{multline}
    Note that the factor of two disappears in the later two terms because the repeated entry means each term in $\mathcal{P}[h^m]$ and $\mathcal{P}[j^m]$ is double counted.
    Recalling that $P(v^k)=v^{P(k)}$, we find that
    \begin{multline} \label{MirrorSymEqn1}
    u(x)
    =
    -\sum_{m=0}^{+\infty}\sum_{P\in\mathcal{P}_3}
    \bigg( 2\phi_m P\left(v^{k^m}\right)
    \sin(2\pi P(k^m)\cdot x) \\
    +\eta_m P\left(v^{h^m}\right)
    \sin(2\pi P(h^m)\cdot x)
    +\zeta_m P\left(v^{j^m}\right)
    \sin(2\pi P(j^m)\cdot x)\bigg).
    \end{multline}
    Next compute that
    \begin{multline}
    u^{-M_\sigma}(x)=
    -\sum_{m=0}^{+\infty}\sum_{P\in\mathcal{P}_3}
    \bigg( 2\phi_m (-M_\sigma P v^{k^m})
    \sin(2\pi (-M_\sigma P k^m)\cdot x) \\
    +\eta_m (-M_\sigma P v^{h^m})
    \sin(2\pi (-M_\sigma P h^m)\cdot x) \\
    +\zeta_m (-M_\sigma P v^{j^m})
    \sin(2\pi (-M_\sigma P j^m)\cdot x)\bigg).
    \end{multline}
    Next we observe that $-M_\sigma$ commutes with all permutations: for all $P\in\mathcal{P}_3$,
    \begin{equation}
    -M_\sigma P=-P M_\sigma,
    \end{equation}
    and that therefore
    \begin{multline} \label{MirrorSymSetp}
    u^{-M_\sigma}(x)=
    -\sum_{m=0}^{+\infty}\sum_{P\in\mathcal{P}_3}
    \bigg( 2\phi_m (-PM_\sigma  v^{k^m})
    \sin(2\pi (-PM_\sigma k^m)\cdot x) \\
    +\eta_m (-PM_\sigma  v^{h^m})
    \sin(2\pi (-PM_\sigma  h^m)\cdot x) \\
    +\zeta_m (-PM_\sigma  v^{j^m})
    \sin(2\pi (-PM_\sigma  j^m)\cdot x)\bigg).
    \end{multline}

    Now we compute how $-M_\sigma$ acts on the various frequencies, concluding that
    \begin{align}
    -M_\sigma k^m
    &=
    \left(\frac{2}{3}\sigma\otimes\sigma -I_3\right)
    \left(2^{2m}\sigma +3^m
    \left(\begin{array}{c}
         1 \\ 0 \\ -1 
    \end{array}\right)\right) \\
    &=
    2^{2m}\sigma -3^m
    \left(\begin{array}{c}
         1 \\ 0 \\ -1 
    \end{array}\right) \\
    &=
    P_{13}(k^m),
    \end{align}
    and
    \begin{align}
    -M_\sigma h^m
    &=
    \left(\frac{2}{3}\sigma\otimes\sigma -I_3\right)
    \left(2^{2m+1}\sigma +3^m
    \left(\begin{array}{c}
         1 \\ 1 \\ -2 
    \end{array}\right)\right) \\
    &=
    2^{2m+1}\sigma -3^m
    \left(\begin{array}{c}
         1 \\ 1 \\ -2 
    \end{array}\right) \\
    &=
    P_{13}(j^m),
    \end{align}
    and
    \begin{align}
    -M_\sigma j^m
    &=
    \left(\frac{2}{3}\sigma\otimes\sigma -I_3\right)
    \left(2^{2m+1}\sigma +3^m
    \left(\begin{array}{c}
         2 \\ -1 \\ -1 
    \end{array}\right)\right) \\
    &=
    2^{2m+1}\sigma -3^m
    \left(\begin{array}{c}
         2 \\ -1 \\ -1 
    \end{array}\right) \\
    &=
    P_{13}(h^m).
    \end{align}
    We also compute that for any $k\in\mathbb{Z}^3$,
    \begin{align}
    -M_\sigma P_k^\perp(\sigma)
    &=
    -M_\sigma\left(\sigma -\frac{\sigma\cdot k}{|k|^2}k\right) \\
    &=
    \sigma -\frac{\sigma\cdot k}{|k|^2}
    (-M_\sigma k)\\
    &=
    \sigma -\frac{\sigma\cdot (-M_\sigma k)}
    {|-M_\sigma k|^2}
    (-M_\sigma k) \\
    &=
    P_{-M_\sigma k}^\perp (\sigma).
    \end{align}
    This in turn implies that
    \begin{align}
    -M_\sigma v^{k^m} &=
    P_{13}\left(v^{k^m}\right) \\
    -M_\sigma v^{h^m} &=
    P_{13}\left(v^{j^m}\right) \\
    -M_\sigma v^{j^m} &=
    P_{13}\left(v^{h^m}\right).
    \end{align}

    Plugging these identities back into \eqref{MirrorSymSetp}, we find that
    \begin{multline}
    u^{-M_\sigma}(x)
    =
    -\sum_{m=0}^{+\infty}\sum_{P\in\mathcal{P}_3}
    \bigg( 2\phi_m PP_{13}\left(v^{k^m}\right)
    \sin(2\pi P(P_{13} k^m)\cdot x) \\
    +\zeta_m P P_{13}\left(v^{h^m}\right)
    \sin(2\pi P(P_{13}h^m)\cdot x)
    +\eta_m PP_{13}(v^{j^m})
    \sin(2\pi P(P_{13}j^m)\cdot x)\bigg).
    \end{multline}
    Note that the image of the group of permutations under composition with any fixed permutation, in this case $P_{13}$, is precisely the whole group of permutations $\mathcal{P}_3$ and so we can conclude
    \begin{multline} \label{MirrorSymEqn2}
    u^{-M_\sigma}(x)
    =
    -\sum_{m=0}^{+\infty}\sum_{P\in\mathcal{P}_3}
    \bigg( 2\phi_m P\left(v^{k^m}\right)
    \sin(2\pi P(k^m)\cdot x) \\
    +\zeta_m P \left(v^{h^m}\right)
    \sin(2\pi P(h^m)\cdot x)
    +\eta_m P(v^{j^m})
    \sin(2\pi P(j^m)\cdot x)\bigg).
    \end{multline}
    Comparing \eqref{MirrorSymEqn1} and \eqref{MirrorSymEqn2}, we can clearly see that $u=u^{-M_\sigma}$ if and only if $\eta_m=\zeta_m$ for all $m\in \mathbb{Z}^+$,
    and this completes the proof.
\end{proof}

\section{Limitations of Fourier space methods} \label{AppendixFourierLimitations}

It can be easily seen that if we take initial data $u^0\in \dot{H}^s_\mathcal{M}$, then this will not be preserved by the dynamics of the full Euler or (hypodissipative) Navier--Stokes equations. This is true in particular because 
$\mathcal{M}+\mathcal{M}\not\subset \mathcal{M}$.
This means the nonlinearity will produce Fourier modes outside of the set $\mathcal{M}$.
In order to restrict the Fourier transform to some subset
$\mathcal{N}\subset \mathbb{Z}^3$,
and to have this property preserved by the dynamics of the Euler or (hypodissipative) Navier--Stokes equation, it is necessary that $\mathcal{N}$ be a subspace. We will show that if we take any permutation symmetric subspace $\mathcal{N}\subset \mathbb{Z}^3$ with
$\mathcal{N} \not\subset
    \spn(\sigma)$ and 
    $\mathcal{N} \not\subset
    \spn(\sigma)^\perp.$
    Then there exists $m\in \mathbb{N}$, such that
    \begin{equation}
    (m\mathbb{Z})^3 \subset \mathcal{N}.
    \end{equation}
This means that taking solutions with Fourier modes supported in any three dimensional, permutation symmetric subspace cannot simplify the dynamics of the full Euler or Navier--Stokes equation, because up to a rescaling we have will have the full complexity of dealing with all of the Fourier modes in $\mathbb{Z}^3$.

\begin{definition}    
    We will define a subspace of $\mathbb{Z}^3$ in the standard way, but restricting to integer scalars. 
    We will say $\mathcal{N}\subset\mathbb{Z}^3$ is a subspace if for all $\lambda\in\mathbb{Z}$ and for all $k,j\in\mathcal{N}$, we have
    $k+j\in\mathcal{N}$ and $\lambda k\in\mathcal{N}$.
    We will additionally say $\mathcal{N}$ is permutation symmetric if for all $P\in\mathcal{P}_3$ and for all $k\in\mathcal{N}$,
    we have $P(k)\in\mathcal{N}$.
\end{definition}

\begin{definition}
    Suppose $\mathcal{N}\subset\mathbb{Z}^3$ is a subspace. Then for all $s\geq 0$, we will define the space $\dot{H}^s_{\mathcal{N}^*}\subset \dot{H}^s_{df}$ as the vector fields $u\in \dot{H}^s_{df}$, such that 
    \begin{equation}
    \supp(\hat{u})\subset \mathcal{N}.
    \end{equation}
\end{definition}

We will show that as long as $\mathcal{N}\subset \mathbb{Z}^3$ is a subspace, vector fields with their Fourier transform supported in $\mathcal{N}$---that is vector fields $u\in \dot{H}^s_{\mathcal{N}^*}$---are preserved by the dynamics of the Euler and Navier--Stokes equations.
First, we will prove a few key lemmas.

\begin{lemma} \label{SubspaceLemma}
    Suppose $\mathcal{N}\subset\mathbb{Z}^3$ is a subspace and
    $u\in \dot{H}^s_{\mathcal{N}^*}$ with $s>\frac{5}{2}$.
    Then $\mathbb{P}_{df}(u\cdot\nabla u)\in \dot{H}^{s-1}_{\mathcal{N}^*}$.
\end{lemma}

\begin{proof}
    First observe that
    \begin{align}
    \left\|\mathbb{P}_{df}(u\cdot\nabla u)
    \right\|_{H^{s-1}}
    &\leq 
    \left\|(u\cdot\nabla) u
    \right\|_{H^{s-1}} \\
    &\leq 
    C\|u\|_{H^{s-1}}\|\nabla u\|_{H^{s-1}} \\
    &\leq 
    C\|u\|_{H^{s-1}}\|u\|_{H^s},
    \end{align}
    so we conclude that $\mathbb{P}_{df}(u\cdot\nabla u)\in \dot{H}^{s-1}$.
    Now it remains only to show that the support of the Fourier transform is contained in $\mathcal{N}$.
    By hypothesis we have
    \begin{equation}
    u(x)=\sum_{k\in \mathcal{N}}
    \hat{u}(k)e^{2\pi i k\cdot x}.
    \end{equation}
    Therefore we can compute that
    \begin{equation}
    (u\cdot\nabla)u
    =
    2\pi i\sum_{j,k\in\mathcal{N}}
    (\hat{u}(j)\cdot k) \hat{u}(k)
    e^{2\pi i (k+j)\cdot x},
    \end{equation}
    and that
    \begin{equation}
    P_{df}(u\cdot\nabla u)
    =
    2\pi i\sum_{j,k\in\mathcal{N}}
    \left(I_3-\frac{(k+j)\otimes(k+j)}{|k+j|^2}\right)
    (\hat{u}(j)\cdot k) \hat{u}(k)
    e^{2\pi i (k+j)\cdot x}.
    \end{equation}
    By hypothesis, $\mathcal{N}$ is subspace and therefore we can conclude that for all $k,j\in\mathcal{N}$, we have $k+j\in\mathcal{N}$,
    and that consequently
    \begin{equation}
        \supp\left(\mathcal{F}P_{df}
        (u\cdot\nabla u)\right)
        \subset \mathcal{N},
    \end{equation}
    which this completes the proof.
\end{proof}

\begin{lemma} \label{MollificationSupportLemma}
Let $J_N$ be the mollifier truncating all frequencies $|k|>N$,
\begin{equation}
J_N(u)(x)= \sum_{|k|\leq N}
\hat{u}(k) e^{2\pi i k\cdot x}.
\end{equation}
    Suppose $\mathcal{N}\subset\mathbb{Z}^3$ is a subspace and $u\in \dot{H}^s_{\mathcal{N}^*}$.
    Then for all $N\in \mathbb{N}$, 
    $J_N(u)\in \dot{H}^\infty_{\mathcal{N}^*}$.
\end{lemma}

\begin{proof}
    It is immediately clear that $J_N(u)\in \dot{H}^\infty$, because the $\widehat{J_N(u)}$ is compactly supported. It is also obvious by construction that
    \begin{equation}
    \supp\left(\widehat{J_N(u)}\right)
    \subset
    \supp\left(\hat{u}\right)
    \subset
    \mathcal{N},
    \end{equation}
    and this completes the proof.
\end{proof}

\begin{theorem} \label{SubspaceEuler}
    Suppose $\mathcal{N}\subset \mathbb{Z}^3$ is a subspace, and suppose $u^0\in \dot{H}^s_{\mathcal{N}^*}$ with $s>\frac{5}{2}$,
    Then the strong solution of the Euler equation
    $u\in C\left([0,T_{max}),\dot{H}^s_{df}\right) 
    \cap 
    C^1 \left([0,T_{max}); \dot{H}^{s-1}_{df}\right)$
    satisfies
    $u\in C\left([0,T_{max}),
    \dot{H}^s_{\mathcal{N}^*}\right)$.
\end{theorem}

\begin{proof}
    Fix a frequency cutoff $N\in \mathbb{N}$. Recall that the solution of the mollified problem for the Euler equation is constructed in \cite{KatoPonce} by a fixed point method using Picard iteration.
    Solutions of the mollified problem are fixed points of the map $Q_N:C_T H^s_x \to C_T H^s_x$
    \begin{equation}
    Q_N[u](\cdot,t)=u^0 -\int_0^t 
    J_N\mathbb{P}_{df}((J_N u\cdot\nabla) J_N u)
    (\cdot,\tau) \diff\tau.
    \end{equation}
    Note that \Cref{SubspaceLemma,MollificationSupportLemma} imply that if $u\in C_T \dot{H}^s_{\mathcal{N}^*}$, 
    then $Q_N[u]\in C_T \dot{H}^s_{\mathcal{N}^*}$.
    Solutions of the mollified problem can be constructed by Picard iteration, with
    \begin{equation}
    u^{m+1}=Q_N[u^{N,m}],
    \end{equation}
    and the solution of the mollified problem $u^N$ satisfying
    \begin{equation}
    u^N=\lim_{m\to +\infty}u^{N,m}.
    \end{equation}
    We can therefore clearly see by induction that if $u^0\in \dot{H}^s_\mathcal{M}$, then for all $m\in\mathbb{N}$, we have
    $u^{N,m}\in C_T \dot{H}^s_{\mathcal{N}^*}$.
    Taking the limit $m\to +\infty$,
    we find that 
    $u^N\in C_T \dot{H}^s_{\mathcal{N}^*}$,
    and therefore the theorem holds for solutions of the mollified Euler equation with frequency cutoff $N$,
    \begin{equation}
    \partial_t u+
    J_N \mathbb{P}_{df}(J_N u\cdot\nabla J_N u)
    =0.
    \end{equation}

    The solution of the full Euler equation can be found as the limit of a subsequence of solutions to the mollified problem
    \begin{equation}
    u=\lim_{k\to +\infty}u^{N_k},
    \end{equation}
    where $N_k \to +\infty$.
    Note this convergence is strong in $C_T H^{s'}_x$ for all $s'<s$, and weak in $C_T H^s_x$. This implies that if 
    $\supp\left(\mathcal{F}u^{N_k}\right)
    \subset \mathcal{N}$,
    then $\supp(\mathcal{F}u)\subset \mathcal{N}$.
    Therefore it was sufficient to prove the result holds for the mollified problem, and this completes the proof.
\end{proof}

\begin{lemma} \label{SubspaceHeatLemma}
    Suppose $\mathcal{N}\subset\mathbb{Z}^3$ is a subspace and
    $u\in \dot{H}^1_{\mathcal{N}^*}$.
    Then for all $\tau>0$, we have $e^{\tau\Delta}\mathbb{P}_{df}(u\cdot\nabla u)\in C_T \dot{H}^1_{\mathcal{N}^*}$.
\end{lemma}

\begin{proof}
    Applying H\"older's inequality and the Sobolev inequality, we can see that
    \begin{align}
    \|(u\cdot\nabla)u\|_{L^\frac{3}{2}}
    & \leq 
    \|u\||_{L^6}\|\nabla u\|_{L^2} \\
    &\leq 
    C \|\nabla u\|_{L^2}^2.
    \end{align}
    We have already shown in \Cref{SubspaceLemma}, that 
    \begin{equation}
    \supp(\mathcal{F}(u\cdot\nabla)u)
    \subset \mathcal{N}.
    \end{equation}
    Note that we do not need the higher regularity assumed in that lemma for the result to hold.
    The smoothing from the heat kernel
    then implies that 
    $e^{\tau\Delta}(u\cdot\nabla)u\in L^2$.
    It is also immediately clear that 
    \begin{equation}
    \supp\left(\mathcal{F}\mathbb{P}_{df}
    e^{\tau\Delta}(u\cdot\nabla)u\right)
    \subset \mathcal{N},
    \end{equation}
    because the operators $e^{\tau\Delta}$ and $P_{df}$ both act on each frequency pointwise in Fourier space, and so will not change the support of the Fourier transform.
    This completes the proof.
\end{proof}

\begin{theorem} \label{SubspaceNS}
    Suppose $\mathcal{N}\subset \mathbb{Z}^3$ is a subspace, and suppose $u^0\in \dot{H}^1_{\mathcal{N}^*}$,
    Then the mild solution of the Navier--Stokes equation
    $u\in C\left([0,T_{max}),\dot{H}^1_{df}\right) 
    \cap 
    C^\infty \left((0,T_{max})\times \mathbb{T}^3\right)$
    satisfies
    $u\in C\left([0,T_{max}),
    \dot{H}^1_{\mathcal{N}^*}\right)$.
\end{theorem}

\begin{proof}
Recall that solutions of the Navier--Stokes equation can be constructed by a fixed point method using Picard iteration \cite{KatoFujita}.
    Solutions of the mollified problem are fixed points of the map 
    $Q_{u^0,\nu}:C_T H^1_x \to C_T H^1_x$
    \begin{equation}
    Q_{u^0,\nu}[u](\cdot,t)=e^{\nu\Delta t}u^0 
    -\int_0^t e^{\nu\Delta(t-\tau)}
    \mathbb{P}_{df}(u\cdot\nabla u)
    (\cdot,\tau) \diff\tau.
    \end{equation}
    Note that \Cref{SubspaceLemma,SubspaceHeatLemma} imply that if $u\in C_T \dot{H}^1_{\mathcal{N}^*}$, 
    then $Q_{u^0,\nu}[u]\in C_T \dot{H}^1_{\mathcal{N}^*}$.
    Solutions of the Navier--Stokes equation 
    are constructed from $Q_{u^0,\nu}$ by Picard iteration, with
    \begin{equation}
    u^{n+1}=Q_{u^0,\nu}[u^n],
    \end{equation}
    and the solution of the Navier--Stokes equation
    $u\in C_T H^1_x$ satisfying
    \begin{equation}
    u=\lim_{n\to +\infty}u^n.
    \end{equation}    
    We can see that for all $0\leq t<T_{max}$ and for all $n\in\mathbb{N}$, we have $\supp(\hat{u}^n)(\cdot,t)\subset \mathcal{N}$,
    from the properties of $Q_{u^0,\nu}$,
    and therefore we can conclude using the limit in $C_t H^1_x$ that
    $\supp(\hat{u})(\cdot,t)\subset \mathcal{N}$.
    This completes the proof.
\end{proof}

\begin{remark}
\Cref{SubspaceEuler,SubspaceNS} may raise some hope that by choosing the appropriate permutation symmetric subspace, we may be able to prove finite-time blowup for the Euler equation---or even for the (hypodissipative) Navier--Stokes equation---in a way that reduces the complexity of the full problem, even if the dynamics will still be considerably more complicated than for the Fourier-restricted Euler equation. Such a hope will prove false, as we will see that any permutation symmetric subspace that is three dimensional cannot have any reduction in complexity in terms of the Fourier modes.
\end{remark}

\begin{lemma} \label{SubspaceUnionLemma}
    Suppose $\mathcal{N}\subset \mathbb{Z}^3$ is a subspace such that
    $\mathcal{N} \not\subset
    \spn(\sigma)$ and 
    $\mathcal{N} \not\subset
    \spn(\sigma)^\perp$. Then
    \begin{equation}
    \mathcal{N} \not\subset
    \spn(\sigma)\cup \spn(\sigma)^\perp,
    \end{equation}
    where $\spn(\sigma)^\perp
    :=\left\{k\in\mathbb{Z}^3: 
    \sigma \cdot k=0\right\}$.
\end{lemma}

\begin{proof}
    Suppose toward contradiction that $N\subset \spn(\sigma)\cup \spn(\sigma)^\perp$ satisfies
    $\mathcal{N} \not\subset
    \spn(\sigma)$ and 
    $\mathcal{N} \not\subset
    \spn(\sigma)^\perp$.
    Then there exists $k\in \mathcal{N}\cap \spn(\sigma), k\neq 0$
    and $j\in \left(\mathcal{N}\cap 
    \spn(\sigma)^\perp\right), j\neq 0$.
    Observe that $k+j\in \mathcal{N}$, but that
    $k+j\notin \spn(\sigma)$ and
    $k+j\notin \spn(\sigma)^\perp$.
    This contradicts the assumption that 
    $N\subset \spn(\sigma)\cup \spn(\sigma)^\perp$,
    which completes the proof.
\end{proof}

\begin{theorem}
    Suppose $\mathcal{N}\subset \mathbb{Z}^3$ is a permutation symmetric subspace such that
    $\mathcal{N} \not\subset
    \spn(\sigma)$ and 
    $\mathcal{N} \not\subset
    \spn(\sigma)^\perp.$
    Then there exists $m\in \mathbb{N}$, such that
    \begin{equation}
    (m\mathbb{Z})^3 \subset \mathcal{N}.
    \end{equation}
\end{theorem}

\begin{proof}
    From \Cref{SubspaceUnionLemma},
    we can conclude that
    \begin{equation}
        \mathcal{N} \not\subset
    \spn(\sigma)\cup \spn(\sigma)^\perp,
    \end{equation}
    and therefore, there must exist $k\in \mathcal{N}$ such that
    $k_1\geq k_2 \geq k_3$, $k_1>k_3$, $k_1>0$ and
    $\sigma\cdot k>0$.
    To see this fix $k\in \mathcal{N}\setminus
    \left(\spn(\sigma)\cup 
    \spn(\sigma)^\perp\right)$.
    We can assume without loss of generality that $k_1\geq k_2 \geq k_3$, because of permutation symmetry, and that $\sigma\cdot k>0$ because $k\in \mathcal{N}$ if and only if $-k\in\mathcal{N}$. This also implies that $k_1>0$, as otherwise $\sigma\cdot k<0$. Finally, we know that $k_1>k_3$, because otherwise $k_1=k_2=k_3$, and then $k\in\spn(\sigma)$.
    
    Let $a=2k_1$ and $b=k_2+k_3$.
    Then we can see that
    \begin{equation}
    \left(\begin{array}{c}
         a  \\ b  \\ b 
    \end{array}\right)
    =
    \left(\begin{array}{c}
         k_1  \\ k_2  \\ k_3
    \end{array}\right)
    +
    \left(\begin{array}{c}
         k_1  \\ k_3  \\ k_2 
    \end{array}\right)
    \in \mathcal{N}.
    \end{equation}
    Also observe that $a>0$, $a>b$, and $a+2b=2\sigma\cdot k>0$.
    By permutation symmetry we can see that we also have
    \begin{equation}
    \left(\begin{array}{c}
         b  \\ a  \\ b 
    \end{array}\right),
    \left(\begin{array}{c}
         b  \\ b  \\ a 
    \end{array}\right)
    \in \mathcal{N}.
    \end{equation}
    Therefore we may conclude that
    \begin{equation}
    \left(a^2(a+b)-2ab^2\right)
    \left(\begin{array}{c}
         1  \\ 0  \\ 0 
    \end{array}\right)
    =
    a(a+b)\left(\begin{array}{c}
         a  \\ b  \\ b 
    \end{array}\right)
    -ab\left(\begin{array}{c}
         b  \\ a  \\ b 
    \end{array}\right)
    -ab\left(\begin{array}{c}
         b  \\ b  \\ a 
    \end{array}\right)
    \in \mathcal{N}.
    \end{equation}
    We note that the scalar multiple of the unit vector in the $x$ direction is a positive integer in this case because
    \begin{equation}
    a^3+a^2b-2ab^2=a(a-b)(a+2b)>0.
    \end{equation}
    
    Let $m=a(a-b)(a+2b)$. Then we can clearly see by permutation symmetry that
    \begin{equation}
        m\left(\begin{array}{c}
         1  \\ 0  \\ 0 
    \end{array}\right),
    m\left(\begin{array}{c}
         0  \\ 1  \\ 0 
    \end{array}\right),
    m\left(\begin{array}{c}
         0  \\ 0  \\ 1 
    \end{array}\right)
    \in \mathcal{N}.
    \end{equation}
    It then follows from the definition of a subspace that
    \begin{equation}
    (m\mathbb{Z})^3\subset \mathcal{N},
    \end{equation}
    and this completes the proof.
\end{proof}

\begin{remark}
    We have in general been working with functions on the three dimensional torus $\mathbb{T}^3$, that have period $1$. If we have $\supp(\hat{f}) \subset (m\mathbb{Z})^3$ for some $m\in\mathbb{N}, m\geq 2$,
    then we can say that $f$ is periodic in all three variables with period $\frac{1}{m}$.
    Because changing the period is just a matter of rescaling space, there is no fundamental difference between periodic solutions of the Euler or Navier--Stokes equations with period $1$ or period $\frac{1}{m}$, at least up to a suitable rescaling of velocity and/or viscosity to preserve Reynolds number.
    We know that if $\mathcal{N}\subset \spn(\sigma)$ then the problem is one dimensional, and if 
    $\mathcal{N}\subset \left\{k\in\mathbb{Z}^3: 
    \sigma \cdot k=0\right\}$
    then the problem is two dimensional.
    In either case, both the Euler and Navier--Stokes equations have smooth solutions globally in time.
    Because any three dimensional, permutation symmetric subspace must contain $(m\mathbb{Z})^3$ for some $m\in\mathbb{N}$,
    this implies that we cannot search for finite-time blowup by means of finding a suitable permutation symmetric subspace that reduces the complexity of the dynamics of the Euler equation in Fourier space. Any three dimensional subspace must include the full complexity of the modes in $(m\mathbb{Z})^3$, which in turn includes the full complexity of the modes $\mathbb{Z}^3$, because there is no difference in complexity between the periodic solutions of the three dimensional Euler equation with period $1$ and with period $\frac{1}{m}$.
\end{remark}

\bibliographystyle{plain}
\bibliography{bib}

\end{document}